%% file: proefschrift2.tex
\documentclass{book}
\usepackage{amssymb, amsmath, amsthm}
\usepackage[all]{xy}
\usepackage{cite}
\usepackage{makeidx}\makeindex
\newcommand\xemph[1]{\emph{#1}\index{#1}}
\usepackage{hyperref}
\usepackage[margin = 2cm, paperwidth = 17cm, paperheight = 24cm]{geometry}

\title{Realizability Categories}
\author{W. P. Stekelenburg}
\date{}

\theoremstyle{plain}
\newtheorem{theorem}{Theorem}[section]
\newtheorem*{theorem*}{Theorem}
\newtheorem{lemma}[theorem]{Lemma}
\newtheorem*{lemma*}{Lemma}
\newtheorem{corollary}[theorem]{Corollary}
\newtheorem{proposition}[theorem]{Proposition}

\theoremstyle{definition}
\newtheorem{definition}[theorem]{Definition}
\newtheorem*{definition*}{Definition}
\newtheorem{example}[theorem]{Example}
\newtheorem{remark}[theorem]{Remark}

\newcommand\comment[1]{}

\begin{document} 

\newcommand\cat{\mathcal}
\newcommand\Cat{\mathsf}
\newcommand\id{\mathrm{id}}
\newcommand\dom{\mathrm{dom}}
\newcommand\cod{\mathrm{cod}}
\newcommand\Set{\Cat{Set}}

\comment{
- Moerdijks vermoeden: Johnstone schijnt dit al opgelost te hebben met tripostheorie. Hoeveel verder kom je daarmee?

- controleer: nodeloos "-ed by" ", that", "form"/"from", "atgeor", "this way"/ "in this way"

- in Longley: probeer een bewijs te vinden voor je relative completion.
Ik zie nog steeds niet of we met surjectieve functies moeten werken of niet. Dit idee van een categorie van realizers spreekt me aan, en het is een voor de hand liggende generalizatie van gepartitoneerde assemblies. Alleen die stomme partiele functies overal verpesten het.

- filters op \Kone: zie Baten en Boerboom 1979. Er bestaan hierom mogelijk oneindig veel niet isomorfe subPCAs van \Kone. We weten niet of alle recursieve functies partieel combinatorisch zijn, omdat dat laatste begrip afhangt van de operator. Ik denk dat het idee van een zwakke PCA hierbij een rol speelt: daarmee kun je bepaalde applicatie gedefinieerd laten zijn die niet 

Wellicht iets met de beslisbaarheid van gelijkheid die combinatorische algebras niet hebben.

- filters op graafmodellen: goed op nog eens op terug te komen?
Ik ben niet zo onder de indruk van mijn resultaten hier.
Rosolini had bezwaar tegen mijn presentatie, maar ik begrijp nog niet waarom.

- soundness classical realizability
Krivine behoort adequacy, maar Streicher bedacht de tripos.
}

\frontmatter
\maketitle
\newpage
\thispagestyle{empty}
\noindent Assessment committee:\\
Prof. dr. Martin Hyland\\
Prof. dr. Bart Jacobs\\
Prof. dr. Giuseppe Rosolini\\
Prof. dr. Thomas Streicher\\
Dr. Benno van den Berg

\vspace{\stretch{2}}

\noindent Cover design by Wouter Stekelenburg. The following remark inspired it:
\begin{quote}Computable and constructive mathematics is like the Matrix. Do you remember the film \emph{Matrix}? Computable
mathematics is the Matrix, a world built by intelligent computers to keep people from seeing the
world as it really is. The realizers are the little green characters that keep falling down the screens.
They seem infinitely more boring and less comprehensible than the Matrix. However, when Neo,
the hero who saves humankind, reaches a higher level of awareness he sees the Matrix as it really
is -- made of little green characters. Intuitionistic logic and category theory are the Architect and
the Oracle, but I am not telling which is which.\end{quote}
Andre Bauer (2005) \emph{Realizability as the Connection between
Computable and Constructive Mathematics} page 33

\vspace{\stretch{1}}

\noindent Printed by: Proefschriftmaken.nl $\Vert$ Uitgeverij BOXPress

\vspace{\stretch{1}}

\noindent\copyright{} Wouter Pieter Stekelenburg 2013\\
ISBN: 978-90-393-5896-2

\newpage
\thispagestyle{empty}
\mbox{}\\
\vspace{\stretch{15}}

\begin{center}
{\Huge{\bf Realizability Categories}
\vspace{\stretch{2}}

Realiseerbaarheidscategorie\"en}
\vspace{\stretch{2}}

(met een samenvatting in het Nederlands)
\end{center}
\vspace{\stretch{15}}

\noindent{\LARGE Proefschrift}
\vspace{\stretch{1}}

\noindent ter verkrijging van de graad van doctor aan de Universiteit Utrecht op gezag van de rector magnificus, prof. dr. G.J. van der Zwaan, ingevolge het besluit van het college voor promoties in het openbaar te verdedigen op maandag 14 januari 2013 des middags te 4.15 uur\vspace{\stretch{1}}

\noindent door\vspace{\stretch{1.5}}

\noindent {\LARGE Wouter Pieter Stekelenburg}\vspace{\stretch{1}}

\noindent geboren op 9 juli 1984 te Huizen 

\newpage
\thispagestyle{empty}
\noindent Promotor: Prof. dr. I. Moerdijk\smallskip

\noindent Co-promotor: Dr. J. van Oosten

\newpage
\mainmatter
\tableofcontents
\chapter*{Introduction}
\addcontentsline{toc}{chapter}{Introduction}
\input{Introduction}

\chapter{Axiomatization}\label{axiom}
\input{Axiomatization}

\chapter{Realizability Categories}\label{realcat}
\input{RealizabilityCategories}

\chapter{Applications}\label{apps}
\input{Applications}

\chapter{Conclusion}
\input{Appendix}

\backmatter
\input{samenvatting3}

\bibliography{realizability}{}
\bibliographystyle{plain}

\printindex
\end{document}

%% file: Introduction.tex
This thesis contains a collection of results of my Ph.D. research in the area of realizability and category theory. My research was an exploration of the intersection of these areas focused on gaining a deeper understanding rather than on answering a specific question. This gave us some theorems that help to define what realizability is, or at least what \xemph{realizability categories} are.

To provide some context, this chapter introduces realizability and category theory and makes a small survey of their intersection. In the end it summarizes our contributions.\comment{terugkomen en herschrijven wordt toch noodzakelijk}

\section*{Realizability}
Realizability is a collection of tools in the study of constructive logic, where it tackles questions about consistency and independence that are not easily answered by other means. We have no overview of this ever growing collection and know no general criterion for what can be considered realizability and what can not. Therefore, instead of giving a definition, we will present the historical starting point of realizability, and a selection of some later developments.

In \cite{MR0015346} Kleene introduces \xemph{recursive realizability}. It interprets arithmetical propositions by assigning sets of numbers to them.
\newcommand\N{\mathbb N}
\newcommand\partar\rightharpoonup
\newcommand\converges{\mathord\downarrow}
\newcommand\rlzs{\mathrel{\mathbf{r.}}}
\begin{definition*} Let $\N$ be the natural numbers. Let $(m,n)\mapsto \langle m,n\rangle:\N\times\N \to \N$ and $n \mapsto (n_0,n_1):\N\to \N \times \N$ be a recursive bijection and let $(m,n)\mapsto mn:\N\times \N \partar \N$ be a \xemph{universal partial recursive function}, i.e., for each partial recursive $f:\N\partar \N$ there is an $e\in \N$ such that for all $n\in\dom f$, $en$ is defined and equal to $f(n)$. We write $mn\converges$ if $(m,n)$ is in the domain of the universal partial recursive function. We define the \xemph{realizability relation} $\rlzs$ as follows.
\begin{itemize}
\item $n\rlzs x=y$ if and only if $x=y$;
\item $n\rlzs p\land q$ if $n_0\rlzs p$ and $n_1\rlzs q$;
\item $n\rlzs p\vee q$ if $n_0 = 0$ and $n_1\rlzs p$, or $n_0=1$ and $n_1\rlzs q$;
\item $n\rlzs p\to q$ if for all $n'\rlzs p$, $nn'\converges$ and $nn'\rlzs q$;
\item $n\rlzs \neg p$ if no $n'\rlzs p$.
\item $n\rlzs \forall x.p(x)$ if for all $n'\in \N$, $nn'\converges$ and $nn'\rlzs p(n')$;
\item $n\rlzs \exists x.p(x)$ if some $n'\in \N$, $n_0=n'$ and $n_1\rlzs p(n')$.
\end{itemize}
A proposition $p$ is \xemph{valid} if there is some $n\in \N$ such that $n\rlzs p$.
\end{definition*}

The realizers encode some justification for the validity of the formulas they realize. In particular, realizers of $p\to q$ are indices of partial recursive functions that send realizers of $q$ to realizers of $p$. The resulting structure has the following features:
\begin{itemize}
\item it is a model of Heyting arithmetic;
\item because every proposition $p$ either has a realizer or doesn't, $p\vee \neg p$ and $(\neg\neg p) \to p$ are valid;
\item nonetheless, there is a predicate $p$ such that $\neg(\forall n.p(n)\vee\neg p(n))$ is realized.
\end{itemize}
We see the paradox that $q(n) = p(n)\vee\neg p(n)$ is valid for all $n$, while $\forall x.q(x)$ can be false, thanks to an interpretation of universal quantification quite different from the one in classical model theory.

Kleene proposed a number of variations on recursive realizability.
\begin{itemize}
\item We can consider whether the existence of realizers is formally provable in Heyting arithmetic or in other formal systems.
\item We can restrict the set of realized negations, implications, or universal quantifications to a preselected set to avoid realizing false propositions like the undecidability of a set of numbers. This restriction allows a more faithful approach to intuitionistic logic.
\item Kleene developed \xemph{function realizability}, where functions $f:\N\to\N$ take the place of numbers. There is a \xemph{universal partial continuous function} $\N^\N\times \N^\N \partar \N^\N$ for the product topology in $\N^\N$, which takes the place of the universal partial recursive function.
\item A further variation on function realizability is that a formula is valid if there is a total recursive function that realizes it \cite{MR0176922}. This idea of using a special set of realizers to determine validity is called \xemph{relative realizability}.
\end{itemize}

Others proposed further extensions.
\begin{itemize}
\item Besides $\N$ and $\N^\N$ other sets are suitable for building realizability interpretations, namely Feferman's \xemph{partial applicative structures} (see \cite{MR0409137}) and their generalizations.
\item Instead of a single set of realizers, one can work with a system of sets of realizers. The first example of this was Troelstra's reformulation of Kreisel's \xemph{modified realizability} \cite{MR0106838, MR0325352}. \comment{Kreisels versie was niet eerst? was het geen realizeerbaarheidsinterpretaie?}
\item Troelstra extended realizability beyond arithmetic, to higher order systems \cite{MR0363826}.
\item Realizability can be combined with \xemph{sheaf semantics} by developing it in the internal language of a Grothendieck topos \cite{MR1139395, MR495115, MR2638259}.
\end{itemize}

An area of application of realizability is computer science, after all, computers are inherently recursive. Practical limitations of computers, in particular the amount of time and memory required to finish a computation, gave us realizability interpretations for languages that are different from first order languages and realizability counter-models for weaker formal systems than classical or intuitionistic first order logic, see \cite{MR2814748}. On the other hand, the desire to extract computational information from proofs in classical mathematics has led Krivine to introduce a realizability interpretation for classical set theory, see \cite{MR2825677}.

\section*{Effective topos}
We combine realizability with category theory. For an introduction to category theory, see \cite{MR1712872}. Category theory started as a part of algebraic topology, as a language for describing the connections between algebraic invariants of topological spaces, see \cite{MR0013131}.  The theory proved useful in other areas of mathematics, in particular in other parts of algebra and geometry, but also in the more remote areas. Lawvere initiated the application of category theory to logic \cite{MR0158921, MR0172807}. 

Several subjects from category theory, in particular from categorical logic, play a prominent role in this thesis: elementary \xemph{toposes}, \xemph{regular}, \xemph{exact} and \xemph{Heyting categories}, \xemph{fibred locales}, \xemph{complete fibred Heyting algebras} and \xemph{triposes}.

Toposes are categories that have finite limits and \xemph{power objects}: an object $PX$ is a power object of $X$, if there is a monomorphism $m:E_X \to X\times P_X$ such that for each monomorphism $n:U\to X\times Y$ there is a unique $f:Y\to PX$ such that $n$ is the pullback of $m$ along $g$.
\[\xymatrix{
U\ar[d]_n \ar[r]\ar@{}[dr]|<\lrcorner & E_X\ar[d]^m \\
X\times Y \ar[r]_(.4){X\times f} & X\times PX
}\]
This definition of toposes comes from Lawvere and Tierney \cite{MR0430021, MR0373888, MR0354800}, although a more restricted notion of toposes appeared earlier in Grothendieck's work. See \cite{MR1300636, MR1953060, MR0470019} for more information on topos theory.

Toposes have an internal language \cite{MR0319757}: a higher order intuitionistic logic. Heyting categories where defined in \cite{MR0360735}. They also have an internal language, but this internal language is a many sorted predicate logic that does not always have higher order quantification.

An early reference of regular and exact categories is \cite{ECnCS}. First Mac Lane developed Abelian categories (see \cite{MR0025464}) for algebraic topology. Subsequent authors looked at categories that omitted parts of the algebraic structure of Abelian categories, while retaining the non-algebraic properties, until Barr settled on the regular and exact categories we use in this thesis. In \cite{MR678508} we find a construction of exact categories out of categories with finite limits -- the ex/lex completion -- and subsequently many similar constructs have been defined \cite{MR1358759, MR1600009}. Menni worked out under which conditions these completion constructions result in toposes \cite{Menni00exactcompletions, MR1900904}.

Lawvere introduced \xemph{hyperdoctrines} in \cite{MR2223032}. Both fibred locales and complete fibred Heyting algebras are -- up to a 2-equivalence of 2-categories -- examples of hyperdoctrines and we could have called them \xemph{regular} and \xemph{first order hyperdoctrines}. We decided to work with the fibred categories instead of category valued (pseudo)functors, in order to make our work less dependent on set theory, therefore new names seemed appropriate. Grothendieck introduced fibred categories (see \cite{MR0146040}) for algebraic geometry. B\'enabou started applying them to logic \cite{MR0393181, MR0393180}.

A tripos is a special type of complete fibred Heyting algebra. In \cite{Pittsthesis, MR578267}, one can find a construction of toposes out of triposes. The tripos-to-topos construction was soon applied to realizability, resulting in Hyland's \xemph{effective topos} \cite{MR717245}. We give a definition of this category here.
\newcommand\Eff{\Cat{Eff}}
\newcommand\set[1]{\{#1\}}
\begin{definition*} The \xemph{effective topos} $\Eff$ is the category whose objects are pairs $(X,E\subseteq \N\times X\times X)$ for which $s,t\in \N$ exists such that for all $(m,x,y)\in E$ and $(n,z,x)\in E$, $sm$ and $tmn$ are defined and $(sm,y,x)$ and $(tmn,z,y)\in E$. Morphisms are defined as follows. For all $X$ and all $U,V\subseteq \N\times X$ we let $U\models_X V$ if there is an $m\in \N$ such that for all $(n,x)\in U$, $mn\converges$ and $(mn,x)\in V$. Then $U\iff_X V$ if $U\models_X V$ and $V\models_X U$. A morphism $(X,E) \to (X',E')$ is an $\mathord{\Longleftrightarrow}_X$-equivalence class $\phi$ of $F\subseteq \N\times X\times X'$ for which $e,r,s_0,s_1,u\in \N$ exists such that
\begin{itemize}
\item for all $(m,x,x)\in E$, there is a $z\in X'$ such that $em\converges$ and $(em,x,z)\in F$;
\item for all $(m,x,y)\in F$, $(n,x,x')\in E$ and $(p,y,y')\in E'$, $((rm)n)p\converges$ and\\ $(((rm)n)p,x',y')\in F$;
\item for all $(m,x,y)\in F$, $s_0m$ and $s_1m$ are defined, $(s_0m,x,x)\in E$ and \\$(s_1m,y,y)\in E'$;
\item for all $(m,x,y),(n,x,z)\in F$, $(um)n\converges$ and $((um)n,y,z)\in E'$.
\end{itemize}
The composition of two morphisms $\phi:(X,E)\to(X',E')$ and $\chi:(X'',E'')\to(X',E')$ is the $\mathord{\Longleftrightarrow}_X$-equivalence class $\chi\circ \phi$ that for all $F\in \phi$ and $G\in \chi$ contains 
\[ G\circ F = \set{(\langle n,m\rangle,x,y)\in \N\times X\times X''|\exists z.(n,x,z)\in F, (m,z,y)\in G } \]
Here, $\langle -,-\rangle$ is the pairing combinator from the definition at the beginning of this introduction.
\end{definition*}

This definition can be understood as follows. Each object $(X,E)$ consists of a set of \emph{names} $X$ and a realizability relation $E$ for the \xemph{extensional equivalence} of names. This extensional equivalence relation is only realized to be symmetric and transitive (by realizers $s$ and $t$) allowing `partial elements' for which $x=x$ is not valid.

By the way, the subcategory of $(X,E)$ such that $x=x$ is valid for each $x\in X$, is equivalent to the effective topos. This is not true for all closely related categories however.

The morphisms are relations that behave like the graphs of functions. Such \xemph{functional relations} have many different representations as realizability relations, which forces us to work with equivalence classes.

We list some subsequent developments after the invention of the effective topos. 
\begin{itemize}
\item Hyland, Grayson and others worked out how to approach some of Kleene's variations of recursive realizability \cite{MR1909029}, e.g., there is an effective topos over any topos with a natural number object, so we can do formal realizability by working over the free topos with natural number object.
\item The theory of partial combinatory algebras can be developed in the internal language of toposes, and we can construct realizability toposes for them too. 
\item There are relative variants, where another property than the existence of realizers determines the $\mathord{\Longleftrightarrow}_X$-equivalence, see \cite{MR1909030}. 
\item The effective topos is both an ex/reg and an ex/lex completion \cite{MR1056382, MR948482}.
\end{itemize}

\section*{Theory of realizability}
Recursive realizability has its own logic, which is different from classical logic, because there are undecidable predicates, but also different from intuitionistic logic, because every proposition is either true or false. The set of realized propositions in recursive realizability is not recursively enumerable by G\"odel's incompleteness theorems. However, the indictive definition of the realizability relation determines a recursive reduction from the set of realized propositions to the set of valid proposition. Maybe we can describe this reduction using a set of axioms.

On the proof theoretic side, we have the following result of Dragalin \cite{MR0258596} and Troelstra \cite{MR0335240}. The schema of \xemph{extended Church's thesis} is, for each almost negative predicate $A$ and every predicate $B$:
\[ \mathrm{ECT_0} \iff [\forall x.A(x) \to \exists y.B(x,y)] \to \exists z.\forall x.A(x) \to (zx\converges \land B(x,zx))]\]
Denoting Heyting arithmetic by $\mathrm{HA}$ we have the following connection between provability and realizability.
\begin{align*}
& \mathrm{HA}+\mathrm{ECT_0}\vdash \phi\leftrightarrow \exists x.x\rlzs \phi \\
& (\mathrm{HA}\vdash \exists x.x\rlzs \phi) \iff (\mathrm{HA}+\mathrm{ECT_0}\vdash \phi)
\end{align*}
Van Oosten extended this to higher order arithmetic in \cite{MR1303664}.

To extends these results to other forms of realizability, we look for categorical properties which characterize a realizability topos (up to equivalence), and determine which of these properties can be expressed in the internal language and which cannot.

Just like the effective topos, we can represent realizability toposes as completions of simpler categories under particular coequalizers, the \xemph{ex/reg} and \xemph{ex/lex completions} \cite{MR948482, MR1056382}. In his thesis \cite{RTnLS}, Longley shows that one of these simpler categories has a universal property. He also introduces \xemph{applicative morphisms}, a preordered set of morphisms between partial combinatory algebras that is equivalent to a category of regular functors between realizability toposes.

Axioms and universal properties give us a clearer view of what realizability is. Therefore, understanding and extending these results have been our aims.

\section*{In this thesis}
We will sketch what we consider to be \xemph{realizability} in this thesis, and then give a survey of our results.

We work with realizability models constructed over arbitrary \xemph{Heyting categories}. We wanted to demand as little structure on the base category as possible, and Heyting categories are both sufficient and necessary to ensure that the resulting realizability categories are Heyting categories too. Although this extra generality has complicated the construction, is has simplified the universal properties: often, more examples means less rules.

Among Heyting categories we find \xemph{syntactic categories}, which model only the provable propositions of some deductive system. Our work applies directly to proof theory through these syntactic categories. We also find \xemph{locally Cartesian closed pretoposes}, some of which are considered to be \xemph{predicative} alternatives to toposes. So the theory can tell us what realizability could (and couldn't) do for intuitionists. All toposes are Heyting categories, and therefore repeating realizability and combining realizability with Kripke semantics or filter quotients constructions falls within the set of models under consideration.

We only have one object of realizers, and it is an \xemph{order partial applicative structure}. Here we follow Hofstra and van Oosten \cite{MR1981211}. Working with more than one object of realizers in combination with the complications from working with general Heyting categories was too daunting. Moreover, there is a modified realizability topos, which occurs as a subtopos of a realizability topos (see \cite{MR1933398}). Other forms of realizability, like Longley's typed realizability \cite{Longley_matchingtyped}, may show up as subcategories of realizability toposes as well.

When working with relative realizability in Heyting categories, i.e., when some subobject of the order partial applicative structure determines validity, it is easier to just collect all objects of realizers and point out the ones that act as truth values in the realizability models. There is a collection of subobjects $\phi$ of the order partial applicative structure such that a proposition $p$ is valid if and only if the realizability relation assigns a member of $\phi$ to $p$. Not all $\phi$ that give a sound interpretation of intuitionistic logic come from relative realizability, but we found it useful to consider the new examples too. It builds the filter quotient construction right into the construction of realizability categories. This is also critical, because the property that sets relative realizability categories apart from this more general class of models is rather complicated; once again we benefit from `more examples means less rules'.

The first half of chapter \ref{axiom} generalizes the construction of \xemph{the effective tripos} -- an intermediate step in the construction of the effective topos -- to all of our realizabilities. It is a structure that assigns a Heyting algebra to each object of a category and this allows us to interpret first order logic. 

We decided to work with \xemph{fibred categories} instead of \xemph{indexed categories}. The 2-categories of small indexed categories and small fibred categories are equivalent thanks to the axiom of choice. If we work with large base categories, the 2-category of indexed categories could be a proper subcategory of the 2-category of fibred categories, depending on how some foundational issues are decided. We decided to work with small fibred categories, so that we can hang on to the intuitions of working with indexed categories while using constructions which work on large categories.

Rather than \xemph{first order hyperdoctrines}, we talk about \xemph{complete fibred Heyting algebras}. Placing `fibred' in front of a type of category seems a convenient way of describing fibred categories that have some extra structure. We use this convention throughout this thesis.

In the second half of chapter \ref{axiom} we determine a list of properties which characterize the \xemph{realizability fibrations} we constructed up to equivalence. This is summarized in theorem \ref{characterization}. Together these properties already are the  universal property we desired and it helps us to derive universal properties of realizability categories in chapter \ref{realcat}.

We determine which properties can be expressed in the internal language and which cannot. One of our axioms is an adaptation of extended Church's thesis, and the others simply generalize well known theorems of the effective topos, so no surprises there. Theorem \ref{axiomatization} is our completeness theorem for realizability.

Chapter \ref{realcat} collects various results on realizability categories. We start with explaining how the tripos-to-topos construction can be applied to general realizability fibrations and similar fibred categories. Using categories-of-fractions constructions forces us to rely on universal properties later on, but this is what we wanted to do anyway.

Working with realizability interpretations where every valid proposition has an inhabited object of realizers has a huge advantage, namely the existence of a left adjoint, see lemma \ref{left adjoint}. Moreover, we can do so without loss of generality, by changing the base category, see theorem \ref{inhabitation}. Both facts are tremendously helpful in the characterization of two categories that can be constructed from a realizability fibration, see theorems \ref{charreg} and \ref{charexact}. 

We can reformulate the universal property of realizability categories in such a way, that it generalizes Longley's applicative morphisms (see \cite{RTnLS}) to our greater class of realizability categories, see definition \ref{appmorph}. They give us an easy way to study regular functors between realizability categories. Theorem \ref{pseudoinitiality} is about a universal property of the realizability fibrations, and corollaries \ref{asm pseudoinitial} and \ref{asmexreg pseudoinitial} translate the universal property to the realizability categories.

Half way through this thesis, we start to discuss realizability toposes. We explain how to exploit the impredicativity of toposes to get different universal properties for realizability toposes. First we consider the advantages of using realizabilities where the objects of realizers of valid propositions have global sections, rather then just being inhabited. A completion construction from Hofstra and van Oosten's \cite{MR1981211} makes sure we never have to work with other kinds of realizability if our base category is a topos. We derive a new universal property in corollary \ref{lems and toposes}, which help us study left exact functors between realizability toposes, and take a look at applicative morphisms that induce geometric morphisms, i.e., the \xemph{computationally dense} applicative morphisms.

The last section of chapter \ref{realcat} demonstrates that most realizability categories are not reg/lex or ex/lex completions of other categories, and are not \xemph{relative completions} (see \cite{MR2067191}) either. The last subsection gives conditions that make these completions work.

Chapter \ref{apps} collects various result on realizability that are not directly related to the matter of the first two chapters, although we start by applying our results on the characterization of realizability to recursive realizability in the first section of this chapter. We derive a list of properties that characterize effective toposes constructed over arbitrary toposes with natural number objects (theorem \ref{definition of effective topos}).

In the second section we redo our paper \cite{MR2729220} using only the characteristic properties of realizability toposes. We can now say that the category of pointed complete extensional PERs is algebraically compact in \xemph{every} effective topos (theorem \ref{compactness}).

The last section of chapter \ref{apps} is about Krivine's \xemph{classical realizability}, which is a realizability for classical logic. We demonstrate that certain Boolean subtoposes of relative realizability toposes are models of classical realizability (theorem \ref{cr sub rr}).

%% file: Axiomatization.tex
This chapter explores the limits of realizability, in particular the bounds to what is logically possible in a realizability model. Since there are many different kinds of realizability, we cannot cover everything that falls under that name, so we make the following choices.
\begin{itemize}
\item We develop realizability internally in arbitrary Heyting categories. This class of categories includes all toposes, but also categories that do not have power objects.
\item We develop realizability with an \xemph{order partial combinatory algebra} of realizers. These include all traditional partial combinatory algebras, but also all meet semilattices. 
\item We allow models where a proposition is valid if its set of realizers satisfies a property other than simply being inhabited. This includes properties like having a global section, or intersecting some special subobject. In this way we have one framework for different types of realizability.
\end{itemize}
It turns out that all resulting models satisfy two axiom schemas that are adaptations of known theorems of the effective topos, namely, \xemph{extended Church's thesis} and the \xemph{uniformity principle}. Since a model that violates these schemas cannot be a realizability model, they describe what the logical limits of realizability are.

\section{Categorical framework}
Our realizability models are going to be \emph{complete Heyting algebras fibred over Hey\-ting categories}, a notion we explain in this section. Our treatment is limited to a definition of the relevant concepts; more on categorical logic can be found in \cite{MR1859470} and \cite{MR1674451}.


\subsection{Regular and Heyting categories}
We run through the definitions of these categories and their characteristic properties.

\begin{definition}[regular categories] Let $\cat C$ be a category with finite limits. For each arrow $f:X\to Y$ in $\cat C$ a \xemph{kernel pair} is a pair of arrows $p:Z\to X$, $q:Z\to X$ such that the square $f\circ p = f\circ q$ is a pullback: 
\[ \xymatrix{
Z \ar[r]^p \ar[d]_q \ar@{}[dr]|<\lrcorner & X \ar[d]^f \\
X \ar[r]_f & Y
} \]
An \xemph{image} of $f$ is a coequalizer $e:X \to \exists_f(X)$ of a kernel pair of $f$. Note that $f$ factors through its image $e$. An image is \xemph{stable under pullback} if for any $g:Y'\to Y$, the image of the pullback of $f$ along $g$ is the pullback of the image of $f$ along $g$:
\[ \xymatrix{
g^*(X) \ar@/^1pc/[rr]^{g^*(f)} \ar[r] \ar[d] \ar@{}[dr]|<\lrcorner & \exists_{g^*(f)}(g^*(X)) \ar[r] \ar[d] \ar@{}[dr]|<\lrcorner & X' \ar[d]^g \\
X \ar[r]  \ar@/_1pc/[rr]_f & \exists_f(X) \ar[r] & Y
} \]
The category $\cat C$ is \xemph{regular} if every arrow has an image and if every image is stable under pullback. Let $F:\cat C \to\cat D$ be a functor between regular categories. If $F$ preserves finite limits and images, then $F$ is a \xemph{regular functor}.
\end{definition}

\begin{remark}[regular-epi-mono factorization] This definition implies that every morphism in a regular category factors as a regular epimorphism followed by a monomorphism. Moreover, for any other factorization of a morphism $f:X\to Y$ as a regular epimorphism $e':X\to Z$ followed by a monomorphism $m':Z\to Y$ there is a unique isomorphism $g:Z\to\exists_f(X)$ such that $e'\circ g =e$ and $m'=m\circ g$.
\[ \xymatrix{
X\ar[r]^(.4)e\ar[d]_{e'} & \exists_f(X) \ar[d]^m \\
Z\ar[r]_{m'} \ar@{.>}[ur]^g & Y
}\]
This is called the regular-epi-mono factorization. An example of a regular category is the category $\Cat{Set}$ of sets and functions, where the regular-epi-mono-factorization is the factorization of a function into a surjection and an injection.
\end{remark}

\newcommand\Sub{\Cat{Sub}}
\newcommand\pre[1]{#1^{-1}}
\begin{definition}[Heyting categories] Let $\cat C$ be a regular category. For each object $X$ in $\cat C$ the poset of subobjects $\Sub(X)$ is the poset reflection of the category of monomorphisms into $X$. For each arrow $f:X\to Y$, pullbacks induce a monotone map $f^{-1}:\Sub(Y) \to\Sub(X)$ called the \xemph{inverse image map}. If each $\Sub(X)$ has finite joins and if $\pre f$ has a right adjoint $\forall_f$, a \xemph{dual image map}, for every arrow $f$ in $\cat C$, then $\cat C$ is a \xemph{Heyting category}.

A regular functor $F$ for which the induced maps $F_X:\Sub(X) \to \Sub(FX)$ preserve joins and right adjoints is a \xemph{Heyting functor}.
\end{definition}

\begin{remark}[size issues] If $\cat C$ is large category, subobjects can be proper classes of monomorphisms, and a `set of subobjects' does not exist. To avoid this problem and other size problems, we will work with small categories througout this thesis. Together with smallness, we will assume the the axiom of choice applies to sets of objects and morphisms of all categories. Large categories like the category of all sets may appear in examples and remarks.
\label{size}
\end{remark}

\newcommand\termo{\mathbf 1}
\newcommand\db[1]{[\![ #1 ]\!]}
\newcommand\im[1]{\exists_{#1}}
\begin{remark}[properties of regular categories]
For every object $X$ in a regular category, $\Sub(X)$ is a meet semilattice, and for each morphism $f:X\to Y$ the preimage map $\pre f$ has a left adjoint: the \xemph{direct image map} $\im f$. The map $\pre f$ preserves all meets. The inverse image maps satisfy the \xemph{Beck-Chevalley condition}: if $f\circ g= h\circ k$ is a pullback square, then $f^{-1}\exists_h =  \exists_g k^{-1}$.
\[ \begin{array}{c}\xymatrix{
\bullet \ar[r]^g \ar[d]_f \ar@{}[dr]|<\lrcorner & \bullet \ar[d]^k \\
\bullet \ar[r]_h & \bullet
}\end{array} \Longrightarrow \begin{array}{c}\xymatrix{
\bullet \ar[r]^{\im g}  & \bullet  \\
\bullet \ar[r]_{\im h} \ar[u]^{\pre f} & \bullet \ar[u]_{\pre k}
}\end{array} \]
The meets and the preimage map together satisfy the \xemph{Frobenius condition}: $\im f(U)\land V = \im f(U\land \pre f(V))$ for all $f:X\to Y$, $U\in \Sub(X)$ and $V\in \Sub(Y)$.

Pullbacks provide meets of subobjects. If for each monic $m:U\to X$, $\db m$ is its equivalence class in $\Sub(X)$, then $\top_X= \db{\id_X}\in \Sub(X)$ is a top element. If $n:V\to X$ is another monic, the meet $\db m\land \db n$ is the pullback $U\times_X V\to X$ of $m$ and $n$.

If $m:U\to X$ is monic, then $\exists_f(\db m)$ is $\db n$ for the monomorphism $n:\im{f\circ m}(U) \to Y$. Both the Beck-Chevalley condition and the Frobenius condition follow from the stability of images under pullback. 

For each regular functor $F:\cat C \to\cat D$ the induced maps $F_X:\Sub(X) \to \Sub(FX)$ preserve meets and direct images, because regular functors preserve finite limits and images.
\end{remark}

\begin{remark}[properties of Heyting categories]
For each object $X$ in a Heyting category, $\Sub(X)$ is a Heyting algebra and for each morphism $f:X\to Y$ the preimage map $\pre f:\Sub(Y) \to \Sub(X)$ is a morphism of Heyting algebras that has both a left adjoint $\exists_f$ and a right adjoint $\forall_f$.

Because Heyting categories are regular categories, $\Sub(X)$ has meets for every object $X$, and $\pre f$ has a left adjoint $\im f$ for every morphism $f$. The inverse image map $\pre f$ preserves all meets and joins because it has both a left and a right adjoint. The Beck-Chevalley condition for the direct image map induces the same condition on the dual image maps, so if $f\circ g= h\circ k$ is any pullback square:
\[ \begin{array}{c}\xymatrix{
\bullet \ar[r]^g \ar[d]_f \ar@{}[dr]|<\lrcorner & \bullet \ar[d]^k \\
\bullet \ar[r]_h & \bullet
}\end{array} \Longrightarrow \begin{array}{c}\xymatrix{
\pre f(\im h(U))\subseteq V\ar@{<=>}[r]\ar@{<=>}[d] &  \im g(\pre k(U))\subseteq V \ar@{<=>}[d]\\
U\subseteq \pre h\forall_f(V)\ar@{<=>}[r] & U\subseteq \forall_k(\pre g V)
}\end{array}\]
This leaves us with \xemph{Heyting implication}. For any monic $m:U\to X$ and any $V\in\Sub(X)$ implication $\db m \to V$ is $\forall_m(m^{-1}(V))$ and this defines it for all of $\Sub(X)$. The map $V\mapsto (\db m \to V)$ is right adjoint to $W\mapsto \im m (\pre m (W))$, which equals $\db m\land W$ because of the Frobenius condition:
\[ \im m (\pre m (W)) = \im m (\pre m (W)\land\top) = W\land \im m(\top) = W\land \db m \]
The Frobenius condition also implies that  the preimage map preserves implication.
\[\xymatrix{
\exists_f(W) \land \db m \subseteq V \ar@{<=>}[r]\ar@{<=>}[d] & \exists_f(W\land \pre f(\db m))\subseteq V \ar@{<=>}[d]\\
W \subseteq \pre f(\db m\to V) \ar@{<=>}[r] &  W\subseteq \pre f(\db m)\to \pre f(V) 
}\]

For each Heyting functor $F:\cat C \to\cat D$ the induced maps $F_X:\Sub(X) \to \Sub(FX)$ are morphisms of Heyting algebras.
\end{remark}

\subsection{Fibred categories}
Fibred categories are categories that depend contravariantly on some base category, roughly in the same way a presheaf is a set that corresponds contravariantly on some base category. Our realizability models will be a kind of fibred categories, so we will develop that notion here.

\begin{definition} Let $F:\cat C\to\cat D$ be a functor. 
\begin{itemize}
\item An arrow $f:X\to Y$ in $\cat C$ is \xemph{prone} or \xemph{Cartesian} relative to $F$, if for every $g:X'\to Y$ and $h:FX'\to FX$ such that $Fg = Ff\circ h$, there is a unique $k:X'\to X$ such that $g=f\circ k$ and $Fk=h$. A functor $F:\cat C \to \cat D$ is a \xemph{fibration} or a \xemph{fibred category} if for each $X$ in $\cat C$ and each $f:Y\to FX$ in $\cat D$ there is a prone $f':Y'\to X$ such that $Ff'=f$.

\item If $f:Y\to X$ is prone for $F^{op}:\cat C^{op}\to\cat D^{op}$, then $f$ is \xemph{supine} or \xemph{coCartesian}. If $F^{op}$ is a fibred category, then $F$ is an \xemph{opfibred category} or an \xemph{opfibration} and if $F$ is both a fibration and an opfibration, then it is a \xemph{bifibration} or a \xemph{bifibred category}.

\item For each object $X\in \cat D$, the \xemph{fibre} $F_X$ is the subcategory of $\cat C$, whose objects are mapped to $X$ and whose arrows are mapped to $\id_X$ by $F$. Arrows that $F$ sends to identities, are called \xemph{vertical}. 

\item For each $f:X\to Y$ there is an up to isomorphism unique functor $F_f:F_Y\to F_X$ such that there is a natural transformation $p:F_f \to \id_{F_X}$ consisting of prone morphisms over $f$. We call it the \xemph{reindexing functor}. 

\item If $F:\cat C\to\cat D$ and $F':\cat C'\to\cat D'$ are fibrations, then a \xemph{morphism of fibrations} is a pair of functors $G_0:\cat D\to \cat D'$ and $G_1:\cat C \to \cat C'$ such that $F'G_1=G_0F$ and such that $G_1$ preserves prone arrows. Given two morphism $G,H:F\to F'$ a \xemph{$2$-morphism of fibrations} is a pair of natural transformations $\eta_0:G_0\to H_0$ and $\eta_1:G_1\to H_1$ such that $\eta_0F = F'\eta_1$.
\end{itemize}
\end{definition}

\begin{example}[discrete fibred categories] For each presheaf $F:\cat C^{op}\to\Cat{Set}$ there is a fibred category of elements $E:\cat E(F) \to\cat C$: its objects are pairs $(X\in\cat C,y\in FX)$ and a morphism $(X,y) \to (X',y')$ is an arrow $f:X\to X'$ such that $Ff(y') = y$. All morphisms are prone, and therefore all vertical morphisms are identities. Fibred categories with this property are called \xemph{discrete fibred categories}.
\end{example}


\begin{example} For each category $\cat C$ that has pullbacks, the \xemph{fundamental bifibration} is $\cod:\cat C/\cat C \to\cat C$, where $\cat C/\cat C$ is the category whose objects are arrows of $\cat C$ and whose morphisms are commutative squares of $\cat C$. Prone morphisms are pullback squares, vertical morphisms are commutative triangles and supine morphisms are squares where the $\dom$-side morphism is an isomorphism.

\[\begin{array}{ccc}
\xymatrix{ \bullet\ar[r]\ar[d]\ar@{}[dr]|<\lrcorner & \bullet \ar[d]\\ \bullet\ar[r] & \bullet } &
\xymatrix{ \bullet\ar[r]\ar[d] & \bullet \ar[d]\\ \bullet\ar@{=}[r] & \bullet } &
\xymatrix{ \bullet\ar[r]^\simeq\ar[d] & \bullet \ar[d]\\ \bullet\ar[r] & \bullet } \\
\textrm{prone} & \textrm{vertical} & \textrm{supine}
\end{array}\]

\newcommand\monos{\Cat{monos}}
Let $\monos(\cat C)$ be the subcategory of $\cat C/\cat C$ whose objects are monomorphisms. The fibred subcategory $\cod:\monos(\cat C) \to\cat C$ is faithful, and its fibres are preorders. So it is a fibred \xemph{preorder}. The \xemph{subobject fibration} $\Sub(\cat C)$ is the fibrewise antisymmetric quotient of $\cod:\monos(\cat C) \to\cat C$, i.e., $\Sub(X) = \Sub(\cat C)_X$ is poset reflection of $\cod_X$. If $\cat C$ is regular the subobject fibration is a bifibration: for any $f:X\to Y$ and any monic $g:Z\to X$ the regular-epi-mono factorization provides a regular epimorphism $e:Z\to W$ and a monomorphism $m:W\to Y$ and $(e,f)$ is a supine morphism $g\to m$.\label{subobject fibration}\end{example}

\begin{example}[fibred category of fibred categories] Let $\Cat{Fib}$ be the category of fibrations, then $\cod:\Cat{Fib}\to \Cat{Cat}$ is itself a fibred 2-category: \xemph{prone morphism of fibred categories} are pullback squares, \xemph{vertical morphisms of fibred categories} are commutative triangles. \label{fibred category of fibred categories} \end{example}

We end this section with some useful facts about finite limits and indexed coproducts in fibred categories. The first one concerns fibred categories with certain limits.

\begin{definition} A fibred category $F:\cat F\to\cat B$ has limits of shape $\cat D$, where $\cat D$ is an arbitrary category, if for each object $X$ of $\cat B$ every functor $D:\cat D\to F_X$ has a limit in $F_X$ and if the reindexing functors preserve these limits. \end{definition}

\begin{lemma}[lifting limits] Let a fibred category $F:\cat F\to\cat B$ have limits of shape $\cat D$ and let $\cat B$ have limits of shape $\cat D$ too. Then $\cat F$ has limits of shape $\cat D$ and $F$ preserves limits of shape $\cat D$. \label{lifting limits}\end{lemma}

\begin{proof} Let $D:\cat D\to \cat F$ be any functor. There is a limit cone $k: \lim FD \to FD$. For each $I\in\cat D$, there is a prone morphism $p_I:F_{k_I}(DI) \to DI$ over $k_I$, and there is a vertical limit cone $k'_I:\lim_{I\in\cat D} F_{k_I}(DI) \to F_{k_I}(DI)$. That $\lim_{I\in\cat D} F_{k_I}(DI)$ is $\lim D$ follows from the fact that any commutative cone over $D$ will first factor uniquely through the prone morphisms $p_I F_{k_I}(DI) \to DI$, and then through $k'$. By the construction of this limit $F\lim D = \lim FD$.
\end{proof}

The second one concerns bifibred categories.

\begin{remark}[adjunction] For every bifibred category $F:\cat E\to\cat B$ and every arrow $f:X\to Y$ in $\cat B$, the reindexing functor $(F^{op})_f: F_X^{op} \to F_Y^{op}$ has a dual $f_! :F_X \to F_Y$: the \xemph{coindexing functor}. This coindexing functor satisfies $f_!\dashv F_f$. The reason is that if $s:X\to f_!(X)$ is supine, $p:F_f(Y)\to Y$ is prone and $Fp = Fs = f$, then for every $g:f_!(X) \to Y$ there is at most one $h:X\to F_f(Y)$ such that $g\circ s= p\circ h$ and vice versa.
\[ \xymatrix{
X \ar[r]^s \ar[d]_h & f_!(X) \ar[d]^g \\
F_f(Y) \ar[r]_p & Y
} \]\label{coindexing}
\end{remark}

\subsection{Fibred locales}
Our realizability models are fibred categories that have the same structure as the subobject fibrations of regular and Heyting categories. We single those out here.

\begin{definition} A bifibration $F:\cat F\to\cat B$ satisfies the \xemph{Beck-Chevalley condition} if for each pullback square $f\circ g= h\circ k$, $F_k h_! \simeq  g_! F_f$. Here $g_!$ and $h_!$ are the \xemph{coindexing functors} from remark \ref{coindexing}.
\[ \begin{array}{c}\xymatrix{
\bullet \ar[r]^g \ar[d]_f \ar@{}[dr]|<\lrcorner & \bullet \ar[d]^k \\
\bullet \ar[r]_h & \bullet
}\end{array} \Longrightarrow \begin{array}{c}\xymatrix{
\bullet \ar[r]^{g_!} \ar@{}[dr]|\simeq & \bullet  \\
\bullet \ar[r]_{h_!} \ar[u]^{F_f} & \bullet \ar[u]_{F_k}
}\end{array} \]
If the Beck-Chevalley conditions holds, the bifibred category \xemph{has indexed coproducts}. The symbols $\coprod_f$, $\sum_f$ or in our case $\exists_f$ then stand for the coindexing functor $f_!$.

A fibred category $F$ \xemph{has finite products} if each of its fibres has them and if the reindexing functors preserve them. In other words: if it has limits of shape $\cat D$ where $\cat D$ is any finite discrete category. A bifibration $F:\cat E\to\cat B$ with finite products satisfies the \xemph{Frobenius condition} if for each arrow $f$, if the canonical map $\exists_f(X\times F_fY) \to \exists_fX \times Y$ is an isomorphism. We get the canonical map from the adjunction $\exists_f\dashv F_f$ that every bifibration has and from the fact that $F_f$ preserves finite products.

A \xemph{fibred locale} is a bifibration $F:\cat F\to\cat B$ 
\begin{itemize}
\item That has indexed coproducts,
\item That has finite products and satisfies the Frobenius condition,
\item That is a faithful functor, so the fibres are preordered sets.
\end{itemize}
A morphism of fibred locales is a morphism of bifibrations that preserves finite products.
\end{definition}

\begin{remark} Faithful fibrations have all equalizers, because any parallel pair of vertical arrows is equal. Therefore fibred locales have all finite limits. \end{remark}

The properties of fibred locales make sure that they are actually fibred preordered sets that have joins indexed over objects in $\cat B$ and finitary meets that distribute over these joins because of the Frobenius condition. These properties characterize \xemph{locales} inside any topos, and therefore we call these structures `fibred locales'.

\begin{example}[subobject bifibration] The subobject bifibration of a regular category is a fibred locale. \end{example}

\begin{definition} A \xemph{complete fibred Heyting algebra} is a fibred locale $F$, where the fibres are Heyting algebras and where the reindexing functors $F_f$ preserve implication and have right adjoints $\forall_f$. A \xemph{Heyting morphism} of complete fibred Heyting algebras is a morphism of fibred locales that preserves joins and right adjoints. Therefore, they are fibred morphisms of Heyting algebras that also preserve indexed meets.
\end{definition}

\begin{remark} We do not assume that fibred Heyting algebras are antisymmetric, because that property does nothing for the theory we develop here. Every fibred Heyting algebra is equivalent to an antisymmetic one, though, because we work with small categories. \end{remark}

\newcommand\oftype{\mathord:}
\begin{remark} We now have a structure for the interpretation of a first order language. Let $F:\cat F\to\cat B$ be a complete fibred Heyting algebra over a category with finite products. The objects of the base $\cat B$ are types, the morphisms are terms, and the objects of $\cat F$ are predicates, where $F$ maps each predicate to the type it applies to. The Heyting algebra structure of the fibres allows us to interpret propositional logic. The diagonal map $\delta:X\to X^2$ provides an equality predicate: $=_X$ is $\exists_\delta(\top_X)$; the projection $\pi_0:X\times Y \to X$ provides quantification: $\forall x\oftype X.\phi$ is $\forall_{\pi_0}(\phi)$ and $\exists x\oftype X.\phi$ is $\exists_{\pi_0}(\phi)$. The elements of the terminal fibre $F_\termo$ are truth values, and $F\models p$ for $p\in F_\termo$ if $p$ is a terminal object.

Note that if we restrict to the fragment with $=,\land,\exists$ called \xemph{regular logic}, we can already give a sound interpretation in any fibred locale. \label{models}
\end{remark}

This is what our realizability model is going to look like.

\section{Order partial applicative structures}
In this section we define the structure of the object of realizers and the \xemph{filters} that determine validity in our realizability models. To get models for which first order intuitionistic logic is sound, we need \xemph{combinatory completeness}. We will define these concepts and give some examples.

\subsection{Order partial applicative structure}
The structure of the object of realizers, which could live in any Heyting category, is the following.

\begin{definition} Let $\cat H$ be a Heyting category. An \xemph{order partial applicative structure} in $\cat H$ is an object $A$ with a preorder $\leq$ and a partial binary operator $(x,y)\mapsto xy:A^2\partar A$ called \xemph{application}. We indicate that a pair $(x,y)\in A\times A$ is in the domain of the application operator by writing $xy\converges$. The application operator must have the following property: if $x\leq x'$, $y\leq y'$ and $x'y'\converges$, then $xy\converges$ and $xy\leq x'y'$.

For each order partial applicative structure $A$ in $\cat H$, a \xemph{filter} is a subobject $C\in \Sub(A)$ that is closed under application, and upward closed for $\leq$.
\end{definition}

\begin{remark}[naming] If $\leq$ is $=$ (the discrete order), then we call the structure a \xemph{partial applicative structure}; if the application operator is total, we call the structure an \xemph{order applicative structure}; an \xemph{applicative structure} is just an object with a binary operator. \end{remark}

\begin{remark} We could work with a more liberal definition of filter, where $C\in\Sub(A)$ is a filter if for all $x,y\in C$ such that $xy\converges$ there is a $z\in C$ such that $z\leq xy$. Let's call these sets \xemph{prefilters} for now. We work towards a realizability relation that is downward closed: if $x\in A$ realizes a proposition $p$, then so does any $y\leq x$. The filter determines validity: $p$ is valid if there is a $x\in C$ that realizes it. For this definition of validity, prefilters are just as sound as filters, because each prefilter $C$ realizes the same set of propositions as the least filter $D$ such that $C\subseteq D$. But this also means that every realizability interpretation with a prefilter is equivalent to a realizability interpretation with a filter. Therefore, we prefer the more restrictive definition we gave above. \end{remark}

We immediately introduce realizability for partial functions $A^n\partar A$ relative to a filter $C$.

\begin{definition} Using $x\vec y$ to denote the repeated application $((xy_1)\dots)y_n$, we define the object of realizers of a partial function $f:A^{n+1}\partar A$ of any arity $n+1$ as follows.
\[ \db f = \{ r\in A | (\forall \vec x\in A^n. r\vec x\converges)\land(\forall \vec y\in \dom f. r\vec y\converges\land r\vec y \leq f(\vec y)) \} \]
So $\db f$ is internally the object of $r\in A$ such that $r\vec x$ is defined for all $\vec x\in A^n$, and $r\vec y$ is defined for all $\vec y\in\dom f$, and $r\vec y\leq f(\vec y)$ for all those $\vec y$. A filter $C$ \xemph{realizes} or \xemph{represents} a partial function $A^n \partar A$, if $\db f$ intersects $C$, i.e., if $C\cap \db f$ is \xemph{inhabited} or \xemph{globally supported}. \label{realizability}
\end{definition}

There is a particular class of functions filters must realize, if we want to get a sound realizability interpretation.

\begin{definition}[combinatory completeness] The class of \xemph{partial combinatory functions} is the set of all partial functions $A^n\partar A$ of any arity $n\in\N$ that can be constructed from projections $\vec x\mapsto x_i$ by pointwise application. In other words, it is the least set of partial functions that is closed under composition, contains all projections of Cartesian powers of $A$ and the partial application operator.

A filter is \xemph{combinatory complete} if it realizes all partial combinatory functions. An order partial applicative structure is combinatory complete if it has a combinatory complete filter. Combinatory complete order partial applicative structures are also called \xemph{order partial combinatory algebras}, where we drop `partial' application is total. If $\leq$ is $=$, then we have a \xemph{partial combinatory algebra}. \label{combinatory completeness} \end{definition}

\begin{remark} Every partial combinatory function is determined by a polynomial in a single binary operator. \end{remark}

\begin{remark} Partial combinatory algebras come from Feferman's \cite{MR0409137}. By one of the theorems in that paper a filter is combinatory complete when it realizes two partial applicative functions, namely $k(x,y) = x$ and $s(x,y,z)=xz(yz)$.
Feferman actually put an stronger condition on partial combinatory algebras: there is a realizer $r$ for each partial combinatory $f:A^n\partar A$ such that for all $\vec x\in A^n$ if $r\vec x\converges$, then $\vec x\in\dom f$. For realizability interpretations, it does not matter if $f(\vec x)$ is undefined when $r\vec x\converges$, so we work with this weaker condition, which is sometimes called \xemph{weak combinatory completeness}. Order partial combinatory algebras were introduced as $\leq$-PCAs in \cite{MR1443487} and further developed as ordered PCAs in \cite{MR1981211}. Also see \cite{MR2479466}.
\end{remark}

\subsection{Preservation}\label{preservation}
We show that finite limit preserving functors preserve order partial applicative structures and filters and that regular functors preserve combinatory completeness.

\begin{lemma} Order partial applicative structure and filters can be defined using only finite limits, and therefore finite limit preserving functors preserve them.
\label{preserving1} \end{lemma}

\newcommand\roso{\mathord\leq}
\begin{proof} For each order partial applicative structure $A$ the partial order $\leq$ and the domain of application $D$ are subobjects of $A^2$. We express their properties by demanding that there are certain arrows between these objects and pullbacks of these objects. We use $\pi_0$ and $\pi_1$ to denote the two projections $A^2\to A$:
\begin{itemize}
\item reflexivity is a map $r:A\to \roso$ that satisfies $\pi_0(r(x)) = \pi_1(r(x)) = x$;
\item let $X_1 = \{ (i,j)\in \roso^2 | \pi_0(i) = \pi_1(j) \}$; transitivity is a map $t:X_1\to A$ that satisfies $\pi_0(t(i,j)) = \pi_0(j)$ and $\pi_1(t(i,j)) = \pi_1(i)$;
\item let $X_2 = \{ (i,j,k) \in \roso\times D | \pi_1(i) = \pi_0(k), \pi_1(j) = \pi_1(k), \}$; downward closure of the domain is a map $d:X_2 \to D$ that satisfies $\pi_0(d(i,j,k)) = \pi_0(i)$ and $\pi_1(d(i,j,k)) = \pi_0(j)$; that application $\alpha:D\to A$ preserves the ordering $p: X_2\to \roso$ is a map that satisfies $\pi_0(p(i,j,k)) = \alpha(\pi_0(i),\pi_0(j))$ and $\pi_1(p(i,j,k)) = \alpha(\pi_1(i),\pi_1(j))$.
\end{itemize}

The objects $X_1$ and $X_2$ are pullbacks, so a functor that preserves finite limits will preserve them. All functors preserve the equations that $r$,$t$,$d$ and $p$ satisfy. Therefore finite limit preserving functors preserve order partial applicative structures.

For each filter $C$ of $A$ we have the following maps:
\begin{itemize}
\item closure under application is a map $c:C^2\cap D \to C$ that satisfies $c(x,y) = xy$;
\item let $Y_1 = \{ (i,j)\in C\times \roso| i=\pi_0(j)\}$; upward closure is a map $u:Y_1\to C$ that satisfies $u(x,y) = \pi_1(j)$.
\end{itemize}
Once again the object $Y_1$ is a pullback, and any functor preserves the equalities satisfied by $c$ and $u$. So finite limit preserving functors preserve filters too.
\end{proof}


\begin{lemma} There is a regular theory whose class of models is the class of ordered partial applicative structure with combinatory complete filters. \end{lemma}

\newcommand\convergesto\downarrow
\begin{proof} We now easily write down a regular theory of ordered partial applicative structures. We use a binary relation $\leq$ for the ordering, but write $xy\converges$ instead of $D(x,y)$ or $x\mathrel D y$ to indicate that $(x,y)$ is in the domain $D$ of the application operator. For the application operator itself, we use juxtaposition, so $xy$ is the application of $x$ to $y$.
\begin{align*}
&\vdash a\leq a & a\leq b\land c\leq a &\vdash c\leq b &
a\leq b\land c\leq d \land bd\converges &\vdash ac\converges\land ac\leq bd
\end{align*}

A filter $C$ becomes a predicate that has to satisfy:
\begin{align*} C(a)\land C(b)\land ab\converges &\vdash C(ab) & C(a)\land a\leq b\vdash C(b) \end{align*}

\newcommand\pred[1]{\ulcorner #1 \urcorner}
We express that $C$ represents any partial combinatory function $f:A^n\partar A$ by extending this theory as follows. We add a predicate $\pred f$ to the language and also add an axiom that says that $\pre f$ intersects $C$, namely $\exists x.\pred f(x)\land C(x)$. We add a list of axioms to say that if $\pred f(a)$ then $((ax_1)\dots)x_n \convergesto y$ for some $y\leq f(\vec x)$:
\begin{align*}
\pred f(a) &\vdash \exists y_1. ax_1\convergesto y_1\\
\pred f(a), ax_1\convergesto y_1 &\vdash \exists y_2. y_1x_2\convergesto y_2 \\
&\ \vdots \\
\pred f(a), ax_1\convergesto y_1, y_1x_2\convergesto y_2,\dotsm &\vdash \exists y_n. y_n\leq f(\vec x) \land y_{n-1}x_n\convergesto y_n
\end{align*}

A model $A$ for these axioms is an order partial applicative structure with a combinatory complete filter and therefore an order partial combinatory algebra.
\end{proof}

\begin{remark} Every order partial combinatory algebra and every combinatory complete filter is a model for this theory, though not necessarily in a unique way. For each partial combinatory $f$ we may interpret the related predicate $F$ as any inhabited subobject of the object of realizers $\db f$. \end{remark}

\begin{corollary} Regular functors preserve combinatory completeness. \end{corollary}

\newcommand\pow{\mathbf P}
\begin{remark} The class of order partial combinatory algebras in the topos $\Set$ of sets is closed under filter products. 
For any set $\kappa$ a filter $\phi$ on the power set $\pow \kappa$ -- which with $\subseteq$ and $\cap$ is an order combinatory algebra -- induces a regular functor $F_\phi:\Set/\kappa \to\Set$. Let $F_\phi(f:X\to \kappa)$ be the set $\Sigma(f)$ of sections of $f$ modulo the equivalence relation $\{(x,y)\in \Sigma(f)| \{i\in \kappa|x(i)=y(i)\}\in \phi\}$. A family of order partial combinatory algebras indexed over $\kappa$ is the same thing as an order partial combinatory algebra in $\Set/\kappa$; if $A = (A_i)_{i\in\kappa}$ is an order partial combinatory algebra in $\Set/\kappa$, then $F_\phi A$ is an order partial combinatory algebra in $\Set$. \label{filter products of OPCAs}
\end{remark}

\subsection{Examples}
We include a list of examples of order partial combinatory algebras.

\newcommand\Kone{\mathcal{K}_1}
\begin{example}[Kleene's first model] According to Kleene's normal form theorem, there is a recursively decidable predicate $T:\N^3\to 2$ and a primitive recursive function $U:\N\to\N$ such that every partial recursive function $f$ is equivalent to $n\mapsto U(\mu x.T(e,n,x))$ for some $e\in \N$. \xemph{Kleene's first model} $\Kone$ is the partial applicative structure whose application satisfies $\alpha(e,n) = U(\mu x.T(e,n,x))$ for all $e,n\in \N$ for which this is defined.

Any Heyting category that has a natural number object has its own version of this partial combinatory algebra, as we will se in chapter \ref{apps}. No other filter then the whole of $\Kone$ represents all partial recursive functions. The partial \xemph{combinatory} functions may be a strict subset of the partial recursive functions, however, and therefore a non-trivial combinatory complete filter may exist. \label{Kone}\end{example}

\begin{example}[Kleene's second model] 
There is a partial function $u:\N^\N\times \N^\N \partar \N^\N$ that is continuous for the product topology, such that for each continuous $f:\N^\N \partar \N^\N$ whose domain is a countable intersection of open sets, there is an $x\in \N^\N$ such that for all $y\in \N$, $u(x,y)\converges$ if and only if $f(y)\converges$, and if $f(y)\converges$ then $u(x,y) = f(y)$. With this operator, $\N^\N$ is a partial combinatory algebra. Total recursive functions form a combinatory complete filter in this algebra.
This algebra and filter are used in \cite{MR0176922} for studying intuitionistic logic. \label{Ktwo} 
\end{example}

\begin{example}[meet semilattices] Every meet semilattice is an order combinatory algebra. In this case the definition of `filter' in this thesis coincides with the traditional order theoretical one, which is the reason we have chosen this name. We just saw the special case of power sets in remark \ref{filter products of OPCAs}.
\end{example}

\newcommand\la{$\lambda$}
\newcommand\FV{\mathrm{FV}}
\newcommand\rto{\to^*}
\begin{example}[\la-terms] Various order partial combinatory algebras consist of \la-terms from the \la-calculus. A \xemph{\la-term} $M$ is either a variable symbol $x$,$y$,$z$\dots from some infinite set of variable symbols, an application of \la-terms $(NP)$ or an abstraction $(\lambda x.N)$. As short hand for $\lambda x_1.(\lambda x_2.\dots(\lambda x_n.M))$ we write $\lambda x_1x_2\dots x_n.M$; similarly, $M_1M_2\dotsm M_n$ is short for $(((M_1M_2)\dotsm)M_n)$.

Together they form the set $\Lambda$ of all \la-terms.  We define the set $\FV(M)$ of \xemph{free variables} of a \la-term $M$ recursively over the set of all terms: $\FV(x) = \{x\}$, $\FV(MN) = \FV(M)\cup\FV(N)$ and $\FV(\lambda x.N) = \FV(N)-\{x\}$. If $\FV(M)=\emptyset$ then $M$ is a \xemph{closed} \la-term.

\xemph{Variable substitution} is an operation on \la-terms that we define as follows.
\begin{itemize}
\item $x[N/x] = N$ but $y[N/x]=y$ if $y\neq x$ for all variable symbols $x$ and $y$, and terms $N$;
\item $(\lambda x.M)[N/x]=\lambda x.M$ but $(\lambda y.M)[N/x]=\lambda y.(M[N/x])$ if $y\neq x$ for all variable symbols $x$ and $y$, and terms $M$ and $N$ such that $y\not\in\FV(N)$.
\item $(MN)[P/x]=(M[P/x])(N[P/x])$ for all terms $M$, $N$ and $P$.
\end{itemize}
Juxtaposition acts as an application operator. To get an order applicative structure we preorder \la-terms.
The \xemph{reduction preorder} on \la-terms is the least preorder $\rto$ that satisfies:
\begin{itemize}
\item $\alpha$-equivalence: $\lambda x.M \rto \lambda y.M[y/x]$ if $y\not\in \FV{\lambda x.M}$;
\item $\beta$-reduction: $(\lambda x.M)N \rto N[M/x]$ if $\FV{N[M/x]} = \FV{(\lambda x.M)N}$;
\item head reduction: if $M\rto M'$, then $MN\rto M'N$;
\item tail reduction: if $N\rto N'$, then $MN\rto MN'$.
\end{itemize}
Of each (partial) combinatory arrow $f:\Lambda^n \to \Lambda$, the term $\lambda x_1\dots x_n.f(x_1,\dotsc, x_n)$ is a realizer. Therefore this preorder $\rto$ makes the set of closed \la-terms a combinatory complete filter and the set of all \la-terms an order combinatory algebra.

We can construct other order partial combinatory algebras by adding reductions rules like the following.
\begin{itemize}
\item $\eta$-expansion: $M\leq \lambda x.Mx$ unless $x\in\FV(M)$.
This rule together with $\beta$-reduction implies $\alpha$-equivalence. Also, if $y\not\in \FV(M)$, then $My\leq N$ if and only if $M\leq\lambda y.N$, turning abstraction and application into adjoint functors.
\item $\zeta$-reduction: if $M\leq M'$, then $\lambda x.M\leq \lambda x.M'$.
With this reduction rule, partial combinatory functions get an up to $\alpha$-equivalence maximal realizer.
\end{itemize}
\end{example}

\newcommand\comb\mathbf
\begin{example}[combinatory logic] We can build simpler term models, based on combinatory logic. Let $T$ be the set of binary trees, with leaves in the set $\{\comb b, \comb c, \comb k, \comb w\}$. If $x$ and $y$ are binary trees, we let $xy$ be the tree whose left subtree is $x$ and whose right subtree is $y$. We use a similar convention as with \la-terms: $x_1\dotsm x_n = (x_1\dotsm)x_n$.

We order this set of trees with the least preorder $\leq$ that satisfies $\comb bxyz  \leq x(yz)$, $\comb cxyz \leq xzy$, $\comb kxy\leq x$, $\comb wxy\leq xyy$ and if $x\leq x'$ and $y\leq y'$ then $xy\leq x'y'$. The set $T$ with this ordering and with the operator $(x,y)\mapsto xy$ is an order combinatory algebra.

Instead of Curry's $\set{\comb b, \comb c, \comb k, \comb w}$ we can use another combinatory basis, e.g. Feferman's $\set{\comb k,\comb s}$ where $\comb sxyz\leq xz(yz)$, and $\comb k$ is as above.
\label{combinatorylogic}
\end{example} 

\begin{example}[graph models]
Graph models are a class of models for the \la-calculus. The model is a powerset $\pow X$ with a topology that is constructed as follows. We take a subset $T\subseteq \pow X$ that is closed under finite unions and contains all finite subsets of $X$. 
Usually $T$ simply is the set of finite subsets, but the constructions below work without this assumption. An open set $U\subseteq \pow X$ is an upward closed set, such that for each $u\in U$ there is a $t\in U\cap T$ such that $t\subseteq u$. A continuous function $f:\pow X \to \pow X$ now has to satisfy: $f(x) = \bigcup\{ f(t)| t\in T\cap \pow x \}$.

If there is an injective map $\mu: T\times X\to X$ we can define a binary operator on $\pow X^2 \to \pow X$.
\[ xy =\{b\in X| \exists y'\in \pow y\cap T. \mu(y',b)\in x\} \]
This operator is continuous, and therefore the partial combinatory functions defined with it are continuous too. For each continuous $f:\pow X^{n+1}\to\pow X$ let $\lambda f(\vec x) = \{\mu(t,y)|y\in f(\vec x,t) \}$. This is a continuous function that satisfies $\lambda f(\vec x)y = f(\vec x,y)$. By iterating this \la-operator, we get a realizer for every continuous function.

We can think of the elements of $X$ as `types', each $\xi\in X$ determining an (upward closed) subset $U_\xi$ of $\pow X$: $U_\xi = \set{u\in \pow X|\xi\in u}$. We can choose which types are inhabited in the realizability model in the following way. Any $F\subseteq X$ such that $FF\subseteq F$ determines a filter $\pow F\subseteq \pow X$. This filter $\pow F$ is combinatory complete if and only if $\lambda xy.x$, $\lambda xyz.zx(yz)\subseteq F$. The exact sense in which the members of $F$ are inhabited in the realizability model should become clear in the remainder of this chapter.
\end{example}

\section{Realizability fibrations} \label{realizability fibrations}
In this part we construct a complete fibred Heyting algebra out of an order partial combinatory algebra. We split the construction into two parts. First we show how to construct a complete fibred Heyting algebra out of a \xemph{fibred} order partial combinatory algebra that satisfies a completeness property. Then we construct a suitable fibred order partial combinatory algebra out of an ordinary (internal) one. Along the way, we find some generalizations of realizability.

\subsection{Complete fibred partial applicative lattices}\label{fcpal}
In this subsection, we show how to construct a complete fibred Heyting algebra out of a complete fibred partial applicative lattice. This is a partial applicative structure in the category of fibred preorders that have all indexed meets and joins. The fibres are applicative structures in the topos of sets and this allows us to use constructions that are available it that topos.

\newcommand\arrow\Rightarrow
\begin{definition} A \xemph{complete fibred lattice} is a fibred category $F:\cat F\to\cat B$ that is a faithful functor, where the fibres are lattices with top and bottom elements and where the reindexing functors have both left and right adjoints that satisfy the Beck-Chevalley condition. A \xemph{complete fibred partial applicative lattice} is a complete fibred lattice that is also a fibred order partial applicative structure, i.e., there is a \xemph{fibred application operator} determined by a \xemph{partial functor} $(X,Y)\mapsto XY:\cat F\times_\cat B\cat F \partar\cat F$ (i.e., a finberd functor defined on a fibered subcategory).
This functor should preserve all joins in each variable separately:
\begin{itemize}
\item always $X\bot\simeq\bot Y\simeq\bot$;
\item if $XZ\converges$ and $YZ\converges$ then $(X\vee Y)Z \simeq XZ\vee YZ$;
\item if $XY\converges$ and $XZ\converges$ then $X(Y\vee Z) \simeq XY\vee XZ$;
\item for every arrow $f$ in $\cat B$, if $XF_f(Y)\converges$ then $\im f(X)Y \simeq \im f(XF_f(Y))$;
\item for every arrow $f$ in $\cat B$, if $F_f(X)Y\converges$ then $X\im f(Y) \simeq \im f(F_f(X)Y)$.
\end{itemize}
Furthermore, in each fibre there is a (total) binary operator $\arrow$ such that if $XY\converges$, then $\cat F(X,Y\arrow Z)\simeq \cat F(XY,Z)$ and else $\cat F(X,Y\arrow Z)=\emptyset$. Reindexing preserves this \xemph{arrow operator}.
\end{definition}

\begin{remark} In the rest of this thesis, we never actually use the fibred meets in the definition above. This is why we don't need a Frobenius condition on them. In fact, the complete fibred applicative lattices could be introduced as complete fibred distributive lattices, where meets have been replaced by an operator that can be nonidempotent, asymmetric, noncommutative, nonassociative and partial (nontotal). We will stick with complete fibred lattices, because they already are more general than we need anyway. \end{remark}

\begin{example} Every complete fibred Heyting algebra is a complete fibred lattice, and with $\land$ as application operator, complete fibred Heyting algebras are complete fibred partial applicative lattices. The difference between complete fibred Heyting algebras and complete fibred lattices is \xemph{Cartesian closure}: complete fibred Heyting algebras have a Heyting implication in each fibre that is preserved by reindexing. \end{example}

\begin{remark} Since the fibres are lattices, we write $X\leq Y$ to indicate the existence of a vertical arrow $X\to Y$ between two objects of $\cat F$. The condition on the arrow operator can now be written as follows: $X\leq Y\arrow Z$ if and only if $XY\converges$ and $XY\leq Z$.
\end{remark}

\begin{definition} A \xemph{fibred filter} on a complete fibred partial applicative lattice $F:\cat F\to \cat B$ is a full fibred subcategory $\cat C\subseteq \cat F$ such that the fibres $C_X = F_X\cap \cat C$ are filters. The filter is \xemph{closed under indexed meets} if $\forall_f(Z)\in C_{\cod f}$ for every $Z\in C_{dom f}$. A fibred filter is \xemph{combinatory complete} if each of its fibres is.
\end{definition}


The following definition is a construction for a complete fibred Heyting algebra and we prove that in the following lemma.

\begin{definition} Let $F:\cat F\to\cat B$ be a complete partial applicative lattice and let $\cat C$ be a combinatory complete fibred filter that is closed under indexed meets. We define the \xemph{filter quotient} $F/\cat C:\cat F/\cat C \to \cat B$ as follows. The domain $\cat F/\cat C$ has the same objects as $\cat F$. The set of morphisms $\cat F/\cat C(X,Y)$ is the set of pairs $(U\in\cat C,f:UX\to Y)$ modulo the equivalence relation $(U,f)\sim (V,g)$ if $Ff=Fg$. Of course, $F$ induces a map $\cat F/\cat C \to\cat B$, which we call $F/\cat C$.
\end{definition}

\begin{lemma} The filter quotient $F/\cat C$ is a complete fibred Heyting algebra.\label{fcHa} \end{lemma}

\begin{proof} We first show that the fibres are Heyting algebras. 

In the fibres of $F/\cat C$, there is an arrow $X\to Y$ if and only if $X\arrow Y\in C_X$, because of the adjunction between $\arrow$ and application. That $C_X$ is closed under the application of $F_X$ implies that if $U\in C_X$ and $U\arrow V\in C_X$, then $V\in C_X$. Combinatory completeness implies the rest.
\begin{itemize}
\item We have $\comb ix\leq x$ and $\comb bxyz\leq x(yz)$, which imply that the fibres are preorders.
\[ \comb i \leq X\arrow X \quad \comb b(X\arrow Y)(Z\arrow X)\leq (Z\arrow Y) \]
\item Note that $\comb i \leq X\arrow Y$ if $X\leq Y$. For that reason, each fibre has a top and a bottom.
\item We have $\comb kxy\leq x$ and $\comb sxyz\leq xz(yz)$, which means that $\arrow$ behaves like logical implication
\begin{align*}
\comb k &\leq X\arrow (Y\arrow X) &
\comb s &\leq (X\arrow (Y\arrow Z))\arrow ((X\arrow Y)\arrow(X\arrow Z))
\end{align*}
\item We have $\comb gxy=yx$, $\comb pxyz\leq zxy$, so fibres have a meet operator that behaves like conjunction in a Heyting algebra. Let $X\times Y = \comb pXY$.
\begin{align*}
\comb{ki} & \leq X\arrow \top & \comb p &\leq X\arrow (Y\arrow (X\times Y)) \\
\comb{gk} &\leq (X\times Y)\arrow X & \comb g(\comb{ki}) &\leq (X\times Y)\arrow Y
\end{align*}
\item We have $\comb lxyz\leq yx$ and $\comb rxyz\leq zx$, so fibres have a join operator that behaves like disjunction in a Heyting algebra. Let $X+Y = \comb lX\vee\comb r Y$.
\[\begin{array}{c}
\begin{array}{cc} \comb l \leq X\arrow (X+Y) & \comb r \leq Y\arrow (X+ Y)\end{array} \\
\comb p \leq (X\arrow Z)\arrow((Y\arrow Z)\arrow((X+Y)\arrow Z))
\end{array}\]
\end{itemize}
This means that the fibres are Heyting algebras.

Reindexing functors preserve all of this structure, because they preserve both application and $\arrow$ and all members of the fibred filter, because the fibred filter is a fibred subcategory. This leaves the adjoints.

Let $f:X\to Y$ be an arrow of $\cat B$. We have $F_f(\forall_f(Y)\forall_f(Y')) \leq YY'$ and therefore, if $SZ\leq Z'$ for some $S,Z,Z'\in F_X$, then $\forall_f(S)\forall_f(Z)\leq \forall_f(Z')$. Because $\cat C$ is closed under indexed meets, this implies that $\forall_f(Z)\arrow \forall_f(Z')\in C_Y$ if $Z\arrow Z'\in C_X$. Preservation of joins by application gives us $\forall_f(S)\exists_f(Z) \simeq \exists_f(F_f\forall_f(S) Z)$. Now $F_f\forall_f(S) Z\leq SZ \leq Z'$ and therefore $\forall_f(S)\exists_f(Z) \leq \exists_f(Z')$. This proves that $\exists_f(Z)\arrow \exists_f(Z')\in C_Y$ if $Z\arrow Z'\in C_X$.

If $U\exists_f(X)\leq Y$, then $\exists_f(F_f(U)X)\leq Y$ because application preserves indexed joins, and therefore $F_f(U)X\leq F_f(Y)$. If $VX\leq F_f(Y)$ then $F_f\forall_f(V)X\leq F_fY$ because $\forall_f$ is right adjoint to $F_f$ relative to $\leq$. This implies $\forall_f(V)\exists_f(X)\leq Y$ by applying the preservation of indexed joins again. Because $\cat C$ is closed both under reindexing and indexed meets, $\exists_f$ is still left adjoint to $F_f$. For universal quantification it is even simpler: $UX\leq \forall_f(Y)$ if and only if $F_f(U)F_f(X)\leq Y$, while $VF_f(X)\leq Y$ implies $\forall_f(V) X\leq \forall_f(Y)$ because $F_f(\forall_f(V))\leq V$.
\end{proof}

\subsection{Lattice of downsets}
Now that we know how to construct a complete fibred Heyting algebra out of a complete fibred partial applicative lattice, we just need to show how to construct a complete fibred partial applicative lattice out of an internal order partial applicative structure with a combinatory complete filter.

\newcommand\ds{\mathsf D}
\begin{definition} Let $A$ be an order partial applicative structure in a Heyting category $\cat H$. The \emph{fibration of families of downward closed subsets of $A$} is the fibred subcategory $\ds A$ of $\Sub(A\times -)$ where $\ds A_X \subseteq\Sub(A\times X)$ contains the $Y\in \Sub(A\times X)$ that satisfy: if $a\leq a'$ and $(a',x)\in Y$ then $(a,x)\in Y$.
\end{definition}

\begin{lemma} The fibred category $\ds A$ is a complete fibred partial applicative lattice.\label{applace} \end{lemma}

\begin{proof} It is a complete fibred lattice because it is a complete fibred Heyting algebra, a property that it inherits from $\Sub$.
The application operator is defined as follows: for $U,V\in \ds A_X$ let $UV\converges$ if $\cat H \models \forall x\in X, (a,x)\in U, (b,x)\in V. ab\converges$ and in that case \[ UV = \set{ (c,x)\in A\times X| \exists a,b\in A. (a,x)\in U\land (b,x)\in V\land c\leq ab } \]
This operator preserves all required joins, because it is $\exists_\alpha$ if $\alpha$ is the application operator, followed by the downward closure map, and both maps preserve joins.

The adjoint operator $\arrow$ is defined as follows.
\[ V\arrow W = \set{ (a,x)\in A\times X | \forall b\in A.(b,x)\in V \to (ab\converges \land (ab,x)\in W) }\]
That $U\subseteq V\arrow W$ if and only if $UV\converges$ and $UV\subseteq W$ follows almost directly from this definition.
\end{proof}

To build a realizability model with this, we need combinatory complete filters that are closed under indexed meets. The property of being closed under indexed meets has the following useful consequence.

\begin{lemma} Let $F:\cat F\to\cat B$ be a complete fibred partial applicative lattice. If $\cat B$ has a terminal object $\termo$ and $\cat C$ is a filter on $F:\cat F\to\cat B$ that is closed under indexed meets, then $X\in \cat C$ if and only if $\forall_!(X)\in C_\termo$, where $!$ is the unique arrow $FX\to\termo$. Hence, $C$ is totally determined by its fibre over $\termo$. \end{lemma}

\begin{proof} Closure under indexed meets and $X\in \cat C$ implies $\forall_!(X)\in C_\termo$. If $\forall_!(X)\in C_\termo$ then $F_!(\forall_!(X))\in F_{FX}$, but $F_!(\forall_!(X))\leq X$ and filters are upwards closed. \end{proof}

\newcommand\ext{\mathbf D}
\begin{definition}
Let $\ext A$ be the set of downward closed subobjects of $A$ in $\cat H$, and define an application operator as follows. For all $U,V\in\ext A$, $UV\converges$ if $U\times V$ is a subobject of the domain $D$ of the application operator; in that case $UV = \set{z\in A| \exists x\in U,y\in V. xy\converges\land z\leq xy}$. This is the \xemph{external completion $A$}.
\end{definition}

\begin{corollary} If $\cat B$ has a terminal object $\termo$, there is an equivalence between filters on $\ext A$ and filters on $\ds A$ that are closed under indexed meets.
\end{corollary}

\begin{proof} Because $A\simeq A\times \termo$, $\ext A \simeq \ds A_\termo$. \end{proof}

We now consider when a filter of $\ext A$ corresponds to a \emph{combinatory complete} fibred filter which is closed under indexed meets.

\begin{lemma} Let $\cat C$ be a fibred filter of $\ds A$ which is closed under indexed meets. The filter $\cat C$ is combinatory complete if and only if $\db f\in C_\termo$ for all partial combinatory $f:A^n \mapsto A$ (here $\db f$ is the object of realizers of $f$ from definition \ref{realizability}).\end{lemma}

\newcommand\A{\mathring A}
\begin{proof} 
We first show that if $C_\termo$ contains $\db f$ for partial combinatory $f$, then $C$ is combinatory complete. 

Let $\pi_i:A^n \to A$ be the projection $\vec x\mapsto x_i$, let $X$ be an arbitray object of $\cat H$ and let $\vec U\in \ds A_X^n$. For convenience, let $\prod_X \vec U = \set{(\vec a,x)\in A^n\times X| (a_i,x)\in U_i}$.
\[ (( \ds A_!(\db\pi_i) U_1)\dotsm)U_n  = \set{(b,x)\in A\times X| \exists (\vec a,y)\in \prod_X\vec U. x'=x \land  b\leq a_i }
\]
This is a subobject of $U_i$. Hence each projection $\vec U \mapsto U_i:\ds A_X^n \to \ds A_X$ is realized by $\ds A_!(\db\pi_i)$ which is a member of $\cat C$ because $\cat C$ is fibered and $\pi_1$ is partial combinatory. 

For any pair $f,g:A^n \partar A$, we have
\begin{align*}
&((\ds A_!(\db{fg})U_1)\dotsm)U_n\! = \!\set{(b,x)| \exists (\vec a,y)\in \prod_X\vec U. y = x\land  b\leq f(\vec a)g(\vec a) }  \\
&((\ds A_!(\db{f})U_1)\dotsm)U_n((\ds A_!(\db{g})U_1)\dotsm)U_n \! = \! \\
& \qquad \set{(b,x)| \exists (\vec a,y),(\vec a',y')\in \prod_X\vec U.y=y'=x\land  b\leq f(\vec a)g(\vec a') }
\end{align*}
So if $\ds A_!(\db f)$ and $\ds A_!(\db g)$ realize $f',g':\ds A^n \partar \ds A$ respectively, then $\ds A_!(\db{fg})$ realizes the pointwise application $f'g'$.

Every partial combinatory function is constructed from projections by pointwise application, and therefore $\cat C$ is combinatorially complete if $C_\termo$ contains $\db f$ for every partial combinatory $f$.

For each $f:A^n\partar A$ let $\hat f\in \ds A_{A^n}$ be
\[ \hat f = \set{(y,\vec x)\in A\times A^n | \vec x\in \dom f \to y\leq f(\vec x) } \]
If $\hat f\hat g$ is defined, then $\hat f\hat g = \widehat{fg}$. Therefore, for any partial combinatory $p:\ds A^n \partar \ds A$ we get $p(\hat \pi_1,\dotsc, \hat \pi_n) = \widehat{p'}$ for the corresponding partial combinatory $p':A^n \partar A$. 

Because of combinatory completeness, $p$ has an object of realizers $R\in C_{A^n}$. So $((R\hat \pi_i)\dotsm)\hat \pi_n\subseteq \widehat{p'}$. By writing out the definitions we find that $\forall_!(R)\subseteq \db{p'}$ if $((R\hat\pi_1)\dotsm)\hat \pi_n \subseteq \widehat{p'}$. Now $\forall_!(R)\in C_\termo$ because of closure under indexed meets, and therefore so is $\db f$. \end{proof}

\begin{definition} An \xemph{external filter} of $A$ is a filter $\phi$ of $\ext A$. We consider it \xemph{combinatory complete} if it contains the object of realizers $\db f$ of all partial combinatory $f:A^n \partar A$. We write $\ds A/\phi$ to denote the filter quotient of $\ds A$ over the combinatory complete fibred filter induced by $\phi$. All complete fibred Heyting algebras $\ds A/\phi$ that arise from external filters on internal order partial applicative structures are \xemph{realizability fibrations}.
\end{definition}

We conclude this section by showing that realizability fibrations are complete fibred Heyting algebras.

\begin{theorem} For every Heyting category $\cat H$, every order partial applicative structure $A\in\cat H$ and every combinatory complete external filter $\phi$ of $A$, the realizability fibration $\ds A/\phi$ is a complete fibred Heyting algebra. \label{rf is ha} \end{theorem}

\begin{proof} This is a special case of lemma \ref{fcHa}. \end{proof}

\begin{remark}[realizability relation] The ordinary way of defining realizability using a realizability relation between combinators and propositions is hidden in the definition of the realizability fibration. It assigns an equivalence class of downsets to each proposition. The realizability relation comes from choosing representatives in these equivalence classes, in a way that depends recursively on the interpreted proposition. We can always do that, as long as we keep in mind that our order partial combinatory algebra may have no global sections and therefore no explicit combinators to point to.

Let $\comb t, \comb f$ be $\db{(x,y)\mapsto x}$ and $\db{(x,y)\mapsto y}$ respectively, then given any map $r$ from atomic propositions (excluding equations) to $\ext A$, we define the realizability relation $\rlzs$ as follows:
\begin{itemize}
\item $x\rlzs p$ if $x\in r(p)$ for each atomic proposition $p$.
\item always $x\rlzs \top$, never $x\rlzs \bot$, and $x\rlzs c=d$ if and only if $c=d$;
\item $x\rlzs p\land q$ if for all $t\in \comb t$ and $f\in\comb f$, $xt\converges$, $xf\converges$, $xt \rlzs p$ and $xf\rlzs q$;
\item $x\rlzs p\vee q$ if for all $t\in \comb t$ and $f\in\comb f$, $xt\converges$, $xf\converges$ and either $xt\in\comb t$ and $xf\rlzs p$, or $xt\in\comb f$ and $xf\rlzs q$;
\item $x\rlzs p\to q$ if for all $y\rlzs p$, $xy\converges$ and $xy\rlzs q$;
\item $x\rlzs \forall u \oftype X.p(u)$ if $x\rlzs p(c)$ for all $c\in X$;
\item $x\rlzs \exists u \oftype X.p(u)$ if $x\rlzs p(c)$ for some $c\in X$.
\end{itemize}
All of this makes sense in the internal language of $\cat H$ and recursively defines an object of realizers for each proposition.
\label{realrela}
\end{remark}

We give some examples for your consideration.

\begin{example} For ordinary realizability in an arbitrary base category the external filter is the filter of inhabited downsets, so a proposition is valid if and only if the object of realizers is inhabited.
\end{example}

\begin{example} For each combinatory complete filter $C\subseteq A$, the filter of downsets that intersect $C$ is combinatory complete. This is for example the case in example \ref{Ktwo}. This is general \xemph{relative realizability}, where a proposition is valid if the base category sees that its object of realizer intersects the filter. In terms of the relation defined in remark \ref{realrela}, a proposition $p$ is valid if $\cat H\models \exists a\in C. a\rlzs p$. One can read more on relative realizability in \cite{MR1769604},\cite{MR1909030} and \cite{MR1948021}.\end{example}

\begin{example} Suppose $\cat H(\termo, A)$ is combinatory complete. Now the filter of subobjects that have global sections is a combinatory complete external filter. In this form of realizability a proposition is valid if the object of realizers has a global section. We can make this relative to any combinatory complete filter of $\cat H(\termo, A)$. \end{example}

\begin{example} Given two combinatory complete external filters $\phi$ and $\chi\subseteq \phi$ on an order partial combinatory algebra $A$, there is a vertical morphism of fibred locales $\ds A/\chi \to \ds A/\phi$ that simply maps each family to itself. This morphism preserves $\forall$ and $\arrow$ and therefore is a morphism of complete fibred Heyting algebras.

Let $\set{\phi_i|i\in\kappa}$ be a set of external filters on $A$, then $\bigcap_{i\in\kappa} \phi_i$ is also an external filter. In fact, there is a least combinatory complete filter, generated by the set of $\db f\in \ext A$ where $f$ is a partial combinatory function.

The category of realizability fibrations for a single order partial combinatory algebra $A$ in a Heyting category $\cat H$ is completeness, because it is equivalent to the poset of combinatory complete external filters. \label{logical functors}
\end{example}

\section{Characterization}
For each Heyting category $\cat H$, order partial applicative structure $A\in\cat H$ and external filter $\phi$, we characterize the fibred locales over $\cat H$ that are equivalent to the realizability fibration $\ds A/\phi$. We don't demand that the fibred locales are complete fibred Heyting algebras. There are six characteristic properties that make a fibred locale $F:\cat F\to\cat H$ equivalent to the realizability fibration. We summarize them here and treat them more extensively in the subsections.
\begin{enumerate}
\item The fibred locale is \xemph{separated}. This means that for each regular epimorphism $e$ of $\cat H$, $\exists_e$ preserves $\top$. Remark \ref{origin of separated} explains this terminology.
\item The fibred locale has a \xemph{weakly generic object} $C$. This means that for each object $X\in\cat F$ there is a span $(p:Y\to C,s:Y\to X)$ such that $p$ is prone and $s$ is supine.
\item The weakly generic object is an `$F$-valued filter' of $A$. This makes sense if we consider that $X\mapsto \top_X\in F_X$ is a finite limit preserving functor, so $\top_A$ is an order partial applicative structure and $C$ is a filter of this structure, for which the inclusion is a vertical morphism.

These first three properties and the fact that $\cat H$ is a regular category together allow us to construct for each object $X\in \cat F$ a $Y\subseteq C\times X$ such that the inclusion is prone, the projection $\pi_1:Y\to X$ is supine, and the fibres of the projection are downsets of $C$. See proposition \ref{Shanin}.

\item Let $X\subseteq C\times Y$ and $Y\subseteq C\times Z$ both be prone and downward closed, let $f=\pi_1:X\to Y$ and $g=\pi_1:Y\to Z$. The fourth property is \xemph{Church's rule}, which says that if $f$ is supine, then there is a $W\subseteq C\times Z$ such that $h = \pi_1: W\to Z$ is supine, and such that $(a,b,x) \mapsto (ab,b,x)$ is a map $W\times_X Y \to X$ that commutes with $f$.

\[ \xymatrix{
X\ar[r]^f & Y \ar[r]^g & Z \\
& W\times_X Y \ar[ul]^{(a,b,z)\mapsto (ab,b,z)} \ar[u]\ar[r]\ar@{}[ur]|<\urcorner & W \ar[u]_h
}\]

\item Let $X\subseteq C\times Y$ be prone an downwards closed, let $f = \pi_1:X\to Y$ and let $g:Y\to Z$ be prone. If $f$ is supine, then there is a $W\subseteq C\times Z$ that is prone and downwards closed such that $h=\pi_1:W\to Z$ is supine and such that $W\times_X Y \subseteq X$. This is the \xemph{uniformity rule}.

\[ \xymatrix{
X\ar[r]^f & Y \ar[r]^g & Z \\
& W\times_X Y \ar[ul]^{(a,z)\mapsto (a,z)} \ar[u]\ar[r]\ar@{}[ur]|<\urcorner & W \ar[u]_h
}\]

\item The last property says that the unique map $!:F_{U\hookrightarrow A}(C)\to \termo$ is supine if and only if $U\in\phi$.
\end{enumerate}
We go through these properties to show that they hold for the realizability fibration, to show some of their consequences and to show that each fibration that has these properties, is equivalent to the realizability fibration.

\subsection{Separated fibred locales}
A separated fibred locale interprets regular epimorphisms as surjective maps.

\begin{definition} A fibred locale $F:\cat F\to \cat B$ is \xemph{separated} if for each regular epimorphism $e$ of $\cat B$, $\im e(\top)\simeq \top$. \end{definition}

\begin{example} If $\cat B$ is a regular category, then its subobject fibration is separated; the realizability fibration $\ds A/\phi$ is separated, because $\Sub(A\times-)$ is. \end{example}

One useful property of separated fibred locales is that $\exists_e$ is a left inverse of $F_e$ for each regular epimorphism.

\begin{lemma} A fibred locale $F$ is separated if and only if $\exists_e\circ F_e\simeq \id$ for each regular epimorphism $e$. \end{lemma}

\begin{proof} If $F$ is separated then $\exists_e(F_e(X))\simeq \exists_e(\top\land F_e(X))\simeq \exists_e(\top) \land X \simeq X$, because of the Frobenius condition. If $\exists_e$ is a left inverse of $F$, then $\exists_e(\top) \simeq \exists_e(F_e(\top)) \simeq \top$. \end{proof}

\begin{remark} The term `separated' comes from topos theory. If $(\cat S,J)$ is a site, $S:\cat S^{op}\to\Set$ is a separated presheaf, $\set{f_i:X_i\to X| i\in \kappa}\in J$, $g,h\in SX$ and $Sf_i(g) = Sf_i(h)$ for all $i\in\kappa$, then $g=h$. A separated fibred category satisfies the same property up to isomorphism. In this particular case we are working with the \xemph{regular topology} on $\cat B$ which is generated by the regular epimorphisms. These ideas are worked out in 
\cite{MR558106}. \label{origin of separated}
\end{remark} \comment{goed dat er niets meer bij staat over locales, want een complete fibred poset is \xemph{geen} complete poset! }

Looking forward to the weak genericity property in the next subsection, we note the following.

\begin{corollary} In a separated fibred locale, each prone supine span $(p:Y\to C,s:Y\to X)$ factors through a prone supine span $(p': Y'\to C,s': Y'\to X)$ for which $(p',s'):Y'\to C\times X$ is a monomorphism. \label{monicspans} \end{corollary} 

\begin{proof} Because the base category is regular $F(p,s) = m\circ e$ for some monomorphism $m: \bullet\to FC\times FX$ and some regular epimorphism $e:FY\to \bullet$. Let $f = \pi_0\circ m$ and $g = \pi_1\circ m$. The span $(p,s)$ indicates that $X \simeq \im{g\circ e}(F_{(f\circ e)}(C)) \simeq \im g\circ (\im e \circ F_ e) \circ F_ f(X) \simeq \im g\circ F_ f(X)$, and this means that there is a prone $p':Y'\to X$ over $f$ and a supine $s':Y'\to X$ over $g$, such that $F(p',s') = m$. Because $F$ is faithful, $(p',s')$ is monic too.
\end{proof}

Separated fibred locales over the same base category form a coslice category of the category of all fibred locales.

\begin{lemma} A fibred locale $F:\cat F\to\cat B$ is separated if and only if there is a vertical morphism of fibred locales $\Sub(\cat B) \to F$. \label{sublocale} \end{lemma}

Vertical morphisms are commutative triangles, as defined in example \ref{fibred category of fibred categories}.

\begin{proof} If $F$ is separated, we determine $\nabla = (\id_{\cat B},\nabla):\Sub(\cat B) \to F$ by $\nabla(\db m) \simeq \exists_{m}(\top)$, where $m$ is some monomorphism, and $\db m$ is the subobject it represents. It is easy to see that this indeed defines a functor that commutes with the fibrations and that preserve finite limits, prone and supine morphisms.

If $m = (\id_{\cat B},m):\Sub(\cat B) \to F$ is a vertical morphism and $e$ in $\cat B$ is a regular epimorphism, then $\exists_e(\top)\simeq \exists_e(m(\top))\simeq m(\exists_e(\top)) \simeq m(\top) \simeq \top$.
\end{proof}

\subsection{Weakly generic filters}\label{wgf}
We treat the property of having a weakly generic object that is also a filter.

\begin{definition} Let $A$ be an order partial applicative structure in a Heyting category $\cat H$ and let $F$ be a fibred locale. 
Let $s,t:\roso \to A$ be the order relation on $A$, with projections $s(x,y)= x$ and $t(x,y)=y$. Let $a,p,q:D\to A$ be the application operator $a(x,y)=xy$ and the projections $p(x,y)=x$, $q(x,y)=y$ of its domain back to $A$. A \xemph{vertical filter} is an object $C\in F_A$ that satisfies $F_s(C)\leq F_t(C)$ and $F_p(C)\land F_q(C) \leq F_a(C)$.
\end{definition} 

\begin{example} The second projection $t:\roso \to A$ is a family of downsets and hence an object of $(\ds A/\phi)_A$. This is a filter which we denote by $\A$. \end{example}

\begin{example} Any filter $C$ of $A$ induces a vertical filter $C'\in \Sub(A)$: $C' = \exists_{C\hookrightarrow A}(C)$. \end{example}

\begin{remark} A fibred locale $F$ always has a right adjoint $\top$ that maps each object to the top element of its fibre. Since right adjoints preserve finite limits, $\top_A$ is a partial applicative structure by lemma \ref{preserving1}. A vertical filter $C$ is a filter of $\top_A$ whose inclusion $C\to\top_A$ is a vertical morphism. \end{remark}

That $\A$ is a filter is useful, because it implies $\A$ is closed under partial combinatory functions. It has a more remarkable property however.

\begin{definition} In any bifibred category $F:\cat F\to\cat B$, an object $G\in\cat F$ is \xemph{weakly generic}, if for each $X\in \cat F$ there is a chain $(p_0,s_0,p_1,s_1,\dots,p_n,s_n)$ of morphisms, where $p_i$ are prone, $s_i$ are supine, $\dom p_i = \dom s_i$, while $\cod p_0 = G$, $\cod p_{i+1} = \cod s_i$, $\cod s_n = X$.

\[ \xymatrix{
& \ar[dl]_{p_0}\bullet \ar[dr]^{s_0} && \ar[dl]_{p_1}\bullet \ar[dr]^{s_1} &&& \ar[dl]_{p_n}\bullet \ar[dr]^{s_n} \\
G && \bullet && \bullet \ar@{}[ur]|\dots &\bullet && X
}\]
\end{definition}

\begin{remark} In a bifibred category $F:\cat F\to\cat B$ over a category with all pullbacks, the Beck-Chevalley condition allows us to simplify this to a single span $(p:Y\to G,s:Y\to X)$. Consider a prone $f:X\to Z$ and a supine $g:Y\to Z$. There is a prone $f':W\to Y$ and a supine $g':W\to X$ over the pullback cone of $Fp$ and $Fs$ that commutes with $f,g$. Therefore, if there is a chain that consists of $n+1$ prone-supine spans (with $n>0$), then there is one that consists of $n$. \end{remark}

\begin{example} Already in $\ds A$, $\A$ is weakly generic, for if $f:Y\to X$ is a family of downward closed subobjects in $\ds A_X$, and $p:Y\to A$ is the projection $(a,x)\mapsto a$, then $f= \im f(F_ p(\A))$ and this is preserved in the fibre product $\ds A/\phi$. \end{example}

\begin{remark} Let $F$ be a bifibration with a weakly generic object $G$ and let $m=(m_0,m_1)$ and $n=(n_0,n_1): F\to F'$ be morphism of bifibrations. If $m_0\simeq n_0$ and $m_1(G)\simeq n_1(G)$ then $m_1\simeq n_1$, because the prone-supine spans are preserved by $m_1$ and $n_1$ and uniquely determined by $m_0$ and $n_0$.
\end{remark}

The last three characteristic properties are defined in terms of families of prone downward closed subobjects of a vertical filter $C$. If $F$ is a fibred locale and $C\in F_A$ a weakly generic filter, then they form a fibration over $\cat F$, which we can describe as follows.

\newcommand\prodo{prodo}
\begin{definition} Because $F:\cat F\to\cat B$ is a bifibration with finite limits and coequalizers of kernel pairs and because $\cat B$ is regular, $\cat F$ has finite limits and coequalizers of kernel pairs which are preserved by $F$. In fact, due to the Frobenius and Beck-Chevalley conditions, $\cat F$ is a regular category and $F$ a regular functor. We have a fibred category $\ds C:\cat DC \to\cat F$ of \xemph{families of downsets of $C$} indexed over the objects of $F$. This fibration has a fibred reflective subcategory $\ds C_{\rm{prone}}$ of families $Y\to X$ for which the inclusion $Y\hookrightarrow C\times X$ is a prone arrow. The unique factorization of any $Y\to C\times X$ into a vertical arrow $Y\to\overline Y$ and a prone arrow $\overline Y\to C\times X$ determines the reflection.

We let $\ds C_{prone}$ be the fibred category of \xemph{families of prone downsets of $C$}. Finally, an arrow $f:Y\to X$ in $\cat F$ that is a family of prone downsets is a \xemph{\prodo morphism}. \end{definition}

`Prodo' is a contraction of `prone downset'.

\begin{lemma} The fibred category $\ds C_{\rm prone}$ is the pullback of $\ds A$ along $F$. \end{lemma}

\begin{proof} For each $Y\in\cat F$ and every downward closed family $f:Y'\to FY$ in $\cat D A$ there is a unique prone $p: F_f(C\times X) \to C\times X$, such that $\cod p = Y$ and $Fp= f$. That's all there is to it. \end{proof}

The three characteristic properties we mentioned so far have one important consequence for the \prodo morphisms in $\cat F$.

\begin{proposition}[Shanin's rule] If $F$ is a separated fibred locale with a weakly generic vertical filter, then there exists a supine \prodo morphism $f:Y\to X$ for each object $X$. \label{Shanin} \end{proposition}

\begin{proof} Because $C$ is weakly generic, there is a supine prone pair $(p:Y\to C,s:Y\to X)$. Because $F$ is separated, we may assume that $(p,s):Y\to C\times X$ is monic, see corollary \ref{monicspans}.

Because $C$ is upward closed, we may assume $Y$ is downward closed. Let $\roso$ be the order of $A$ and let $\pi_0,\pi_1:\roso \to A$ be the projections, then the downward closure $\overline{FY}$ of $FY$ is the following pullback:
\[ \xymatrix@dr{ \overline{FY} \ar@{.>}[r]\ar@{.>}[d] & \roso \ar[r]^{\pi_1} \ar[d]_{\pi_0} & A \\
FY \ar[r]^{Fs}\ar[d]_{Fp} & A \\
FX }\]
Upward closure of $C$ implies that there is a $\roso' \in F_{\roso}$ with a prone $\pi'_0:Z\to C$ over $\pi_0$ and a supine $\pi'_1$ over $\pi_1$ and the Beck-Chevalley condition now determines that there is an $\overline Y\in F_{\overline{FY}}$ with a prone arrow to $Y$ and a supine arrow to $Z$:
\[ \xymatrix@dr{ \overline{Y} \ar@{.>}[r]\ar@{.>}[d] & \roso' \ar[r]^{\pi'_1} \ar[d]_{\pi'_0} & C \\
Y \ar[r]^s\ar[d]_p & C \\
X }\]

Finally, $(p,s):Y\to C\times X$ is prone. We note that $p:Y \to C$ is prone, and that we can always factor $(p,s)$ in a vertical arrow $v:Y\to Y'$ followed by a prone arrow $(q,t):Y'\to C\times X$. Because $p$ is prone, and $Fq = Fp$, there is a unique vertical arrow $w:Y'\to Y$ such that $q=p\circ w$. Now $p\circ w\circ v = q\circ v = p$, and therefore $w\circ v = \id_Y$. Because $F$ is faithful, the split epimorphism $w$ is also monic, and $Y\times Y'$. By this isomorphism $(p,s):Y\to C\times X$ is prone too.
\[\xymatrix{
& Y \ar[dl]_p \ar[dr]^s\ar@{.>}@<.5ex>[d]^v \\
C & Y'\ar[l]_{q} \ar[r]^t\ar[d]|{(q,t)} \ar@{.>}@<.5ex>[u]^w & X\\
& C\times X\ar[ur]_{\pi_0}\ar[ul]^{\pi_1}
}\]
\end{proof}

\begin{remark} In the realizability interpretation this means that each predicate $P\in (\ds A/\phi)_X$ is equivalent to $\exists c\in C.(c,x)\in \top_Y$ for some family of downsets $Y\to X$ (where $\top_Y$ is the terminal object of the fibre over $Y$), a property that we name after \xemph{Shanin's principle} (see \cite{MR1303664, MR2479466}), which is a similar statement about predicates in recursive realizability.
\end{remark}

\subsection{Tracking principle}
The last three characteristic properties express the fact that morphisms in the realizability fibration have a set of realizers in the combinatory complete external filter $\phi$.

\begin{definition}[Church's rule] Let $F:\cat F\to\cat H$ be a fibred locale and let $C\in F_A$ be a vertical filter. They satisfy \xemph{Church's rule} if for each pair of \prodo morphisms $f:X\to Y$ and $g:Y\to Z$ such that $f$ is supine, there is a supine \prodo morphism $e:W\to Z$ such that for all $(a,z)\in W$ and $(b,z)\in Y$, $ab\converges$ and $(ab,b,z)\in X$. 
\[ \xymatrix{
& W\times_Z Y \ar@{.>}[r] \ar@{.>}[d] \ar@{.>}[dl]_{(a,b,z)\mapsto (ab,b,z)} \ar@{}[dr]|<\lrcorner  & W \ar@{.>}[d]^e \\
X \ar[r]_f & Y \ar[r]_g & Z
}\]
\end{definition}

\begin{example} The realizability fibration $\ds A/\phi$ with the filter $\A$ satisfies this property. The inclusion of $Y$ into $\exists_{Ff}(Y)$ has some downset of realizers $U\in \phi$ and we can let $W = C_U\times Z$, where $C_U = F_{U\hookrightarrow A}(C)$, and $e:W\to Z$ is the projection. The pullback of $e$ along $g$ is the projection $C_U\times Y\to Y$.

Application determines the desired map from $C_U\times Y$ into $X$ because $Y$ is a prone subobject of $C\times Z$, and $X$ of $C\times Y$. Without loss of generality, we may assume that realizers for $(b,z)\in Y$ are pairs $\langle b',c \rangle$ where $b'\leq b$ and $c\rlzs z\in Z$. Because the elements of $U$ realize the inclusion of of $Y$ into $\exists_{Ff}(Y)$, if $a\rlzs ((b,z)\in Y)$ and $u\in U$ then $\langle (ua)_0,\langle a_0,(ua)_1 \rangle\rangle \rlzs ((ua)_0,b,z)\in X$, for any choice of pairing and unpairing combinators $\langle -,-\rangle$, $(-)_0$ and $(-)_1$. \end{example}

\begin{remark} Church's rule is a categorical version of \xemph{extended Church thesis}. The \prodo morphisms $f:X\to Y$ and $g:Y\to Z$ represent a family of relations $f_z:g_z \nrightarrow C$ that are total because $f$ is supine. The supine $e:W\to Z$ tells us that for all $z\in Z$ there is a $c\in C$ such that for all $a\in f_z$, $ca\converges$ and $ca\in g_z$. Shanin's rule covers some relations that are not prone.
\end{remark}

\begin{definition}[uniformity rule] Let $F:\cat F\to\cat H$ be a fibred locale and let $C\in F_A$ be a vertical filter. They satisfy the \xemph{uniformity rule} if for each pair of arrows $f:X\to Y$ and $g:Y\to Z$ where $f$ is a supine \prodo morphism and $g$ is prone, if there is a supine \prodo morphism $e:W\to Z$ such that for all $(a,z)\in W$ and $(y,z)\in Y$, $(a,y,z)\in X$.
\[ \xymatrix{
& W\times_Z Y \ar@{.>}[r] \ar@{.>}[d] \ar@{.>}[dl]_\id \ar@{}[dr]|<\lrcorner  & W \ar@{.>}[d]^e \\
X \ar[r]_f & Y \ar[r]_g & Z
}\]
\end{definition}

\begin{example} The realizability fibration $\ds A/\phi$ with the filter $\A$ satisfies this property too, and for similar reasons as it satisfies Church's rule. The inclusion of $Y$ into $\exists_{Ff}(Y)$ has some downset of realizers $U\in \phi$ and we can let $W = C_U\times Z$, where $C_U = F_{U\hookrightarrow A}(C)$, and $e:W\to Z$ is the projection. The pullback of $e$ along $g$ is the projection $C_U\times Y\to Y$. 

In this case, because $g$ is prone, we may assume that a realizer $b$ for $(y,z)\in Y$ is a realizer for $z\in Z$. Because $u\in U$ realize the inclusion of $Y$ into $\exists_{Ff}(X)$, if $a\rlzs (y,z)\in Y$ then $\langle a,az \rangle \rlzs (a,y,z)\in X$.
\end{example}

\begin{remark} The uniformity rule is a categorical version of the \xemph{uniformity principle}. The supine \prodo morphism $f:X\to Y$ and prone $g:Y\to Z$ represent a family of relations $f_z:g_z \nrightarrow C$ that are total because $f$ is supine. The supine $e:W\to Z$ tells us that for all $z\in Z$ there is a $c\in C$ such that for all $x\in f_z$, $c\in g_z$. This expresses the uniformity of fibres of prone morphisms.
\end{remark}

\begin{definition} Let $U\in \phi$.  A filter $C\in F_A$ \xemph{represents} $\phi$ if for each prone $U\subseteq C$, $FU\in\phi$ if and only if $!:U\to\termo$ is supine. \end{definition}

\begin{lemma} The filter $\A\in (\ds A/\phi)$ represents $\phi$. \end{lemma}

\begin{proof} Any prone downset of $\A$ is equivalent to a downset of the form $\A_U = F_{U\hookrightarrow A}(\A)$ where $U\in \ext A$, and for there downsets $\exists_!(\A_U) = U$. Because $\comb k U\subseteq A\arrow U$ and $\comb i\subseteq U\arrow A$, $U$ is equivalent to $A$ if $U$, and therefore $\comb k U$, is a member of $\phi$.
\end{proof}

\begin{remark} This property helps explain how the fibred locale $F$ can interpret universal quantification inside existential quantification over $C$: $\exists x\oftype C. \forall y\oftype Y.p(x,y)$ is valid if and only if there is a $U\in \phi$ such that $\forall x\oftype U,y\oftype Y.p(x,y)$ and thus we get a simple inclusion. The last two characteristic properties deal with quantification that is more deeply nested.
\end{remark}

We now use these principles to show that arbitrary maps have an object of realizers.

\begin{theorem}[tracking principle] Let $F:\cat F\to\cat H$ be a fibred locale and let $C\in F_A$ be a filter that represents $\phi$, such that Church's rule and the uniformity rule hold. Let $f:X\to X'$ be any arrow, let $s:Y\to X$ be any supine \prodo morphism and let $s:Y'\to X'$ be any supine \prodo morphism. Then there is a $U\in \phi$ such that $(a,b,x)\mapsto(ab,f(x))$ defines a total map $C_U\times Y\to Y'$ that commutes with $s\circ \pi_1$ and $s'$; here $C_U=F_{U\hookrightarrow A}(C)$. \label{tracking}
\[\xymatrix{
C_U \times Y\ar@{.>}[rr]^(.55){(a,b,x)\mapsto (ab,f(x))}\ar@{.>}[d] &  & Y'\ar[d]^{s'} \\
Y\ar[r]_s & X\ar[r]_f & X'
}\]
\end{theorem}

\begin{proof} We pull $s'$ back along $f\circ s$ and get a supine prone map $t = (f\circ s)^*(s'):Z\to Y$ and a projection $t':Z\to Y'$ as a result. Let $u:X\to \top_{FX}$ be the unit of the adjunction $F\dashv \top$ over $X$ and note that $u\circ s: Y\to \top_{FX}$ is still a \prodo morphism, just not a supine one. We apply Church's rule to the sequence $Z\stackrel t\to Y \stackrel{u\circ s}\to \top_{FX}$ and find a supine \prodo morphism $r:W\to \top_{FX}$, such that $(a,x)\in W$ and $(b,x)\in Y$ implies $(ab,f(x))\in Y'$.
We then apply the uniformity rule to the sequence $W \stackrel r\to\top_{FX} \stackrel !\to \termo$ to find a supine \prodo morphism $!:V \to \termo$ such that if $v\in V$ then $(v,x)\in W$. Because $C$ represents $\phi$, $V=C_U$ for some $U\in \phi$. We see that $C_U\times Y\subseteq W$ and we can restrict $r$ to this subobject.
\[ \xymatrix{
& C_U\times Y \ar@{.>}[rrr] \ar@{.>}[d]^r \ar@{}[dr]|<\lrcorner \ar@{.>}[dl]_{(a,b,x)\mapsto (ab,b,x)} &&& C_U \ar@{.>}[d] \\
Z\ar@{.>}[r]^t \ar@{.>}[d]_{t'} \ar@{}[dr]|<\lrcorner & Y\ar[r]^s\ar[d]_{f\circ s} &
X\ar[r]^u \ar[dl]^f & \top_{FX}\ar[r] & \termo \\
Y' \ar@{->}[r]_{s'} & X'
}\]
This is the way the last three characteristic properties imply the realizability of each morphism.
\end{proof}

\begin{remark} `Tracking' is a synonym of realizing a function, hence the name \xemph{`tracking principle'}. \end{remark}

\subsection{Equivalence}
We close this section by proving that the listed properties of realizability fibrations are indeed characteristic, i.e. that a fibred locale that satisfies them is equivalent to the realizability fibration.

\begin{theorem}[characterization theorem] Let $\cat H$ be a Heyting category, let $A$ be an order partial applicative structure in $\cat H$ and let $\phi$ be a combinatory complete external filter of $A$. Each separated fibred locale $F:\cat F\to\cat H$ with a weakly generic filter $C\in F_A$ where Church's rule and the uniformity rule hold and where $C$ represents $\phi$ is equivalent to the realizability fibration.\label{characterization} \end{theorem}

\begin{proof} We find a vertical morphism $m= (\id,m): (\ds A/\phi) \to F$, by sending a family $Y\subseteq A\times X$ of downsets of $A$ to $\exists_{\pi_1|_Y} F_{\pi_0|_Y}$, which is the only way up to isomorphism that a morphism can map the family. Because of theorem \ref{Shanin}, this mapping of objects is essentially surjective.

For each arrow $f:X\to X'$ in $\cat H$, there is an arrow $g:Y\to Y'$ in $\cat DA/\phi$ that $\ds A/\phi$ maps to $g$, if $U = (Y\arrow Ff(Y'))\in \phi$. Noting that $Y$,$Y'$ are families of subobjects of $A$, let $p:Y\to A$ and $p':Y'\to A$ be the projections to $A$. Because $C$ is closed under application, we have is a morphism $(a,b,x) \mapsto (ab,f(x)) C_U\times F_p(C) \to F_p'(C)$. We compose this with the supine map $s':F_p'(C) \to m(Y')$ and then factor it through the supine maps $\pi_1:C_U\times F_p(C) \to F_p(C)$ and $F_p(C)\to m(Y)$, in order to find $m(g)$. That $\pi_1$ is supine follows from the fact that $U\in \phi$ and that $\phi$ represents $U$.

Because of theorem \ref{tracking}, this mapping is full. Because both $F$ and $\ds A\phi$ are faithful, and $Fm\simeq \ds A\phi$, $m$ is faithful. So $(\id,m)$ is an equivalence of fibred categories.
\end{proof}

\section{Axiomatization}
We will now show what the consequences of the characteristic properties of realizability fibrations are for the logic of realizability. We connect the characteristic properties of realizability interpretation to axiom schemas in the internal language of complete fibred Heyting algebras, as far as this is possible, and explain why the remaining properties cannot be expressed internally. But first we make remark \ref{models} precise by defining what it means for a fibred Heyting category to satisfy a proposition.
\begin{definition} Let $F:\cat C\to\cat B$ be a complete fibred Heyting algebra over a category $\cat B$ with finite products. We define its internal language as follows. 
\begin{itemize}
\item The objects of $\cat B$ are types.
\item For each type $X$, the set of terms $T_X$ consists of variable symbols $x,y,z\dots$ in an infinite set of variable symbols $V_X$, and terms of the form $f(t)$ where $f$ is an arrow $Y\to X$ and $t$ is a term of type $Y$.
\item For each type $X$ the set of formulas $\Phi_X$ consist of 
 \begin{itemize}
 \item \xemph{constants} $\top$ and $\bot$;
 \item \xemph{equations} $t=s$ where $t,s$ are terms of type $X$;
 \item \xemph{simple predicates} $R(x)$ where $R\in F_X$;
 \item \xemph{compound predicates} $p\land q$, $p\vee q$, $p\to q$ where $p$ and $q\in \Phi_X$;
 \item \xemph{quantified predicates} $\forall y\oftype Y.p$, $\exists y\oftype Y.p$ where $p\in \Phi_{X\times Y}$ and $y\in V_Y$.
 \end{itemize}
\end{itemize}

The interpretation of these terms and formulas is defined as follows. 
\begin{itemize}
\item For terms we let $\db x = \id_X$ if $x\in V_X$, and $\db{f(t)} = f\circ \db t$.
\item For formulas:
 \begin{itemize}
 \item for equations let $\delta:X\to X^2$ be the diagonal map; \[ \db{t=u} \simeq F_{(\db t,\db u)}(\im\delta(\top)) \]
 \item for simple predicates $\db{R(t)} = F_{\db t}(R)$;
 \item compound predicates and constants:
 \begin{align*}
 \db{p\land q} &\simeq \db p\land\db q & \db\top &\simeq \top_X \\
 \db{p\vee q} &\simeq \db p\vee\db q & \db\bot &\simeq \bot_X \\
 \db{p\to q} &\simeq \db p \to \db q
 \end{align*}
 \item for quantified predicates let $i:I\times X\to I$ be the first projection and $x:I\times X \to X$ the second; 
 \[ \db{\forall x\oftype X.p} \simeq \forall_{i}(\db p) \quad \db{\exists x\oftype X.p} \simeq \exists_{i}(\db p)  \]
 \end{itemize}
\end{itemize}

A well formed formula without free variables is a \xemph{proposition}.  We say $p$ is \xemph{valid}, that it \xemph{holds} in $F$ and that $F$ \xemph{satisfies} it, and write $F\models p$ if $\db p \simeq \top_\termo$.
\end{definition}

\begin{remark} Even if we assume the complete fibred Heyting algebra is antisymmetric, we need to distinguish the equality of predicates from the equality in the internal language and we we already have the isomorphism symbol $\simeq$ to do that. We assume that the interpretations are isomorphism classes of predicates.
\end{remark}

\begin{remark} The following convention makes some formulas more readable: if $Y\subseteq X$ in $\cat H$, then $Y(x) = (\exists y\oftype Y.x=y)$. \end{remark}

\subsection{Theory of realizability}
We express some of the characteristic properties of realizability fibrations in the internal language of a complete fibred Heyting algebra. Not all characteristic properties can be expressed in this way, therefore we select some complete fibred Heyting algebras that have the inexpressible properties build into them.


\begin{definition} Let $A$ be a partial applicative structure and let $\phi$ be a combinatory complete external filter of $A$. A \xemph{candidate} for $(A,\phi)$ is a complete fibred Heyting algebra $F:\cat F\to\cat H$ with a weakly generic object $C\in F_A$ such that if $U\in\ext A$ and the unique map $!:F_{U\hookrightarrow A}(C)\to \termo$ is supine, then $U\in \phi$. 
\end{definition}

\begin{example} The realizability fibration for any combinatory complete external filter $\chi\subseteq \phi$ is a candidate. \end{example}

We now introduce a handful of axioms and axiom schemas that correspond to characteristic properties of the realizability fibrations.

\newcommand\axiom\mathrm
\begin{definition}[surjection schema] The \xemph{surjection schema} is the following formula for each regular epimorphism $e:X\to Y$ in $\cat H$.
\[ \axiom{S}_e \simeq \forall y\oftype Y.\exists x\oftype X.e(x)=y \]
\end{definition}

\begin{lemma} A complete fibred Heyting algebra satisfies the surjection schema if and only if it is separated. \end{lemma}

\begin{proof} Let $e:X\to Y$ be a regular epimorphism in $\cat H$, let $p:X\times Y\to X$ be the first projection and let $q:X\times y\to X$ be the second projection. Then $\db{\axiom S_e} = \forall_!(\exists_q(F_{(e\circ p,q)}(\exists_\delta(\top))))$. We can simplify this using the Beck-Chevalley condition, because the following square is a pullback:
\[ \xymatrix{
X \ar[r]^e \ar[d]_{(\id,e)} \ar@{}[dr]|<\lrcorner & Y \ar[d]^\delta \\
X\times Y \ar[r]_{(e\circ p,q)} & Y^2 
}\]
This means $\db{\axiom S_e} = \forall_!(\exists_e(F_e(\top)))$. Reindexing preserves finite limits so $F_e(\top) = \top$, and $\top \leq \forall_!(p)$ if $\top = F_!(\top) \leq p$, so $F\models \db{\axiom S_e}$ if $\top = \exists_e(\top)$. Therefore being separated is equivalent to satisfying the surjection schema.
\end{proof}

\begin{definition}[filter axioms] Let $D$ be the domain of the application operator of $A$, let $a:D\to A$ be the application operator. Let $\set{\leq}$ be the ordering and let $e:\roso \to A^2$ be in inclusion. Let $\pi_0$ and $\pi_1$ be the first and second projections $A^2 \to A$.
\begin{align*}
\axiom F_1 &\simeq \forall x\oftype A.\forall y\oftype A. D(x,y)\land C(x)\land C(y) \to C(a(x,y)) \\
\axiom F_2 &\simeq \forall x\oftype A.\forall y\oftype A.\roso(x,y)\land C(x)\to C(y)
\end{align*}
\end{definition}


\begin{lemma} A candidate $(F,C)$ for $(A,\phi)$ satisfies $\axiom F_1$ and $\axiom F_2$ if and only if $C$ is a vertical filter. \end{lemma}

\begin{proof} Trivial. \end{proof}

\begin{definition}[modified Church's thesis] For arbitrary candidates, \xemph{modified Church's thesis} is the following schema. Let $Y\hookrightarrow A\times Z$ and let $X\hookrightarrow A\times Y$ be families of downsets of $A$, then:
\[ \axiom{MCT}_{X,Y,Z}\simeq \left(\begin{array}{l} 
\forall z\oftype Z.(\forall y\oftype A. Y(y,z)\land C(y) \to \exists x\oftype A. X(x,y,z)\land C(x)) \to \\
\exists w\oftype A. C(w)\land \forall y\oftype A.Y(y,z) \to D(w,y) \land X(a(w,y),y,z)) \end{array}\right)
\]
\end{definition}

\begin{remark} Starting with extended Church's thesis, the modifications involve replacing $\Kone$ by $A$,  adding a parameter $z$, and only allowing `prone predicates', i.e. predicates of the form $X(x,y,z)\land C(x))$ for $X\in \ds A_Y$ to appear on the right of the antecedent. Besides that, we need $\axiom F_1$ to show that $Y(y,z)\land C(x)$ implies $X(a(w,y),y,z))\land C(a(w,y))$ in the consequent. In the case of \xemph{Kleene's first model} the relevant order of $\Kone$ is discrete, and by Shanin's rule any predicate $P(x,y,z)\in \Phi_{\N\times Y}$ is equivalent to $\exists n\oftype\N. X(n,y,z)\land n_0=x$. In this way we get extended Church's thesis back in the case of recursive realizability, see example \ref{Kone}.
\end{remark}\comment{effective topos}

\begin{lemma} A candidate $(F,C)$ for $(A,\phi)$ where $C$ is a vertical filter, satisfies Church's rule if and only if it satisfies $\axiom{MCT}$. \end{lemma} 

\begin{proof} Assume Church's rule. Let $Y\hookrightarrow A\times Z$ and let $X\hookrightarrow A\times Y$ be families of downsets of $A$. We let $Z' = \db{(\forall y\oftype A. Y(y,z)\land C(y) \to \exists x\oftype A. X(x,y,z)\land C(x))}_{(Z,z)\mapsto \id_Z}$, $Y' = F_{Y\hookrightarrow A\times Z}(C\times Z')$ and $X' = F_{X\hookrightarrow A\times Y}(C\times Y')$, so that we get two \prodo morphisms $f:X'\to Y'$ and $g:Y'\to Z'$, and $f$ is supine. We can apply Church's rule to find a supine \prodo morphism $e:W' \to Z'$ such that for all $(w,z)\in W'$ and $(y,z)\in Y'$, $wy\converges$ and $(wy,y,x)\in X'$.

Let $V = \set{(w,z)\in A\times Z | \forall y\in A. (y,z)\in Y, wy\converges, (wy,y,z)\in X }$ and let $V' = F_{V\hookrightarrow A\times Z}(C\times Z')$, to get a \prodo morphism $d:V' \to Z'$.  Now $e$ factors through $d$, making $d$ a supine morphism. For this reason we have $F\models \forall z\oftype Z. Z'(z) \to \exists w\oftype A. C(w)\land V(w,z)$. The predicate $V$ is equivalent to $\forall y\oftype A.Y(y,z)\to D(w,y)\land X(a(w,y),y,z)$, so Church's rule implies $\axiom{MCT}$.

Assume $\axiom{MCT}$. Let $f:X'\to Y'$ and $g:Y'\to Z'$ be \prodo morphisms such that $f$ is supine. We let $X=FX'$, $Y=FY'$, $Z=FZ'$, and find that $V = \set{(w,z)\in A\times Z | \forall y\in A. (y,z)\in Y, wy\converges, (wy,y,z)\in X }$ and $V' = F_{V\hookrightarrow A\times Z}(C\times Z')$ provide a supine \prodo morphism $V'\to Z'$ such that for all $(w,z)\in V'$ and $(y,z)\in Y'$, $wy\converges$ and $(wy,y,z)\in W'$, according to $\axiom{MCT}_{X,Y,Z}$ and $\axiom{F}_1$. In this way Church's rule is satisfied. \end{proof}

\begin{definition}[uniformity principle for prone arrows] The \xemph{uniformity principle} for prone arrows $\axiom{UP}_{\mathrm{prone}}$ is the following schema. For each family of downwards closed sets $X\hookrightarrow A\times Y$ and each $g:Y\to Z$, 
\[ \axiom{UP_{X,Y,Z}} \simeq \left\{\begin{array}{ll} 
\forall z\oftype Z.(\forall y\oftype Y. g(y)=z \to \exists x\oftype A. X(x,y)\land C(x)) \to \\
\exists x\oftype A. C(x)\land \forall y\oftype Y. g(y)=z \to X(x,y)) \end{array}\right.
\]
\end{definition}

\begin{remark} The uniformity principle originally applied to power objects, but since we don't have those in all realizability models we use the fibres of prone arrows instead. Once again the axiom is made parametric. \end{remark}

\begin{lemma} A candidate satisfies $\axiom{UP}_{\mathrm{prone}}$ if and only if it satisfies the uniformity rule. \end{lemma}

\begin{proof} Assume the uniformity rule. Let $Z' = \db{(\forall y\oftype Y. f(y)=z \to \exists x\oftype A. X(x,y)\land C(x))}_{(Z,z)\mapsto \id_Z}$, $Y'=F_f(Z')$ and $X' = F_{X\hookrightarrow A\times Y}(C\times Y')$, which gives us a \prodo morphism $f:X'\to Y'$ and a prone arrow $g':Y'\to Z'$, and $f$ is supine thanks to our choice of $Z'$. 

We let $W = \set{ (a,z)\in A\times Z | \forall y\in Y. g(y)=z \to (a,y)\in X }$, and $W' = F_{W\hookrightarrow A\times Z}(C\times Z')$. The \prodo morphism $W'\to Z'$ is supine because according to the uniformity rule, there is a supine \prodo morphism to $Z'$ that factors through $W'$. In this way we get $F\models \forall z.Z'(z) \to \exists a\oftype A. C(a)\land W(a,z)$, and $W(a,z)$ is equivalent to $\forall y\oftype Y. g(y)=z \to X(a,x)$. Therefore, the uniformity rule implies $\axiom{UP}_{\rm prone}$

Assume $\axiom{UP}_{\rm prone}$. Let $f:X'\to Y'$ be a supine \prodo morphism, and let $g':Y'\to Z'$ be prone. Let $X=FX'$, $Y=FY'$ and $Z=FZ'$. Let $W = \set{ (a,z)\in A\times Z | \forall y\in Y. g(y)=z \to (a,y)\in X }$, then $\axiom{UP}_{\rm prone}$ tells us that $\forall z\oftype Z. Z'(z)\to \exists a\oftype A. C(a)\land W(a,x)$. Also, the preimage $(\id_A\times g)^{-1}(W)$ is a subobject of $X$ because $\forall_{\id_A\times g} \dashv (\id_A\times g)^{-1}$. Let $W' = F_{W\hookrightarrow A\times Z}(C\times Z')$, then for each $(a,z)\in W'$ and $y\in Y'$ such that $g'(y) = z$ we have $(a,y)\in X'$; that means the candidate satisfies the uniformity rule.
\end{proof}

\begin{definition}[intersection] The \xemph{intersection schema} is for each $U\in \ext A$,
\[ \axiom I_U \simeq \exists x\oftype A.C(x)\land U(x) \]\label{intersection}
\end{definition}

\begin{lemma} In a candidate $(F,C)$ for $(A,\phi)$, $C$ represents $\phi$ if and only if $\axiom I_U$ holds for all $U\in\phi$. \end{lemma}

\begin{proof} Let $\chi$ be the set of those $U\in \ext A$ for which $(F,C)$ satisfies $\axiom I_U$. The definition of candidates gives us $\chi\subseteq \phi$, and the $\phi$-completeness schema gives us $\phi\subseteq \chi$. \end{proof}

\begin{theorem} A candidate $(F,C)$ for $(A,\phi)$ is equivalent to the realizability fibration if and only if it satisfies $\axiom S$, $\axiom F_1$, $\axiom F_2$, $\axiom{MCT}$, $\axiom{UP}$ and $\axiom{R}$. \end{theorem}

\begin{proof} This is a straightforward consequence of the lemmas of this subsection. \end{proof}

\comment{
\begin{corollary} Let $F:\cat F\to\cat H$ be any separated complete fibred Heyting algebra, and let $C\in FA$ be a combinatory complete filter that is a weakly generic object of $F$. If $\axiom{MCT}$ and $\axiom{UP}$ are valid, then there is an external filter $\phi$ such that $(F,C)$ is equivalent to $(\ds A/\phi,\A)$. \end{corollary}

\begin{proof} Let $\phi = \set{ U\in \ext A| F\models \exists x\oftype A.C(x)\land U(x)}$. \end{proof}
}

We have expressed several characteristic properties of realizability fibrations in the internal language. The next subsection concerns the expressibility of the remaining ones.

\subsection{Inexpressible} \label{inexpressible}
The remaining characteristic properties are the properties that $C$ is weakly generic and that it represents $\phi$. Here we will show that these properties cannot be expressed in the internal language of the realizability fibration.

\begin{lemma} Let $F:\cat F \to\cat H$ be a complete fibred Heyting algebra, let $F^2: \cat F\times_\cat H\cat F \to\cat H$ be the fibred category where $F^2_X = (F_X)\times (F_X)$. The fibrewise diagonal $\Delta: F \to F^2$ is a logical morphism that does not preserve weakly generic objects, unless $F\cong\id_{\cat H}$. \end{lemma}

\begin{proof} The morphism $\Delta:F\to F^2$ is a fibred monomorphism of Heyting algebras that commutes with both adjoints of the reindexing functors and therefore is a Heyting morphism. If $C$ is a weakly generic object of $F$, then $\Delta C$ can only reach other objects in the image of $\Delta$. If every object of $\cat F\times_\cat H\cat F$ is isomorphic to $\Delta X$ for some $X\in \cat F$, then we must conclude that all objects of all fibres are isomorphic and that $F$ and $\id_X$ are equivalent fibred categories. If not $F\cong \id_{\cat H}$, then $\Delta$ cannot preserve weakly generic objects. \end{proof}

\begin{corollary} Weak genericity is not expressible in internal logic. \end{corollary}

Now we turn to representability with another argument.

\begin{lemma} Let $\phi$ and $\chi$ be external filters of $A$ such that $\phi\subseteq \chi$. The inclusion induces a vertical Heyting morphism $\ds A/\phi\to\ds A/\chi$. \label{Heyting morphisms}
\end{lemma}

\begin{proof} The morphism $(\id,m):\ds A/\phi\to\ds A/\chi$ is the identity on objects and this means $m(U\arrow V) = m(U)\arrow m(V)$ and $m(\forall_f(U)) = \forall_f(m(U))$ and so on. \end{proof}

As a consequence, $\ds A/\chi$ satisfies every proposition that $\ds A/\phi$ satisfies. We simply cannot tell the difference between the two based on theorems alone. What we can say is that $\ds A/\phi$ does not satisfy some of the propositions that $\ds A/\chi$ does, which leads us to the following solution for expressing representability.

\begin{lemma} Let $F$ be a complete fibred Heyting algebra and $C\in F_A$ a filter. The filter $C$ represents $\phi$ if for all $U\in\ext A$, $F\models \axiom I_U$ if and only if $U\in \phi$.
\end{lemma}

\begin{proof} Trivial. \end{proof}

\begin{remark} The Heyting morphism $\ds A/\phi\to\ds A/\chi$ is one of the reasons for working with external filters. We cannot find a \xemph{theory or realizability} whose models are equivalent to a particular realizability fibration, but we can limit the variance to a difference of external filters.
\end{remark}

\subsection{Conclusions}
We have reached the ultimate goal of this chapter.

\begin{theorem}[axiomatization] Let $F:\cat F\to\cat H$ be a separated complete fibred Heyting algebra over a Heyting category $\cat H$ and let $C\in F_A$ be a combinatory complete vertical filter that is weakly generic. If $(F,C)$ satisfies the schemas $\axiom{MCT}$ and $\axiom{UP}$ then it is equivalent to $(\ds A/\phi,\A)$ for some external filter $\phi$.\label{axiomatization}
\end{theorem}

\begin{proof} Let $\phi$ be the set of all $U\in \ext A$ such that $!:F_{U\hookrightarrow A}(C) \to \top$ is supine. Because $C$ is combinatory complete, so is $\phi$. Because of the satisfies schemas, $(F,C)$ has all of the characteristic properties of $(\ds A/\phi,\A)$ and therefore is equivalent.
\end{proof}

At this point we see that realizability is almost completely axiomatized by the axiom Kleene had in mind when he defined recursive realizability \cite{MR0015346} and the axiom that Troelstra used to extend realizability to higher order logic \cite{MR0363826}. We can now say that a wide variety of realizability structures has to satisfy the schemas of modified Church's thesis and the uniformity principle.

\section{Further thoughts}
We write up some loose ideas about realizability models to conclude this chapter.

\subsection{Markov's principle} \label{Markov}
Kleene's first model $\Kone$ (see example \ref{Kone}) exists in any Heyting category $\cat H$ with a \xemph{natural number object} $\N$, (see definition \ref{nno}). Using this model we can construct the \xemph{effective fibration} $\ds\Kone/\Kone$, and this gives us a form of recursive realizability in $\cat H$. The characterization results in this chapter now tell us the following facts about these effective fibrations.
\newcommand\nno{\mathbf N}
\begin{itemize}
\item They are separated complete fibred Heyting algebras, with a weakly generic object $\nno$ in the fibre over $\N$ such that for all $U\subseteq \N$, the unique map $!:(\ds\Kone/\Kone)_{U\hookrightarrow \N}(\nno)\to \termo$ is supine if and only if $U$ is inhabited.
\item They satisfy a version of extended Church's thesis for $\neg\neg$-stable predicates, and a uniformity principle.
\item If $\cat H$ is Boolean, then prone arrows are precisely the arrows whose fibres are $\neg\neg$-stable; therefore prone subobjects are $\neg\neg$-stable subobjects. This not only allows us to recover a version of Shanin's principle and extended Church's thesis, but also \xemph{Markov's principle}, which is the following schema for $P\in (\ds \Kone/\Kone)_\N$:
\[ (\forall n\oftype \N.\nno(n)\to P(n)\vee \neg P(n)) \land \neg(\forall n\oftype\N.\nno(n)\land P(n)) \to \exists n\oftype\N.\nno(n)\land P(n) \]
The proof relies on a much stronger principle, namely that the schemas $\neg\neg p\to p$ and $p\vee \neg p$ are valid, because for each proposition $p$, the set of realizers $\db p$ is either inhabited or empty according to the internal language of $\cat H$.
\end{itemize}

Markov's principle can help characterize the effective fibration over a Boolean category with a projective terminal object.

\begin{lemma} Let $F$ be a separated complete fibred Heyting algebra $F$ over a two valued Boolean category $\cat H$ with a natural number object $\N$ and a projective terminal object. Let $\nno \in F_\N$ be a weakly generic object, that is also a combinatory complete filter, and let $(F,\nno)$ satisfy Church's rule, the uniformity rule, and Markov's principle. Then $F\cong \ds\Kone/\Kone$ or $F\cong \id_{\cat H}$. \end{lemma}

\begin{proof} Our characterization theorem tells us that $F\cong \ds \Kone/\phi$, where $\phi$ contains all subobjects of $\N$ that have a recursively decidable inhabited subobject. But because of the projective terminal object, every inhabited subobject has a global section. This global section is a recursively decidable inhabited subobject, and therefore $\phi$ contains all inhabited subobjects of $\N$. Because of two-valuedness, every subobject is either inhabited or empty, and this leaves two options: $\emptyset\not\in \phi$ and $\ds\Kone/\phi = \ds\Kone/\Kone$ or $\emptyset \in \phi$. In that last case, the fibration collapses, because $\emptyset U\converges$ and $\emptyset U=\emptyset\subseteq V$ for all $U,V\in (\ds\Kone)_X$ for all $X$ in $\cat H$. The result is equivalent to the terminal fibred category $\id_{\cat H}$.
\end{proof}
 
In Boolean categories where the terminal object is not projective, inhabited subobjects of the natural numbers may have no decidable subobjects. Therefore, Markov's principle is not strong enough to characterize effective fibrations.

\subsection{Complete fibred partial applicative lattices}
The realizability fibration is a quotient of the complete fibred partial applicative lattice $\ds A$. We can construct such lattices in other ways, which we will explore here.

\begin{example} We define a \xemph{partial applicative lattice} $L$ inside $\cat H$ as follows.
\begin{itemize}
\item A partial applicative lattice $L$ has a binary join operator called $\vee$. For $f,g:X\to L$, we let $f\vee g = \mathord\vee\circ(f,g)$. We let $f\leq g$ if $f\vee g=g$.
\item We interpret completeness as a schema that says that for each map $f:X\to Y$ and $g:X\to L$ there is both a least $h:X\to L$ such that $g\leq h\circ f$, namely $\sup_f g$, and a greatest $k:X\to L$ such that $k\circ f \leq h$, namely $\inf_f g$.
\item A partial applicative lattice $L$ has a partial application operator $a:L^2\partar L$. For $f,g:X\to L$, we let $fg\converges$ if $f\times g:X\to L^2$ factors through $\dom a$ and let $fg = a\circ (f\times g)$ in that case. We demand that this operations preserves binary joins, so it is an order partial applicative structure, and that is preserves least upper bounds: if $f(g\circ h)\converges$, then $(\sup_h f)g = \sup_h(f(g\circ h))$; if $(f\circ g)h\converges$, then $f\sup_g(h) = \sup_g((f\circ g)h)$.
\end{itemize}

Partial applicative lattices are lattices because they have binary meets: let $B = \set{(x,y,z)\in L| x\leq y \land x\leq z}$ and consider $\land = \sup_{\pi_{12}}(\pi_0):A^2\to A$ where $\pi_0:B\to A$ and $\pi_{12}:B\to A^2$ are the projections. By definition $\land(x,y)\leq z$ if and if $z\leq x$ and $z\leq y$. Similarly, partial applicative lattices have an \xemph{arrow operator} $\arrow$, that is left adjoint to application: let $C = \set{(x,y,z)\in L| xy\converges \land xy\leq z}$ and $\arrow = \sup_{\pi_{12}}(\pi_0):A^2\to A$. This time $x\leq y\arrow z$ if and only if $xy\converges$ and $xy\leq z$. We have to require that $\inf$ exists, because some Heyting categories are too weak to construct these from $\sup$.

For each partial applicative lattice $L$ in $\cat H$, we let $\cat L$ be the category where object are arrows into $L$, and a morphism $f: g \to h$ is an arrow $f:\dom g\to \dom h$ that satisfies $g\leq h\circ f$. Now $\dom:\cat L\to\cat H$ is a complete fibred partial applicative lattice.

If $\cat H$ is a topos and $A$ is an order partial applicative structure, then $DA$, the object of downsets of $A$ in $\cat H$, is a complete partial applicative lattice. The related fibration is equivalent to the realizability fibration. This gives us an alternative way to construct realizability fibrations over toposes, which is exhibited in \cite{MR2265872}. \label{cpals}
\end{example}

\begin{proposition} Every partial applicative lattice is an order partial combinatory algebra. \end{proposition}

\begin{proof} For each partial combinatory function $f:L^n\partar L$ let $r_f = \sup\db f$: this is a realizer for $f$. \end{proof}

\begin{corollary} Let $L$ be a complete partial applicative lattice. Let $\comb k  = r_{(x,y)\mapsto x}$ and $\comb s = r_{(x,y,z)\mapsto xz(yz)}$. A filter $\phi$ is combinatory complete if and only if $\comb k\in \phi$ and $\comb s\in \phi$. \end{corollary}

\begin{proof} See Feferman \cite{MR0409137}, for a proof that these two combinators generate realizers for all partial combinatory arrows. This takes care of the `if' part. For the `only if' part, consider that we defined $r_f$ to be the greatest realizer of these functions. A filter that is combinatory complete, contains realizers for $(x,y)\mapsto x$ and $(x,y,z) \mapsto xz(yz)$; because of upward closure it also contains $\comb k$ and $\comb s$. \end{proof}

\begin{remark} Note that in any topos $\comb k$ and $\comb s$ are global elements of their complete partial applicative lattices. \end{remark}

\begin{example} We consider what structure we need on $(A,\leq)$ to make $\ds A$ a complete partial applicative lattice. The domain of the fibrewise application operator is a problem, so we assume that there is a downward closed $D\subseteq A^2$ such that for all $U,V\in \ds A$, $UV\converges$ if and only if $U\times_X V\subseteq D\times X$. Let $P_0\in \ds_D$ be $\set{(x,y,z)\in A\times D| x\leq y}$ and $P_1 = \set{(x,y,z)\in A\times D| x\leq z}$. Now $P_0P_1\converges$, and $T=P_0P_1$ defines a relation $T:D\nrightarrow A$. We can reconstruct the fibred partial application operator from this relation. For all $X$ in $\cat H$ and $U,V\in\ds A_X$, $UV = \set{(a,x) |\exists (b,c)\in D. (b,x)\in U, (c,x)\in V, (b,c,a)\in T }$, because of the preservation properties of the application operator.

Without an application operator, we cannot construct partial combinatory functions, but there are \xemph{partial combinatory relations} $A^n\nrightarrow A$, or partial combinatory families of downsets $\ds A_{A^n}$. Fibred filters can then be called combinatory complete, if they contain these relations.\end{example}

\subsection{Necessity of combinatory completeness}
The soundness of realizability models relies on the combinatory completeness of the partial applicative structures, see proposition 1.2.2 in \cite{MR2479466}. We will give our own account of this fact here.

\begin{theorem} Let $A$ be an order partial applicative structure $A$ and let $\phi\subseteq \ext A$ be an external filter. For all  $U,V\in A$, let $U\ll V$ if $U\arrow V\in \phi$. If $\ds A$ ordered by $\ll$ is a complete fibred Heyting category with $\arrow$ as Heyting implication, then $\phi$ is combinatory complete.\label{why combinatory completeness}
\end{theorem}

\begin{proof} Consider the predicate $E = \exists_\delta(\A)$, where $\delta$ is the diagonal map. It has the peculiar property that $E(x,y)E(x',y') \subseteq E(xx',yy')$ when $xx'\converges$ and $yy'\converges$. For any partial arrow $f:A^n\partar A$, we see that $\forall \vec x,\vec y\oftype \dom f.E(x_1,y_1) \to \dotsm \to E(x_n,y_n) \to E(f(\vec x),f(\vec y))$ is $\db f$. 

Note that for all $\vec x$ and $\vec y\in \dom f$ in such a way that $x_i=y_i$ for all $i<n$, we need $z\in \db{f}$ to satisfy $zx_1\dotsm x_{n-1}\converges$ in order to realize $E(x_n,y_n) \to E(f(\vec x),f(\vec y))$. This explains the extra condition on the set of realizers $\db f$ of a partial arrow $f$.

Now we use induction on the class of partial combinatory arrows.
\begin{itemize}
\item Let $f(\vec x) = x_i$ for all $\vec x\in A^n$ with $n\in\N$ arbitrary. The realizers of $f$ realize the following formula: 
\[ \forall \vec x,\vec y\in A^n. E(x_1,y_1) \to \dotsm \to E(x_n,y_n) \to E(x_i,y_i)\] But this is valid, and hence $\db f\in \phi$.
\item Let $f(\vec x) = g(\vec x)h(\vec x)$ and let $U$ be the domain of this function. Assume that $\db g$ and $\db h$ are in $\phi$. Now $E(g(\vec x),g(\vec y))E(h(\vec x),h(\vec y)) \subseteq E(f(\vec x),f(\vec y))$; because $\db g\in \phi$, 
\[ \forall \vec x,\vec y\in U. E(x_1,y_1) \to \dotsm \to E(x_n,y_n) \to E(h(\vec x),h(\vec y)) \to E(f(\vec x),f(\vec y))\] because $\db h\in \phi$, \[ \forall \vec x,\vec y\in U. E(x_1,y_1) \to \dotsm \to E(x_n,y_n) \to E(h(\vec x),h(\vec y)) \] Using modus ponens, we see that $\forall \vec x,\vec y\in U. E(x_1,y_1) \to \dotsm \to E(x_n,y_n) \to E(f(\vec x),f(\vec y))$ is valid, and hence $\db f\in \phi$.
\end{itemize}
Since $\phi$ represents all partial combinatory functions, it is a combinatory complete external filter.
\end{proof}

\begin{remark} The order partial applicative structure is built into this construction, which means that this proof does not apply to forms of realizability that are not based on them. Relaxing the combinatory completeness condition may still result in interesting regular and coherent categories, because those have no implications or universal quantifications in the internal language.
\end{remark}

%% file: RealizabilityCategories.tex
In this chapter we develop realizability in a higher categorical setting. We characterize realizability fibrations in relation to other fibred locales, show how to construct regular and exact categories out of fibred locales, and apply these constructions to realizability fibrations. The resulting \xemph{realizability categories} and the regular functors between them are the main subject of this chapter. In particular, we will clarify and generalize the following theorems about realizability categories.
\begin{itemize}
\item In examples of realizability categories, the base category is a reflective subcategory. We will show that this is a characteristic property of realizability categories.
\item There is an equivalence between Longley's \xemph{applicative morphisms} and some category of regular functors between realizability categories. We generalize applicative morphisms to our filtered realizability models, and prove a similar equivalence.
\item It is natural to consider geometric morphisms between realizability toposes, even when the direct image functor does not preserve regular epimorphisms. We consider how to do this.
\item The \xemph{effective topos} has enough projectives and the subcategory of projectives is closed under finite limits. This implies that it is the ex/lex completion of its category of projective objects. This is a consequence of the weak genericity property of the realizability fibration combined with the axiom of choice in the topos of sets. Unfortunately, this result won't generalize to general realizability categories. We will show why.
\end{itemize}

\section{Categories of fractions}\label{categories of fractions}
In this section we discuss the construction of regular and exact categories out of fibred locales. In both cases we characterize them as \xemph{categories of fractions} \cite{MR0210125} where a suitable class of morphisms is formally inverted. We discuss conditions under which the category of fractions is a subcategory. Ultimately we will show that the construction of a regular category out of a fibred locale is left biadjoint to the 2-functor that turns regular categories into fibred locales, and that the exact completion construction is left biadjoint to the forgetful functor to regular categories.

Our approach relies on the category-of-fractions constructions. Other constructions can be found in \cite{MR1674451}.

\subsection{Calculus of fractions} 
In a localization a class of morphisms of a category is inverted, like in the ring of rational numbers the non-zero integers are inverted.

\begin{definition} Let $W$ be a set of morphisms of $\cat C$.
The \xemph{localization} is a functor $Q:\cat C \to \cat C[W^{-1}]$, such that 
\begin{itemize}
\item For each $w\in W$, $Qw$ is an isomorphism.
\item For any functor $F:\cat C\to \cat D$ that sends all $w\in W$ to isomorphisms, there is a functor $G:\cat C[W^{-1}]\to \cat D$ and a natural isomorphism $GQ\to F$.
\item For any pair $G,H:\cat C[W^{-1}] \to\cat D$ and every natural transformation $\eta:GQ\to HQ$ there is a unique natural transformation $\theta:G\to H$ such that $\theta Q=\eta$.
\end{itemize}
\end{definition}

If the class of morphisms to be inverted \xemph{admits a calculus of fractions}, then there is a simple construction for the localization.

\begin{definition} A set of morphisms $W$ of $\cat C$
\xemph{admits a calculus of right fractions} if
 \begin{enumerate}
 \item $W$ contains all identities and is closed under composition.
 \item for every $w:X\to Y$ in $W$ and $f:Y'\to Y$ in $\cat C$, there are $w':X'\to Y'$ in $W$ and $f':X'\to X$ such that $f\circ w'=w\circ f'$.
\[ \xymatrix{ X' \ar@{.>}[r]^{f'}\ar@{.>}[d]_{w'} & X \ar[d]^w \\ Y' \ar[r]_f & Y } \]
 \item for every parallel pair $f,g:X\to Y$ and every $w:Y\to Z$ such that $w\circ f=w\circ g$, there is a $v:W\to X$ in $W$ such that $f\circ v=g\circ v$.
\[ \xymatrix{ W \ar@{.>}[r]^v & X \ar@<1ex>[r]^f\ar@<-1ex>[r]_g & Y \ar[r]^w & Z } \]
 \end{enumerate}\label{calculus of right fractions}

Dually, a set of morphisms $W$ admit a calculus of \xemph{left} fractions, if $W^{op}$ admits a calculus of right fractions for $\cat C^{op}$.
\end{definition}

For a class of morphisms that admits a calculus of right fractions we can construct a localization by using equivalences of a kind of span as morphisms.

\newcommand\fun{\Cat{fun}}
\begin{definition} Let $W$ be a set of morphisms in a category $\cat C$ that admits a calculus of right fractions. A \xemph{functional span} is a span $(F,s:F\to S,t:F\to T):S\to T$ such that $s\in W$. Two functional spans $(F,s,t)$ and $(F',s',t')$ are equivalent, if there are $f:G\to F$ and $f':G\to F'$ such that $s\circ f=s'\circ f'\in W$ and $t\circ f=t'\circ f'$. A \xemph{functional relation} is an equivalence class of functional spans.

Let $\alpha:X\to Y$, $\beta:Z\to X$ and $\gamma:Z\to Y$ be functional spans, then we define $\alpha\circ\beta = \gamma$ if for all spans $(A,a,a')\in\alpha$, $(B,b,b')\in\beta$, there is a $(C,c,c')\in\gamma$ and a span $f:C\to A$, $g:C\to B$ such that $a\circ f = c$, $a'\circ f=b\circ g$ and $b'\circ g = c'$
\[\xymatrix{
&& \ar[dl]^f \ar@/_2ex/[ddll]_{c} C\ar[dr]_g \ar@/^2ex/[ddrr]^{c'} \\
& A\ar[dl]^a \ar[dr]_{a'} && B\ar[dl]^b \ar[dr]_{b'} \\
X && Y && Z
}\]
The second condition on the category of fraction ensures existence and the third ensures uniqueness, see \cite{MR0210125}.

We let $\fun(\cat C,W)$ be the category with the same objects as $\cat C$, but where morphisms are functional relations. We let $Q:\cat C \to\fun(\cat C,W)$ be the functor that satisfies $QX=X$ for all objects and $Qf$ is the equivalence class that contains $(\id_X,f)$ for all morphisms $f:X\to Y$.
\end{definition}

\begin{lemma} The functor $Q:\cat C\to\fun(\cat C,W)$ has the universal property of a localization. Moreover, if $\cat C$ has finite limits, so has $\fun(\cat C,W)$ and $Q$ preserves them. \label{fun is le}\end{lemma}

\begin{proof} See 
\cite{MR0210125}.
\end{proof}


It is nice to know that $\fun(\cat C,W)$ is equivalent to a subcategory of $\cat C$ sometimes, especially considering that the category-of-fractions construction does not preserve local smallness of large categories.

\begin{definition}[coarse objects] Relative to a set of morphisms $W$ an object $X$ of $\cat C$ is \xemph{coarse} if for every $w:Y\to Z$ in $W$ and $x:Y\to X$ in $\cat C$ there is a unique $y:Z\to X$ such that $x=y\circ w$.
\[ \xymatrix{
Y \ar[d]_w \ar[r]^x & X \\
Z \ar@{.>}[ur]_y
}\]
We say $\cat C$ \xemph{has enough coarse objects} if for every object $Y$ there is a coarse object $X$ and a $w:Y\to X$ in $W$. \end{definition}

\begin{lemma} If there are enough coarse objects, $\fun(\cat C,W)$ is equivalent to the subcategory of coarse objects of $\cat C$. \label{coarse} \end{lemma}

\begin{proof} Let $\cat C_{\rm coarse}$ be the full subcategory of coarse objects. The functor $Q:\cat C_{\rm coarse}\to\fun(\cat C,W)$ is essentially surjective, because there are enough coarse objects. It is full because for every functional span $(s:F\to S,t:F\to T)$ where $T$ is coarse there is an $f:S\to T$ such that $t=f\circ s$. It is faithful because if $f,g:S\to T$ are mapped to equivalent functional spans $(\id_S,f)\sim (\id_S,g)$, then there is a $w:F\to S$ in $W$ such that $h = f\circ w=g\circ w$, and the unique factorization property of $T$ then forces $f=g$. Therefore $Q$ is an equivalence of categories.\end{proof}

\begin{remark} If there are enough coarse objects, then they automatically form a reflective subcategory. If we choose for each object $X$ an arrow $X\to X'$ in $W$ such that $X'$ is coarse, then because of the unique factorization property of coarse objects, there is a unique endofunctor that turns our choice of arrows into a natural transformation. Any two choice functions are isomorphic because of the unique factorizations. In this way we get a left adjoint to the inclusion of coarse objects. \label{coarse reflection}
\end{remark}

The dual version works just as well.

\begin{definition}[fine objects] Relative to $W$ an object $X$ of $\cat C$ is \xemph{fine} if for every $w:Y\to Z$ in $W$ and $y:X\to Z$ there is a unique $x:X\to Y$ such that $y=w\circ x$.
\[ \xymatrix{
& Y \ar[d]^w \\
X \ar[r]_y\ar@{.>}[ur]^x & Z
}\]
We say $\cat C$ \xemph{has enough fine objects} if for every object $Y$ there is a fine object $X$ and a $w:X\to Y$ in $W$. \end{definition}

\begin{lemma} If there are enough fine objects, $\fun(\cat C,W)$ is equivalent to the category of fine objects. \label{fine}
\end{lemma}

\begin{proof} Let $\cat C_{\rm fine}$ be the full subcategory of fine objects. The functor $Q:\cat C_{\rm fine}\to\fun(\cat C,W)$ is essentially surjective, because there are enough fine objects. It is full because in every functional span $(s,t)$ between fine objects, $s$ has a section $\pre s$ and $(s,t)$ is equivalent to $(\id,t\circ s^{-1})$. If $f,g:X\to Y$ are two morphisms between fine objects and $(\id,_X,f)$ and $(\id_X,g)$ are equivalent spans, then there is a $w:X'\to X$ in $W$ such that $f\circ w=g\circ w$. This $w$ is an isomorphism because $X$ is fine however. Hence the functor is faithful. \end{proof}

\begin{remark} This terminology comes the theory of quasitoposes. Inverting the class of morphisms that are both monic and epic, turns quasitoposes into toposes. In turn that terminology comes from topology, where coarse and fine correspond to coarse and fine topologies in categories of topological spaces.
\end{remark}

\subsection{Assemblies}
Subobject fibrations are a construction that turn regular categories into fibred locales. This section shows a canonical way of doing the opposite: turning fibred locales into regular categories. The idea behind the construction is the following. For any regular category, the fibred locale $\cod:\Cat{monos}(\cat C)\to\cat C$ has some extra structure in the form of the functor $\dom:\Cat{monos}(\cat C)\to\cat C$. This functor $\dom$ can be characterized as the universal way of turning supine morphisms into regular epimorphisms. In particular, $\dom$ inverts supine monomorphisms. Therefore, the canonical way of constructing a regular category out of a fibred locale is inverting the supine monomorphisms of the domain.

\begin{lemma} The supine monomorphisms of a fibred locale admit a calculus of right fractions. \end{lemma}

\begin{proof} Identities are supine monomorphisms, and composition of supine monomorphisms are too. Furthermore, if $s$ is a supine monomorphism and $s\circ f=s\circ g$, then $f=g$, which means that the set of supine monomorphisms satisfies the first and third properties of definition \ref{calculus of right fractions}.

Lemma \ref{lifting limits} tells us $\cat F$ has finite limits. Monomorphisms are stable under pullback and the Beck-Chevalley and Frobenius conditions imply that supine morphisms are stable under pullback too. In this way the second property is also satisfied.
\end{proof}

\newcommand\Asm{\Cat{Asm}}
\begin{definition} Let $F:\cat F\to\cat B$ be a fibred locale over a category $\cat B$ with finite limits. Let $S$ be the class of supine monomorphisms in $\cat F$. The \xemph{category of assemblies} $\Asm(F)$ is $\fun(\cat F,S)$. 
\end{definition}

\begin{lemma} The category $\Asm(F)$ is regular, and Heyting if $F$ is a complete fibred Heyting algebra. \end{lemma}

\begin{proof} Lemmas \ref{lifting limits} and \ref{fun is le} give us finite limits in $\fun(\cat F,S)$. The localization $Q:\cat F\to\fun(\cat F,S)$ turns supine morphisms into regular epimorphisms, for the following reasons. Let $(p,q)$ be a kernel pair of a supine $s:X\to Y$ in $\cat F$ and let $f\circ p = f\circ q$ for some $f:X\to Y'$. We can split the map $(s,f):X\to Y\times Y'$ into a supine morphism $s': X\to \im{F(s,f)}(X)$ and vertical morphism $v:\im{F(s,f)}(X)\to Y\times Y'$. The projection $\pi_0: \im{F(s,f)}(X) \to X$ is supine because $\pi_0\circ s' = s$ and both $s$ and $s'$ are supine morphisms. Seeing that it is monic requires some diagram chasing. We pull back any pair $x,y$ such that $\pi_0\circ x=\pi_0\circ y$ along $s'$ to get a pair of morphisms $x',y'$ such that $s\circ x' = s\circ y'$. Now $(x',y')$ factors through $(p,q)$ and therefore $f\circ x' = f\circ y'$. We may then conclude that $(s,f)\circ x' = (s,f)\circ y'$. Because $F$ is faithful, vertical morphisms are monic, and because $s'$ is the supine part of $(f,s)$, we get $s'\circ x' = s'\circ y'$. Then $x=y$ follows from the fact that supine morphisms are stable under pullback, and $(x,y)$ is the unique factorization of $(s'\circ x',s'\circ y')$ over a supine morphism. We see that $f$ factors through a supine morphism $s'$ that differs from $s$ by a supine monomorphism. This `factorization' is unique up to equivalence of spans. The conclusion is that supine morphisms become coequalizers and hence regular epimorphisms in the localization. 
\[ \xymatrix{
X \ar[r]^s \ar[dr]^{s'} \ar[d]_f & Y \\
Y' & \im{F(s,f)}(X) \ar[l]^{\pi_1} \ar[u]_{\pi_0}
}\]
Because $\Asm(F)$ has finite limits, it has a subobject fibration. We now show that $\Sub(X)\cong F_{FX}/X$. For this we consider what monomorphisms in $\Asm(F)$ are like. 

In $\fun(\cat F,S)$ supine monomorphisms have become isomorphisms. This means that if a functional span $(s:X\to X',m:X\to Y)$ is a monomorphism, then $(\id_X,m) = Qm$ is an isomorphic monomorphism. Here $Q:\cat F\to\fun(\cat F,S)$ is the localization.

We can split $m$ into a supine part $s_m: X\to \im m(X)$ followed by a vertical part $v_m:\im m(X)\to Y$. Vertical morphisms are monic because $F$ is faithful, and $Qv_m$ will therefore be monic too. Because $Qv_m\circ Qs_m = Qm$, $Qs_m$ is a monomorphism that is also a regular epimorphism and therefore an isomorphism. For these reasons every monomorphism $(s,m)$ is isomorphic to $(\id,v_m)$. Thus we find a natural equivalence between $\Sub(X)$ and $F_{FX}/X$ for every object $X$ of $\fun(\cat F,S)$.

Let $\cat F^v$ be the category whose objects are vertical morphisms for the fibration $F$, then $Q\cod:\cat F^v \to \fun(\cat F,S)$ is equivalent to the subobject fibration of $\fun(\cat F,S)$. It is a fibred locale and it is a complete fibred Heyting algebra if $F$ is. Therefore $\Asm(F)=\fun(\cat F,S)$ is regular or even Heyting.
\end{proof}

Regular categories are a reflective subcategory of the category of fibred locales in a suitable 2-categorical sense. We prove this now.

\newcommand\floc{\Cat{floc}}
\newcommand\reg{\Cat{reg}}
\begin{theorem} Let $\floc$ be the category of fibred locales over finite limit categories and morphisms $m=(m_0,m_1)$ of fibred locales, where $m_0$ preserves finite limits. Let $\reg$ be the category of regular categories and functors. Subobject fibrations determine a 2-functor $\Sub:\reg \to\floc$ such that $\reg(\cat C,\cat D)\cong \floc(\Sub(\cat C),\Sub(\cat D))$ naturally and assemblies determine a 2-functor $\Asm:\floc \to \reg$, such that $\reg(\Asm(F),\cat C)\cong \floc(F,\Sub(\cat C))$ naturally.
\end{theorem}

\begin{proof} The equivalence of regular functors $\cat C \to\cat D$ and morphisms of fibred locales $\Sub(\cat C) \to \Sub (\cat D)$ is trivial. On one hand, all we need to know is that regular functors preserve subobjects and images. On the other, a morphism of fibred locales $m:\Sub(\cat C) \to \cat (\cat D)$ consists of two functors, $m_0:\cat C \to\cat D$ and $m_1:\Cat{monos}(\cat C) \to \Cat{monos}(\cat D)$. The upper functor $m_1$ forces the lower $m_0$ to be regular.

The second right adjoint $\cod\dashv\id\dashv \dom$ of the subobject fibration determines a fully faithful functor $G:\floc(F,\Sub(\cat C)) \to \reg(\Asm(F),\cat C)$. For each morphism $(m_0,m_1):F\to \Sub(\cat C)$, $\dom m_1: \cat F \to \cat C$ maps supine morphisms to regular epimorphisms, and hence supine monomorphism to isomorphisms. For these reasons, $\dom m_1$ factors through $\fun(\cat F,S)$ by an up to isomorphism unique regular functor. The functor $G$ is uniquely determined by a choice of regular functor for each morphism $(m_0,m_1)$.

Any fibred locale $F$ has a right adjoint $R$ that sends each object $X$ in $\cat B$ to the terminal object $\top_X$ of $F_X$. We derive two functors from this map, namely $\nabla = QR$ and $\nabla':\cat F \to \cat F^v$ that sends $X$ in $\cat F$ to the unit $X \to RFX$. Together $\nabla$ and $\nabla'$ are a morphism of fibrations $F \to \Sub(\Asm(F))$. So for any regular functor $H:\Asm(F) \to \cat C$ we get the morphism $\Sub(H)(\nabla,\nabla'): F\to \Sub(\cat C)$. We now see that $G$ is surjective on objects and hence an equivalence of categories.
\end{proof}

\begin{definition}[constant object functor] For each fibred locale $F:\cat F\to\cat B$ the unit of the biadjunction $\Asm\dashv \Sub$ is a left exact functor $\nabla: \cat B \to \Asm(F)$, which is regular when $F$ is separated. We call this functor the \xemph{constant object functor}. \label{constant object functor} 
\end{definition}

\subsection{Ex/reg completion} \label{exreg}
In this subsection we introduce the ex/reg completion of a regular category, by first introducing its universal property, and then showing how to construct such a completion using a category-of-fractions construction. A construction of the exact completion that is not a category of fractions, along with similar completion constructions, is found in \cite{MR1600009}.

\begin{definition} Note that categories with finite limits have enough structure to internally define \xemph{equivalence relations}, in the form of certain monics $E\to X\times X$. An equivalence relation $(\pi_0,\pi_1):E\to X\times X$ is \xemph{effective} if $(\pi_0,\pi_1)$ has a coequalizer $e:X\to Y$ and if $(\pi_0,\pi_1)$ is a kernel pair of $e$. A regular category $\cat R$ is \xemph{exact}, if every equivalence relation $E\to X\times X$ is effective.
\end{definition}

Informally, exact categories have \xemph{quotient objects} for all equivalence relations.

\newcommand\exreg{\textrm{ex/reg}}
\begin{definition} The \xemph{ex/reg completion} of a regular category $\cat C$ is a regular functor to an exact category $I:\cat C\to\cat C_\exreg$ such that
\begin{itemize}
\item For every regular functor $F:\cat C\to\cat E$ to an exact category, there is a regular functor $G:\cat C_\exreg \to\cat E$ and a natural isomorphism $GI\to F$ .
\item For every natural transformation $\eta:GI\to HI$ there is a unique natural transformation $\theta:G\to H$ such that $\theta I = \eta$.
\end{itemize}
\end{definition}

\newcommand\ex{\Cat{ex}}
\begin{remark} By this definition, the ex/reg completion is left biadjoint to the inclusion of the \emph{2-category $\ex$ of exact categories} in the 2-category $\reg$ of regular categories: it induces a natural equivalence $\reg(\cat C,\cat E)\cong \reg(\cat C_\exreg,\cat E)$, for every exact $\cat E$. As a consequence, the 2-functor $\Sub:\ex\to\floc$ has a bireflector $\cat C \mapsto \cat C[F] = \Asm(F)_\exreg:\floc \to \ex$. For triposes, this construction is known as the \xemph{tripos-to-topos} construction, because it is the standard way of constructing a topos out of a tripos. See \cite{MR578267, MR2479466}. \label{ttt2}
\end{remark}

We split the construction of the ex/reg completions of a (small) regular category into two steps. First, we freely add quotients of equivalence relation to a regular category $\cat C$. This results in a new regular category $\cat C_q$ and a fully faithful finite limit preserving functor $\cat C\to\cat C_q$ that unfortunately only preserves regular epimorphism that are split. We then use a category-of-fractions construction to get a category of fractions $\cat C_\exreg$ with an embedding $I:\cat C\to\cat C_\exreg$ that \xemph{does} preserve regular epimorphisms.

\newcommand\er{\Cat{er}}\newcommand\wave{\mathord{\sim}}
\begin{definition} Let $\cat C$ be a category with finite limits. The category $\er(\cat C)$ of \xemph{equivalence relations} is defined as follows. The objects of $\er(\cat R)$ are pairs $(X\in\cat R, \wave_X\subseteq X^2)$ where $\wave_X$ is an equivalence relation on $X$. A morphism $(X,\wave_X) \to (Y,\wave_Y)$ is an arrow $f:X\to Y$, such that $(f,f):X\to Y^2$ factors through $\wave_Y$.

There is an embedding $\Delta:\cat R\to\er(\cat E)$ that sends each object $X$ to $(X,=)$ where $=$ stands for the diagonal subobject.

A parallel pair of morphisms $f,g:(X,\wave_X)\to(Y,\wave_Y)$ is \xemph{equivalent} if $(f,g):X\to (X')^2$ factors through $\wave_Y$; we write $f\sim g$ to denote this. 

Because composition preserves $\wave$, there is a category $\cat C_q = \er(\cat C)/\wave$ whose morphisms are equivalence classes for $\wave$. This is the \xemph{free quotient completion} of $\cat C$. We define a functor $I:\cat C\to\cat C_q$, letting $IX = (X,(\id,\id):X\to X\times X)$ for every object $X$ and $If = \set f$ for every morphism $f$. \label{fqc}
\end{definition}

Maietti and Rosolini are writing papers on quotient completions \cite{EQC, QCftFoCM}. These papers cover the construction of quotient completions directly out of fibred locales, which they call \xemph{elementary doctrines}.

\begin{lemma} The category $\cat C_q$ is regular, the functor $I:\cat C\to \cat C_q$ is full and faithful and preserves finite limits, and $I\wave_X \subseteq IX\times IX$ is an effective equivalence relation for all equivalence relation $\wave_X\subseteq X\times X$ in $\cat C$.
 \end{lemma}

\begin{proof} There is a forgetful functor $E:\er(\cat C)\to\cat C$. Because equivalence relation are stable under pullback, $E$ is a fibred category. It is a fibred meet semilattice with top and bottom, because equivalence relations on a single object $X$ form a meet semilattice and because pullbacks preserve the lattice structure.

Some properties of $\er(\cat C)$ make it a regular category up to $\wave$:
\begin{itemize}
\item It has finite limits, because $E:\er(\cat C) \to \cat C$ is a fibred category with finite limits over a category with finite limits.
\item It is also 2-category, where 2-morphisms are equivalences of arrows.
\item As a 2-category is has \xemph{inserters} for every parallel pair of arrows. This means that for every pair of arrows $f,g:(Y,\wave_Y) \to (Z,\wave_Z)$ there is an arrow $i:(X,\wave_X)\to(Y,\wave_Y)$ such that $f\circ i \sim g\circ i$ and for each $h:(X',\wave_{X'}) \to (Y,\wave_Y)$ factors uniquely through $i$. 

Using inserters we can build other pseudolimits, like the \xemph{comma square}. For two morphisms $f:\alpha\to \gamma$ and $g:\beta\to \gamma$ this is a span $p:\delta\to\alpha$ and $q:\delta \to\alpha$ such that $f\circ p\sim g\circ q$ and such that every other span $(p',q')$ that commutes with $f$ and $g$ up to $\wave$, factors uniquely through $(p,q)$.

\item Any morphism factors into a prone morphism and a vertical morphism. Prone morphisms reflect $\wave$, and therefore become monic in $\cat C_q$, whereas vertical morphisms are coinserters (inserters in the dual category) that become regular epimorphism in $\cat C_q$.

Let $\wave$ and $Y\subseteq X^2$ be equivalence relation on $X$ such that $\wave\subseteq Y$. Define on the element of $Y$ the relation $(x,y)\wave_Y(x',y')$ if $x\wave x'$ and $y\wave y'$. Then the projections $\pi_0$ and $\pi_1:Y\to X$ are morphisms $(Y,\wave_Y) \to (X,\wave)$, such that any morphism $f:(X,\wave) \to (X',\wave')$ for which $f\circ \pi_0\sim f\circ \pi_1$ factors uniquely though $(X,Y)$. Therefore vertical morphisms are coequalizers, and regular epimorphisms.

\item Coinserters are stable in comma squares. 

Let $f:(X,\wave_X) \to (Y,\wave_Y)$ be a morphism, and let $\mathord\equiv\subseteq \wave_Y$ be another equivalence relation on $Y$. We can pull back $\mathord\equiv$ along $f$, and we can intersect equivalence relations, so $f^*(\mathord\equiv)\cap \wave_X$ is a new equivalence relation on $X$. Of course $f: (X,f^*(\mathord\equiv)\cap \wave_X) \to (X,\mathord\equiv)$ and the resulting square commutes up to equivalence. 
This is a comma square, since any pair of morphisms $(g,h)$ such that $f\circ h \sim g\circ \id$ factors uniquely through $(X,f^*(\mathord\equiv)\cap \wave_X)$ in $h$.
\[\xymatrix{
\bullet \ar@/^/[drr]^g \ar@/_/[ddr]_h \ar@{.>}[dr]^g \\
& (X,f^*(\mathord\equiv)\cap \wave_X) \ar[r]^f \ar[d]_{\id} \ar@{}[dr]|\wave & (X,\mathord\equiv) \ar[d]^\id \\
& (X,\wave_X) \ar[r]_f & (Y,\wave_Y)
}\]
\end{itemize}
For all of these reasons $\er(\cat R)/\wave$ is a regular category, and the full functor $\er(\cat R) \to\er(\cat R)/\wave$ turns pseudolimits into real ones. Functor $\Delta:\cat R\to\er(\cat R)$ preserves limits, and maps to objects for which inserters and equalizers coincide. An equalizer and an inserter of $f,g:(X,\wave_X)\to(Y,=)$ are exactly the same thing. Therefore $I$ preserves finite limits. That $I$ is full and faithful is because whenever $f\sim g:(X,=)\to(Y,=)$ then $f=g$. 

The vertical morphism $(X,=)\to(X,\wave_X)$ is the coequalizer of the projections $(\wave_X,=)\rightrightarrows(X,=)$ and this pair of projections is the kernel pair of the vertical morphism $(X,=)\to(X,\wave_X)$. Thus every equivalence relation becomes effective.
\end{proof}

\begin{lemma} The functor $I:\cat C \to\cat C_q$ has the universal property that 
\begin{itemize}
\item for each finite limits preserving $F:\cat C\to\cat D$ that sends all equivalence relations to effective equivalence relations, there is a regular functor $G:\cat C_q \to\cat D$ and a natural isomorphism $GI\to F$, and 
\item for each pair of regular functors $H,K:\cat C_q\to\cat D$ and each natural transformation $\eta:HI\to KI$ there is a natural transformation $\theta:H\to K$ such that $\theta I = \eta$.
\end{itemize}
\end{lemma}

\begin{proof} For each $(X,\wave_X)\in\cat C_q$, choose a coequalizer $e_{(X,\wave_X)}:FX\to Y$ of the projection $F\wave_X\rightrightarrows FX\times FX$ in $\cat D$. For each $\phi:(X,E) \to (X',E')$ there is a unique $f:\cod e_{(X,E)} \to \cod e_{(X',E')}$ such that for all $g\in \phi$, $f\circ e_{(X,E)} = e_{(X',E')}\circ Fg$, thanks to the morphisms $h:E\to E'$. This determines a functor $G(X,E) = e_{(X,E)}$ and $G\phi = f$. 
Because each object $X$ is isomorphic to the coequalizer of its diagonal subobject $X\to X\times X$, which is preserved by the functor $F$, and $IX = (X,X\to X\times X)$, there is a unique natural isomorphism $GI\to F$.

Let $H,K:\cat C_q\to\cat D$ be regular functors and let $\eta:GI\to HI$ be any natural transformation. Because they are regular, $F, G$ preserves the kernel pair/coequalizer diagram $(\wave_X,=) \rightrightarrows (X,=) \to (X,\wave_X)$ for each $(X,\wave_X)$ in $\cat C_q$. There is a unique $\theta_{(X,\wave_X)}$ that commutes with these diagrams, $\eta_X:HIX \to KIX$ and $\eta_{\wave_X}:HI\wave_X\to KI\wave_X$
\[\xymatrix{
H(\wave_X,=) \ar@<.5ex>[r]\ar@<-.5ex>[r]\ar[d]_{\eta_{\wave_X}} & H(X,=) \ar[r]\ar[d]_{\eta_X} & H(X,\wave_X)\ar@{.>}[d]^{\theta_{(X,\wave_X)}} \\
K(\wave_X,=) \ar@<.5ex>[r]\ar@<-.5ex>[r] & K(X,=) \ar[r] & K(X,\wave_X)
}\]
This determines a unique natural transformation $\theta$ such that $\theta I = \eta$.
\end{proof}

This first step takes us almost immediately to the exact completion of $\cat C$. There are two problems. Firstly, $I$ is not a regular functor. In fact it preserves no regular epimorphisms except the split ones: if $p,q: E\rightrightarrows X$ is the kernel pair of a regular epimorphism $e:X\to Y$, then $(X,E)$ is the image of $Ie$; an inverse of the canonical monomorphism $\set e:(X,E)\to (Y,=)$ is precisely the set of all sections of $e$, and if this set is empty $e$ is not longer a regular epimorphism.

By the way, because the objects in the image of $I$ cover every object of $\cat C_q$, any regular epimorphism $e:X\to IY$ is split, and hence that $IY$ is \xemph{projective} (see definition \ref{proj}) for all $Y\in\cat C$. Moreover, $\cat C_q$ has enough projectives.

Secondly, the new quotients in $\cat C_q$ may be new equivalence relations that do not have quotients in $\cat C_q$ itself.

A category-of-quotients construction solves the first problem: just invert maps of the form $(X,\pre{(e\times e)}(=))\to (Y,=)$ for every regular epimorphism $e$. The resulting localization has quotients for all equivalence relations, if $\cat C$ is regular category.

We define a coverage consisting of singletons in $\cat C_q$. We will show that it admits a category of right fractions, but that is satisfies an extra condition that causes the category of fractions to preserve regular epimorphisms.

\begin{definition} In $\cat C_q$, a \xemph{cover} is any morphism that contains a regular epimorphism of $\cat C$. Let $W$ be the set of all \xemph{monic} covers. \end{definition}

\begin{lemma} Monic covers admit a calculus of right fractions. 
Moreover, if $e:Y\to Z$ is a regular epimorphism, and $m:X\to Y$ a monic cover, then the monomorphism $\exists_{e\circ m}(X) \to Z$ is a cover too.\label{special cover} \end{lemma}

\begin{proof} Every identity is a monic cover, as is any composition of monic covers. Also monic covers are stable under pullback, and if $m\circ f = m\circ g$ and $m$ is monic, then $f=g$, so $W$ satisfies the three conditions of definition \ref{calculus of right fractions}.

An $m\in W$ followed by a regular epimorphism $e$ in $\cat C_q$ corresponds to a regular epimorphism $m:(X,\pre{(m\times m)}(\wave)) \to (Y,\wave)$ followed by a vertical morphism $(Y,\wave)\to (Y,\mathord\equiv)$ in $\er(\cat C)$. If we pull back $\mathord\equiv$ along $m$, we get a vertical morphism $(X,\pre{(m\times m)}(\wave))\to (X,\pre{(m\times m)}(\mathord\equiv))$: this is the reindexing functor at work. Now $m:(X,m^*(\mathord\equiv)) \to (Y,\mathord\equiv)$ is a monic cover and also the image of $m$ under the regular epimorphism $e$.
\[\xymatrix{
(X,m^*(\wave)) \ar[r]^m \ar[d]_{\id} \ar@{}[dr]|\wave & (Y,\wave) \ar[d]^\id \\
(X,m^*(\mathord\equiv)) \ar[r]_m & (Y,\mathord\equiv)
}\]
\end{proof}

This category-of-fractions construction also takes care of the new equivalence relations that appear in the free quotient completion, solving the second problem.

\begin{theorem} If $\cat C$ is a regular category, then $\fun(\cat C_q,W)$ is an exact completion of $\cat C$. \end{theorem}

\begin{proof} The quotient functor $Q:\cat C_q \to \fun(\cat C_q,W)$ preserves finite limits by lemma \ref{fun is le}. It preserves regular epimorphism in $W$ due to the extra condition in lemma \ref{special cover}: it ensures that if an object is isomorphic to the coequalizer of a kernel pair, then it is a coequalizer of an isomorphic kernel pair. We finally get a regular functor $QI:\cat C \to \fun(\cat C_q,W)$.

Monic covers between equivalence relations in $\cat C_q$ make $\fun(\cat C_q,W)$ an exact category. Consider a map $f:(X,\wave_X) \to (Y,\wave_Y)\times (Y,\wave_Y)$ in $\er(\cat C)$ that becomes an equivalence relation in $\cat C_q$. Because it is monic, we may assume that $f$ is prone. Its image  $\exists_f(X)$ is an equivalence relation on $Y$ in $\cat C$. Because $\cat C$ is regular $f = m\circ e$ with $m$ monic and $e$ regular epic. Now we can split $f:(X,\wave_X) \to (Y,\wave_Y)\times (Y,\wave_Y)$ into a monic cover $e:(X,\wave_X) \to (\exists_f(X),\wave_{\exists_f(X)})$ and an effective equivalence relation $m:(\exists_f(X),\wave_{\exists_f(X)}) \to (Y,\wave_Y)\times (Y,\wave_Y)$, where $\wave_{\im f(X)}$ is the unique relation that makes both morphisms well defined. Because every equivalence relation is isomorphic to an effective equivalence relation in $\fun(\cat C_q,W)$, this category is exact.

Every regular functor $F:\cat C\to\cat E$ to an exact category $\cat E$ makes all equivalence relations effective, so there is a $G:\cat C_q\to \cat E$ such that $GI\simeq F$. This $G$ sends monic covers to isomorphisms: if $\mu:(X,\wave_X) \to (Y,\wave_Y)$ is a monic cover, then there is a regular epimorphism $m\in\mu$, such that $Im$ commutes with $\mu$ and the covers $c_X:IX\to (X,\wave_X)$ and $c_Y:IY \to (Y,\wave_Y)$.
\[\xymatrix{
IX \ar[r]^{Im}\ar[d] & IY \ar[d] \\
(X,\wave_X) \ar[r]_\mu & (Y,\wave_Y)
}\]
Because $G$ is regular and $GI\simeq F$, $G(c_Y\circ m)$ and $G(c_X)$ are regular epimorphisms while the difference $G\mu$ is a monomorphism. This forces $G\mu$ to be an isomorphism. Because $G$ sends all monic covers to isomorphisms, there is an $H:\fun(\cat C_q,W) \to \cat E$ such that $HQ\simeq G$, and of course $HQI\simeq F$, so this condition is satisfied.

If $\eta:KQI \to LQI$ is a natural transformation, there is a unique $\theta: KQ \to LQ$ such that $\theta I = \eta$ and a unique $\iota:K\to L$ such that $\iota Q = \theta$. We now have demonstrated that $QI:\cat C \to \fun(\cat C_q,W)$ satisfies all conditions on ex/reg completions.
\end{proof}

In the remainder of this section we will demonstrate that ex/reg completions of Heyting categories are Heyting categories, and show that if a regular category has enough projectives, the free quotient completion has enough fine objects.

\begin{definition} Let $\wave_X$ be an equivalence relation on $X$. A $U\in \Sub(X)$ is \xemph{($\wave_X$)-saturated} if $\pre{\pi_0}(U)\cap\wave_X \subseteq \pre{\pi_1}(U)$ for the projections $\pi_i:X\times X\to X$. We let $\overline U = \set{x\in X|\exists u\in U.u\sim_X x}$ be the \xemph{$\wave_X$-saturation} of $U$. 
\end{definition}

\begin{lemma}[$\Sub$ in ex/reg completions] For all objects $(X,\wave_X)$ of $\cat C_\exreg$ and all $Y\subseteq X$, the subobject lattice $\Sub(X,\wave_X)$ is isomorphic to the lattice $\Sub(X,\overline{(-)})$ of $\wave_X$-saturated subobjects of $X$.\label{sub in exreg} \end{lemma}

\begin{proof} Any monic $(Y,\wave_Y) \to (X,\wave_X)$ comes from a prone morphism $m:(Y,\wave_Y) \to (X,\wave_X)$ in $\er(\cat C)$ and can be split into a monic cover and an prone monomorphism using the regular-epi-mono factorization in $\cat C$. The monic cover becomes an isomorphism in $\fun(\cat C_q,W)$, and therefore each monomorphism is isomorphic to one that contains a prone monomorphism. Hence $\Sub(X,\wave_X)$ is a sublattice of $\Sub(X)$.

Two subobjects $Y$ and $Y'\subseteq X$ induce equivalent subobjects of $(X,\wave_X)$ if and only if they contain the same elements up to the equivalence $\wave_X$. Firstly, from an isomorphism between $(Y,\wave_X)$ and $(Y',\wave_X)$ we can derive a pair of monic covers $p_0:(Z,\mathord\equiv)$ and $p_1:(Z,\mathord\equiv)$ that prove that $Y$ has the same elements as $Y'$ up to equivalence. Secondly, consider $Z = \wave_X\cap (Y\times Y')$, with the projections $\pi_0:Z\to Y$ and $\pi_1:Z\to Y'$. There is a unique equivalence relation $\mathord\equiv$ on $Z$ such that $\pi_0:(Z,\mathord\equiv) \to (Y,\wave_X)$ and $\pi_1:(Z,\mathord \equiv) \to (\overline Y,\wave_X)$ are prone morphism, namely $\mathord\equiv = \set{(a,b,c,d)\in Z\times Z | a\sim_X c, b\sim_X d}$. This turns $(Z,\equiv)$ into the intersection of $(Y,\wave_X)$ and $(Y',\wave_X)$. If $Y$ and $Y'$ contain equivalent elements, then $\pi_0$ and $\pi_1$ are regular epimorphisms, and therefore monic covers, proving that $Y$ and $Y'$ induce the same subobjects.

We conclude that each $Y$ induces the same subobject as $\overline Y$, but that if $\overline Y$ induces the same subobject as $\overline Z$, then $\overline Y = \overline Z$. This makes $\Sub(X,\wave_X)$ isomorphic to the lattice of saturated subobjects of $X$. \end{proof}

\begin{corollary} The ex/reg completion of a Heyting category is a Heyting category. \end{corollary}

\begin{proof} Reflective subcategories have all limits and colimits of their ambient category, because the adjunction of the reflector and the inclusion is monadic. We just showed that $\Sub(X,\wave_X)$ is a reflective subcategory of $\Sub(X)$, and this implies that $\Sub(X,\wave_X)$ has all meets and joins that $\Sub(X)$ has. 

Because ex/reg completions are regular, we can focus now on the existence of the dual image map. Though morphisms are sets of spans in $\fun(\cat C_q,W)$, every morphism is isomorphic to one in the image of $Q:\cat C_q \to \fun(\cat C_q,W)$. So let $f:(X,\wave_X)\to(Y,\wave_Y)$. The inverse image map correspond to the map $V\mapsto \overline{\pre f(V)}$ between the lattices of saturated subobjects, something which we can show using the construction of the pullback above. The right adjoint is $\forall'_f(U) = \set{y\in Y|\forall x\oftype X.f(x)\sim_X y \to U}$, which exists because it is definable in the internal language of the Heyting category. Therefore, ex/reg completions of Heyting categories are Heyting categories.
\end{proof}

We mention some fine objects to simplify the construction of the category of fractions.

\begin{definition} In a regular category $\cat C$ a \xemph{projective object} $P$ is a object for which $\cat C(P,-):\cat C\to\Cat{Set}$ is a regular functor. This means that for every regular epimorphism $e:X\to Y$, and every $y:P\to Y$ there is an $x:P\to X$ such that $e\circ x= y$. \label{proj}
\[ \xymatrix{
& X\ar[d]^e \\
P \ar@{.>}[ur]^{x} \ar[r]_y & Y
}\]
\end{definition}

\begin{lemma} If $P$ is projective then for every equivalence relation $\wave_P\subseteq P\times P$, $(P,\wave_P)$ is a fine object relative to monic covers. \end{lemma}

\begin{proof} Let $m:(X,\wave_X)\to(Y,\wave_Y)$ be a monic cover, and let $y:(P,\wave_P)\to(Y,\wave_Y)$. There is an $x:P\to X$ such that $m\circ x=y$. Now $x\times x:\wave_P\to\wave_X$ because $\wave_X = \pre{(m\times m)}(\wave_Y)$ and because $m$ is monic, this factorization is unique. \end{proof}

\begin{corollary} If there are enough projectives in $\cat C$, then the subcategory of $\cat C_q$ of equivalence relations over projective objects, is an exact completion of $\cat C$. \label{fine2} \end{corollary}

\begin{proof} See lemma \ref{fine}. \end{proof}

\section{Realizability Categories}
We apply the constructions of the previous section to realizability fibrations. Our characterization of realizability fibrations helps to characterize the resulting \xemph{realizability categories}.

The first subsections show some advantages of working with external filters of inhabited subobjects, and prove that we can do so without loss of generality.

\subsection{Inhabited filters}
We apply $\Asm$ to the realizability fibration $\ds A/\phi:\cat DA/\phi \to\cat H$. Here $\cat H$ is an arbitrary Heyting category, $A$ an arbitrary order partial applicative structure in $\cat H$ and $\phi$ an arbitrary combinatory complete external filter of $A$. This subsection shows that if all downsets in $\phi$ are inhabited, then $\cat DA/\phi$ has enough fine objects. In fact, in that case $\Asm(\ds A/\phi)$ is a coreflective subcategory of $\cat DA/\phi$, and in turn the base category $\cat H$ is a reflective subcategory of $\Asm(\ds A/\phi)$.

\newcommand\supp\Sigma
\begin{definition} For any $Y\in \ds A_X$ let $\supp Y = \set{x\in X|\exists a\in A.(a,x)\in Y}$. \end{definition}

\begin{lemma} The map $\supp$ defines a subfunctor of $\ds A/\phi:\cat DA/\phi \to \cat H$ if and only if all $U\in\phi$ are inhabited. This subfunctor is regular and factors uniquely through $\Asm(\ds A/\phi)$. \label{left adjoint}
\end{lemma}

\begin{proof} Assume all $U\in \phi$ are inhabited. For each $f:Y\to Y'$ in $\cat DA/\phi$, we already have a map $(\ds A/\phi)f:(\ds A/\phi)Y\to(\ds A/\phi)Y'$ and inclusions $\supp Y \to (\ds A/\phi)Y$ and $\supp Y'\to (\ds A/\phi)Y'$. Because of the tracking principles, we can find a $U\in \phi$ such that $(a,b,x)\mapsto(ab,(\ds A/\phi)f(x))$ is a well defined map $U\times Y\to Y'$ that commutes with the projections $U\times Y \to (\ds A/\phi)Y$ and $Y'\to (\ds A/\phi)Y'$. Because $U$ is inhabited, $\supp Y$ is the image of the projection $U\times Y \to (\ds A/\phi)Y$, while $\supp Y'$ is the image of $Y'\to (\ds A/\phi)Y'$ by definition. This proves that the composition of $(\ds A/\phi)f:(\ds A/\phi)Y\to (\ds A/\phi)Y'$ with the inclusion $\supp Y\to (\ds A/\phi)Y$ factors uniquely through the inclusion  $\supp Y' \to (\ds A/\phi)Y'$.
\[ \xymatrix{
U\times Y \ar[r]^{(a,b,x)\mapsto x} \ar[d]_{(a,b,x)\mapsto (ab,(\ds A/\phi)f(x))} & \supp Y \ar@{^(->}[r]\ar@{.>}[d]_{\supp f} & (\ds A/\phi)Y \ar[d]_{(\ds A/\phi)f} \\
Y' \ar[r]_{(c,x)\mapsto x} & \supp Y' \ar@{^(->}[r] & (\ds A/\phi)Y'
}\]
So $\supp$ is a subfunctor of $(\ds A/\phi)$ if all $U\in\phi$ are inhabited.

Now assume $\supp$ is a functor. The fact that $!:(\ds A/\phi)_{U\hookrightarrow A}(\A) \to \termo$ is supine implies that $\im{!}((\ds A/\phi)_{U\hookrightarrow A}(\A)) \simeq \termo$. The functor $\supp$ preserves this isomorphism while sending $\im{!}((\ds A/\phi)_{U\hookrightarrow A}(\A))$ to $\set{x\in \termo|\exists a\in U.\top}$. So $U$ is inhabited if $U\in \phi$.

The functors $\supp$ is regular. Equalizers are prone, due to shared unique factorization properties. The functor $\supp$ preserves them, because $\ds A/\phi$ does and because for prone morphisms the naturality square for the inclusion $\supp\to \ds A/\phi$ is a pullback. Similarly, coequalizers are supine, and $\supp$ maps supine morphisms to regular epimorphisms. This leaves products.

The object $A$ is inhabited for combinatory completeness, and therefore $\supp(\top_\termo)=\termo$. If $\comb p$ is a set of pairing combinators in $\phi$, then $\comb p UV$ is inhabited if and only if $U$ and $V$ are, because $\comb p$,  $\comb p UV\arrow U$ and $\comb p UV\arrow V$ are inhabited. Therefore $\supp$ sends fibred products to intersections, which in turn implies that $\supp$ preserves binary products.

If $s:Y\to Y'$ is a supine morphism, $\supp s:Y\to Y'$ is going to be a regular epimorphism, because of the inhabited set of realizers. Therefore, $\supp$ sends supine monomorphisms to isomorphisms, which means that $\supp$ has an up to isomorphism unique factorization through $\fun(\cat DA/\phi,S) = \Asm(\ds A/\phi)$.
\end{proof}

We lift $\supp:\cat D A/\phi \to\cat H$ up along $\ds A/\phi$ and thus create a \xemph{fine coreflection}.

\begin{lemma} For each $Y\in\cat DA/\phi$ let $Y_{\rm fine} = (\ds A/\phi)_{\supp Y\hookrightarrow (\ds A/\phi)Y}(Y)$. Now $Y_{\rm fine}$ is fine relative to supine monomorphisms, and the prone map $Y_{\rm fine} \to Y$ is a supine monomorphism. \end{lemma}

\begin{proof} By definition $\supp Y_{\rm fine} = \supp Y$. For this reason $Y_{\rm fine} \to Y$ is monic, but also supine, because the family $Y$ is the composition of the regular epimorphism $Y\to\supp Y$ and the monomorphism $\supp Y \to (\ds A/\phi)Y$. Therefore the map $Y_{\rm fine} \to Y$ is a supine monomorphism.

Let $f:X\to X'$ be a supine monomorphism in $\cat DA/\phi$. For each $g:Y_{\rm fine} \to X'$, $\supp g$ factors uniquely though $\supp f$ because $\supp f$ is an isomorphism. The natural monomorphism $\supp\to\ds A/\phi$ makes that $\ds A/\phi g$ factors uniquely through $(\ds A/\phi)(f)$. The Beck-Chevalley condition on the kernel pairs of monomorphisms makes supine monomorphisms prone. Therefore the unique factorization of $(\ds A/\phi)g$ through $(\ds A/\phi)f$ lifts to $\cat DA/\phi$.  This is a unique factorization of $g$ through $f$ since $\supp Y_{\rm fine} = (\ds A/\phi)Y_{\rm fine}$.
\end{proof}

This leads us to the following conclusion.

\begin{theorem}[assemblies are inhabited families] Let $A$ be a partial applicative structure and let $\phi$ be a combinatory complete filter such that all $U\in\phi$ are inhabited. The category of assemblies is the equalizer of $\ds A/\phi$ and $\supp$.\end{theorem}

\begin{proof} Apply lemma \ref{fine} to the lemma above. \end{proof}

\begin{remark} This is why we identify the category of assemblies with the subcategory of $\cat DA/\phi$ whose objects are families of inhabited downsets of $A$. \label{assemblies are inhabited families} \end{remark}

\begin{corollary} There is a string of inclusions $\cat H \to \Asm(\ds A/\phi) \to \cat DA/\phi$, the first reflective, the second coreflective. \end{corollary}

\begin{proof} The second embedding is the fine coreflection, the dual of the coarse reflection of remark \ref{coarse reflection}. We can therefore identify $\Asm(\ds A/\phi)$ with the subcategory of fine objects of $\cat D A/\phi$.

For each object $X$, $\supp \top_X = X$ and therefore $\top_X$ is fine. This establishes the functor $\nabla:\cat H\to \Asm(\ds A/\phi)$. The functor $\supp$ determines a left adjoint left inverse because the vertical maps $Y_{\rm fine} \to \top_{\supp Y}$ form a unit for this adjunction.
\end{proof}

\subsection{Filter quotients}
In the previous chapter external filters contained uninhabited subobjects. We will now show that when we construct realizability categories, we can restrict to filters of inhabited downsets without loss of generality, a property which we will use to characterize of realizability categories later on.

Let $\phi$ be a combinatory complete filter of an order partial applicative structure $A$ in a Heyting category $\cat H$. In the case that $\phi$ contains uninhabited downsets, we can do the following.

\begin{definition} Let $\sigma\subseteq \Sub(\termo)$ be the least filter that contain all subterminals $V$ for which there is a $U\in \phi$ such that $!:U\to\termo$ factors through $V$. Since $\termo$ is a (very simple) order partial applicative structure, there is a category $\Asm(\ds\termo/\sigma)$ with a functor $\nabla:\cat H \to \Asm(\ds\termo/\sigma)$.
\end{definition}

Note that $\Asm(\ds\termo/\sigma)$ is the result of factoring the constant object functor $\nabla:\cat H \to \Asm(\ds A/\phi)$ into a functor that is surjective on objects and a functor that is full and faithful, i.e. it is the \xemph{full image} of $\nabla$.

\begin{lemma} There is a full and faithful functor $I:\Asm(\ds\termo/\sigma)\to\Asm(\ds A/\phi)$ that commutes with the constant object functors $\nabla:\cat H\to \Asm(\ds\termo/\sigma)$ and $\nabla':\cat H\to \Asm(\ds A/\phi)$. Moreover $\nabla:\cat H\to \Asm(\ds\termo/\sigma)$ is essentially surjective.
\end{lemma}

\begin{proof} We need a suitable map $i:\ds \termo/\sigma \to \ds A/\phi$. A family of downsets in $\cat D\termo/\sigma$ is just a pair $(X,Y\subseteq X)$ and we send it to the trivial family $(X,A\times Y)$ in $\cat DA/\phi$. This is a functor: for every $U\in \sigma$ there is a $V\in\phi$ such that $!:V\to U$. Therefore, if $f:(X,Y)\to (X',Y')$ has an object of realizers $U\in \sigma$, then $f:(X, A\times Y) \to (X',A\times Y')$ has one in $\phi$. We could take for example $\comb k V$, where $\comb k = \db{(x,y)\mapsto x}$. 

This functor $i:\cat D \termo/\sigma \to \cat D A/\phi$ is faithful, because it commutes with the faithful realizability fibrations. It is full because if $f:(X,A\times Y)\to(X',A\times Y')$ has an object of realizers $U$, then $\im !(U)\in\sigma$ tracks $(X,Y)\to (X,Y')$. We can use the biadjunction $\Asm\dashv \Sub$ to turn the vertical morphism $(\id_{\cat H},i):\ds \termo/\sigma \to\ds A/\phi$ into a functor $I:\Asm(\ds\termo/\sigma)\to\Asm(\ds A/\phi)$. It is bijective on subobjects; both fullness and faithfulness follows from that.

A supine monomorphism $(X,Y)\to (X',Y')$ in $\cat D\termo/\sigma$ is a monomorphism $m:X\to X'$ such that $Y' = \exists_m(Y)$ or $Y'=Y$. Every object $(X,Y)$ in $\Asm(\ds A/\phi)$ is isomorphic to $(Y,Y) = \nabla Y$ because of the supine morphism $(Y,Y) \to (X,Y)$. This establishes that $\Asm(\ds\termo/\sigma)$ is the full image of $\nabla:\cat H \to\Asm(\ds A/\phi)$. 
\end{proof}

The construction of $\Asm(\ds\termo/\sigma)$ can be simplified.

\begin{definition} A partial $f:X\partar Y$ in $\cat H$ is \xemph{almost total} relative to $\sigma$, if for some $U\in\sigma$, $U\times X\subseteq \dom f$. Two almost total arrows $f,g:X\partar Y$ are \xemph{almost equal} relative to $\sigma$, if there is a $U\in\sigma$ such that $Y=\set{x\in X|U}\subseteq \dom f\cap \dom g$ and $f(x)=g(x)$ for all $x\in Y$. The \xemph{filter quotient} $\cat H/\sigma$ is the category of almost total arrows modulo almost-equality relative to $\sigma$.
\end{definition}

\begin{remark} The category $\cat H/\sigma$ is a \xemph{filter quotient} of $\cat H$, described in section V.9 of \cite{MR1300636}.\end{remark}

\begin{lemma} $\Asm(\ds\termo/\sigma)$ is equivalent to $\cat H/\sigma$. \end{lemma}

\begin{proof} Because $(X,Y)\simeq (Y,Y)$ in $\Asm(\ds\termo/\sigma)$, we get a functor $\Asm(\ds\termo/\sigma) \to \cat H/\sigma$ that maps $(X,Y)$ to $Y$ and $f:(X,Y) \to (X',Y')$ to $f: \pre f(Y') \to  Y'$ that is surjective on objects. 

This functor if full, because an almost total arrow $f:Y\partar Y'$ determines an arrow $(U\times Y,U\times Y)\to (Y',Y')$ for some $U\in\phi$. Now $(U\times Y,U\times Y)$ is isomorphic to $(Y,U\times Y)$ because of a supine monomorphism, while $(Y,U\times Y) \simeq (Y,Y)$ because $\id_Y$ is a morphism in both directions.

If $f,g:(X,Y)\to (X,Y')$ determine almost equal $f,g:\pre f(Y')\cap \pre g(Y') \to Y'$, then there is some $U\in\sigma$ such that the restrictions $f,g:U\times Y \to Y'$ are equal. We just saw that $(U\times Y,U\times Y)\simeq (Y,Y)$, and this forces $f=g$ in $\Asm(\ds\termo/\sigma)$.
\end{proof}

The realizability fibration $\ds A/\phi$ composed with the constant object functor $\nabla:\cat H\to\Asm(\ds\termo/\sigma)$ is a new realizability fibration, but in the composite, all members of $\phi$ are inhabited. We will demonstrate this now.

\begin{lemma} For all $U\in\sigma$ and $X\in \cat H$, reindexing along $\pi_1:U\times X\to X$ is an equivalence of the fibres. \end{lemma}

\begin{proof} Note that $\pi_1$ is a monomorphism, because $U\subseteq \termo$. By the Beck-Chevalley condition, $(\ds A/\phi)_{\pi_1}(\exists_{\pi_1}(Y))\simeq Y$ in a natural way. Now $\exists_{\pi_1}(\top)\simeq \top$, because there is a $V\in \phi$ such that $!:V\to U$; this implies that $(\ds A/\phi)_{V\hookrightarrow A}(\A)\times \top_ X \to \top_U\times \top_X$ and because $!:(\ds A/\phi)_{V\hookrightarrow A}(\A)\to \termo$ is supine, so must $\top_{U\times X}\to \top_X$ be. By the Frobenius condition $\im{\pi_1}((\ds A/\phi)_{\pi_1}(Y)\land \top)\simeq Y\land \im{\pi_1}(\top) \simeq Y$. We now see that $(\ds A/\phi)_{\pi_1}$ and $\exists_{\pi_0}$ are an equivalence of fibres. \end{proof}

\begin{corollary} The composed functor $\ds A/\phi: \cat DA/\phi \to \cat H/\sigma$ is a realizability fibration. \end{corollary}

\begin{proof} For $U\in \sigma$ the fibres over $U\times X$ and $X$ are equivalent. Therefore $\im e$ preserves $\top$ if $e:X\partar Y$ is an almost total regular epimorphism. Also, all pairs of almost isomorphic objects $X$ and $Y$ have equivalent fibres. This means that the composite is a separated fibred locale. Now, $\A$ is still weakly generic, Church's rule and the uniformity rule still hold, and $\A_U$ is inhabited if and only if $U\in \phi$. That is all we need to know. \end{proof}

We conclude that every category of assemblies that comes from a realizability fibration, comes from a realizability fibration for an external filter that contain only inhabited downsets. We summarize this as follows.

\begin{theorem}[inhabitation] For every Heyting category $\cat H$ every order partial applicative structure $A$ in $\cat H$ and every combinatory complete external filter $\phi$ of $A$, there is a Heyting category $\cat H'$, an order partial combinatory algebra $A'$ and a combinatory complete external filter $\phi'$ all of whose members are inhabited, such that $\Asm(\ds A/\phi)\cong \Asm(\ds A'/\phi')$. \label{inhabitation}
\end{theorem}

\begin{remark} If an order partial applicative structure $A$ has a combinatory complete external filter of inhabited objects, then it has a combinatory complete (internal) filter, namely $A$ itself. That is why we can finally talk about order partial combinatory algebras instead of order partial applicative structures (see definition \ref{combinatory completeness}). \end{remark}

\begin{proof} Let $\sigma = \set{\im !(U)\in \Sub(\termo)|U\in\phi}$, then $\cat H' = \cat H/\sigma$. Let $Q:
\cat H \to \cat H'$ be the quotient functor, then $A' = QA$ and $\phi' = \set{QU\in\ext A'|U\in\phi}$. Because every member of $\phi'$ is inhabited in $\cat H/\sigma$, the objects of realizers of partial combinatory functions are inhabited, which makes $A$ combinatory complete, and hence an order partial combinatory algebra.
\end{proof}

We will use these facts to characterize realizability categories in the next subsection.

\subsection{Characterization}
The material in the previous two subsections implies that a realizability category contains (one of its) base categories as a reflective subcategory with a regular inclusion functor. Here, we characterize the reflective inclusions of categories that are equivalent to such inclusions.

\begin{definition} Let $\cat R$ be a Heyting category, let $C\in \cat R$ be an order partial combinatory algebra, let $\cat H$ be a reflective subcategory, with fully faithful $I:\cat H\to\cat R$ and reflector $R:\cat R\to\cat H$. We start with the assumptions that $I$ is regular and that $R$ is faithful and preserves finite limits.

A finite limit preserving reflector is a fibration: prone morphisms are those morphisms for which the naturality square of the unit of the reflection is a pullback. Using these prone morphisms we can define \xemph{\prodo morphisms}: $f:X\to Y$ is a \prodo morphism if $f$ factors through $\pi_1:A\times Y\to Y$ in a prone arrow $p:X\to A\times Y$ and the fibres of $f$ are isomorphic to downsets of $C$.

We now say that $(\cat R,C,\cat H)$ is a \xemph{regular realizability category} if
\begin{enumerate}
\item \xemph{Weak genericity}: For each object $X$ there is a span $(e:Y\to X,p:Y\to C)$ where $e$ is a regular epimorphism and $p$ is prone.
\item \xemph{Tracking principle}: For each chain of arrows $e:X\to Y$ and $f:Y\to Z$ where $e$ is a regular epic \prodo morphism and $f$ is either a \prodo morphism or a prone morphism, there is a regular epic \prodo morphism $g:W\to Z$, such that the pullback of $e$ along $g$ is split.
\[ \xymatrix{
\bullet \ar[r]\ar[d]\ar@{}[dr]|<\lrcorner & \bullet \ar[r]\ar[d]\ar@{}[dr]|<\lrcorner\ar@{.>}@/_/[l] & W\ar[d]^g \\
X \ar[r]_e & Y\ar[r]_f & Z
}\]
Moreover, this splitting takes a particular form: if $f$ is a \prodo morphism, then this splitting is a restriction of the application map; if $f$ is prone, then this splitting is an inclusion.
\end{enumerate} \label{regular realizability category}
\end{definition}

\begin{theorem} Any regular realizability category $(\cat R,C,\cat H)$ is equivalent to $\cat H\to\Asm(\ds RC/\chi)$ for some external filter $\chi$. \label{charreg} \end{theorem}

\begin{proof} Pull back the subobject fibration of $\cat R$ along $I$. The first property tells us that this fibration is separated and that $\cat R$ is a reflective subcategory of its domain, and that $X\in \Sub(IRX)$ for all $X\in \cat R$. Also, because $IR$ is both faithful and regular, $C\in \Sub(IRC)$ is a filter. The second property tells us that $C$ is weakly generic, and the third property tells us that $\Sub(I-)$ satisfies Church's rule and the uniformity rule. If we now let $\chi$ be the set of $U\in \ext RC$ such that $C$ intersects $IU$, $(\Sub(I-),C)\simeq (\ds RC/\chi,\mathring{RC})$, because of theorem \ref{characterization}. Because of the adjunction $\Asm\dashv\Sub$, $\cat R$ must be equivalent to $\Asm(\ds RC/\chi)$.
\end{proof}

\begin{corollary} Slices of regular realizability categories are regular realizability categories.\label{slice} \end{corollary}

\begin{proof} Let $(\cat R,C,\cat H)$ be a regular realizability category, and let $X\in \cat R$. Prone arrows form a reflective subcategory of $\cat R/X$, whose inclusion $I_X:\mathrm{prones}(\cat R/X) \to \mathrm{prones}(\cat R/X)$ is regular, and equipped with a faithful finite limit preserving left adjoint $R_X$.

The diagonal functor $\cat R \to\cat R/X$, which sends $Y$ to the projection $Y\times X\to X$, is a Heyting functor. Therefore it preserves all of the internal characteristic properties of $C$, in particular combinatory completeness and the tracking principles. The weak genericity principle cannot be preserved this way, as subsection \ref{inexpressible} showed. We will now explain why it holds in slices of $\cat R$. It turns out that the proof of theorem \ref{Shanin} almost gives the reason.

Let $f:Z\to X$ be any morphism. By weak genericity there is a prone $p:Y\to C$ be and a regular epimorphism $e:Y\to Z$. The morphism $(p,f\circ e)$ factors as a vertical $v:Y\to Y'$ followed by a prone $(q,g):Y'\to C\times X$, and $q$ factors though $p$ in a unique vertical morphism $w:Y'\to Y$. We note that $p\circ w\circ v = q\circ v = p$ and therefore $w\circ v=\id_Y$, but because $R$ is faithful, this means $w$ and $v$ are inverses of each other. Now $R(f\circ e\circ w) = Rg$, because $f\circ e = g\circ v$ and $Rv = Rw = \id_{FY}$. But faithfulness of $R$ now implies that $f\circ e\circ w = g$.
\[\xymatrix{
& Y \ar[dl]_p \ar[dr]^e\ar@{.>}@<.5ex>[d]^v  \\
C & Y'\ar[l]^{q} \ar[r]_{e\circ w}\ar[d]|{(q,g)} \ar@{.>}@<.5ex>[u]^w & Z\ar[d]^f \\
& C\times X\ar[r]_{\pi_0}\ar[ul]^{\pi_1}& X
}\]
We now have a prone map $(q,g): Y' \to C\times X$ and a morphism $e\circ w:Y'\to Z$, commuting with $\pi_1:C\times X\to X$ and $f:Z\to Y$. The arrow $e\circ w$ is a regular epimorphisms because $w$ is a split epimorphism and $e$ is regular. So, thanks to the faithfulness of $R$, weak genericity extends to all slices of $\cat R$.

The slice categories have all characteristic properties of regular realizability categories, and must themselves be realizability categories. \end{proof}

We will extend the characterization to exact completions of regular realizability categories. 

\begin{remark}[exact base categories]\label{exact base}
Note that for a non-exact Heyting category $\cat H$, partial applicative structure $A$ and combinatory complete filter $\phi$ of inhabited downsets, the functor $\supp: \Asm(\ds A/\phi)_\exreg \to \cat H$ does not exist because $\cat H$ lacks quotients. In particular, $\Asm(\ds A/\phi)_\exreg$ contains quotients of equivalence relations in $\cat H$ that $\cat H$ does not contain because it is not exact.
This is no serious problem, because the inclusion $I:\cat H \to \cat H_\exreg$ is a regular functor. Therefore, 
$IA$ is a partial applicative structure in $\cat H_\exreg$ and $\im I(\phi) = \set{IU|U\in\phi}$ a filter of inhabited subobjects, so $\ds A/\phi$ has a canonical extension $\ds IA/\im I(\phi)$ to $\cat H_\exreg$. This extension coincides with the ordinary $\ds A/\phi$ over $\cat H$, because $\Sub(X)\simeq \Sub(IX)$ for all $X\in\cat H$ (see lemma \ref{sub in exreg}).

There is a much more problematic issue: the functor $\supp$ is no longer faithful!
\end{remark}

\begin{definition} An \xemph{exact realizability category} is an exact Heyting category $\cat E$, with an order partial combinatory algebra $C\in \cat E$, and a reflective subcategory $\cat H$, with fully faithful $I:\cat H\to\cat E$ and reflector $R:\cat E\to\cat H$, that satisfy:
\begin{enumerate}
\item $I$ is regular, $R$ preserves finite limits and $\eta_C:C\to IRC$, where $\eta:\id_{\cat E}\to IR$ is the unit of the reflection, is a monomorphism.
\item Weak genericity: for each object $X$ there is a span $(e:Y\to X,p:Y\to C)$ where $e$ is a regular epimorphism and $p$ is prone.
\item Tracking principle: For each chain of arrows $e:X\to Y$ and $f:Y\to Z$ where $e$ is a regular epic \prodo morphism and $f$ is either a \prodo morphism or a prone morphism, there is a \prodo morphism $g:W\to Z$, such that the pullback of $e$ along $g$ is split.
\[ \xymatrix{
\bullet \ar[r]\ar[d]\ar@{}[dr]|<\lrcorner & \bullet \ar[r]\ar[d]\ar@{}[dr]|<\lrcorner\ar@{.>}@/_/[l] & W\ar[d]^g \\
X \ar[r]_e & Y\ar[r]_f & Z
}\]
Moreover, this splitting takes a particular form: if $f$ is a \prodo morphism, then this splitting is a restriction of the application map; if $f$ is prone, then this splitting is an inclusion.
\end{enumerate}
\end{definition}

\begin{theorem} For each exact realizability category $(\cat E,C,\cat H)$ there is an external filter $\chi$ such that $\cat E \cong \Asm(\ds RC/\chi)_\exreg$. \label{charexact} \end{theorem}

\begin{proof} The subcategory $\cat R$ of $\cat E$ of the objects for which a monomorphism to an object of $\cat H$ exists, has all the properties definition \ref{regular realizability category}: the reflector $R$ is faithful there, and $C$ lives in this subcategory.
The objects that have a prone monomorphism to $C$ cover every object in $\cat E$, although they all live in $\cat R$. For this reason every $X\in\cat E$ is the quotient of some equivalence relation that exists in $\cat R$, and hence a member of $\cat R_\exreg$. But $\cat E$ is already exact, so $\cat E \cong \cat R_\exreg$.
\end{proof}

The obvious relation between regular and exact realizability categories holds.
\begin{proposition} The triple $(\cat R_\exreg,IC,\cat H_\exreg)$ is an exact realizability category if $(\cat R,C,\cat H)$ is a regular realizability category. \end{proposition}

\begin{proof} The inclusion $I:\cat R \to \cat R_\exreg$ is a Heyting functor just because it is an inclusion that is bijective on subobjects. Hence first order properties are preserved. Once again, the only problem is weak genericity, but the objects of $\cat R$ cover all objects of $\cat R_\exreg$, which gives us the weak genericity property right back. \end{proof}

\begin{remark} Because $R$ is not faithful, there is no regular epic \prodo morphism to every object in $\cat E$. Only slices over objects $X$ for which the unit $X\to IRX$ is a monomorphism are exact realizability categories, just because $(\cat R/X)_\exreg \cong \cat R_\exreg/IX$.
\end{remark}

\subsection{Examples}
We end this section with some classes of examples of realizability categories.

\begin{example} Let $\cat E$ be a well pointed topos with natural number object $N$. We can construct Kleene's first model $\Kone$ (see example \ref{Kone}) inside $\cat E$. We let the \xemph{effective topos} relative to $\cat E$ be $\Eff(\cat{E}) = \Asm(\ds \Kone/\Kone)_{\it ex/reg}$. Because $\cat E$ is two valued and has a projective terminal the results of subsection \ref{Markov} tell us that $\Eff(\cat E)$ can be characterized as the non-degenerate topos that satisfies \xemph{weak genericity}, \xemph{extended Church's thesis}, the \xemph{uniformity principle}, and \xemph{Markov's principle}. \end{example}

\begin{example} Any filter quotient of a realizability category is a realizability category by our definition. If $\chi$ is a filter of subterminals of $\Asm(\ds A/\phi)$, we construct a related filter $\chi'\subseteq \ext A$, by considering what downsets represent the subterminals contained in $\chi$. Because $\phi \subseteq \chi'$, there is a vertical morphism $m:\ds A/\phi \to \ds A/\chi'$ and this induces a functor $\Asm(m):\Asm(\ds A/\phi)\to\Asm(\ds A/\chi')$ that satisfies $\Asm(m)(U)\simeq \termo$ if $U\in\chi$. For this reason, $\Asm(m)$ factors through the quotient $Q:\Asm(\ds A/\phi) \to \Asm(\ds A/\phi)/\chi$ in an up to isomorphism unique functor $F:\Asm(\ds A/\phi)/\chi \to \Asm(\ds A/\chi')$. On the other hand, $(Q\nabla,\A)$ is a regular model for $(A,\chi')$, because for all $U\in \chi'$, the supports $\exists_!(\A\cap\nabla U)$ have become terminal objects. Hence there is an up to isomorphism unique functor $G:\Asm(\ds A/\chi') \to \Asm(\ds A/\phi)/\chi$. The functors $F$ and $G$ are pseudoinverses of each other. \end{example}

\begin{example} Any slice of a regular realizability category is a regular realizability category, by lemma \ref{slice}. \end{example}

\begin{example}[presheaves]
The category of internal presheaves on an internal meet semilattice $L\in \cat H$ is equivalent to $\Asm(\ds L/\set L)_\exreg$. \end{example}

\section{Pseudoinitiality}
Let $G$ be a weakly generic object of a fibred locale $F$. Since 1-morphisms of bifibred categories have to preserve prone and supine morphisms, each 1-morphism $(m_0,m_1):F\to F'$ is determined up to unique isomorphism by the combination of $m_0$ and $m_1(G)$. In this section we explain what kind of objects are images of the weakly generic filters of realizability fibrations and we explore morphisms into other fibred locales using these objects. Afterward, we combine these results with those of the previous section, in order to talk about regular functors between realizability categories.

\subsection{Fibred models}
We introduce a (2,1)-category of fibred models for an order partial applicative structure with a combinatory complete external filter, and show that the realizability fibration is an \xemph{pseudoinitial object}. We start be introducing these higher categorical concepts.

\begin{definition} A \xemph{$(2,1)$-category} is a $2$-category where all 2-morphisms are invertible. A \xemph{pseudoinitial object} is a 2-category is an object $I$ such that there is an up to isomorphism unique 1-morphism into every object. \end{definition}

Now we define a category of fibred locales combined with an object to send the weakly generic filter to.

\begin{definition} Let $\cat H$ be a Heyting category, let $A$ be an order partial combinatory algebra in $\cat H$ and let $\phi$ be a combinatory complete external filter. A \xemph{fibred model} is a separated fibred locale $F:\cat F \to\cat H$ together with a vertical filter $C\in F_A$, such that $!:F_{U\hookrightarrow A}(C) \to \termo$ is supine for all $U\in\phi$.

We let a \xemph{morphism of fibred models} $(F,C)\to (G,D)$ be a vertical morphism of fibred locales $H=(\mathrm{id}_{\cat H}, H)$ together with a vertical isomorphism $h:HC \to D$. A 2-morphism $(H,h)\to (K,k)$ is a 2-isomorphism of fibred categories $\eta:H\to K$, such that $k\circ\eta = h$.\end{definition}

\begin{remark} Consider the regular theory $\Theta$ that says $A$ has an internal filter that intersects the members of the external filter $\phi$. Regular theories can be interpreted in fibred locales like first order theories can be interpreted in complete fibred Heyting algebras. Fibred models are models of $\Theta$.
\end{remark}


\begin{theorem}[pseudoinitiality] The realizability fibration $\ds A/\phi$ together with the filter $\A$ is a pseudoinitial fibred model. \label{pseudoinitiality}\end{theorem}

\begin{proof} That the realizability fibration is a fibred model, is part of its characteristic properties, see theorem \ref{characterization}.

Let $(F:\cat F\to\cat H,C)$ be a fibred model. For each family of downsets $Y\hookrightarrow A\times X$, we exploit the fact that $Y = \im{x}(\ds A/\phi)_{a}(\A)$ if $x:Y\to X$ and $a:Y\to A$ are the projections: we choose a prone supine span $(p:Z\to C,s:Z\to Y')$ in the fibre of $(a,x)$, and let $m(Y) = Y'$. For each prone morphism $p:Y\to Y'$, there is a unique prone morphism $m(Y) \to m(Y')$ in the fibre over $(\ds A/\phi)(p)$ and in a similar way $m$ determines a unique mapping of supine morphisms. 

Let $v:Y\to Y'$ be a vertical morphism. There is an object of realizers $U\in \phi$ for this morphism. Let $\A_U = (\ds A/\phi)_{U\hookrightarrow A}$, $Z = (\ds A/\phi)_{(b,x)\mapsto b}(\A)$ and $Z' = (\ds A/\phi)_{(a,(b,x)) \mapsto ab}(C)$ and consider the maps:
\begin{align*}
f\oftype & (a,(b,x)) \mapsto ab: \A_U\times Z \to \A \\
g\oftype & (a,(b,x)) \mapsto x: Z' \to Y'
\end{align*}
The arrow $f$ factors through the prone morphism $\dom g \to \A$ in a unique vertical morphism $w$, and $g\circ w$ is equal to $v\circ (!\times s)$ for the supine $!\times s: (a,(b,x))\mapsto x: Z\to Y$. Now we note that there is a choice for $m(f)$ because $C$ is closed under application, and this choice in unique because $F$ is faithful, and $F(m(f)) = (\ds A/\phi)(f)$. This means there is a unique choice for $m(w)$, since the prone map $\dom g\to \A$ is preserved, and there is a unique factorization of $m(f)$ through it. Meanwhile $g$ is a supine map, which determines $m(g)$ and $m(g\circ w) = m(v\circ (!\times s))$. 
We get $m(!\times s): m(Z) \to m(Y)$, because $Z$ is the canonical family of prone downsets over $Y$, and $!:m(\A_U) \to \termo$ is supine because $m(\A_U) = F_{U\hookrightarrow A}(C)$ and $U\in \phi$. This determines $m(v):m(Y)\to m(Y')$ uniquely as the factorization of $m(g\circ w)$ over $m(s)$.
\end{proof}

There is an alternative set of morphisms between fibred models, determined by inclusions of filters, and they correspond to natural inclusions of morphisms between fibred models.

\begin{lemma} Let $C\subseteq D\in F_A$ be filters such that $(F,C)$ is a fibred model. Then $(F,D)$ is a fibred model, and there is a natural monomorphism $\mu:f_C \to f_D$ between the vertical morphisms $f_C,f_D:\ds A/\phi \to F$ induced by $(F,C)$ and $(F,D)$. \label{inclusion} \end{lemma}

\begin{proof} Reindexing preserves vertical morphisms, and for this reason $!:F_{U\hookrightarrow A}(D)\to \termo$ is supine if $!:F_{U\hookrightarrow A}(C)\to \termo$ is. Therefore $(F,D)$ is a fibred model if $(F,C)$ is.

Using the prone-supine spans, we turn the vertical morphism $C\to D$ into a natural transformation $f_C \to f_D$. Since $F$ is faithful, getting all desired naturality squares to commute is trivial. \end{proof}

Fibred models actually form a double category, with a set of squares determined by inclusions of vertical morphisms.

\begin{definition} Let $(F,C)$ and $(G,D)$ be fibred models, let $v:C\to C'\in F_A$ and $w:D\to D'\in G_A$ be inclusions of filters, and let $(H,h):(F,C)\to (G,D)$ and $(K,k):(F,C') \to (F,D')$ be morphisms of fibred models. A \xemph{square} between these is a 2-morphism of fibred locales $\sigma: H\to K$, such that $k\circ\sigma_{C'}\circ Hv = h\circ w$.
\end{definition}

\begin{lemma} Let $(H,h): (F,C)\to(G,D)$ and $(K,k): (F,C') \to (G,D')$ with $C\subseteq C'$ and $D\subseteq D'$. Let $i_C:f_C\to f_{C'}$ be the transformation induce by the inclusion $C\subseteq C'$ and $i_D:g_D\to g_{D'}$ the transformation induces by the other inclusion. A square $\sigma:(H,h) \to (K,k)$ satisfies $\sigma\circ i_C=i_D$.
\end{lemma}

\begin{proof} The one condition on squares $k\circ\sigma_{C'}\circ Hv = h\circ w$ induces the equality of natural transformations $\sigma\circ i_C=i_D$. \end{proof}

We summarize the results in this subsection as follows.

\begin{definition} Let $(\ds A/\phi)/\floc(\cat H)$ be the double category whose objects are vertical morphisms from $(\ds A/\phi)$ to arbitrary fibred locales $F$ over $\cat H$, whose horizontal arrows are commutative triangles of vertical morphisms of fibrations, whose vertical morphisms are 2-morphisms between vertical morphisms, and whose squares are 2-morphisms that commute with everything. \end{definition}

\begin{corollary} The double category of fibred locales, morphisms of fibred locales, inclusions and squares, is equivalent to
$(\ds A/\phi)/\floc(\cat H)$. \end{corollary}

\subsection{Regular and exact models}
In section \ref{categories of fractions} we saw that the 2-category $\reg$ of regular categories and regular functors is a bireflective subcategory of the 2-category $\floc$ of fibred locales over left exact categories and that the 2-category $\ex$ of exact categories is bireflective subcategory of $\reg$. Thanks to these reflections, the pseudoinitiality results of the previous subsection extend to the realizability categories $\Asm(\ds A/\phi)$ and $\Asm(\ds A/\phi)_\exreg$.

\begin{definition} Let $\cat H$ be a Heyting category, $A$ an order partial applicative structure and let $\phi$ be a combinatory complete external filter. A \xemph{regular model} $(\cat R,F:\cat H\to\cat R,C\subseteq FA)$ is a regular category $\cat R$ together with a regular functor $F:\cat H\to\cat R$ and a filter $C\subseteq FA$ such that $C$ intersects $FU$ for all $U\in \phi$. An \xemph{exact model} is a regular model $(\cat R,F,C)$ where $\cat R$ is an exact category. 

Regular and exact models form a (2,1)-category. A functor $H:\cat R\to \cat S$ with a natural isomorphism $h:HF\to G$ such that $h_A$ restricts to an isomorphism $HC\to D$ is a \emph{1-morphism} $(\cat R,F,C) \to (\cat S,G,D)$. A \emph{2-morphism} $(H,h)\to(K,k)$ is an natural isomorphism $\eta:H\to K$ such that $\eta\circ h = k$.
\end{definition}

\begin{example} The constant object functor $\nabla: \cat H \to\Asm(\ds A/\phi)$ is regular because $\ds A/\phi$ is separated. Also $\A\in (\ds A/\phi)_A$ turns into a filter $\A\subseteq \nabla A$ in $\Asm(\ds A/\phi)$. Together $\Asm(\ds A/\phi)$, $\nabla$ and $\A$ are a regular model. \end{example}

\begin{example} If $(\cat R,F:\cat H\to\cat R,C)$ is a regular model and $G:\cat R \to\cat R'$ a regular functor, then $(\cat R',GF,GC)$ is a regular model, because regular functors preserve models of regular theories. \label{image model} \end{example}


We just use a general result.
\begin{lemma} Left biadjoints preserve pseudoinitial objects. \end{lemma}

\begin{proof} Let $\cat A$ and $\cat B$ be 2-categories, let $I:\cat A\to\cat B$ be a 2-functor, and let $R:\cat B \to\cat A$ be a left biadjoint, i.e. $\cat A(RX,Y)\cong \cat B(X,IY)$ naturally. Of course, the $R$ preserves all pseudocolimits up to isomorphism. In particular, if $J$ is a pseudoinitial object of $\cat B$, the for each object $X$ of $\cat A$, there is an up to isomorphism unique 1-morphism $J\to IX$, with an up to isomorphism unique transpose $RJ\to X$. Therefore $RJ$ is a pseudoinitial object of $\cat A$. \end{proof}

\begin{corollary} Let $\nabla:\cat H \to \Asm(\ds A/\phi)$ be the constant object functor, then $(\Asm(\ds A/\phi),\nabla,A)$ is the pseudoinitial regular model. \label{asm pseudoinitial} \end{corollary}

\begin{proof} For each fibred model $(G,D)$ let $A(G,D) = (\Asm(G),\nabla',D)$, where $\nabla'$ is the constant object functor $\cat H \to\Asm(G)$. For each regular model $(\cat R,F,C)$ let $S(\cat R,F,C) = (\Sub(F-),C)$ where $\Sub(F-)$ is the pullback of $\Sub(\cat R)$ along $F:\cat H\to\cat R$. These mappings extend to bifunctors between the (2,1)-categories of fibred models (with only isomorphism as inclusions) and regular models. They are biadjoint because a morphism $H:A(G,D) = (\Asm(G),\nabla',D) \to (\cat R,F,C)$ corresponds to a morphism $H^\dagger: G\to \Sub(R)$ that sends $D$ to $C$ and that factors uniquely through $\Sub(F-)$ because of the pull back. In the other direction, any morphism $K:(G,D) \to S(\cat R,F,C) = (\Sub(F-),C)$ can be composed with the pullback morphism $\Sub(F-) \to \Sub(\cat R)$ and this functor has an up to isomorphism unique transpose $K':\Asm(G) \to \cat R$ that satisfies $K'D\simeq C$ and $K'\nabla'\simeq F$. So the adjunction between fibred locales and regular categories induces an adjunction between fibred models and regular models. For this reason $A(\ds A/\phi,\A) = (\Asm(\ds A/\phi),\nabla,\A)$ is a pseudoinitial regular model. 
\end{proof}

\begin{corollary} Let $\Delta:\cat H\to\Asm(\ds A/\phi)_\exreg$ be the composition of $\nabla$ with the embedding $I:\Asm(\ds A/\phi)\to\Asm(\ds A/\phi)_\exreg$. Now $(\Asm(\ds A/\phi)_\exreg,\Delta,A)$ is the pseudoinitial exact model. \label{asmexreg pseudoinitial}\end{corollary}

\begin{proof} The forgetful functor the sends exact models to regular models has a left biadjoint thanks to the ex/reg completion. If $(H,h):(\cat R,F,C)\to (\cat E,G,D)$ is a morphism of regular models, and $\cat E$ is exact, then there is an up to isomorphism unique functor $K:\cat R_\exreg \to \cat E$ that is a morphism of exact models $(\cat R_\exreg,F,C)\to (\cat E,G,D)$. So ex/reg completion is a left biadjoint to the forgetful functor. We see that $(\Asm_\exreg,\Delta,\A)$ is a pseudoinitial object, because $(\Asm(\ds A/\phi),\nabla,\A)$ is.
\end{proof}

Up to now we have ignored inclusions of filters and the morphisms they induce between regular models. We will say something about them now.

\begin{lemma} Let $F:\cat H\to\cat R$ and let $C\subseteq C'\subseteq FA$ be two filters such that $(\cat R,F,C)$ is a regular model. Then $(\cat R,F,C')$ is a regular models and there is a natural inclusion between the regular functors $F_C$ and $F_{C'}:\Asm(\ds A/\phi) \to \cat R$ that are induced by the regular models. \end{lemma}

\begin{proof} We can apply lemma \ref{inclusion} in combination with the adjunction in the proof of corollary \ref{asm pseudoinitial}. \end{proof}

\begin{remark} In fact we can extend the (2,1)-categories of regular and exact models to double categories, which are equivalent to certain complicated double categories of regular functors, namely $\nabla/(\cat H/\reg)$ and $\Delta/(\cat H/\reg)$ with all imaginable morphisms in between. \end{remark}

\begin{example}[characters] Let $\cat H$ be a topos, and let $(f^*\dashv f_*):\cat G\to\cat H$ be a geometric morphism between toposes. Let $A$ be an order partial applicative structure in $\cat H$ and let $\phi$ be an external filter. Consider the following type of maps $\chi:A\to f_*\Omega$ where we use the locale structure of $f_*\Omega$:
\begin{itemize}
\item if $a\leq b$, then $\chi(a)\leq \chi(b)$;
\item if $ab\converges$, then $\chi(a)\land\chi(b)\leq \chi(ab)$;
\item if $U\in\phi$, then $\bigvee_{a\in U}\chi(a) = \top$.
\end{itemize}
Such a map $\chi$, which we call a \xemph{character}, induces a filter on $C\subseteq f^*(A)$, by the isomorphisms:
\[ \Sub(f^*(A))\simeq \cat G(f^*(A),\Omega) \simeq \cat H(A,f_*(\Omega)) \]
Because $f^*$ is regular and $\cat G$ is exact, the filter $C$ induces a regular functor $g^*:\Asm(\ds A/\phi)_\exreg \to \cat G$ that satisfies $g^*\nabla\simeq f^*$ and $g^*(\A) \simeq C$.

If $(f^*\dashv f_*)$ is \xemph{localic}, then $g^*$ has a right adjoint $g_*$. It is constructed as follows. There is an assembly $G = \{(a,p) \in A\times f_*\Omega | \gamma(a) \leq p\}$, and this is a filter on $\nabla f^*\Omega$, because it has a meet semilattice structure. A property of localic geometric morphisms $f:\cat G \to\cat H$ is that every object $X$ of $\cat G$ can be represented as a partial equivalence relation $e:Y^2 \to f^*\Omega$ in $\cat H$. This determines a partial equivalence relation $\pre{\nabla e}(G)$ on $\nabla Y$ and there is an up to isomorphism unique functor $g_*$ that sends $(Y,e)$ to a subquotient for this partial equivalence relation.

Filters $C$ of $A$ that intersect all members of $\phi$ are characters relative to the identity of $\cat H$. Such filters determine all geometric morphisms $(h^*\dashv h_*):\cat H\to \Asm(\ds A/\phi)_\exreg$ such that $h^*\nabla \simeq \id_\cat H$.\end{example}

\begin{remark}[relative realizability categories] We had a characterization of categories of the form $\Asm(\ds A/\phi)$ in theorem \ref{charreg}. The relative realizability category $\Asm(\ds A/A')$, where $\ds A/A' = \ds A/\phi$ for the external filter $\phi$ of downsets that intersect a filter $A'\subseteq A$, stands out as the terminal realizability category over $\cat H$ and $A$ for which a regular functor $F:\Asm(\ds A/A') \to\cat H$ exists such that $F\nabla\simeq \id_\cat H$ and $F\A\simeq A'$. If $\chi$ contains more downsets than $\phi$, then $(\id_{\cat H},C)$ is no fibred model, and if $\chi$ contains less, there is a Heyting functor $\Asm(\ds A/\chi) \to \Asm(\ds A/A')$ by lemma \ref{Heyting morphisms}.
\end{remark}

\subsection{Applicative morphisms}
By applying the equivalence of the former subsection to morphisms between different realizability fibrations, we can generalize Longley's \xemph{applicative morphisms} \cite{RTnLS} to our more general setting.

\begin{definition}[applicative morphisms] Let $(A,\phi)$ and $(B,\chi)$ be pairs of order partial combinatory algebras with external filters. An \xemph{applicative morphism} $(A,\phi) \to (B,\psi)$ is a subobject $C\subseteq B\times A$ that satisfies:
\begin{itemize}
\item If $(b,a)\in C$ and $b'\leq b$ then $(b',a)\in C$.
\item There is a $U\in \chi$ such that if $(b,a)\in C$ and $a\leq a'$, then $ub\converges$ for all $u\in U$, and $(ub,a')\in C$.
\item There is an $R\in \chi$ such that if $(b,a)$ and $(b',a')\in C$ and $aa'\converges$, then $rbb'\converges$ for all $r\in R$ and $(rbb',aa')\in C$.
\item If $U\in \phi$ then $\{b\in B|\exists a\in U. (a,b)\in C \} \in \chi$
\end{itemize} \label{appmorph}

The ordinary composition operation on relations determines the composition of applicative morphisms, i.e. for $D:(A,\phi) \to (B,\chi)$ and $D':(B,\chi) \to (C,\psi)$ we let:
\[ D'\circ D = \set{(c,a)\in C\times A| \exists b\in B. (b,a)\in D\land (c,b)\in D' }\]

The set of all applicative morphisms $(A,\phi) \to (B,\chi)$ has a preorder: let $C$ and $D$ be morphisms $(A,\phi) \to (B,\chi)$, then $C\ll D$ if there is an $R\in\phi$ such that $rb\converges$ and $(rb,a)\in D$ for all $r\in R$ and all $(b,a)\in \phi$.
\end{definition}


\begin{lemma} There is an equivalence between applicative morphisms $(A,\phi) \to (B,\chi)$ and vertical morphisms $\ds A/\phi \to \ds B/\chi$.
\end{lemma}

\begin{proof} Any morphism $C:(A,\phi) \to (B,\chi)$ is a family of downsets of $B$ indexed over $A$ and therefore represents a downset $C\in \ds B_A$. This subobject is a filter for $(A,\phi)$, because the definition of applicative morphisms provides realizers for upward closure, closure under application and the inclusion of $\termo$ into $\im !((\ds B/\phi)_{U\hookrightarrow A}(C))$ for all $U\in \phi$. We then use the fact that $(\nabla,\A)$ is a pseudoinitial model to derive the equivalence.
\end{proof}

We now show some examples of applicative morphisms and functors derived from them.

\begin{example} When $\phi\subseteq \chi \subseteq \ext A$ are two combinatory complete external filters on $A$, the order of $A$ determines an applicative morphism: $\roso: (A,\phi) \to (A,\chi)$. This is the way we get an applicative identity morphism on every order partial combinatory algebra. The resulting morphisms are Heyting morphisms, as we have seen in lemma \ref{Heyting morphisms}. 

Note that we get a complete category of realizability fibrations with varying external filters and vertical Heyting morphisms between them: the downsets of realizers of partial combinatory arrows generate the least combinatory complete external filter, which is a terminal object, and filters are closed under arbitrary intersections, which means that we have arbitrary products. Posets always have all equalizers.
\end{example}

\begin{example} Consider the applicative structure $\termo\in \cat H$. The least combinatory complete filter is $\set\termo$, for which $\Asm(\ds\termo/\set\termo) \cong \cat H$, because $\ds\termo/\set\termo$ is equivalent to the subobject fibration of $\cat H$. Every applicative structure $A$ is a morphism $(\termo,\set\termo) \to (A,\phi)$, which induces the morphism $\Sub(\cat H) \to \ds A/\phi$ from lemma \ref{sublocale}. If all members of $\phi$ are inhabited, then $A:(A,\phi) \to (\termo,\set\termo)$.
\end{example}

\begin{example}[single valued morphisms] Another type of applicative morphisms arises from monotone maps $f:A\to B$ between order partial applicative structures that \xemph{laxly preserve application}, i.e. if $xy\converges$ then $f(x)f(y)\converges$ and $f(x)f(y)\leq f(xy)$. Let $\phi$ be an external filter on $A$ and let $\phi_f = \set{V\in \ext B| \pre f(V)\in \phi}$. For any combinatory complete external filter $\chi$ such that $\phi_f\subseteq \chi$, $B/f = \set{(b,a)\in B\times A | b\leq f(a) }$ is an applicative morphism $(A,\phi) \to (B,\phi_f)$.

If $f$ is a monotone map that \xemph{laxly reflects application}, i.e. if $f(x)f(y)\converges$ then $xy\converges$ and $f(xy)\leq f(x)f(y)$, then $f$ induces an applicative morphism in the opposite direction. Let $\chi$ be an external filter on $B$.
For every combinatory complete external filter $\phi$ such that $\chi\subseteq \phi_f$, $f/B = \set{(a,b)\in A\times B | f(a)\leq b }: (B,\chi)\to (A,\phi)$ is an applicative morphism.

In case $f$ does both, $B/f$ is a left adjoint of $f/B$, because $a\leq a'$ implies $(a,a')\in f/B\circ B/f$, and $(b,b')\in B/f\circ f/B$ implies $b\leq b'$. The adjunction is between $(A,\phi)$ and $(B,\phi_f)$ for any suitable $\phi$. \label{single valued morphisms}
\end{example}

\section{Realizability toposes}
Exact realizability categories for order partial applicative structures and combinatory complete external filters in \xemph{toposes} are toposes themselves. This section is ultimately about geometric morphisms of realizability toposes. In order to study those, we collect some properties of realizability categories that help us to understand irregular finite limit preserving functors between realizability categories.

The road to geometric morphisms twists as follows.

For a special type of external filter $\phi$ the inhabited prone downsets of the canonical filter $\A\in\Asm(\ds A/\phi)$ have global sections. We call these filters \xemph{generated by singletons}. This forces left exact functors from $\Asm(A/\phi)$ to other categories to preserve realizers, which makes the study of such functors easier. If the base category is a topos, then we can embed each realizability category into one where the external filter is generated by singletons. This is worked out in example \ref{completiontopos}. Moreover, realizability fibrations over toposes are triposes. For these two reasons, every left exact functor $F:\Asm(\ds A/\phi)_\exreg\to\cat E$ into an exact category is determined up to unique isomorphism by the composed left exact functor $F\nabla$ and the object $FG$ where $G$ is a \xemph{generic assembly}, in a way similar to how $F\nabla$ and $F\A$ determine a regular functor $F:\Asm(\ds A/\phi))_\exreg\to\cat E$ up to isomorphism.

The conclusion generalizes van Oosten and Hofstra's \xemph{computationally dense} applicative morphisms (see \cite{MR1981211}) to all of our realizability toposes.

\subsection{Generated by singletons}
For some external filters $\phi$ the inhabited prone downsets of the canonical filter $\A$ in $\Asm(\ds A/\phi)$ have global sections. This property is useful because all finite limit preserving functors preserve inhabited subobjects that have global sections.

\begin{definition} Let $\cat H$ be a Heyting category and let $A$ be a partial applicative structure. An external filter $\phi$ on a partial combinatory algebra is \xemph{generated by singletons} if for each $U\in\phi$ there exists a global section $u:\termo \to U$, such that the \xemph{principle downset} $\set{a\in A|a\leq u}$ is a member of $\phi$. \end{definition}

\begin{remark} There is an isomorphism between the poset of filters of $\cat H(\termo,A)$ and the poset of external filters of $A$ that are generated by singletons. The set of all global sections $u:\termo\to A$ such that $\set{a\in A|a\leq u}\in \phi$ is a filter. For each filter $C\subseteq \cat H(\termo,A)$ there is a least filter $\mu_C$ such that $\set{a\in A|a\leq u}\in \mu_C$ for all $u\in C$. These constructions are inverses of each other, an therefore we say that these filters are `generated by singletons'. For globals sections $u:\termo \to A$ we will often write $u:\termo \to A\in \phi$ instead of $\set{a\in A|a\leq u}\in \phi$ to exploit this isomorphism.
\end{remark}

We check that these indeed have the desired property.
\begin{lemma} If $\phi$ is generated by singletons, then inhabited prone downsets of $\A$ have global sections. \end{lemma}

\begin{proof} A prone downset $U\subseteq \A$ is inhabited if and only if $FU\in \phi$. If $FU\in\phi$ then there is a global section $u:\termo \to FU$ and a set of realizers $\set{a\in A|a\leq u}\in \phi$ that makes $u$ a global section of $U$ in $\Asm(\ds A/\phi)$. \end{proof}

That morphisms and propositions have global sections for realizers has far reaching consequences for realizability categories, some of which we will explore now. The following subcategories of realizability categories become far better behaved for filters that are generated by singletons.

\newcommand\Pasm{\Cat{Pasm}}
\begin{definition} A \xemph{partitioned assembly} is any $P\in \cat DA/\phi$ for which a prone morphism $p:P\to\A$ exists. The \xemph{category of partitioned assemblies} is the full subcategory $\Pasm(A,\phi)$ of $\cat DA/\phi$ whose objects are partitioned assemblies.
\end{definition}

\begin{lemma} Assuming external filters of inhabited downsets, partitioned assemblies are fine objects relative to supine monomorphisms, and hence assemblies. \end{lemma}

\begin{proof} Remember we have an adjunction $\supp\dashv\nabla$. Because $p:P\to \A$ is prone, it is the pullback of $\nabla\supp p: \nabla \supp P \to\nabla A$ along the unit of the adjunction. If we apply $(\ds A/\phi)$ to this pullback diagram, we see that $(\ds A/\phi)P\simeq \supp P$, which is a property that characterizes fine objects. \end{proof}

Ordinarily the category of partitioned assemblies is not that well behaved. For example, if $A$ has no global sections, then the terminal object is never a partitioned assembly. This also happens when $\phi$ contains no principle downsets. Another problem is binary products: a prone arrow $p:\A\times \A \to \A$ determines a paring operator $\supp p:A^2 \to A$ that may not exist, or may not have a realizer in $A$. However, we just introduced a property of filters that solves these problems.

\begin{proposition} Let $\phi$ be generated by singletons and combinatory complete. The category $\Pasm(A,\phi)$ has all finite limits. \end{proposition}

\begin{proof} Combinatory completeness implies $\phi$ is non-empty and therefore contains at least one principle downset, in particular one generated by a global section $p$ of $\comb p = \set{(x,y,z)\mapsto zxy}$. This not only turns $\termo$ into a partitioned assembly, but it also determines a paring operator $(a,b) \mapsto pab: A^2 \to A$ that lifts to a prone morphism $\A\times \A \to \A$. Partitioned assemblies are closed under equalizers in any circumstance, because equalizers are prone morphisms. \end{proof}

\begin{corollary} The constant object functor $\nabla$ factors through $\Pasm(A,\phi)$. \end{corollary}

\begin{proof} All maps $!:\nabla X\to \termo$ are prone. \end{proof}


We now make the connection with left exact functors.

\begin{definition} A \xemph{left exact model} is a left exact functor $F:\cat H\to\cat C$ together with a filter $C\subseteq FA$ such that for each $U\in\phi$ there is a global section $\termo\to C\cap FU$. Like regular models, they form a $(2,1)$-category where a morphism $(F,C)\to (G,D)$ is a functor $H:\cod F\to\cod G$ together with a natural isomorphism $h:HF\to G$ that restricts to an isomorphism $h:FC\to D$.
\end{definition}

\begin{proposition} The constant object functor $\nabla:\cat H \to \Pasm(A,\phi)$ together with $\A$ is the pseudoinitial left exact model.\label{pasm pseudoinitial} \end{proposition}

\begin{proof} The mapping of each partitioned assembly $P$ is determined up to isomorphism by any prone map $p:P\to \A$: 
\[ F_C(p) \simeq \pre{F\supp p}(C) \]

Each map of partitioned assemblies $f:P\to Q$ has a global section $\termo \to A$ tracking it. In other words, if we choose prone $p:P\to \A$ and $q:Q\to \A$, then there is an $r:\termo \to A$ such that $(r\supp p,\supp(q\circ f)):P \to A^2$ factors through the ordering $\roso$ of $A$. The global sections $Fr$ factors through $C$. Because $C$ is a filter, $F\supp f:F\supp P\to F\supp Q$ restricted to $F_C(p)$ factors uniquely through $F_C(q)$. 
\[ \xymatrix{
& \roso\ar[dl]_{\pi_0} \ar[drr]^{\pi_1} \\
C \ar[d] & F_C(P)\ar[d]\ar[l]\ar@{}[dl]|<\llcorner \ar@{.>}[r]\ar[u]_{(rp,q\circ f)} & F_C(P)\ar[d]\ar[r]\ar@{}[dr]|<\lrcorner & C \ar[d] \\
FA & F\supp P \ar[r]_{F\supp f}\ar[l]^{F\supp p} & F\supp Q \ar[r]_{F\supp q} & FA 
}\]
\end{proof}

\begin{remark} We should look for properties that characterize $\Pasm(A,\phi)$ up to isomorphism. This list may work:
\begin{itemize}
\item $F$ has a faithful left exact left adjoint left inverse.
\item $C\to FA$ is a \xemph{generic monomorphism}. See Menni \cite{Menni00exactcompletions}.
\item Arrows $\termo \to C$ represent all morphisms $FU\cap C^n\to C$ for all $U\subseteq A^n$
\item Every morphism $FX\to C$ has a lower bound $\termo \to C$.
\item For each $U\in\ext A$, $FU\cap C$ has a global section if and only if $U$ is a member of $\phi$.
\end{itemize}
\end{remark}

Proposition \ref{pasm pseudoinitial} applies to a small subcategory of $\Asm(\ds A/\phi)$ in this one special case where $\phi$ is generated by singletons. The impredicative nature of toposes allows us to extend it to whole realizability toposes, as we see in the coming subsections.

\subsection{Inhabited join completion}\label{downsetmonad}
If $\cat H$ is a topos then we can pull the following trick. For each order partial combinatory algebra $A$ and each combinatory complete external filter $\phi$ whose members are inhabited,  $\Asm(\ds A/\phi)\cong\Pasm(B,\chi)$ for some other order partial combinatory algebra $B$ and a filter $\chi$ that is generated by singletons. This subsection explains how.

\newcommand\ijc{\partial}
\begin{definition} For each order partial applicative structure $A$ in each topos $\cat H$ let $\ijc A$ be the order partial combinatory algebra of inhabited downsets. Its ordering is the inclusion ordering. Application is defined as follows. If $\alpha:A^2\supseteq D\to A$ is the partial application operator of $A$, then $UV\converges$ if $U\times V\subseteq \dom \alpha$ and 
\[ UV = \set{x\in A|\exists u\oftype U,v\oftype V.x\leq \alpha(u,v)} \]
The resulting algebra $\ijc A$ is the \xemph{inhabited join completion} of $A$.

Let $\phi$ be an external filter of $A$ whose members are inhabited. We let $\phi^+$ be the filter on $\ijc A$ generated by the global sections $\termo \to \ijc A$ that corresponds to the members of $\phi$. \end{definition}

\begin{remark} The order partial combinatory algebra $\ijc A$ is kind of a complete applicative lattice. As an internal poset of $\cat H$, it has joins for all inhabited subobjects and the application operator distributes over those joins. Also, any subobject $U\subseteq \ijc A$ that has a lower bound, has a greatest lower bound $\inf U$. We also have the operator $\arrow$, such that $x\subseteq y\arrow z$ if and only if $xy\converges$ and $xy\subseteq z$.
\end{remark}

\begin{proposition} $\Pasm(\ijc A,\phi^+)\cong \Asm(\ds A/\phi)$\label{asm is pasm} \end{proposition}

\begin{remark} This generalizes a result of Hofstra and van Oosten, see section 4 of \cite{MR1981211}. \end{remark}

\begin{proof} We use the pseudoinitiality properties to establish a functor and then argue why it is an equivalence. 

For $\Pasm(\ijc A,\phi^+) \to \Asm(\ds A/\phi)$ the functor is $\nabla:\cat H \to \Asm(\ds A/\phi)$. The family $G = \set{(a,U)\in A\times \ijc A| a\in U}$, which is a family of inhabited downsets and therefore an assembly, determines the filter. Now $(\nabla,G)$ is a left exact model for $(\ijc A,\phi^+)$ for the following reasons:
\begin{itemize}
\item if $(a,U)\in G$ and $U\subseteq V\in \ijc A$, then $(a,V)\in G$; therefore $\db{a\mapsto a}$ realizes upward closure;
\item if $(a,U)\in G$ and $(b,V)\in G$ and $UV\converges$, then $ab\converges$ and $(ab,UV)\in G$; therefore $\db{(a,b)\mapsto ab}$ realizes closure under application;
\item if $U\in \phi^+$, then there is a $u:\termo \to U$ such that $\set{a\in \ijc A|a\leq u}\in\phi^+$, and this is the case if and only if $\set u = \set{a\in A|a\in u}\in \phi$; therefore $\comb k\set u$ realizes $u:\termo \to U$, where $\comb k = \db{(x,y)\mapsto x}$.
\end{itemize}
So there is an up to isomorphism unique functor $F:\Pasm(\ijc A,\phi^+)\to\Asm(\ds A/\phi)$ such that $F\nabla\simeq \nabla$ and $F\mathring{\ijc A}\simeq G$, where $\mathring{\ijc A}$ is the (weakly) generic filter of $\Pasm(\ijc A,\phi^+)$.

The functor $F$ is essentially surjective, because every assembly $X$ is a pullback of $G$ along some map $\nabla \supp X\to \nabla \ijc A$. As we noted in remark \ref{assemblies are inhabited families} $X$ is an inhabited family of downsets, and because $\ijc A$ is the object of inhabited downsets, we get the required map $\supp X\to \ijc A$. The functor $F$ is faithful because $\nabla\supp:\Pasm(\ijc A,\phi^+) \to\Asm(\ds A/\phi)$ is faithful and $F$ is a subfunctor. The functor is full, because by theorem \ref{tracking} each morphism between assemblies is tracked, and its object of realizers corresponds to a global section of $\ijc A$, which generates a principle downset in $\phi^+$. For all these reasons $F$ is an equivalence of categories.
\end{proof}

\begin{remark} The object $G$ is a \xemph{generic object} of $\supp = \ds A/\phi:\Asm(\ds A/\phi) \to \cat H$, i.e. a \xemph{generic assembly}, because prone morphisms connect every assembly to $G$. There is an inclusion $G\to\nabla DA$ where $DA$ is the objects of downsets of $A$. This inclusion a \xemph{generic monomorphism} (see \cite{Menni00exactcompletions}) for $\Asm(\ds A/\phi)$.
\end{remark}

\subsection{Tripos-to-topos}
Here we demonstrate that realizability fibrations over toposes are triposes and explain the connection between morphisms of triposes and finite limit preserving functors on the related toposes. Most of this can be found in \cite{Pittsthesis} and \cite{MR578267}.

\begin{definition} A \xemph{tripos} is a complete fibred Heyting algebra $F:\cat F\to \cat B$ over a category with finite products, such that for each object $X\in\cat B$, there is an $\pi X\in \cat B$ and a \xemph{membership predicate} $\epsilon_X\in F_{X\times \pi X}$ such that for each $P\in F_{X\times Y}$ there is an $p:Y\to \pi X$ and a prone arrow $q:P\to \epsilon_X$ with $Fq=\id\times p$. \label{tripos} 
\end{definition}

\begin{example}
Let $\cat H$ be a topos, let $A$ be a partial applicative structure and let $\phi$ be a combinatory complete external filter. The realizability fibration $\ds A/\phi$ is a tripos.
Theorem \ref{rf is ha} tells us that $\ds A/\phi$ is a complete fibred Heyting algebra. Let $DA$ be the object of downsets of $A$. Note that $\pi_1:\set{(a,b)\in A\times DA|a\in b}\to DA$ classifies families of downsets. For each object $X$ of $\cat H$, let $\pi X = DA^X$ and let $E_X = \set{(a,f,x)\in A\times \pi X\times X| a\in f(x) }$. This determines a membership predicate, because for every family of downsets $Z \subseteq A\times X\times Y$ there is a morphism $z:Y \to \pi X$ such that $Z = \pre{(\id\times z)}(E_X)$. Therefore, $E_X$ is a membership predicate.
\end{example}


\begin{theorem}[Pitts] For each tripos $F:\cat F\to\cat B$, let $\cat B[F]=\Asm(F)_\exreg$. Then $\cat B[F]$ is a topos.\end{theorem}

\begin{proof} See \cite{Pittsthesis}. \end{proof}

\newcommand\RT{\Cat{RT}}
\begin{definition} For each combinatory complete external filter $\phi$ on each order partial applicative structure $A$ in each topos $\cat H$, $\RT(A,\phi) = \cat H[\ds A/\phi]$, where $\RT$ stands for \xemph{realizability topos}. \end{definition}

This construction turns finite limit preserving morphisms of bifibrations between triposes into regular functors between toposes. A surprising fact about triposes, however, is that a morphism of fibred categories $(m_0,m_1):F\to F'$ such that $m_0$ does not preserve regular epimorphisms and $m_1$ does not preserve supine morphisms, still determines a finite limit preserving functor $\cat B[F] \to \cat B'[F']$. We derive the following useful fact.

\begin{lemma} Let $F:\cat F\to\cat B$ be a tripos, and let $G:\Asm(F) \to\cat E$ be a left exact functor to an exact category. There is an up to isomorphism unique $H:\cat B[F] \to\cat E$ such that $HI\simeq G$, if $I:\Asm(F) \to \cat B[H]$ is the canonical embedding in the ex/reg completion. \end{lemma}

\begin{proof} The functor $G$ induces a morphism of fibred categories $(G\nabla, G):F\to\Sub(\cat E)$, i.e. the maps preserve limits and prone morphisms respectively. The proof in \cite{Pittsthesis} that such morphism lift to morphisms of triposes, relies on the fact that $F$ is a tripos, and does not require $\Sub(\cat E)$ to be one. \end{proof}

\begin{proposition} Let $\cat H$ be a topos, let $A$ be a partial applicative structure and let $\phi$ be a combinatory complete external filter. Each left exact model $(F:\cat H\to\cat E,C)$ for $\ijc A$ and $\phi^+$, where $\cat E$ is exact, induces an up to isomorphism unique finite limit preserving functor $F_C:\RT(A,\phi) \to \cat E$. \label{lems and toposes}
\end{proposition}

\begin{proof} This is the combination of propositions \ref{pasm pseudoinitial} and \ref{asm is pasm} with the lemma above. \end{proof}

\subsection{Computational density}
In order to analyze geometric morphisms between realizability toposes, we determine which kind of applicative morphisms are induced by left exact models.

\begin{definition} A \xemph{left exact morphism} $(A,\phi) \to (B,\chi)$ is a filter $C\subseteq \ijc A$
together with a map $f:C\to \ijc B$ that satisfies:
\begin{itemize}
\item there is a $q:\termo \to \ijc B$ in $\chi^+$ such that if $x\subseteq y$ and $x\in C$ then $qf(x)\converges$ and $rf(x) \subseteq f(y)$;
\item there is an $r:\termo \to \ijc B$ in $\chi^+$ such that if $xy\converges$, and $x,y\in C$, then $rf(x)f(y)\converges$ and $(rf(x))f(y)\leq f(xy)$;
\item for all $x:\termo \to \ijc A$ in $\phi^+$, $x$ factors through $C$ and $f\circ x\in \chi^+$.
\end{itemize}

Left exact morphism are composed like partial morphisms, i.e. if $(D,f):(A,\phi) \to (B,\chi)$ and $(E,g):(B,\chi) \to (C,\psi)$, then $(E,g)\circ (D,f) = (f^{-1}(E),g\circ f)$. They are ordered as follows: $(C,f)\leq (D,g)$ if $C\subseteq D$ and for some $r:\termo\to \ijc B  $ in $\chi^+$, $rf(a)\subseteq g(a)$ for all $a\in C$. \end{definition}

\begin{remark} We choose to work with $\ijc A$ instead of the object $DA$ of all downsets of $A$ because of the connection with partitioned assemblies. The disadvantage is that we have to work with partial maps $X\partar \ijc A$ to describe subassemblies of $\nabla X$, while we could equivalently represent them with total maps $X\to DA$. \end{remark}

\begin{example} Any applicative morphism $C:(A,\phi) \to (B,\chi)$ induces a left exact morphism. Let $f(u) = \set{b\in B| \exists a\in u. (b,a)\in C}$, and let $C' = \set{u\in \ijc A|f(u)\in \ijc B  }$. Now $(C',f):\ijc A\to \ijc B  $ satisfies the properties listed above. \end{example}

\begin{example} The downward closure map $d:\ijc A \to \ijc\ijc A$ determines a morphism $(A,\phi) \to (\ijc A,\phi^+)$, while $\set{(a,u)\in A\times \ijc A| a\in u }$ is an applicative morphism $(\ijc A,\phi^+) \to (A,\phi)$, corresponding to the union map $U: \ijc\ijc A\to \ijc A$. Now $U$ is a left adjoint to $d$ relative to $\subseteq$, and $d$ is injective. This adjunction induces the geometric inclusion of realizability toposes $\RT(A,\phi)\to \RT(\ijc A,\phi^+)$ we hinted at at the beginning of this section. \label{completiontopos}
\end{example}

\begin{lemma} There is an equivalence between left exact morphisms $(A,\phi) \to (B,\chi)$ and left exact functors $\RT(A,\phi) \to \RT(A,\chi)$ that commute with $\nabla$. \end{lemma}

\begin{proof} A left exact morphism $(C,c):(\ijc A,\phi^+)\to (\ijc B  ,\chi^+)$ corresponds to a filter $C$ of $\nabla \ijc A$ with enough global sections in $\Pasm(\ijc B  ,\chi^+)$ and $(\nabla, C)$ corresponds to some left exact functor $\Pasm(\ijc A,\phi^+) \to \Pasm(\ijc B  ,\chi^+)$ that commutes with $\nabla$. See proposition \ref{lems and toposes}.

Each left exact functor $F$ that commutes with $\nabla$ induces a left exact model $(\nabla, F\mathring{\ijc A})$. For each prone map $f: F\mathring{\ijc A} \to \mathring{\ijc B}$, $(\supp F \mathring{\ijc A},\supp f)$ is a left exact morphism $(A,\phi)\to(B,\chi)$. \end{proof}

Since left exact morphisms are ordered, they can be adjoints, and adjunctions of left exact morphisms correspond to adjunctions of left exact functors between realizability toposes. Left adjoints are regular and therefore determined by applicative morphisms. However, not all applicative morphisms induce geometric morphisms.

\begin{definition} A left exact morphism $(C,f):(\ijc A,\phi^+) \to (\ijc B  ,\chi^+)$ is \xemph{computationally dense} if it preserves arbitrary unions and if there is an $\comb m:\termo \to \ijc B$ in $\phi^+$ such that for all $x:\termo\to \ijc B  \in \chi^+$ there is a $x':\termo\to \ijc A  \in \phi^+$ such that for all $y\in \ijc A$ if $xf(y)\converges$, then $\comb mf(x'y)\subseteq xf(y)$. \end{definition}

\begin{proposition}[Hofstra, van Oosten] A left exact morphism $(C,f)$ is computationally dense if and only if there is a right adjoint $(D,g)$. \end{proposition}

\begin{proof} Assume computational density. Let $g:\ijc B \to DA$ satisfy $g(y) = \bigcup \set{x\in \ijc A| \comb mf(x)\leq y}$ and let $D = \pre g(\ijc A)$, so $g:D\to \ijc A$ as required for left exact morphisms. This map $g$ is monotone; if $x\subseteq y$ then $g(x)\subseteq g(y)$, for all $x\in D$. Closure under application is harder to show. It requires some programming.

To start with, let $\comb p$, $\pi_0$ and $\pi_1:\termo\to\ijc A$ be paring and unparing combinators in $\phi^+$, i.e. $\pi_0(\comb pxy) \leq x$ and $\pi_1(\comb pxy) \leq y$. Let $r_f:\termo\to \ijc B\in \chi^+$ be one of the combinators that satisfy $r_ff(x)f(y) \leq f(xy)$. The first program is:
\[ wz \subseteq \comb m (r_ff(\pi_0) z) (\comb m (r_ff(\pi_1) z)) \]
There is a $w:\termo\to\ijc B\in \chi^+$ that satisfies this, because of combinatory completeness and closure under application. The following reduction explains the purpose of this program:
\begin{align*}
w f(\comb pxx') &\subseteq \comb m (r_ff(\pi_0) f(\comb pxx')) (\comb m (r_ff(\pi_1) f(\comb pxx')))\\
& \subseteq \comb m (f(\pi_0(\comb pxx')) (\comb m (f(\pi_1(\comb pxx')))) \\
& \subseteq \comb m f(x) (\comb m f(x'))
\end{align*}
It takes the image of a pair, pulls it apart, applies $\comb m$ to both parts, and then applies the results to each other.

We still need some more programming. For $y\in \ijc B$ let $y^t = \bigcup\set{z\in \ijc A| \forall z\in \ijc A. \comb mf(xz)\leq yf(z)}$. This is a total arrow $\ijc B\to\ijc A$, because combination density implies that these unions are inhabited. Moreover, if $y:\termo\to \ijc B\in \chi^+$ then $y^t:\termo\to\ijc A\in \phi^+$ for the same reason. The second program is:
\[ r_gxx' \subseteq w^t(\comb p xx') \]
This time there is a $r_g:\termo\to\ijc A\in \phi^+$. This program satisfies:
\[ \comb mf(r_gxx') \subseteq \comb m f(w^t (\comb p xx')) \subseteq wf(\comb pxx') \subseteq \comb m f(x) (\comb m f(x')) \]
Now if we let $x=g(y)$ and $x' = g(y')$, we get:
\[ \comb mf(r_gg(y)g(y')) \subseteq \comb m f(g(y)) (\comb m f(g(y')) \subseteq yy' \]
The last inequality follows from the definition of $g$ in combination with the fact that $f$ preserves arbitrary joins. Another consequence of the definition of $g$, is that we can conclude that $r_gg(y)g(y') \subseteq g(yy')$ for all $y,y'\in \ijc B$. Therefore $g$ is in fact closed under application, and a left exact morphism as promised.

For all $y\in D$, $\comb m f(g(y))\converges$ and $\comb mf(g(y))\subseteq y$, because $f$ preserves unions. Since there is an $\comb i:\termo \to \ijc A \in \chi^+$ satisfying $\comb i y\subseteq y$, there is some $n:\termo \to \ijc B\in \chi^+$ such that $\comb m f(nx)\subseteq \comb i f(x)\subseteq f(x)$. This means $nx \leq g(f(x))$ for all $x\in \ijc A$. These combinators determine that $f$ and $g$ are adjoint. So each computationally dense morphism $(C,f)$ has a right adjoint $(D,g)$, and we have proved one direction of the equality in the proposition.

Assume that $(C,f)$ has a right adjoint $(D,g)$. The map $f$ preserves arbitrary joins because of the adjunction, and the fact that application preserves arbitrary joins. First note that $f(y)\subseteq z$ if and only if $ny\subseteq g(x)$ for a fixed $n$ that realizes the unit of the adjunction.
\[\begin{array}{ccccc}
f(\bigcup Y) \subseteq z &\iff & n\bigcup Y = \bigcup_{y\in Y}ny \subseteq g(z)
&\iff & \forall y\in Y. ny \subseteq g(z) \\
&\iff & \forall y\in Y. f(y) \subseteq z
&\iff & \bigcup_{y\in Y} f(y) \subseteq g(z) \\
\end{array}\]

There must also be a realizer $\comb m:\termo \to \ijc B\in \chi^+$ for the counit of the adjunction $f\circ g \leq \id$. It satisfies computational density because for all $x:\termo\to \ijc B  \in \chi^+$ there is an $x':\termo\to \ijc A  \in \phi^+$ such that $x(fy)\converges$ and $xf(y)\subseteq z$ if and only if $x'y\converges$ and $x'y\subseteq g(z)$. Therefore $x'y\subseteq g(xf(y))$ and
\[ \comb m f(x'y) \subseteq \comb m f(g(xf(y))) \subseteq xf(y)\]
This proves that left adjoint left exact morphisms are computationally dense.
\end{proof}

\begin{example} Let $A$ be a filter of $B$, let $\phi$ be a combinatory complete external filter of $A$ and let $\chi$ be $\set{U\in \ext B| U\cap A\in \phi}$. The map $U\mapsto U\cap A: \chi \to \phi$ is surjective, because each $V\in \phi$ equals $U\cap A$ if $U$ is the downward closure $V$. 

The inclusion $A\to B$ determines an adjoint pair of applicative morphisms between $(A,\phi)$ and $(B,\phi)$, because this is a special case of the last part of example \ref{single valued morphisms}. The left adjoint $(A,\phi) \to (B,\phi)$ is full and faithful because the surjection $\chi \to \phi$ determines that the same set of morphisms is tracked.

The left exact morphism $x\mapsto x\cap \ijc A: (\ijc B  ,\chi^+) \to (\ijc A,\phi^+)$, which is induced by the right adjoint, is computationally dense, because $x(y\cap \ijc B  )\subseteq z$ for some $x:\termo \to \ijc B  $ in $\chi^+$ if and only if $x'(y\cap \ijc B  )\subseteq z$ for some $x':\termo \to \ijc A$ in $\phi^+$. Therefore the right adjoint has its own right adjoint. This determines a \xemph{local geometric morphism} $\RT(B,\chi) \to \RT(A,\phi)$.

A special case is the morphism $\RT(B,A) \to \RT(A,A)$, which was studied in \cite{MR1769604, MR1909030, MR1948021}.
\end{example}

\section{Projectives} 
In this section we consider an alternative approach to realizability categories, using a construction on the category of partitioned assemblies and its underlying object functor, rather then on the realizability fibration. Our main result is that it doesn't work, because the realizability fibration is not a free completion of this underlying set functor, \xemph{unless} we limit ourselves to filters that are generated by singletons.

\subsection{Reg/lex and ex/lex completions}
We freely add stable images to categories with finite limits to make them regular, and compose this operation with the ex/reg completion in order to turn left exact categories into exact categories. We then show that the original category is embedded as the category of projective objects. This is a drawback for realizability.

\newcommand\reglex{\textrm{reg/lex}}
\begin{definition} A \xemph{reg/lex completion} of a category $\cat C$ with finite limits, is a regular category $\cat C_\reglex$ with a functor $I:\cat C\to\cat C_\reglex$ that preserves finite limits, such that:
\begin{itemize}
\item for every finite limit preserving $F$ from $\cat C$ to a regular category $\cat R$, there is a regular functor $G:\cat C_\reglex \to\cat R$ such that $GI\simeq F$,
\item for every pair of regular functors $H,K:\cat C_\reglex \to \cat D$ and every natural transformation $\eta:HI\to KI$ there is a unique $\theta: H\to K$ such that $\theta I = \eta$.
\end{itemize}
\end{definition}

\begin{proposition}[Rosolini, Carboni (see \cite{MR1358759})] Every (small) category $\cat C$ with finite limits has an reg/lex completion. \end{proposition}

\begin{proof} We present two different constructions based on the material in the beginning of this chapter.
\begin{enumerate}
\item For every left exact category $\cat C$ the functor $\cod: \cat C/\cat C \to\cat C$ is a bifibration with finite limits that satisfy the Beck-Chevalley and Frobenius conditions. The functor $\cod$ factors as a surjective-on-objects and full functor $\cod_0:\cat C/\cat C \to \cat D$ followed by a faithful one $\cod_1:\cat D \to \cat C$ in an up to isomorphism unique way, and the faithful part is a fibred locale, because $\cod$ has all the required properties except faithfulness, and those properties are preserved in this factorization. 

The category of assemblies $\Asm(\cod_1)$ is a regular category with a left exact embedding $\nabla: \cat C \to\Asm(\cod_1)$. Every functor $F:\cat C\to\cat R$ that preserves finite limits induces a fibred locale $\Sub(F-)$ on $\cat C$: the pullback of subobject fibration. Because of the biadjunction $\Asm \dashv \Sub$, $\Asm(\cod_1)$ has the universal property of a reg/lex completion.

\item The free quotient completion $\cat C_q$ already has a similar universal property, because it makes all equivalence relation in $\cat C$ effective, and images are quotients of kernel pairs, which are equivalence relations. The regular category $\cat C_q$ has a full regular subcategory of quotients of kernel pairs, and this is yet another reg/lex completion of $\cat C$
\end{enumerate}
\end{proof}

Before we rush onto the ex/lex completion, it is useful to remember that if there are enough projectives, the exact completion can be simplified, as we proved in corollary \ref{fine2}. Therefore this is a useful fact:

\begin{lemma} For all $X\in\cat C$, $IX$ is projective. Also, objects in the image of $I$ cover every object of $\cat C_\reglex$. \end{lemma}

\begin{proof} See proposition 9 of \cite{MR1600009}. 
\end{proof}

\newcommand\exlex{\textrm{ex/lex}}
\begin{definition} An \xemph{ex/lex completion} of a category $\cat C$ with finite limits, is an exact category $\cat C_\exlex$ with a functor $I:\cat C\to\cat C_\exlex$ that preserves finite limits, such that:
\begin{itemize}
\item for every finite limit preserving $F$ from $\cat C$ to an exact category $\cat E$, there is a regular functor $G:\cat C_\exlex \to\cat R$ such that $GI\simeq F$,
\item for every pair of regular $H,K:\cat C_\exlex \to \cat D$ and every $\eta:HI\to KI$ there is a unique $\theta: H\to K$ such that $\theta I = \eta$.
\end{itemize}
\end{definition}

\begin{lemma} The ex/lex completion $I:\cat C\to\cat C_\exlex$ is the composition of the reg/lex $I_0:\cat C\to\cat C_\reglex$ and the ex/reg $I_1:\cat C_\reglex\to (\cat C_\reglex)_\exreg$ completions. \end{lemma}

\begin{proof} This is a question of checking the universal properties. 
\begin{itemize}
\item The composed functor is a finite limit preserving functor into an exact category.
\item For every finite limit preserving functor $F:\cat C\to\cat E$ to an exact category, there exists regular $G:\cat C_\reglex\to\cat E$ and $H:(\cat C_\reglex)_\exreg\to\cat E$ such that $HI_1\simeq G$ and $HI_1I_0\simeq GI_0\simeq F$.
\item For every pair of regular $K,L:(\cat C_\reglex)_\exreg\to\cat E$ and every natural transformation $\eta:KI\to LI$ there are unique $\theta: KI_1 \to LI_1$ and $\iota: K\to L$ such that $\iota I_1 = \theta$ and $\iota I = \theta I_0 = \eta$.
\end{itemize}
\end{proof}

In \cite{MR1056382} Robinson and Rosolini show that $\Asm(\ds\Kone/\Kone)\cong \Pasm(\Kone,\Kone)_\reglex$ and $\Eff = \RT(\Kone,\Kone)\cong \Pasm(\Kone,\Kone)_\exlex$ for $\Kone$ in $\Set$. In the coming subsection we explain whether this holds for other realizability categories.

\subsection{Projectives of $\Asm(\ds A/\phi)$}
If $\Asm(\ds A/\phi)$ is a reg/lex completion, then it must have enough projective objects. In order to judge whether this is the case, we analyze what objects in realizability categories are projective.

\begin{lemma}  Let $\cat H$ be a Heyting category, $A$ an order partial applicative structure and $\phi$ a combinatory complete external filter whose members are inhabited. If $P$ is a projective object in $\Asm(\ds A/\phi)$, then there is a prone morphism $P\to \A$. \end{lemma}

\begin{proof} There is a prone supine span $(s:Y\to P,Y\to \A)$ by weak genericity, and for assemblies, supine morphisms are regular epimorphisms. So $s$ is split by some $t:P\to Y$. That $t$ is prone, follows from the fact that for any $f:X\to Y$ such that $\supp f = \supp t\circ g$, we have a map $s\circ f$ such that $f = t\circ s\circ f$, and $\supp(s\circ f) = g$ because $t$ is monic and $\supp t\circ g = \supp f = \supp t\circ \supp(s\circ f)$. \end{proof}

We extend our analysis with the following useful fact.

\begin{lemma} All left adjoints of regular functors preserve projective objects. \end{lemma}

\begin{proof} Suppose $L\dashv R$ where $R$ is a regular functor. Let $P$ be projective object of $\dom L$, and let $e:X\to Y$ be a regular epimorphism in $\dom R$. Each $y:LP\to X$ has a transpose $y^\dagger:P\to RX$. Because $Re:RX\to RLP$ is a regular epimorphism and $P$ is projective, there is an $x:P\to RX$ such that $Re\circ x = y^\dagger$. There is a transpose $x^\dagger:LP\to X$ which satisfies $e\circ x^\dagger = y$. Because every $y:LP\to Y$ factors through every regular epimorphism $e:X\to Y$ in this way, $LP$ is a projective object. \end{proof}

Because of this, the base category of a realizability topos inherits a lot of projective from the base category.

\begin{lemma} If $P\in \Asm(\ds A/\phi)$ is projective, then so is $\supp P\in\cat H$. \end{lemma}

Unfortunately, this leads us to the following corollary.

\begin{proposition} If $\Asm(\ds A/\phi)$ has a projective terminal object, then $\phi$ is generated by singletons. \end{proposition}

\begin{proof} For each $U\in\phi$, $\A_U = (\ds A/\phi)_{U\hookrightarrow A}(\A)$ is inhabited and therefore has a global section $u:\termo \to \A_U$. Since $\set{a\in \A| a\leq u}$ has the same global section it is inhabited and therefore, $\supp\set{a\in \A| a\leq u}\in \phi$.
\end{proof}

\begin{remark} Even if the terminal object of $\Asm(\ds A/\phi)$ is projective, $\Asm(\ds A/\phi)$ is not a reg/lex completion unless the subcategory of projectives is closed under pullbacks and if every object in $\Asm(\ds A/\phi)$ is the image of a morphism between projectives. That requires $\cat H$ to be a reg/lex completion itself. \end{remark}

In the case of filters that are generated by singletons, we can characterize the projective objects, as we will see in the following lemma.

\begin{proposition} If the terminal object in $\Asm(\ds A/\phi)$ is projective, $p:P\to \A$ is prone, and $\supp P$ is projective, then $P$ is projective.\label{charproj} \end{proposition}

\begin{proof} Because of pullbacks, it suffices to consider regular epimorphisms to $P$. Because of weak genericity, a partitioned assembly covers each object, and therefore it suffices to consider regular epimorphisms from other partitioned assemblies to $P$. In particular the other partitioned assembly $Q$ may be \xemph{downward closed}, i.e. not only is there a prone map $q:Q\to \A$, but also a map $r: \set{(a,x)\in \A\times Q| a\leq q(x) } \to Q$ such that $q\circ r(a,x) = a$. Now let $e:Q\to P$ be a regular epimorphism. We make the following factorization: $(q,e): Q\to \im{(q,e)}(Q)\subseteq \A\times P$ and $\pi_1: \im{(q,e)}(Q) \to P$. Because $Q$ is downward closed, so is $\im{(q,e)}(Q)$ and this makes $\pi_1$ a \prodo morphism. Because of the uniformity rule and Church's rule, the prone map $p:P\to \A$ and the regular epic \prodo morphism $!:\A\to \termo$, there is an inhabited prone downset $U\subseteq \A$ such that $(a,x)\mapsto (ap(x),x)$ defines a map $U\times P \to \mathord\downarrow\im{(q,e)}(Q)$. Because $U$ has a global section, this determines a section $s:P \to \im{(q,e)}(Q)$ of $\pi_1$.
By projectivity $\supp s:\supp P\to \supp\im{(q,e)}(Q)$ factors through the epimorphism $\supp(q,e):\supp Q \to \supp\im{(q,e)}(Q)$ in some morphism $t:\supp P\to \supp Q$. Now consider the other projection $\pi_0:\im{(q,e)}(Q)\to \A$. Because $\supp s = \supp(q,e)\circ t$ and $\pi_0\circ(q,e) = q$, $\supp(\pi_0 \circ s) = \supp q\circ t$. Because $q$ is prone, there is a unique lifting $t':P\to Q$ of $t$ and this is our splitting of $e$.
\[ \xymatrix{
Q\ar[r]_e\ar[d]_q\ar[dr]_{(q,e)} & P \ar@{.>}@/^/[d]^s \ar@{.>}@/_/[l]_{t'}\\
\A & \im{(q,e)}(Q) \ar[l]^{\pi_0} \ar[u]^{\pi_1}
}\]
\end{proof}

This is an adaptation of the proof of proposition 3.2.7 in \cite{MR2479466}.

\begin{corollary} If $\phi$ is generated by singletons, then the projective objects of $\Asm(A,\phi)$ are the partitioned assemblies $P$ for which $\Sigma P$ is projective in $\cat H$. Therefore, $\Asm(A,\phi)$ has enough projectives if only if $\cat H$ has. \end{corollary}

\subsection{Alternative completions}
The characteristic properties of realizability fibration force certain arrows between partitioned assemblies to be regular epimorphisms:
\begin{itemize}
\item Realizability fibrations are separated. If $f:P\to Q$ in $\Pasm(A,\phi)\subseteq \cat DA/\phi$ is prone and $\supp f$ is a regular epimorphism, then $f$ must be one too.
\item For each $U\in \phi$, $\A_U = (\ds A/\phi)_{U\hookrightarrow A}(\A)$ is inhabited, i.e. $!:\A_U\to\termo$ and its pullbacks are regular.
\end{itemize}
The reg/lex completion preserves only split epimorphisms and the morphism above aren't always split. Therefore $\Asm(\ds A/\phi)$ is not always a reg/lex completion.

Besides the problem of lost regular epimorphisms, there is a somewhat smaller problem of the closure of the category of projective objects under pullbacks. Carboni and Vitale's reg/wlex completion in \cite{MR1600009} solves this problem. We first show why.

\newcommand\exwlex{\textrm{ex/wlex}}
\begin{definition} For any diagram $D:\cat D\to\cat C$ a \xemph{weak limit cone} is a cone $\kappa_i:W \to D(i)$ through which every other cone factors, but not necessarily in a unique way. A category $\cat C$ is \xemph{weakly left exact} if every finite diagram has a weak limit cone.

Let $\cat C$ be a weakly left exact category. A functor $F:\cat C \to \cat R$ to a regular category is \xemph{left covering} if for every finite diagram $D:\cat D\to\cat C$ and every weak limit cone $\kappa:W\to D$ the factorization of $F\kappa$ through $\lim FD$ is a regular epimorphism.

The \xemph{ex/wlex completion} of $\cat C$ is an exact category $\cat C_\exwlex$ with a left covering functor  $I:\cat C \to \cat C_\exwlex$ such that:
\begin{itemize}
\item for each left covering $F:\cat C\to\cat E$ to an exact category, there is a regular $G:\cat C_\exwlex\to\cat E$ such that $ GI\simeq F$,
\item for each pair of regular functors $H,K:\cat C_{\exwlex} \to\cat E$ and each $\eta:HI\to KI$ there is a unique $\theta:H\to K$ such that $\theta I = \eta$.
\end{itemize}
\end{definition}

\begin{remark} Regular and exact completions are related in the following way: if $\cat R$ is the reg/something completion of $\cat C$ then $\cat R_\exreg$ is the ex/something completion of $\cat C$; if $\cat E$ is the ex/something completion then the least regular subcategory $\cat R$ of $\cat E$ that contains all of $\cat C$ is the reg/something completion. The reason is that exact categories are a reflective subcategory of regular categories. Below we switch between regular and exact completions, depending on which of these completions are easiest to describe. \end{remark}

\begin{proposition} Let $\cat H$ be a Heyting category with enough projectives, let $A$ be an order partial combinatory algebra in $\cat H$ and let $\phi$ be a combinatory complete external filter of $A$ that is generated by singletons. Now $\Asm(\ds A/\phi)_\exreg$ is the ex/wlex completion of its subcategory of projective objects.\label{enproj} 
\end{proposition}

\begin{proof} By lemma \ref{charproj}, a projective object in $\Asm(\ds A/\phi)_\exreg$ is a partitioned assembly $P$ for which $\supp P$ is projective. The category $\Asm(\ds A/\phi)_\exreg$ has enough projectives, because partitioned assemblies cover all objects and projective objects cover all partitioned assemblies. The result now follows from theorem 16 in  \cite{MR1600009}.
\end{proof}

The free quotient completion in subsection \ref{exreg} also lost regular epimorphisms. A category-of-fractions constructions to got them back. In \cite{MR2067191} Hofstra presents the relative completion constructions, which does roughly the same thing for prone regular epimorphisms. Therefore, relative completions work without enough projective objects in the base category.

\begin{definition} Let $F:\cat B\to \cat C$ be a finite limit preserving functor between categories with finite limits. The \xemph{relative completion} is a finite limit preserving  functor $I:\cat C \to\cat C_{\cat B/\textrm{reg}}$ to a regular category such that:
\begin{itemize}
\item the composite $IF$ preserves regular epimorphisms, i.e. is a regular functor,
\item for each finite limit preserving $G:\cat C\to\cat R$ to a regular category such that $GI$ is regular, there is a regular $H:\cat C_{\cat B/\textrm{reg}}\to\cat R$ such that $ GI\simeq F$,
\item for each pair of regular functors $K,L:\cat C_{\cat B/\textrm{reg}} \to\cat R$ and each $\eta:KI\to LI$ there is a unique $\theta:K\to L$ such that $\theta I = \eta$.
\end{itemize}
\end{definition}

\begin{theorem}[Hofstra] Let $\cat H$ be a Heyting category, let $A$ be a partial combinatory algebra and let $\phi$ be a filter that is generated by singletons. Then $\Asm(\ds A/\phi)\cong \Pasm(A,\phi)_{\cat B/\textrm{reg}}$. \label{relcomp}
\end{theorem}

\begin{proof} This is theorem 7.1 in \cite{MR2067191}. The condition of singleton generated filters is stated right above, by the way.\end{proof}

Actually, Hofstra's construction is a \xemph{stack completion}.

\begin{definition} Let $F:\cat F\to\cat B$ be a fibred meet semilattice with top over a regular category. 
A \xemph{descent datum} is an equivalence relation $\wave$ on an object $X$ in $\cat F$, such that the inclusion $\wave\subseteq X\times X$ is prone, and $F\wave$ is effective. The fibred meet semilattice $F$ is a \xemph{stack} (for the regular topology of $\cat B$) if it is separated, and every descent datum is effective in $\cat F$. \label{stacks}
\end{definition}

\begin{remark} Stacks can be defined on all sites, and are not necessarily fibred meet semilattices with top. Those other stacks play no role here.  For more on stacks, see \cite{MR558105, MR558106}. \end{remark}

\begin{example} The functor $\supp = \ds A/\phi:\Asm(\ds A/\phi) \to \cat H$ is a stack, because and assemblies are closed under quotients of prone equivalence relations $\wave\subseteq X^2$ for which $\supp X$ has a quotient. In fact, let $q:\supp X \to \supp X/\supp\wave$ be the quotient map, then $\exists_q(X)$ is the quotient for the descent datum.\label{asm is stack} \end{example}

\begin{definition} The \xemph{stack completion} of $F$ is a stack $F^+:\cat F^+\to \cat B$ with a vertical morphism $(\id,i): F\to F^+$ that satisfies the following conditions:
\begin{itemize}
\item for every stack $G:\cat G \to \cat B$ and every morphism $(f_0,f_1):F\to G$ for which $f_0$ is regular, 
there is a morphism $(g_0,g_1):F^+\to G$ such that $g_0\simeq f_0$ and $g\circ (\id,i) = f$.
\item for every pair of morphisms of stacks $(k_0,k_1),(h_0,h_1): F^+ \to G$ for which $k_0\simeq h_0$ is regular and $k_1\circ i\simeq h_1\circ i$, $k\simeq h$.
\end{itemize}\label{stack completion}
\end{definition}

\begin{remark} Morphisms of fibred categories preserve quotients of descent data, because the regular functors preserve the underlying quotients, all morphisms preserve the descent data, and the descent data determine their quotient up to unique isomorphism.
\end{remark}

\begin{theorem} Let $\cat H$ be a Heyting category, let $A$ be a partial combinatory algebra and let $\phi$ be a combinatory complete external filter that is generated by singletons. Then $\supp_\Asm: \Asm(\ds A/\phi) \to \cat H$ is a stack completion of $\supp_\Pasm:\Pasm(A,\phi)\to\cat H$. \end{theorem}

\begin{proof} Note that $\supp_\Pasm$ is a fibred meet semilattice with top, because we have a global paring operator, and so on. This is also the reason why $\nabla:\cat H\to\Asm(\ds A/\phi)$ factors through $\Pasm(A,\phi)$. The inclusion into $\supp_\Asm$ provides the required morphism $(\id,i):\supp_\Pasm\to \supp_\Asm$.

If we have a morphism $(m_0,m_1): \supp_\Pasm \to G$ for some stack $G:\cat G \to \cat H$ then we have a functor 
$m_1\nabla:\cat H\to\cat G$, which is regular because $m_0$ is regular, and $G:\cat G\to\cat H$ is a stack. The filter $m_1\A$ intersects $m_1\nabla$ for all $U\in\phi$, because $\A\cap \nabla U$ has a global section, and global sections are also preserved. That means $(m_1\nabla,m_1\A)$ is a regular model, and that there is a regular functor $\overline{m_1}:\Asm(\ds A/\phi) \to \cat G$ satisfying $\overline{m_1}\A=m_1\A$ and $\overline{m_1}\nabla = m_1\nabla$. Now $(m_0,\overline{m_1}):\supp_\Asm \to G$ is the desired factorization.

If $h,k:\supp_\Asm \to G$ satisfy $h_0=k_0$ and $h_1\circ i = k_1\circ i$, then $h_1\circ i \circ \nabla\simeq h_1\circ i \circ \nabla$ and $h_1\circ i(\A)\simeq k_1\circ i(\A)$ so we can use our theory of regular models to show that these morphisms are isomorphic.
\end{proof}

\begin{remark} Relative completions do not preserve the inhabited prone downsets unless all of those downsets have global sections. Moreover, partitioned assemblies are not always closed under finite products and if partitioned assemblies are not closed under finite products, $\supp_\Pasm$ is not a fibred meetsemilattice with $\top$, and $\nabla$ does not factor through $\Pasm(A,\phi)$.
\end{remark}

There is a regular completion construction that does work.

\newcommand\Sh{\Cat{Sh}}
\begin{definition} The exact completion of a site $(\cat C,J)$ is the up to equivalence least exact subcategory $(\cat C,J)_{\rm ex}$ of the category $\Sh (\cat C,J)$ of sheaves over the site that contains all the representable sheaves. The objects are quotients of finite limits of representables.
\end{definition}

\begin{lemma} Let $R$ be the set of morphisms of $\Pasm(A,\phi)$ that are regular epimorphisms in $\Asm(A,\phi)$. Now $(\Pasm(A,\phi),R)$ is a site and $\Asm(\ds A/\phi)_\exreg \cong (\Pasm(A,\phi),R)_{\rm ex}$. \end{lemma}

\begin{proof} The regular topology $R'$ on $\Asm(\ds A/\phi)_\exreg$ makes the Yoneda embedding $\Cat{Sh}(\Asm(\ds A/\phi)_\exreg,R')$ a regular functor. Therefore, the least exact subcategory containing all representables is equivalent to $\Asm(\ds A/\phi)_\exreg$.

By weak genericity partitioned assemblies cover all assemblies, and this makes $(\Pasm(A,\phi),R)$ a \xemph{dense subsite} of $(\Asm(\ds A/\phi)_\exreg,R')$. The inclusion of categories $\Pasm(A,\phi) \to \Asm(\ds A/\phi)_\exreg$ therefore induces an equivalence of categories $\Sh(\Pasm(A,\phi),R)\cong \Sh(\Asm(\ds A/\phi),R')$. See theorem C2.2.3 in \cite{MR1953060}.

The least exact category that contains all partitioned assemblies contains all of $\Asm(\ds A/\phi)_\exreg$ because every object of $\Asm(\ds A/\phi)_\exreg$ is the quotient of an equivalence relation on a partitioned assembly, and those equivalence relations are the images of morphisms between partitioned assemblies. Therefore $(\Pasm(A,\phi),R)_{\rm ex} \cong \Asm(\ds A/\phi)_\exreg$. \end{proof}

\begin{remark} In \cite{ECnSS}, Shulman works out this type of exact completion in great generality. \end{remark}


%% file: Applications.tex
This chapter contains interesting facts about realizability toposes that I have come across during my research.

The first section is about deriving well-known properties of the effective topos from the abstract characterization of realizability categories which we gave in the previous chapters. We show that some properties of the effective topos generalize to \xemph{effective categories} constructed over other Heyting categories than the topos of sets.

In the second section I revisit the topic of my master thesis and my paper \cite{MR2729220}, which was the algebraic compactness of a full internal subcategory of the effective topos.

The last section connects relative realizability and \xemph{classical realizability}, a realizability for classical second order arithmetic and set theory, developed by Krivine.

\section{Effective Categories}
This section is about realizability categories that are constructed from internal versions of \xemph{Kleene's first model}, which we have met in example \ref{Kone}. Some of the topics we want to discuss are:
\begin{itemize}
\item conditions for constructing $\Kone$ and $\Eff(\cat H)$ for a Heyting category $\cat H$;
\item characteristic properties of categories constructed in this way, especially when the base category has additional properties, e.g. having a subobject classifier or satisfying the axiom of choice.
\end{itemize}

\subsection{Internal recursion theory}
We demonstrate that a natural number object suffices to have a partial combinatory algebra of partial recursive functions in any Heyting category.

\begin{definition} In an arbitrary Heyting category $\cat H$ a \xemph{natural number object} is an object $\nno$ with morphisms $0:\termo\to \nno$ and $s:\nno\to\nno$ such that for each $f:X\to X$ there is a unique $g:\nno\times  X\to X$ which satisfies $g(0,y) = y$ and $g(s(x),y) = f(g(x,y))$.
\[ \xymatrix{
X\ar[r]^(.4){(0,\id)}\ar[dr]_{\id} & \nno\times X \ar[r]^{s\times \id}\ar@{.>}[d]^g & \nno\times X\ar@{.>}[d]^g \\
& X \ar[r]_f & X
}\]\label{nno}
\end{definition}

\begin{remark} The natural number object is usually defined as follows: for each pair $f_0:\termo \to X$ and $f_1:X\to X$ there is a unique $h:\nno\times X$ such that $h(0)=f_0$ and $h(s(n))=f_1(h(n))$.
\[ \xymatrix{ \termo\ar[r]^0\ar[dr]_{f_0} & \nno \ar[r]^s\ar@{.>}[d]^h & \nno\ar@{.>}[d]^h \\ & X \ar[r]_{f_1} & X } \]
In a Cartesian closed category, the definitions are equivalent. To get from the traditional definition to ours, let $f_0:\termo \to X^X$ be the transpose $f^t$ of $f:X\to X$, and let $f_1 = f^\id:X^X\to X^X$, then the transpose $h^t$ of $h$ is our $g$; to get from our definition to the traditional, let $f=f_1$ and let $h = g\circ(\id\times f_0):\nno \to X$.
\[ \begin{array}{cc}
\xymatrix{ \termo\ar[r]^0\ar[dr]_{*\mapsto f} & \nno \ar[r]^s\ar@{.>}[d]^h & \nno\ar@{.>}[d]^h \\ & X^X \ar[r]_{f^\id} & X^X } &
\xymatrix{ \nno \ar@/^1.5ex/[rr]^{\id\times f_0} \ar@/_2ex/@{.>}[drr]_h & X \ar@/_1.5ex/[dr]^\id \ar[r]_(.4){(0,\id)} & \nno\times X \ar[r]_{s\times \id}\ar@{.>}[d]^g & \nno\times X\ar@{.>}[d]^g \\ & & X \ar[r]_{f_1} & X }
\end{array}\]
A lack of exponentials in general Heyting categories forces us to use this stronger definition.
\end{remark}

We define the partial recursive functions as follows.

\begin{definition} In any Heyting category $\cat H$, the \xemph{partial recursive functions} is the least class of partial arrows $\nno^n \partar \nno$ that satisfies the following conditions.
\begin{itemize}
\item The zero function $0:\termo\to\nno$, the successor function $s:\nno\to\nno$ and all projections $\vec x\mapsto x_i$ are partial recursive.
\item The partial recursive functions are closed under composition, i.e. if $f:\nno^m\partar \nno$ and $g_i:\nno^n\partar \nno$ for $i<m$ are partial recursive, then so is $f\circ(g_0,\dotsm,g_m):\nno^n\partar \nno$. 
\item For every partial recursive $f:\nno^k\partar \nno$ and $g:\nno^{k+2}\to\nno$, the partial function $h:\nno^{k+1}\partar \nno$ that satisfies the following conditions is partial recursive:
\begin{itemize}
\item $(0,\vec x)\in \dom h$ and $h(0,\vec x) = f(\vec x)$ if and only if $\vec x\in \dom f$;
\item $(n+1,\vec x)\in \dom h$ and $h(n+1,\vec x) = g(n,h(n,\vec x),\vec x)$ if and only if $(n,\vec x)\in \dom h$ and $(n,h(n,\vec x),\vec x)\in\dom g$.
\end{itemize}
\item For every partial function $f:\nno^{k+1} \partar \nno$ the partial function $g:\nno^k\partar \nno$ that satisfies the following condition is partial recursive: $\vec x\in\dom g$ and $g(\vec x)=y$ if for all $y'\leq y$, $(\vec x,y')\in \dom f$ and $f(\vec x,y') = 0$ if and only if $y'=y$.
\end{itemize}
\end{definition}

We can now use Kleene's normal form theorem to define a partial application operator $\nno^2\partar \nno$ and this is exactly Kleene's first model $\Kone$.

\begin{theorem}[Kleene's normal form theorem] There are primitive recursive functions $T:\nno^3\to\nno$ and $U:\nno \to \nno$ such that for each partial recursive $f:\nno \partar \nno$ there is an $i\in N$ such that for all $j$, there is a $k$ such that $T(i,j,k)=0 \land U(k)=f(j)$. \end{theorem}

\begin{proof} This is theorem IV in \cite{MR0007371}. \end{proof}

\begin{definition}[Kleene's first model] Let $T:\nno^3 \to\nno$ and $U:\nno\to\nno$ be as in the theorem. For all $x,y\in \nno$ let $xy\converges$ if $\exists z\oftype\nno. T(x,y,z)=0$ and let $xy=z$ if $\exists w\oftype\nno.T(x,y,w)=0\land U(w)=z$. The applicative structure that is $\nno$ with this application operator is \xemph{Kleene's first model}. 
\end{definition}

\begin{lemma} This is a partial combinatory algebra. \end{lemma}

\begin{proof} Theorem XXIII in \S 65 of \cite{MR0051790} is the \xemph{S-$m$-$n$ theorem}, which implies combinatory completeness, because all combinatory function are partial recursive by definition. \end{proof}

\subsection{Weakly relational natural number objects}
We have no proof that natural number objects are preserved in ex/reg completions. The problem is that we need to do recursion with functional relations, and this is not part of the definition of a natural number object. Sadly, this means that $\Asm(\cat H)_\exreg$ may not always have a natural number object, which makes characterizing effective categories harder. Therefore, we will now consider a remedy.


An obvious solution is to allow recursive relations. However, if a Heyting category has all recursively defined relations, then the ex/reg completion has all coequalizers, because we can construct transitive closures of arbitrary relations. For each parallel pair of arrows $f,g:X\to Y$, the transitive closure of $\im{(f,g)}(X)\cup \im{(g,f)}(X)\cup\set=\subseteq Y^2$ is precisely the kernel pair of the coequalizer of $f$ and $g$, which therefore exists in the ex/reg completion. Because this seems too strong, we consider a weaker condition.

\begin{definition}[weakly relational natural number objects] A natural number object $\nno$ is \xemph{weakly relational}, if for each total relation $R: X\nrightarrow X$ there is a total relation $S:\nno\times X\nrightarrow X$ such that $S(0,x,y)$ implies $x=y$ and $S(n+1,x,z)$ implies that $S(n,x,y)$ and $R(y,z)$ for some $y\in X$. We say that $S$ exists by \xemph{weak relational recursion}. \end{definition}

\begin{lemma} Let $I:\cat H \to\cat H_\exreg$ be the ex/reg completion of a regular category $\cat H$. If $\nno$ is a weakly relation natural number object in $\cat H$, then $I\nno$ is one in $\cat H_\exreg$. \end{lemma}

\begin{proof} Let $R:X/\wave\nrightarrow X/\wave$ be a total relation on a new quotient object in the ex/reg completion and let $q:X\to X/\wave$ be the quotient map. We pull back $R$ along $q\times q$ to get a total relation $R' = \pre{(q\times q)}(R)\subseteq X^2$. By weak relational recursion there is an $S': \nno\times X\nrightarrow X$, such that $S'(0,x,y)$ implies $x=y$ and such that if $S'(n+1,x,y)$, then there is a $z$ such that $S'(n,x,z)$ and $R'(z,y)$. We let $S = \im{q\times q}(S'):\nno\times (X/\wave) \nrightarrow (X/\wave)$. This is a total relation such that $S(0,x,y)$ implies $x=y$ and $S(n+1,x,y)$ implies $S(n,x,z)\land R(y,z)$ for some $z\in X/\wave$.

A function $f:X/\wave \to X/\wave$ is a total relation, and therefore there is an $S\subseteq \nno\times (X/\wave)^2$ such that $S(0,x,y)$ implies $x=y$ and $S(n+1,x,y)$ implies $S(n,x,z)$ and $f(z)=y$. By induction, $S$ is single-valued just like $f$ is. That means that the projection $\pi_{01}:S\to \nno\times X/\wave$ is not only a regular epimorphism (due to totality) but also a monomorphism and therefore an isomorphism. So $\nno$ is in fact a natural number object and weakly relational.
\end{proof}

\begin{remark} The canonical functor $I:\cat H\to\cat H_\exreg$ is fully faithful and regular, so if $\cat H_\exreg$ has a natural number object in the image of $I$, it is automatically a natural number object of $\cat H$. We have no proof that if an exact category has a natural number object, then the natural number object is weakly relational. Therefore we cannot say whether a weaker principle than weak relationality forces $I$ to preserve natural number objects.
\end{remark}

We require that $\Asm(\ds\Kone/\Kone)$ has a weakly relational natural number object, so that $\RT(\Kone,\Kone)$ has one too. This translates to the same requirement on the base topos $\cat H$.

\begin{lemma} The category $\Asm(\ds \Kone/\Kone)$ has a weakly relational natural number object if and only if the base category $\cat H$ has one. \end{lemma}

\begin{proof} Consider that $\cat H$ is a regular reflective subcategory of $\Asm(\ds\Kone/\Kone)$, with $\nabla$ for an inclusion and $\supp$ for a reflection. Therefore, if $\Asm(\ds\Kone/\Kone)$ has a weakly relational natural number object, we can first find a suitable $S:\nno\times \nabla X\nrightarrow \nabla X$ for any total relation $R:X\nrightarrow X$, and $\supp S:\supp\nno \times X\to X$ will then have the required properties, namely that $\supp S(0,x,y)$ implies $x=y$ and $\supp S(n+1,x,y)$ implies that there is a $z\in X$ such that $S(n,x,z)$ and $R(z,y)$. Therefore $\supp \nno$ is a weakly relational natural number object in $\cat H$.

In the other direction, let $R:X\nrightarrow X$ be a total relation in $\Asm(\ds\Kone/\Kone)$. By theorem \ref{Shanin} there is a $Y$ with prone $p:Y\to \nno$ and a \prodo morphism $Y\to X$. Using the tracking principles, we can prove that there is a total relation $R':Y\nrightarrow Y$ which tracks $R$. There is a total relation $S:\nno\times \supp Y \nrightarrow \supp Y$ such that $S(0,x,y)$ implies $x=y$, and if $S(n+1,x,y)$ then there is a $z$ such that $S(n,x,z)$ and $\supp R'(z,y)$. We can construct $S$ in such a way, that there is an inhabited $U\subseteq \N = \supp\nno$ such that for all $m\in U$, $n\in \N$ and $x\in \supp Y$ there is a $y\in \supp Y$ such that $mn\supp p(x)\converges$ and $mn\supp p(x)=\supp p(y)$. The idea is to take an inhabited set $V\subseteq \N$ of realizers for the totality of $R'$ and consider the total relation $R'': U\times \supp Y\nrightarrow U\times \supp Y$ that satisfies \[ R''((m,n,x),(m',n',y)) \iff \supp R'(x,y), m=m', mn\converges, mn=n' \]
The resulting $S$ determines a total relation $S':\nno\times Y\nrightarrow Y$, which tracks a total relation $S'':\nno\times X\rightarrow X$. We can do this in a way that that gives us a set of realizers $U\subseteq \N = \supp\nno$ for the totality of $S$ as a relation $Y\nrightarrow Y$.
\end{proof}

\begin{remark} It would seem that by similar reasoning, $\Asm(\ds\Kone/\Kone)_\exreg$ has a natural number object if the base category $\cat H$ is exact, because then $\cat H$ is a reflective subcategory of $\Asm(\ds\Kone/\Kone)_\exreg$ and we have tracking principles here too. But the fact that the unit of $\supp\dashv\nabla$ is no longer a monomorphism blocks attempts to represent endomorphisms in $\Asm(\ds\Kone/\Kone)_\exreg$ as endomorphisms in $\cat H$, and we still lack the ability to define relations inductively. \end{remark}

We consider conditions on categories which ensure that their natural number objects are weakly relational. Basically there are two kinds: weak versions of the axiom of choice and weak versions of impredicativity. Both allow us to substitute total relations by arrows between related objects.

\begin{lemma} Let $\cat H$ be a Heyting category with natural number object $\nno$ and enough projective objects. Then $\nno$ is weakly relational. \end{lemma}

\begin{proof} Let $R:X\nrightarrow X$ be a total relation and let $p:P\to X$ be a projective cover. We pull back $R$ along $p\times p$ to get a total relation $R' = \pre{(p\times p)}(R):P\nrightarrow P$. Projectivity implies that $R'$ contains an arrow $r:P\to P$, and by recursion there is an arrow $s:\nno\times P\to P$, such that $s(0,x)=x$ and $s(n+1,x) = r(s(n,x))$ for all $n\in\nno$ and $x\in X$. Let $S$ be the following relation.
\[ S = \set{(n,x,y) \in\nno \times X^2| \exists z\in P. p(z)=x\land p(s(n,z))) = y}\]
Now $S$ is a total relation, because $p$ is a regular epimorphism. If $S(0,x,y)$ then $x=y$, because $s(0,z)=z$ for all $z\in P$; if $S(n+1,x,y)$ then for some $z\in P$, $p(z) = x$ and $p(s(n+1,z)) = p(r(s(n,x))) = y$, which implies that $S(n,x,p(s(n,z)))$ and $R(p(s(n,z)),y)$; therefore $S$ satisfies both of our requirements.
\end{proof}


For a weak version of impredicativity, we weaken the notion of power objects in toposes.

\begin{lemma} Let $\cat H$ be a Heyting category with a natural number object $\nno$, and \xemph{weak power objects}, i.e. for each object $X$ there is an $E_X\subseteq X\times P_X$ such that for each $Y \subseteq X\times Z$, there is an $m:Z\to P_X$ such that $Y = \pre{(\id_X\times m)}(E_X)$
\[ \xymatrix{
Y\ar[r]\ar[d] \ar@{}[dr]|<\lrcorner & E_X\ar[d] \\
X\times Z \ar[r]_{\id_X\times m} & X\times P_X
}\]
Then $\nno$ is weakly relational.
\end{lemma}

\begin{proof} Let $R:X\nrightarrow X$. Since $R\subseteq X^2$ there is an $r:X\to P_X$ such that $R = \pre{(\id_X\times r)}(E_X)$. Moreover, $(\id_X,r):X\to X\times P_X$ is a monomorphism, so there is an $f:P_X\to P_X$ such that $\im{(\id_X,r)}(X) = \pre{(\id_X\times f)}(E_X)$. We apply recursion to find a $g:\nno\times P_X \to P_X$ such that $g(0,\xi)=\xi$ and $g(n+1,\xi) = f(g(n,\xi))$. Let $s:X\to P_X$ be an arrow for which $\pre{(\id_X\times s}(E_X)$ is the diagonal subobject, let $h:\nno\times X^2 \to X\times P_X$ satisfy $h(n,x,y) = (y,g(n,s(x)))$ and let $S = \pre h(E_X)$. 
Now $S(0,x,y)$ if and only if $(y,g(0,s(x))) = (y,s(x))\in E_X$ if and only if $x=y$. Also, $S(n+1,x,y)$ if and only if $(y,g(n+1,s(x))) = (y,f(g(n,s(x))))\in E_X$ if and only if $r(y) = g(n,s(x))$; that means $R(z,y)$ if and only if $S(n,x,z)$ for all $z$. Since $R$ is total, there is a $z$ such that $R(z,y)$ and $S(n,x,z)$ if and only if $S(n+1,x,y)$ holds. \end{proof}

\begin{corollary} If $F$ is a tripos, and $\Asm(F)$ has a natural number object, then so does $\Asm(F)_\exreg$. \end{corollary}

Moerdijk and van den Berg have shown \cite{MR2474446} that natural number objects are preserved by ex/reg completions of locally Cartesian closed categories. The following lemma explains why.

\begin{proposition}[van den Berg] Any locally Cartesian closed Heyting category $\cat H$ with a natural number object $\nno$ has all recursively defined relations. \end{proposition}

\begin{proof} For $R: X\nrightarrow X$ let $S$ be:
\[ \set{(n,x,y)| \exists p\oftype\set{0,n} \to X. p(0)=x\land p(n)= y \land \forall i<n. R(p(i),p(i+1)) } \]
We use the family of initial segments of $n$ to interpret this. Now $S(0,x,y)$ if and only if $x=y$, and $S(n+1,x,z)$ if and only if $S(n,x,y)$ and $R(y,z)$. \end{proof}

\subsection{Effective categories}
In this section we consider properties of the effective topos that follow directly from our characterization theorems, and therefore hold in all realizability categories of the form $\Asm(\ds\Kone/\Kone)_\exreg$. We also consider what the axiom of choice and what power objects in the base category allow us to do.

We will reserve $\nno$ for natural number objects in realizability categories and use $\N$ for natural number objects in base categories.

\begin{theorem} Let $\cat H$ be an exact Heyting category with a weakly relational natural number object $\N$. Let $\Eff(\cat H) = \Asm(\ds \Kone/\Kone)_\exreg$. Then the following set of properties characterize $\Eff(\cat H)$:
\begin{itemize}
\item There is a regular full and faithful $\nabla:\cat H\to\Eff(\cat H)$ with a left adjoint $\supp$.
\item $\Eff(\cat H)$ has a natural number object $\nno$ and the unit $\eta_\nno:\nno \to \nabla\supp\nno$ is monic.
\item $\nno$ is weakly generic, i.e. partitioned assemblies cover all objects, where a partitioned assembly is an object $P$ with a morphism $p:P\to\nno$ that is prone relative to $\supp$.
\item let $p:P\to\nno$ and $q:Q\to \nno$ be prone, let $e:Q\to Y$ be a regular epimorphism and let $f:P\to Y$ be any morphism; there is a prone subobject $U\subseteq \nno$ and a map $g:U\times P \to Q$ such that $e(g(n,x))=f(x)$ and $q(g(n,x)) = np(x)$.
\[\xymatrix{
U\times \exists_p(P) \ar[d]_{(m,n)\mapsto mn} & U\times P \ar[l]_(.4){\id\times p}  \ar[r]^{\pi_1}\ar[d]^g & P \ar[d]^f \\
\nno & Q\ar[l]^{q} \ar[r]_e & Y
}\]
\end{itemize}
\end{theorem}

\begin{proof} This is essentially theorem \ref{charexact} applied to this situation. \end{proof}

A natural number object has infinitely many global sections, which has the following important implication.

\begin{lemma} The category of partitioned assemblies in $\Eff(\cat H)$ is closed under finite limits, and contains the image of $\cat H$. \end{lemma}

\begin{proof} Categories of partitioned assemblies are always closed under equalizers, because equalizers are prone (see \ref{lifting limits}). There are recursive isomorphisms between powers of $\nno$ and subobjects of $\nno$ in any Heyting category. In fact, with the exception of $k=0$, $\nno^k\simeq \nno$ by primitive recursive bijections. The terminal object is a partitioned assembly because $0$ is a global section of $\nno$, and all maps $\termo \to \nno$ are prone. Let $(m,n)\mapsto\langle m,n\rangle: \N^2\to\N$ be a recursive paring surjection in $\cat H$. This lifts to an isomorphism $\nno^2\to\nno$ and isomorphisms are prone. Therefore, for each two prone maps $p:X\to \nno$ and $q: Y\to \nno$ the map $\langle\cdot,\cdot\rangle\circ p\times q: X\times Y \to \nno$ is another prone map, making the product of two partitioned assemblies a partitioned assembly.

For each $X\in\cat H$, $!:\nabla X\to\termo$ is prone, and therefore $0:\nabla X\to\nno$ is prone too. \end{proof}

\begin{corollary} If $\cat H$ has a projective terminal object, then $\Eff(\cat H)$ is the relative completion of $\Pasm(\Kone,\Kone)$. \end{corollary}

\begin{proof} If $\termo$ is projective in $\cat H$ then the filter of inhabited subobjects of $\Kone$ in $\cat H$ is generated by singletons. This result follows from theorem \ref{relcomp}. \end{proof}

\begin{remark} If the object $\nno$ is projective, then so is $\termo$, because projective objects are closed under retracts, and $\termo$ is a retract of $\nno$. 
\end{remark}

The following corollary tells us what the axiom of choice in $\cat H$ can do for us.

\begin{corollary} If all regular epimorphisms split in $\cat H$, then $\Eff(\cat H)$ is the ex/lex completion of $\Pasm(\Kone,\Kone)$. \end{corollary}

\begin{proof} The assumption implies that the projective objects of $\Eff(\cat H)$ and the partitioned assemblies coincide, as we saw in lemma \ref{charproj}. The result now follows from theorem 16 in \cite{MR1600009}. \end{proof}

\begin{remark} If $\Eff(\cat H)$ is the ex/lex completion of $\Pasm(\Kone,\Kone)$ then all regular epimorphisms split in $\cat H$ because $\nabla$ is regular. If $\cat H$ has enough projectives and its terminal object is projective, then we can apply proposition \ref{enproj} to see that $\Eff(\cat H)$ is still the ex/wlex completion of its category of projective objects. \end{remark}

We will now consider what happens if $\cat H$ is a topos regardless of whether the terminal object is projective or whether there are enough projective objects.

\begin{definition} A \xemph{local operator} $j$ in a topos $\cat E$ is \xemph{regular} if $j(\exists x\oftype X.\phi(x))\to \exists x\oftype X.j(\phi(x))$ holds for all $j$-sheaves $X$. This is equivalent to: the inclusion $\Sh(\cat E,j)\to\cat E$ is a regular functor.
\end{definition}

For the theory of local operators see chapter V of \cite{MR1300636} where they are called \xemph{Lawvere-Tierney topologies}.

\newcommand\Sep{\Cat{Sep}}
\begin{lemma} Let $\cat H$ be a topos with a natural number object. The effective topos $\Eff(\cat H)$ has a regular local operator $j$ such that the category $\Sep(\Eff(\cat H),j)$ of $j$-separated objects is equivalent to $\Asm(\ds\Kone/\Kone)$ and the category $\Sh(\Eff(\cat H),j)$ of $j$-sheaves is equivalent to $\cat H$.\label{asm is sep}\end{lemma}

\begin{proof} Any geometric inclusion of toposes induces a local operator, and this operator is regular if the direct image functor is regular. The assemblies are the objects for which the unit of the reflection is a monomorphism, a characteristic property of separated objects. \end{proof}

\begin{theorem} Let $\cat E$ be a topos with a natural number object $\nno$ and a regular local operator $j$ satisfying the following conditions.
\begin{itemize}
\item Internal $\nno$-indexed coproducts of $j$-sheaves cover all objects.
\item The canonical map $\nno\to j\nno$, where $j\nno$ is the $j$-sheaf associated to $\nno$, is a monomorphism whose image intersects every inhabited $j$-stable subobject of $j\nno$.
\item The following axiom schemas, where $jX$ is any sheaf for $j$, are valid.
\begin{align*}
\axiom{ETC(j)}\quad & [\forall m\oftype \nno. j(\phi(m))\to \exists n\oftype \nno. \chi(m,n))] \to [\exists e\oftype \nno. j(\phi(m))\to em\converges \land \chi(m,em)] \\
\axiom{UP(j)}\quad & [\forall x\oftype jX. \exists n\oftype\nno.\chi(x,n)] \to [\exists n\oftype \nno.\forall x\oftype jX.\chi(x,n)]
\end{align*}
\end{itemize}
Then $\cat E \cong \Eff(\Sh(\cat E,j))$. \label{definition of effective topos}
\end{theorem}

\begin{proof} Any local operator induces a geometric embedding of toposes, and for a regular local operator, this embedding has a regular direct image map. In this context $\nno$-indexed coproducts of $j$-sheaves are partitioned assemblies, because both are characterized by the existence of a prone morphism to $\nno$. The properties of $\nno\to j\nno$ tell us that $\nno$ is a filter that represents the external filter of inhabited subobjects of $j\nno$ in $\Sh(\cat E,j)$. The two axioms are extended Church's thesis and the uniformity principle reformulated in terms of $j$. So $\cat E$ has all the characteristic properties of $\Eff(\Sh(\cat E,j))$ and must be the same \xemph{effective topos}. \end{proof}


The following theorem explains the uselessness of recursive realizability in propositional logic.

\begin{theorem} For every $p\in\Omega$, $\Eff(\cat H)\models jp$ if and only if $\Eff(\cat H)\models p$.\end{theorem}

\begin{proof} The direction $p\models jp$ is trivial. Every proposition $p$ is equivalent to the inhabitation of some $j$-stable subobject $U$ of $\nno$, because $p\subseteq \termo$ is a separated object. By the intersection schema, see definition \ref{intersection}, $\exists n\in\nabla\supp \nno.n\in U\models \exists n\in\nno.n\in U$. We may conclude that $j(\exists n\in\nno.n\in U) \models \exists n\in \nno. n\in U$ holds because $j$ is a regular local operator, and $U$ is $j$-stable. Hence, $jp\models p$ for all propositions.
\end{proof}

If $\cat H$ is a topos we can say things about left exact functors from $\Eff(\cat H)$ into other categories, using the left exact morphism from corollary \ref{lems and toposes}.

\begin{proposition} Let $\cat H$ be a topos with natural number object $\N$, and let $\pow^*\N$ be the object of inhabited subobjects. Note that $\pow^*\N$ is an inhabited join complete applicative lattice. Let $F:\cat H\to\cat K$ be a left exact functor, and let $C\subseteq F\pow^*\N$ be a filter such that for every $x:\termo\to\pow^*\N$, $Fx$ factors through $C$. Then there is an up to isomorphism unique functor $F_C:\Eff(\cat H)\to\cat K$ such that $F_C\nabla\simeq F$ and $F_C\mathring{\pow^*\N} \simeq C$, where $\mathring{\pow^*\N}$ is the assembly determined by $\set{(n,x)\in \N\times \pow^* \N|n\in x}$.
\end{proposition}

\begin{proof} See corollary \ref{lems and toposes}. Note again that the ordering of $\Kone$ is the discrete one ($=$).\end{proof}

\section{Complete Internal Categories}
The purpose of this subsection is to revisit the work I did for my master thesis, on the algebraic compactness of the full internal subcategory of \xemph{extensional PERs} \cite{MR1167031, MR2729220}. We derive the result directly from the characteristic properties of effective categories.

\subsection{Internal categories}
In this section we go through the definition of internal categories and discuss different notions of completeness for these categories.

\begin{definition}[Ehresmann \cite{MR0213410}] An \xemph{internal category} of a category with finite limits, consists of two objects $C_0$, $C_1$, and four arrows $i:C_0\to C_1$, $s,t:C_1 \to C_0$ and $c: C_2 \to C_1$ where $C_2$ is the pullback of $s$ and $t$. These arrows satisfy $s\circ i = t\circ i = \id$, $t\circ c = t\circ p_1$ and $s\circ c=s\circ p_0$, where $p_i:C_2\to C_1$ are the projections of the pullback.
\[\begin{array}{cccc}
\xymatrix{
C_0 \ar[r]^i \ar[d]_i \ar[dr]^\id & C_1 \ar[d]^s \\
C_1\ar[r]_t & C_0
}&
\xymatrix{
C_2 \ar[r]^{p_0}\ar[d]_{p_1} \ar@{}[dr]|<\lrcorner & C_1\ar[d]^s \\
C_1 \ar[r]_t & C_0
}&
\xymatrix{ 
C_2 \ar[r]^{p_0} \ar[d]_c & C_1 \ar[d]^s \\
C_1 \ar[r]_s & C_0 \\
}&
\xymatrix{ 
C_2 \ar[r]^{p_1} \ar[d]_c & C_1 \ar[d]^t \\
C_1 \ar[r]_t & C_0 \\
}\end{array}\]
Moreover, $c$ is an associative composition operator and $i$ is a unit for composition. We use a pullback $(C_3,q_0,q_1)$ of $p_0$ and $p_1$ to help express associativity.
\[\begin{array}{ccc}
\xymatrix{
C_3 \ar[r]^{q_0}\ar[d]_{q_1} \ar@{}[dr]|<\lrcorner & C_2\ar[d]^{p_1} \\
C_2 \ar[r]_{p_0} & C_1
}&
\xymatrix{
C_3\ar[rr]^{(c\circ q_0,p_1\circ q_1)}\ar[d]_{(p_0\circ q_0,c\circ q_1)} && C_2 \ar[d]^c \\
C_2\ar[rr]_c && C_1
}&
\xymatrix{
C_1 \ar[r]^{(i\circ t,\id)} \ar[d]_{(\id,i\circ s)} \ar[dr]^\id & C_2 \ar[d]^c \\
C_2\ar[r]_c & C_1
}\end{array}\]

Let $\cat C = (C_0,C_1,i,s,t,c)$ and $\cat D = (D_0,D_1,i',s',t',c')$ be internal categories. 
An \xemph{internal functor} $\cat C \to \cat D$ is a pair of maps $f_0:C_0\to C_0$ and $f_1:C_1 \to C_1$. If $p'_i:D_2\to D_0$ is the pullback of $s',t':D_1\to D_0$, then there is a unique $f_2:C_2\to D_2$. The maps $f_0,f_1,f_2$ have to commute with all the structure maps.
\[\xymatrix{
C_2\ar[d]_{f_2} \ar[r]^c & C_1\ar[d]|{f_1} \ar@<2ex>[r]|s\ar@<-2ex>[r]_t & C_0\ar[d]^{f_0} \ar[l]|i\\
D_2 \ar[r]_{c'} & D_1 \ar@<2ex>[r]|{s'}\ar@<-2ex>[r]|{t'} & D_0 \ar[l]|{i'}
}\]

Let $F=(f_0,f_1)$ and $G=(g_0,g_1)$ be internal functors $\cat C\to\cat D$. 
An \xemph{internal natural transformation} between internal functors $F \to G$ is a map $\eta:C_0\to D_1$ such that $s'\circ \eta = f_0$ and $t'\circ\eta = g_0$. It has to satisfy the \xemph{naturality condition} that for each $x\in C_1$, $c(\eta(t(x)),f_0(x)) = c(g_0(x),\eta(s(x)))$.
\[\xymatrix{
C_1 \ar[r]^{(\eta\circ t,f_0)} \ar[d]_{(g_0,\eta\circ s)} & D_2\ar[d]^{c'} \\
D_2 \ar[r]_{c'} & D_1
}\]
\end{definition}


The properties of internal categories, especially completeness properties, are studied through the fibred category of families of objects and morphisms of the internal categories.

\newcommand\Ext{\Cat{Ext}}
\begin{definition} Let $\cat C = (C_0,C_1,i,s,t,c)$ be an internal category in $\cat E$. Its \xemph{externalization} $\Ext(\cat C): [\cat C]\to \cat E$ is the following fibred category. The domain $[\cat C]$ has arrows $f:X\to C_0$ as objects. If $g:Y\to C_0$ is another object then a morphism $f\to g$ is a pair $m=(m_0:X\to Y, m_1: X\to C_1)$ where $s\circ m_1 = f$ and $t\circ m_1 = g\circ m_0$.
\[ \xymatrix{& \ar[dl]_f X \ar[r]^{m_0}\ar[d]^{m_1} & Y\ar[d]^g \\
C_0 & \ar[l]^s C_1 \ar[r]_t & C_0}\]
Composition of morphisms is defined as follows. If $h:Z\to C_0$ is yet another object, $m:f\to g$ and $n:g\to h$ then $m_1:X\to C_1$ and $n_0\circ m_1: X\to C_1$ factors uniquely through $(C_2,p_0,p_1)$, because $t\circ m_1 = g\circ m_0 = s\circ n_1\circ m_0$. If $(n_1\circ m_0,m_1)$ is the factorization, we let $n\circ m = (n_0\circ m_0, c(n_1\circ m_0,m_1))$. 
\[ \xymatrix{
X\ar[d]_{m_0}\ar@{.>}[r]\ar@/^2ex/[rr]^{m_1} & C_2 \ar[r]_{p_1} \ar[d]_{p_0} \ar@{}[dr]|<\lrcorner & C_1 \ar[d]^t \\
Y \ar@/_2ex/[rr]_g \ar[r]^{n_1} & C_1 \ar[r]^s & C_0
}\]
The functor $\Ext(\cat C):[\cat C] \to \cat E$ sends each $f:X\to C_0$ to its domain $X$ and each morphism $(m_0,m_1)$ to $m_0$. Prone morphisms $(m_0,m_1):f\to g$ are the ones that satisfy $m_1 = i\circ g\circ m_0$.
\end{definition}

\begin{example}[internal full subcategories]
Let $f:X\to Y$ be any arrow in a topos $\cat E$. We consider $f$ to be a family of objects indexed over $Y$: $\set{X_i|i\in Y}$. There is an internal category of fibres of $f$ that is a full subcategory of $\cat E$, in a suitable sense. Let $C_0 = Y$. Take the exponential of $Y\times f:Y\times X\to Y\times Y$ and $f\times Y: X\times Y\to Y\times Y$ in $\cat E/Y \times Y$, which is the projection $\sum_{(i,j)\in Y\times Y} Y_j^{Y_i} \to Y\times Y$. The resulting arrow is $(s,t): C_1 \to C_0\times C_0$, and $(C_2,p_0,p_1)$ is the pullback of $s$ and $t$. Note the map $(s\circ p_1,s\circ p_0=t\circ p_1, t\circ p_0): C_2 \to C_0\times C_0\times C_0$. The is a fibred compositions operator $c_{ijk}: X_k^{X_j}\times X_j^{X_i} \to X_k^{X_i}$ which determines a composition operator $c:C_2 \to C_1$. \label{internalfullsubcategories}
\end{example}

\begin{example}[PERs] 
If $\cat H$ is a topos, we can define the internal category of $j$-separated subquotients of $\nno$. The \emph{category $\cat P$ of PERs} in $\Eff(\cat H)$ is the full internal subcategory generated by the following arrow. We let $P_0$ be the objects of prone partial equivalence relations on $\nno$, i.e. the symmetric and transitive binary relations $R$ for which the inclusion into $\nno\times\nno$ is a prone morphisms relative to $\supp$. We let \[B = \set{(\xi,R)\in\nno\times P_0| \exists m\oftype \nno.R(m,m) \land \forall n\oftype\nno. R(m,n) }\] Now the projection $B\to P_0$ generates $\cat P$. 

Prone relations and $j$-stable relations are the same thing by the way. The object $P_0$ is a sheaf and modest sets actually form an internal category in $\Asm(\ds\Kone/\Kone)$.\label{PERs}\index{PERs}
\end{example}

There are different notions of completeness for internal categories, because limits are usually only unique up to unique isomorphism. The completeness condition that the externalization of an internal category $\cat C$ has all finite limits and all indexed products, is equivalent to the condition that the constant functors $\Delta:\cat C\to\cat C^\cat D$ have a right adjoint for all internal categories $\cat D$. If $\cat C$ has all limits of shape $\cat D$, then such a right adjoint may not exist if the axiom of choice is not valid. This leads us to introducing a weaker notion of completeness for fibred categories.

\begin{definition} A fibred category $F:\cat F \to\cat B$ is \xemph{strongly complete} if:
\begin{itemize}
\item each fibre has finite limits, and reindexing functors preserve these limits;
\item each reindexing functor $F_f$ has a right adjoint $\prod_f$ and these right adjoints satisfy the Beck-Chevalley condition over each pullback square.
\end{itemize}
A fibred category $F:\cat F \to\cat B$ over a topos $\cat B$ is \xemph{weakly complete} if there is a full subcategory $\cat B'\subseteq \cat B$ such that
\begin{itemize}
\item the restriction of $F$ to $\cat B'$ is strongly complete,
\item objects of $\cat B'$ cover all objects of $\cat B$.
\end{itemize}

A fibred category $F$ is strongly or weakly \xemph{cocomplete} if $F^{op}:\cat F^{op}\to\cat B^{op}$ is strongly or weakly complete.
\end{definition}

\begin{remark}[stacks] Stacks $F:\cat F\to\cat B$ for the regular topology on $\cat B$ are examples of fibred categories for which strong and weak completeness coincide. Stacks are the 2-categorical counterpart of sheaves and the \xemph{`stack completion functor'} sends weakly complete fibred categories to strongly complete stacks. \end{remark}

\begin{example} Every complete fibred Heyting algebra is strongly complete and strongly cocomplete. \end{example}

\begin{example} The fibred subcategory of $\cod:\cat E/\cat E \to\cat E$ \xemph{generated by} $f:X\to Y$ consists of pullbacks of $f$ and commutative squares between these pullbacks. i.e. it is the greatest subfibration where $f$ is a \xemph{generic object}. This fibration is equivalent to the externalization of the internal full subcategory that we constructed for $f$ in example \ref{internalfullsubcategories}, at least when $\cat E$ is a small topos. \end{example}

\begin{example} For a \xemph{locally Cartesian closed category} $\cat C$, the codomain fibration $\cod:\cat C/\cat C \to\cat C$ is strongly complete. If $\cat C$ is regular and \xemph{extensive}, i.e. pullbacks preserve finite colimits, then the codomain fibration is \xemph{cocomplete}. Every topos has both properties.
\end{example}

Some completeness properties of complete preorders in the topos of sets extend to all weakly complete internal categories in other categories.

\begin{lemma} Weakly complete internal categories are weakly cocomplete. \label{comp then cocomp}\end{lemma}

\begin{proof} Every weakly complete internal category has an initial object: the limit of the identity functor $\id:\cat C\to \cat C$. If $\cat C$ is weakly complete then so is $\cat C^\cat D$ for all other internal categories $\cat D$. If each of those has an initial object, then $\cat C$ is cocomplete. \end{proof}

The desire to recursively define objects and functors motivates some other completeness conditions that we consider below.

\begin{definition} For any functor $F$, a \xemph{Lambek algebra} is an object $X$ together with an arbitrary morphism $f:FX\to X$. Lambek algebras form a category with the obvious morphisms. A category $\cat C$ is \xemph{weakly algebraically complete} if every functor $F:\cat C \to \cat C$ has an initial Lambek algebra. A category $\cat C$ is \xemph{strongly algebraically complete} if there is a functor $\cat C^\cat C \to \cat C$ that sends each functor to an initial Lambek algebra. 

A category $\cat C$ is \xemph{weakly algebraically compact} if both $\cat C$ and $\cat C^{op}$ are \xemph{weakly algebraically complete} and moreover the canonical morphism from the initial algebra to the terminal coalgebra is an isomorphism. The structure maps of initial algebras are isomorphisms by a result of Lambek, and their inverses are coalgebras, which is where the canonical morphism to the terminal coalgebra comes from.

A category $\cat C$ is \xemph{strongly algebraically complete} if there is a functor $\cat C^\cat C\to\cat C$ that sends each functor to an object that is both its initial algebra and its terminal coalgebra. \end{definition}

Completeness implies algebraic completeness.

\begin{lemma} Weakly complete internal categories are weakly algebraically complete. Strongly complete internal categories in locally Cartesian closed categories are strongly algebraically complete. \label{comp then algcomp} \end{lemma}

\begin{proof} For each endofunctor $F:\cat C \to\cat C$, we construct the category of algebras $\Cat{Alg}(F)$ and take the limit of the underlying object functor $U:\Cat{Alg}(F)\to \cat C$. The category $\Cat{Alg}(F)$ consists of pairs $(X,f)$ where $f$ is an arrow $f:FX\to X$, and a morphism $(X,f) \to (Y,g)$ is an arrow $m:X\to Y$ such that $m\circ f = g\circ Fm$. The underlying object functor $U$ simply sends $(X,f)$ to $X$ and is the identity on morphisms. Note that $\alpha_{(X,f)} = f$ determines the canonical natural transformation $\alpha:FU \to U$. In fact, it is a universal property of $U$ that every functor $G:\cat D\to\cat C$ with a natural transformation $\eta:FG\to G$ factors uniquely through $U$.

If $\kappa:\lim U \to U$ is a limit cone, we have a cone $\alpha \circ F\kappa: F\lim U \to U$ that factors uniquely through $\kappa$ by a structure map $a:F\lim U \to \lim U$. This is an initial algebra because of the unique morphism $\kappa_{(X,f)} (\lim U,a) \to (X,f)$.

For the strong version: in a locally Cartesian closed category, we can construct the functor category $\cat C^\cat C$ internally. All we have to do now is make a bundle of categories of algebras over $\cat C^\cat C$. Construct $\cat A$ as follows. The objects are triples $(F,X,f:FX\to X)$, a morphism $(F,X,f) \to (G,Y,g)$ is a pair $(\mu:F\to G, m:X\to Y)$ where $m\circ f = g\circ Gm\circ \eta_X = g\circ \eta_Y\circ Fm$. The bundle is the functor $A:\cat A \to \cat C^{\cat C}$ that sends $(F,X,f)$ to $F$ and $(\mu, m)$ to $\mu$. There is a fibred underlying object functor $U:\cat A \to \cat C^{\cat C}\times \cat C$ satisfying $U(F,X,f) = (F,X)$ and $U(\mu, m) = (\mu,m)$. 

We take the same limit in the fibre over $\cat C^{\cat C}$ and use this universal limit to construct an initial algebra. Now $\lim U$ is a global section of the bundle $\cat C^{\cat C}\times \cat C \to \cat C^{\cat C}$ and induces the desired functor $\cat C^\cat C \to\cat C$. \end{proof}

This almost concludes our introduction to internal categories and their completeness properties. We just want to add some remark on internal categories in the category of sets.

\begin{remark} The following is a proposition of Freyd. If $\cat C$ is a complete small category, i.e. a complete internal category of $\Set$, then $\cat C(x,y)^{C_0}\simeq \cat C(x,\prod_{f\in C_0} y)\subseteq C_0$ and for cardinality reasons $\cat C$ is a preordered set. \end{remark}

\subsection{Modest sets}
We discuss a complete internal category of $\Asm(\ds \Kone/\Kone)$ which is equivalent to the category of PERs from example \ref{PERs} if we work in a realizability topos, where the latter category exists. We follow \cite{MR1023803}, in that we will characterize this category as a category of objects that is orthogonal to a particular class of arrows. By picking a larger class of arrows and limiting our results to the category of assemblies, we are able to derive the result directly from the characteristic properties of realizability categories.

\begin{definition} An arrow $f:X\to Y$ is \xemph{left orthogonal} to $g:X'\to Y'$ (and $g$ is \xemph{right orthogonal} to $f$) if for any pair of arrow $h:X\to X'$ and $k:Y\to Y'$ such that $k\circ f=g\circ h$ there is a unique arrow $l:Y\to X'$ such that $h=g \circ l$ and $k=l\circ f$.
\[ \xymatrix{
X\ar[d]_f \ar[r]^h & X'\ar[d]^g \\
Y \ar[r]_k \ar@{.>}[ur]^l & Y'
}\]
We write $f\perp g$ in this situation.
\end{definition}

\begin{example} In regular categories, regular epimorphisms are left orthogonal to monomorphisms. \end{example}

The paper \cite{MR1023803} explains that the set of all arrows right orthogonal to some set of arrows has a lot of nice properties. In particular, such a class is automatically a complete fibred subcategory of the codomain fibration.

\newcommand\Mod{\Cat{Mod}}
\begin{definition} In $\Asm(\ds\Kone/\Kone)$, a \xemph{modest arrow} is any arrow that is right orthogonal to all prone regular epimorphisms. An object $X$ for which $!:X\to\termo$ is modest is called a \xemph{modest set}. We use $\Mod$ to denote the full subcategory whose objects are modest sets. \label{modest} \end{definition}

\begin{remark} We choose a large set of arrows instead of the arrow $!:\nabla(\termo+\termo)\to \termo$, which is used in \cite{MR1023803}. Hyland, Robinson and Rosolini prove that when the base category is $\Set$ they resulting classes of right orthogonal arrows are the same. But this may not apply to all base toposes. \end{remark}

\begin{lemma} Modest arrows form a complete fibred category over $\Asm(\ds \Kone/\Kone)$, which contains the natural number object and is closed under all subquotients. \label{mod is good} \end{lemma}

\begin{proof} Because $\Kone$ has the discrete order, the uniformity principle implies that for every prone morphism $p:X\to Y$, the morphism $\nno^p:\nno^Y \to \nno^X$ is a regular epimorphism. For prone regular epimorphisms, this becomes an isomorphism, and this implies that the natural number object is modest.

Let $\set{d_i:D_i \to X|i\in\kappa}$ be a family of modest arrows for which a product $d:\prod_X D_i\to X$ exists in $\cat E/X$, let $p:Y\to Z$ be any prone regular epimorphism and let $f:Y\to \prod_X D_i$ and $g:Z\to X$ be any pair of morphisms such that $d\circ f = g\circ p$. The composition of $f$ with the pullback cone $\pi_i: \prod_X D_i \to D_i$ is a family of arrows $f_i:Y\to D_i$ such that $d_i\circ f_i = g\circ p$. Because $d_i$ is left orthogonal to $p$, there is a unique $h_i$ such that $h_i\circ p = f_i$ and $f_i\circ h_i = d_i$. The $h_i$ form a cone $Z\to D_i$ that factors uniquely through $\prod_X D_i$ in a single arrow $h$ that satisfies $h\circ p=f$ and $f\circ h=d$. This proves that $\Mod$ is closed under all products that exist in $\Asm(\ds\Kone/\Kone)$.
\[\xymatrix{
Y \ar[d]_p\ar[r]^f & \prod_X D_i \ar[d]_(.4)d \ar[r]^{\pi_i} & D_i \ar[dl]^{d_i} \\
Z \ar[r]_g\ar@{.>}[ur]|h\ar@{.>}[urr]|(.3){h_i} & X \\
}\]

Because the reindexing functor $\cat E/Y\to \cat E/X$ preserves prone regular epimorphisms, its right adjoint preserves modest arrows. We use this fact to show that modest arrows are closed under internal products.

Let $f:X\to Y$, let $d:D\to X$ be modest, let $p:P\to Z$ be a prone regular epimorphism, let $g: P \to \prod_f(D)$ and let $h:Z \to Y$ such that $\prod_f(d)\circ g = h\circ p$. Because of the adjunction, there is a transposed diagram in $\cat E/X$ with a prone regular epimorphism $f^*(p): f^*(P)\to f^*(Z)$ and morphisms $g^t: f^*(P)\to D$ and $h^t: f^*(Z) \to X$ where $f^*(p)$ is still a prone regular epimorphism. Therefore, there is a unique $k:f^*(Z)\to D$ such that $k\circ f^*(p) = g^t$ and $d\circ k = h^t$. The transposes of these $k$'s show that $\prod_f(d)$ is modest whenever $d$ is.
\[ \xymatrix{
f^*(P)\ar[d]_{f^*(p)} \ar[r]^{g^t} & D\ar[d]^d & P\ar[d]^p \ar[r]^g & \prod_f(D)\ar[d]^{\prod_f(d)} \\
f^*(Z)\ar[r]_{h^t} \ar@{.>}[ur]|k & X\ar@/_2ex/[rr]_f & Z\ar[r]^h \ar@{.>}[ur]|{k^t} & Y
} \]

Now we show that modest arrows are closed under subobjects, because prone regular epimorphisms are regular epimorphisms.

Let $e:P\to Z$ be a prone regular epimorphism. Let $d:D\to Y$ be modest and let $m:S\to D$ be a monomorphism.
Finally, let $f:P\to S$ and $g:Z\to Y$ be such that $d\circ m \circ f = g\circ e$. Because $e\perp d$ there is a unique morphism $h:Z\to D$ such that $h\circ e = m\circ f$ and $d\circ h = g$. Now $e\perp m$ because $e$ is a regular epimorphism and $m$ is a monomorphism, so there is a unique $k:Z\to S$ such that $k\circ e = f$ and $m\circ k = h$, which implies $d\circ m \circ k = g$. Therefore subobjects of modest arrows are right orthogonal to all prone regular epimorphisms.
\[ \xymatrix{
P \ar[d]_e & S\ar[d]^m \\
Q \ar[d]_r\ar[ur]^f & D\ar[d]^d \\
Z \ar[r]_g\ar@{.>}[ur]|h\ar@{.>}[uur]|k & X
}\]
\end{proof}

We now start our comparison of modest arrows and PERs.

\begin{lemma} Modest arrows $X\to Y$ are quotients of prone subarrows of the projections $\nno\times Y \to Y$.  \label{mod is per} \end{lemma}

\begin{proof} By lemma \ref{slice}, each slice category of $\Asm(\ds\Kone/\Kone)$ is a realizability category for its own natural number object and subcategory of prone arrows. We may therefore apply weak genericity. Let $p:P\to \nno\times Y$ be prone, let $e:P\to D$ be a regular epimorphism and let $d:D\to Y$ be a modest arrow. Split $p$ into a prone regular epimorphism $f:P\to \im p(P)$ and a prone monomorphism $m:\im p(P) \to \nno\times Y$. Due to orthogonality of $d$, $e:P\to D$ factors through $f$ in a new regular epimorphism.
\[\xymatrix{
P\ar[dr]_e \ar[r]^p & \im p(P)\ar[r]\ar@{.>}[d] & \nno\times Y \ar[d]\\
& D \ar[r]_d & Y
}\]
\end{proof}

We show that suitable subquotients of $\nno$ are modest.

\begin{lemma} A subquotient of $\nno$ by a prone equivalence relation is modest. \label{per is mod} \end{lemma}

\begin{proof} For such a subquotient $D$ there is a regular epic \prodo morphism $c:U\to D$ where $U\subseteq \nno$. Now consider a prone epimorphism $e:Y\to Z$ and an arrow $f:Y\to X$. The pullback of $c$ along $f$ is another regular epic \prodo morphism so that we can now apply the uniformity rule to the sequence $e\circ f^*(c)$. We get a prone $W\subseteq \nno\times Z$ such that $W\times_Z Y\subseteq U\times_D Y (\subseteq \nno\times Y)$. 

For each pair $y,y'\in Y$ such that $e(y) = e(y')$ we have $n\in\nno$ such that $(n,e(y))$ and $(n,e(y'))\in W$ and therefore $(n,f(y)),(n,f(y'))\in U$. But that implies $f(y)=f(y')$ and hence $f$ factors through $e$ in a unique $g:Z\to D$. This proves $D$ is modest.

\[ \xymatrix{
& U\times_D Y \ar[r]\ar[d]\ar@{}[dr]|<\lrcorner & U \ar[d]^c \\
W\times_Z Y \ar[r]\ar[d]\ar@{}[dr]|<\lrcorner \ar@{.>}[ur] & Y \ar[d]_e\ar[r]^f & D \\
W \ar[r] & Z \ar@{.>}[ur]_g 
}\]
\end{proof}

The lemmas show that modest arrows are families of subquotients of $\nno$. If the base category $\cat H$ is a topos then we can construct the internal category $\cat P$ of $j$-separated subquotients of $\nno$ from example \ref{PERs} in $\Eff(\cat H)$. This category is actually an internal category of the category of assemblies $\Asm(\ds\Kone/\Kone)$. Therefore we can say the following about modest arrows and $\cat P$.

\begin{proposition} Let $\cat H$ be a topos. Over $\Asm(\ds \Kone/\Kone)$ the fibration of modest arrows is equivalent to the externalization of the category $\cat P$ of PERs. \end{proposition}

\begin{proof} Lemmas \ref{mod is per} and \ref{per is mod} establish the equivalence between the modest arrows. \end{proof}

\begin{corollary} For any base topos $\cat H$, the category $\cat P$ of PERs is a strongly complete internal category of $\Asm(\ds \Kone/\Kone)$ and a weakly complete internal category of $\Eff(\cat H)$. \end{corollary}

\begin{proof} The equivalent fibred category of modest sets is complete by lemma \ref{mod is good}. Because $\Eff(\cat H) = \Asm(\ds\Kone/\Kone)_\exreg$, $\Asm(\ds\Kone/\Kone)$ is a full subcategory whose objects cover each object of $\Eff(\cat H)$. Here weak completeness comes from.
\end{proof}

\begin{remark} We want to note some things about this completeness result. Not all subquotients of $\nno$ in $\Eff(\cat H)$ are $j$-separated, so $\cat P$ is not the category of all subquotients. Similarly, for arrows that are right orthogonal to prone regular epimorphisms the domain and codomains don't have to be separated. The proof in \cite{MR1023803} that all subquotients of $\nno$ are right orthogonal relies critically on the existence of enough internally projective objects among sheaves, a condition that could fail for many effective toposes. On the other hand, the slices of $\Eff(\cat H)$ over non-separated objects may not satisfy weak genericity, so the proof that right orthogonal objects are subquotients can fail in these slices. This means that we might have two fibred categories neither of which is a subcategory of the other. In \cite{MR1023803} the \xemph{discrete arrows} are the right orthogonal ones, while in \cite{MR2479466} the discrete arrows are families of subquotients of $\nno$. Either way, they do not form a subcategory that is as well-behaved as the fibred category of modest arrows over $j$-separated objects. Unfortunately, the category of PERs cannot be strongly complete in $\Eff(\cat H)$ by proposition 7.5 of \cite{MR1023803} (which is proposition 3.4.14 in \cite{MR2479466}).
\end{remark}

\subsection{Pointed complete extensional PERs}
We will now single out a subcategory of the category of modest sets which is algebraically compact. This category was identified in \cite{MR1167031}, but the proof for algebraic completeness given there has a gap. I patched that gap in \cite{MR2729220}. Here, we reproduce the result using the characteristic properties of effective categories. The first step is to exploit the algebraic completeness of complete internal categories.

\begin{lemma} The category $\cat P$ of PERs is strongly algebraically complete and cocomplete. \end{lemma}

\begin{proof} For the strong algebraic completeness of modest sets, we are in luck that both $P_0$ and $P_1$ are separated, and that the functor category $\cat P^{\cat P}$ also consists of separated objects. The object $\Omega_j^{\nno\times\nno}$ of $j$-stable subobjects of $\nno\times \nno$ is a sheaf, because $\nno$ is separated, and there is a bijection between $j$-stable subobjects of $\supp\nno$ and $\nno$. Evidently $P_0$ is separated because it is a subobject of a sheaf.

The reason that $P_1$ is separated is that the separated objects are a locally Cartesian closed category. The topos of sheaves is locally Cartesian closed, and we can simply define exponentials in every slice of the category of separated object, by 
\[ X^Y = \set{f\in \supp X^{\supp Y} | \forall x\oftype Y. f(x)\in X }\]
Here we exploit the fact that the unit $\eta_X: X\to \nabla\supp X$ is a monomorphism for separated objects.
We defined $P_1$ as an exponential, and hence it is a separated object. We note this result implies also that internal categories of $\Asm(\ds \Kone/\Kone)$ are closed under functor categories, and that in particular $\cat P^{\cat P}$ is separated.

By lemma \ref{comp then cocomp}, $\cat P^{op}$ is strongly cocomplete. We now use lemma \ref{comp then algcomp} to conclude that $\cat P$ is strongly algebraically complete and cocomplete. \end{proof}

\begin{remark} The category $\cat P$ of modest sets is not algebraically compact. The initial algebra of $\id:\cat P\to\cat P$ is the initial object whereas its terminal coalgebra is the terminal object, and the unique map between them is not an isomorphism. \end{remark}

We now introduce a new subcategory of $\Eff(\cat H)$ and show that is it is a strongly algebraically compact internal subcategory.

\begin{definition} In $\Eff(\cat S)$ a \xemph{decidable subobject} $U\subseteq X$ is any subobject for which a function $k:X\to \nno$ exists such that $kx=0$ if and only if $x\in U$. An $\nno$-indexed union of decidable subobjects is a \xemph{semidecidable subobject}. Semidecidable subobjects have a classifier:
\[ \Sigma = \set{\sigma\in\Omega| \exists f\oftype\nno\to\nno. (\sigma \leftrightarrow \exists n\oftype\nno. f(n)=0) } \]

An object $X$ is \xemph{extensional} if it is a $j$-stable subobject of $\Sigma^Y$ for some $Y$. It is \xemph{complete} if for every monotone $f:\nno\to X$, there is a join in $X$. It is \emph{pointed} if $\emptyset\in X$.
\index{extensional object!pointed}
\end{definition}

\begin{lemma} Every extensional object is modest. \end{lemma}

\begin{proof} We first note that $\Sigma$ is a subquotient of $\nno$, because $\nno^\nno$ is, and there is an evident epimorphism $\nno^\nno \to\Sigma$, because $\Sigma$ is the image of $f\mapsto (\exists n\oftype\nno.f(n)=0):\nno^\nno\to\Omega$.

The object $\Sigma$ is separated. We have $j((\exists n\oftype\nno.f(n)=0)) \to (\exists n\oftype\nno.f(n)=0)$ because $\nno$ is a separated object and hence $j(f(n)=0)\to f(n)=0$. This show that the fibres of the epimorphism $\nno^\nno \to\Sigma$ are $j$-stable sets, and therefore that $\Sigma$ is separated.

All $j$-stable subobjects of $\Sigma^X$ are modest, because modest sets are closed under subobjects. \end{proof}

\begin{definition} A \xemph{pointed complete extensional PER} is a PER $R$ such that the subquotient $\nno/R$ is a pointed complete extensional object. \end{definition}

\begin{theorem} The category of pointed complete extensional PERs is strongly algebraically compact. \label{compactness} \end{theorem}

\begin{proof} Pointed complete extensional objects form a reflective subcategory of the category of modest sets, and this induces a reflection of the category of coalgebras of any functor. The reflector preserves the terminal coalgebra, so that we get a pointed complete extensional version of it. Other properties of the category of pointed complete extensional objects then imply that it is algebraically compact.

For each modest set $X$ let $RX$ be the least pointed complete $j$-stable subobject of $\Sigma^{\Sigma^X}$ that contain all sets of the form $\set{\xi\in \Sigma^X | x\in \xi }$ for $x\in X$. For each $f:X\to Y$, $\Sigma^{\Sigma^f}$ is an inverse image map, and therefore preserves all joins. For this reason the restriction of $\Sigma^{\Sigma^f}$ to $RX$ factors through $RY$. 

The natural transformation $\eta_X:X\to \Sigma^{\Sigma^X}$ that satisfies $\eta_X(x) = \set{\xi\in \Sigma^X | x\in \xi }$, factors through $R$. That $R\eta_X$ is an isomorphism stems from the fact that $\eta_X:X\to\Sigma^{\Sigma^X}$ is a monomorphism, and the closure of the image of $X$ in $\Sigma^{\Sigma^{\Sigma^{\Sigma^X}}}$ is therefore isomorphic.

We consider all functors at once, to handle stability issues. In the category of modest sets, let a \xemph{Lambek coalgebra} be a triple $(F,X,a:X\to FX)$, where $F$ is a functor, $X$ is a modest set and $a$ is any morphism $a:X\to RX$. Morphisms $(F,X,a) \to (G,X,b)$ are pairs $(\mu,m)$ where $\mu:F\to G$, $m:X\to Y$ and $b\circ m = \mu m\circ a$, where $\mu m = \mu_Y\circ Fm = Gm\circ \mu_X$. Composition goes componentwise.
\[ \xymatrix{
X \ar[d]^m\ar[r]^a & FX\ar[r]^{\mu_X} \ar[d]_{Fm} \ar[dr]^{\mu m}  & GX \ar[d]^{Gm}  \\
Y\ar@/_2ex/[rr]_b & FY \ar[r]^{\mu_Y} & GY 
}\]
This is still equivalent to an internal category, in fact some subcategory $\cat L\subseteq \cat P^{\cat P}\times\cat P\times (\cat P/\cat P)$. The projection $\Pi:\cat L\to\cat P^{\cat P}$ has a section $Y:\cat P^{\cat P} \to \cat L$ that picks out terminal coalgebras for each functor in $\cat P^{\cat P}$ thank to the algebraic cocompleteness of modest sets.

Let $\cat C$ be that category of pointed complete extensional $j$-stable partial equivalence relations on $\nno$. The map $F\mapsto FR$ determines an inclusion $R^*:\cat C^\cat C \to\cat P^\cat P$. Let $F:\cat C\to\cat C$, then we have a terminal coalgebra $y_{FR}:Y_{FR} \to FRY_{FR}$. Now $RF=F$ because the objects in the image of $F$ are already pointed complete and extensional. This gives us the algebra $RY_{FR} \to RFRY_{FR} = FRY_{FR}$, and the morphism $(\id,\eta):(FR,Y_{FR},y_{FR})\to (FR,Y_{FR},Ry_{FR})$ which has to be an isomorphism. So the terminal coalgebra of a functor $\cat C\to\cat C$ is already a pointed complete extensional object.

For the last part we refer to \cite{MR1099200}: $\cat C$ is enriched in complete posets, and complete poset enriched categories are algebraically compact if they are algebraically cocomplete. \end{proof}

\section{Classical Realizability}
Krivine's \xemph{classical realizability} \cite{MR2825677} is a new construction of models of second order arithmetic and set theory that is based on the \la-calculus. Classical realizability explains the properties of certain Boolean subtoposes of relative realizability categories and relative realizability categories are a source of new classical realizability models. This section contains the construction of a \xemph{classical realizability tripos} from a partial applicative structure with a combinatory complete filter. We show that it is a subtripos of the realizability tripos.

Our underlying category will be a topos $\Set$ of sets, but the sets may come from arbitrary models of Zermelo's set theory, just as classical realizability models are constructed relative to an arbitrary model of set theory. We can probably generalize to other categories, but we haven't checked that no problems arise there.

\subsection{Classical realizability}
Classical realizability still assigns a set $x$ to every proposition $p$, but this time $x$ contains \xemph{challenges} to the validity of $p$. A realizer for $p$ encodes an algorithm for defeating these challenges.

The basic structure we work with extends an applicative structure with a set of stacks. Note that the stacks here are not fibred categories.

\newcommand\terms\Lambda
\newcommand\stacks\Pi
\newcommand\cc{\underline{\comb{cc}}}
\begin{definition} A \xemph{stack structure} is an applicative structure $\terms$ of \xemph{terms} with a set of \xemph{stacks} $\stacks$ which is connected to $\terms$ by the following structure:
\begin{itemize}
\item there is an operator called \xemph{push} $(t,\alpha)\mapsto t\cdot \pi: \terms\times \stacks \to\stacks$;
\item there is an operator called \xemph{continuation} $\pi\to k_\pi: \stacks\to\terms$;
\item there is a special element $\cc\in \terms$;
\item the set $\terms\star\stacks = \terms\times \stacks$ is called the set of \xemph{processes}; we write $t\star \pi$ instead of $(t,\pi)$ to denote its members; $\terms$ has an ordering $\succ$ which satisfies:
\[ \begin{array}{ccc}
tu\star\pi \succ t\star u\cdot \pi & \cc\star t\cdot\pi\succ t\star k_\pi\cdot \pi & k_\pi\star t\cdot\rho\succ t\star\pi
\end{array}\]
\end{itemize}

A filter $C\subseteq \terms$, which is still a subset closed under application, is \xemph{operationally complete} if it contains $\cc$ and if for each combinatory function $f:\terms^n\to\terms$ there is a $t\in C$ such that $t\star u_1\cdot\dotsm\cdot u_n\cdot \pi \succ f(\vec u)\star \pi$. A stack structure that has an operationally complete filter is a \xemph{realizability algebra} (Krivine) or an \xemph{abstract Krivine structure} (Streicher). 
\end{definition}

\begin{remark} Krivine's and Streicher's definitions are not equivalent, see Krivine \cite{MR2825677}, Streicher \cite{KCRfaCP}. What we defined is a well-behaved subclass of both realizability algebras and abstract Krivine structures. We also made the definition more similar to our definition of partial combinatory algebras, by replacing the existence of certain \xemph{basic combinators} with the more abstract property of operational completeness. \end{remark}

To define a tripos using an abstract Krivine structure, we need one additional bit of structure. \newcommand\pole{{\bot\mkern-10.5mu\bot}}
\begin{definition} Let $(\terms,\stacks,\cdot,k_*,\cc,\succ)$ be a stack structure. A \xemph{pole} is a $\pole\subseteq \terms\star\stacks$ that is \xemph{saturated}: if $p\in \pole$ and $q\succ p$ then $q\in\pole$.

Using the pole, we define the \xemph{orthogonal complement} $U^\pole$ of a set of stacks by $U^\pole = \set{t\in\terms| \forall \pi\in U. t\star \pi\in \pole}$. Similarly, a set of terms $V$ has a complement $V^\pole = \set{\pi \in\stacks| \forall t\in U. t\star \pi\in \pole}$. 
\end{definition}

With the help of this pole we can make the notion of challenge precise. If $S\subseteq \stacks$ is the set of stacks assigned to a proposition $p$, then $S^\pole$ is the set of realizers. The $\pi\in S$ challenge the validity of $p$; the $t\in S^\pole$ send each challenge $\pi$ to $\pole$, thereby defeating $\pi$. If $t$ is also member of the filter $C$, then $p$ is valid. We turn this informal description into a tripos.

\newcommand\CR{\Cat{CR}}
\begin{definition} Let $\stacks = (\terms,\stacks,\cdot,k_*,\cc,\succ) $ be an abstract Krivine structure. For all $f,g:X\to\pow\stacks$ we let $f\Vdash_X g$ if there is a $t\in C$ such that for all $x\in X$, $u\in f(x)^\pole$ and $\pi\in g(x)$, $t\star u\circ \pi \in \pole$. The resulting indexed preorder $(\pow\stacks,\Vdash)$ is the \xemph{classical realizability tripos} $\CR(\stacks,C,\pole)$ for $(C,\pole)$.
\end{definition}

\begin{proposition}[Krivine, Streicher] Let $\stacks = (\terms,\stacks,\cdot,k_*,\cc,\succ)$ be a stack structure, let $C$ be an operationally complete filter and let $\pole\subseteq \terms\times\stacks$ be any pole. The classical realizability tripos $\CR(\stacks, C, \pole)$ is a Boolean tripos. \end{proposition}

\begin{proof} We start by defining implication as follows.
\[ f\to g(x) = \set{t\cdot \pi\in\stacks|t\in f(x)^\pole, \pi\in g(x)}\]
By definition, $t\in C$ realizes $f\Vdash g$ if and only if $t\in (f\to g(x))^\pole$ for all $x$. If $u\in f^\pole(x)$, then $tu\in g^\pole(x)$ because for all $\pi \in g(x)$, $tu\star\pi \succ t\star u\cdot \pi \in \pole$, because $u\cdot \pi\in f\to g(x)$ and $t\in (f\to g)^\pole$. We have found a modus ponens rule, so now we will just show that the axioms of classical implication hold.

Let $\comb k\in C$ be a combinator that satisfies $\comb k \star t\cdot u\cdot \pi \succ t\star \pi$ and let $f,g:X\to \pow\stacks$. For all $x \in X$, $\comb k\in (f\to g\to f(x))^\pole$, because if $t\in f(x)^\pole$, $u\in g(x)^\pole$ and $\pi\in f(x)$, then $\comb k\star t\cdot u\cdot \pi \succ t\star \pi$ and because $t\in f(x)^\pole$ and $\pi\in f(x)$, $t\cdot \pi\in \pole$.

Let $\comb s\in C$ be a combinator that satisfies $\comb s\star t\cdot u\cdot v\cdot \pi\succ tv(uv)\star \pi$ let $f,g,h:X\to \pow\stacks$. For all $x\in X$, $\comb s \in ((f\to g\to h)\to (f\to g) \to f\to h(x))^\pole$ because if $t\in (f\to g\to h(x))^\pole$, $u\in (f\to h(x))^\pole$ and $v\in f(x)^\pole$, then $tv(uv)\in h(x)^\pole$ and hence if $\pi\in h(x)$, then $tv(uv)\cdot \pi\in \pole$.

The special combinator $\cc$ is a member of $((f\to g)\to f)\to f(x))^\pole$ for all $x\in X$ and $f,g:X\to\pow\stacks$. If $\pi\in f(x)$ then $k_\pi\in (f\to g)^\pole$, because if $t\in f(x)^\pole$ and $\rho\in g(x)$, then $k_\pi t\star\rho\succ t\star \pi\in \pole$. Hence if $u\in ((f\to g)\to f)^\pole$ then $u\star k_\pi \cdot \pi\in\pole$, and therefore $\cc\star u\cdot \pi \succ u\star k_\pi\cdot \pi \in\pole$. This proves that Pierce's law holds and that the tripos is Boolean.

Let $\comb i\star t\cdot\pi \succ t\star \pi$. If $f(x)\subseteq g(x)$, then $\comb i \in (g\to f(x))^\pole$, because if $t\in g(x)^\pole$ and $\pi\in f(x)$ then $\pi\in g(x)$ and $\comb i \star t\cdot \pi \succ t\star \pi \in\pole$. We see at once that $\comb i\in (\stacks\to f(x))^\pole$, which is the principle of explosion. This allows us to define negation by letting $\neg g = f\to\stacks$, and using negation, we can define all other connectives.

Now we know that $(\pow\stacks^X,\Vdash_X)$ is a Boolean algebra for each set $X$. 

Let $f:X\to Y$ be any function, and let $g,h:Y\to \pow\stacks$. If $t\in (g\to h(y))^\pole$ for all $y\in Y$, then $t\in (g\to h(f(x)))^\pole$ for all $x\in X$, and hence reindexing preserves $\Vdash$.

We focus on defining the right adjoint now, since classical negation determines a left adjoint given any right adjoint.

Define $\forall_f(g)(y) = \bigcup_{f(x)=y} g(x)$ for all $g:X\to \pow \stacks$. Let $g,h:X\to\pow \stacks$. If $t\in (g\to h(y))^\pole$ for all $x\in X$, then $t\in (\forall_f(g)\to\forall_f(h)(y))^\pole$ for all $y\in Y$. If $\pi \in \forall_f(h)(y)$, then there is an $x\in X$ such that $f(x)=y$ and $\pi \in h(x)$. But if $u\in \forall_f(g)(x)^\pole$, then for all $x'$ such that $f(x')=y$, $u\in g(x')^\pole$ and therefore in particular in $g(x)^\pole$. For these reasons $t\star u\cdot \pi\in \pole$ if $t\in (g\to h(x))^\pole$ for all $x\in X$. This means that $\forall_f$ preserves $\Vdash$. Let $k:Y\to \pow\stacks$. Now $\bigcup_{y\in Y}k \to \forall_f(h)(y) = \bigcup_{x\in X}k\circ f\to h(x)$ by the definition of $\forall_f$. Hence $\forall_f$ is right adjoint to reindexing.

The way implication is defined, it is clear that $(f\to g)\circ h = f\circ h \to f\circ h$. This implies the Frobenius condition:
\begin{align*}
f\Vdash (g\to h)\circ k &\iff f\Vdash g\circ k \to h\circ k \\
\exists_k(f) \Vdash g\to h &\iff f\land g\circ k \Vdash h\circ k \\
\exists_k(f)\land g \Vdash h &\iff \exists_k(f\land g\circ k) \Vdash h
\end{align*}

Now we know that we are dealing with a complete fibred Boolean algebra. The map $\epsilon_X: \pow\stacks^X\times X\to \pow\stacks$ determines the membership predicate. For every $f:X\times Y \to\pow\stacks$ there is a $f^t: X\to \pow\stacks^Y$ such that $f = \epsilon_X\circ (f^t\circ X)$. Therefore the classical realizability tripos is in fact a tripos.
\end{proof}

This is classical realizability. For examples see \cite{MR2825677}.

\subsection{Implementation}
We are going to construct a stack structure out of an order partial applicative structure. The relation $\succ$ between processes of a stack structure give the operational semantics of \xemph{Krivine's machine}, which we can implement in combinatory logic. Order partial combinatory algebras model combinatory logic (partially) and this is where we derive our construction from.

\newcommand\whrto{\twoheadrightarrow^*_w}
\newcommand\CT{\mathrm{CT}}
\begin{definition}[combinatory logic] A \xemph{combinatory term} $M,N,\dots$ is a variable $x,y,z,\dots$, one of the \xemph{basic combinators} $\comb b,\comb c,\comb k,\comb w$, or an application of combinators $MN$. A string of combinators $M_1M_2\dotsm M_n$ stands for the nested application $((M_1M_1)\dotsm)M_n$. The set of all such terms is $\CT$.

We order the combinatory terms by \xemph{weak head reduction} $\whrto$, which is the least preorder that satisfies: if $M\whrto M'$ then $MN\whrto M'N$ and
\begin{align*}
\comb bxyz &\whrto x(yz) & \comb cxyz &\whrto xzy &
\comb kxy &\whrto x & \comb wxy&\whrto xyy
\end{align*}

We define a \xemph{substitution operator} by the following rules.
For every term $N$ every variable $x$, every variable or constant $y$ that differs from $x$ and every pair of terms $M,M'$ we let
\begin{align*}
x[N/x] &= N & y[N/x] &= y & (MM')[N/x] &= (M[N/x])(M'[N/x])
\end{align*}

We define an \xemph{abstraction operator} by the following rules.
For every term $N$ every variable $x$, every variable or constant $y$ that differs from $x$ and every pair of terms $M,M'$ we let
\begin{align*}
\hat x(x) &= \comb{wk} & \hat x(y) &= \comb ky & \hat x(Mx) &= \comb w\hat x(M) \\
 \hat x(My) &= \comb c\hat x(M)y & \hat x(M(NP)) &= \hat x(\comb bMNP)
\end{align*}
\end{definition}

We have seen $\CT$ with a different ordering as an example (\ref{combinatorylogic}) of an order combinatory algebra.

\begin{remark} We only use weak head reduction because it is sufficient for Krivine's machine. Later we will see that we can weaken the structure of order partial combinatory algebras along the same lines, and still get a good realizability interpretation. The abstraction function $\hat x$ is total, despite the recursion in its definition; induction over the depth of application shows this.
\end{remark}

\begin{lemma} For all terms $M,N\in\CT$ and every variable $x$, $\hat x(M)N\whrto M[N/x]$.\label{abstraction} \end{lemma}

\begin{proof} Induction. \end{proof}

After that exposition of combinatory logic and its interpretation in order partial combinatory algebras, we can now start implementing Krivine's machine.

\begin{definition} We build a syntactic stack structure out of $\CT$, with the following combinators and operators. Note that we define a new application operator $\bullet$ for the stack structure and that we also introduce an \xemph{abstraction operator} $\lambda^\bullet$, which will later help to establish operational completeness.
\begin{align*}
\comb i &= \comb{wk} & \comb g &= \comb{ci} & \comb p &= \comb{bcg} \\
M\cdot N &= \comb p MN & M\bullet N &= \comb b M(\comb p N) & \lambda^\bullet x.M &= \comb g\hat x(M) \\
\kappa &= \hat x(\lambda^\bullet y.\comb k(ux)) & k_M &= \kappa M & \cc &= \lambda^\bullet x.\hat y(x(k_y\cdot y))
\end{align*}
\end{definition}

\begin{remark} This translation of classical realizability to combinatory logic is based on the one found in \cite{MR2550048}. \end{remark}

\begin{lemma}
Let $M,N,\pi,\rho\in\CT$, then 
\begin{align*}
(M\bullet N)\pi &\whrto M(N\cdot \pi) & (\lambda^\bullet x.M)(N\cdot \pi)&\whrto M[N/x]\pi \\
k_\pi(M\cdot\rho) &\whrto M\pi & \cc(M\cdot \pi) &\whrto M(k_\pi\cdot \pi)
\end{align*}\label{KrivinesMachine}
\end{lemma}

\begin{proof} Writing out the definitions proves this. \end{proof}

We have implemented Krivine's machine in a sort of generic combinatory algebra. Now we show how we can turn order partial combinatory algebras into abstract Krivine structures.

\begin{definition} Let $A$ be an order partial applicative structure. A \xemph{partial interpretation} of combinatory terms, is a partial function $f:\CT \partar A$ that satisfies $f(MN)\converges$ if and only if $f(M)f(N)\converges$ and in that case $f(MN) = f(M)f(N)$; also for all $x,y,z\in A$:
\begin{align*}
&f(\comb b)xy\converges & x(yz)\converges \Longrightarrow & f(\comb b)xyz\converges, f(\comb b)xyz \leq x(yz) \\
&f(\comb c)xy\converges & xzy\converges \Longrightarrow & f(\comb c)xyz\converges, f(\comb c)xyz \leq xzy \\ 
&& & f(\comb k)xy\converges, f(\comb k)xy \leq x \\ 
&f(\comb w)x\converges & xyy\converges \Longrightarrow &  f(\comb w)xy\converges, f(\comb b)xyz \leq x(yy)
\end{align*}
\end{definition}

\begin{lemma} For all terms $M,N\in \CT$ and every partial interpretation $f$, if $M\whrto N$ and $f(N)\converges$, then $f(M)\converges$ and $f(M)\leq f(N)$.\label{monotony} \end{lemma}

\begin{proof} Induction. \end{proof}

\begin{definition} Let $C$ be a combinatory complete filter on $A$. Note that a partial interpretation $f:\CT\partar A$ exists, such that $f(\comb b),f(\comb c),f(\comb k),f(\comb w)\in C$. The \emph{abstract Krivine structure induced by $A$, $C$, and $f$} satisfies
\begin{itemize}
\item terms and stack are both $\terms_A = \stacks_A = A$;
\item application on $\terms_A$ is $x\bullet y = f(\comb b) x(f(\comb p) y)$; note that it is total;
\item push is $x\cdot y = f(\comb p)xy$;
\item continuation is $k_x = f(\kappa)x$;
\item $\cc = f(\cc)$
\item we define $(x,y)\succ(x',y')$ as follows: if $x'y'\converges$, then $xy\converges$ and $xy\leq x'y'$.
\item we let $C=C$
\end{itemize}
\end{definition}

\begin{theorem} With this ordering and this filter $(\terms_A,\stacks_A,\cdot,k_*,\cc,\succ)$ is an abstract Krivine structure.\end{theorem}

\begin{proof} Lemmas \ref{KrivinesMachine} and \ref{monotony} together imply that $\succ$ satisfies
\begin{align*}
t\bullet u\star \pi &\succ t(u\cdot \pi) & 
k_\pi(t\cdot\rho) &\succ t\pi & \cc(t\cdot \pi)&\succ t(k_\pi\cdot \pi)
\end{align*}
For operational completeness, note that for each partial combinatory function $g$ we have a \la-term $\lambda \vec x.g(\vec x)$, which we can interpret in $\CT$ using $\lambda^\bullet$ and the abstraction operators $\hat x$. The result is a term $M$ that has no free variables. This means the partial interpretation $f$ sends it to a member of $C$. Lemma \ref{abstraction} now helps to show that:
\[ f(M)\star t_1\cdot\dotsm\cdot t_n \cdot \pi \succ g(\vec t)\star \pi \]
Therefore the filter is operationally complete.
\end{proof}

We just need to add a pole to get a classical realizability tripos. The next subsection will show that natural choice of pole makes the classical realizability tripos equivalent to a subtripos of $\ds A/C$.

\subsection{Negative translation}
In any topos $\cat E$ subterminal objects induce local operators . e.g. if $u\subseteq\termo$ we have $x\mapsto u\to x$, $x\mapsto x\mapsto u\vee x$, $x\mapsto (x\to u)\to u$. The $x\mapsto u\to x$ operator corresponds to the slice topos $\cat E/u$, the restriction of the topos to this subterminal. Geometrically $x\mapsto u\vee x$ is the complement of $x\mapsto u\to x$: in a topos of sheaves over a topological space $x\mapsto u\vee x$ correspond to the restriction of the topos to the closed complement of $u$, while $x\mapsto u\to x$ corresponds to sheaves over $u$. 

The local operator $x\mapsto (x\to u) \to u$ it the result of combining $x\mapsto u\vee x$ with $\neg\neg$. In this way subterminals generate a family of Boolean subtoposes of $\cat E$. In this subsection we will show that such a Boolean subtopos of a relative realizability topos is a classical realizability topos. However, we work this out at the tripos level.

\newcommand\Cl{\Cat{Cl}}
\begin{definition} Let $A$ be an order partial combinatory algebra, $C$ a combinatory complete filter and let $U$ be any downset. A family of downsets $Y\subseteq A\times X$ is \xemph{$U$-stable}, if $C$ intersects $((Y\arrow U\times X)\arrow U\times X)\arrow Y$. Here $\arrow$ is as in lemma \ref{applace}:
\[ V\arrow W = \set{ (a,x)\in A\times X | \forall b\in A.(b,x)\in V \to ab\converges \land (ab,x)\in W }\]

The fibred locale $\Cl(\ds A/C,U)$ is the restriction of $\ds A/C:\cat DA/C \to\Cat{Set}$ to the $U$-stable objects.
\end{definition}

\begin{remark} Reasoning from the perspective of the relative realizability topos, a family of $U$-stable downsets correspond to $(*\to U)\to U$-stable subobjects of $\neg\neg$-sheaves. 
\end{remark}

Now we just need to pick a suitable pole.

\begin{definition} Consider the stack structure $\stacks_A = ((A,\bullet),A,\cdot,k_*,\cc,\succ)$ constructed from $A$, and the operationally complete filter $C$. For each downset $U$ of $A$, we have $\pole_U = \set{ (t,\pi)\in A\times A| t\pi\converges\land t\pi\in U }$.
\end{definition}

\begin{theorem} The triposes $\Cl(\ds A/C,U)$ and $\CR(\stacks_A,C,\pole_U)$ are equivalent. \label{cr sub rr}\end{theorem}

\begin{proof} We start with an analysis of the classical realizability tripos. Let $f:X\to X'$, $g:X\to\pow A$ and $h:X'\to \pow A$, then $f$ is a morphism $(X,g) \to (Y,h)$ if $g\Vdash_X h\circ f$, which means there is some $t\in C$ such that for all $x\in X$, $u\in g(x)^\pole$ and $\pi \in h(f(x))$, $t\star u\cdot \pi\in \pole_U$. Writing out the definitions gives us $t(\comb pu\pi)\converges$ and $t(\comb pu\pi)\in U$, where $\comb p$ is the paring combinator. If fact we can say $t\rlzs \forall x\oftype X.(\overline{g(x)}\to U \land \overline{h(f(x))}) \to U$ if $\overline{g(x)}$ and $\overline{h(f(x))}$ are the downward closures of $g(x)$ and $h(f(x))$.

On the relative realizability side, every family of downsets $Y\subseteq A\times X$ has a characteristic function $\chi_Y:X\to DA$.
In terms of these characteristic functions, $f:X\to X'$ represent a morphism $(X,Y)\to(X',Y')$ if there is a $t\in C$ such that $t\rlzs \forall x\oftype X. \chi_Y(x) \to \chi_{Y'}(f(x))$. Moreover, if $Y$ is $U$-stable, there is some $t\in C$ such that $t\rlzs \forall x\oftype X. ((\chi_Y(x)\to U)\to U)\to \chi_Y(x)$.

We will go on using characteristic functions.

Let $n:\pow A\to DA$ satisfy $n(V) = \set{a\in A|\forall b\in V.ab\converges\land ab\in U}$. The set $n(V)$ is downward closed because $U$ is, and because application preserves the ordering. Note that $n(V) = V^\pole$ at the classical side and $V\arrow U$ at the relative side.

If $f:(X,g:X\to \pow A) \to (X',h:X'\to \pow A)$ then $\forall x\oftype X.(\overline{g(x)}\to U \land \overline{h(f(x))}) \to U$ is realized and therefore $\forall x\oftype X.n(g(x)) \to n(h(f(x)))$. This implies $n\circ-$ determines a vertical morphism $\CR(\stacks_A,C,\pole_U) \to \Cl(\ds A/C,U)$. If $f:(X,Y)\to (X,Y')$ then $\forall x\oftype X. \chi_Y(x) \to \chi_{Y'}(f(x))$ holds, and therefore $\forall x\oftype X. (n(\chi_Y(x))\to U) \land n(\chi_{Y'}(f(x))) \to U$ is valid too, showing that $n\circ -$ also determines a vertical morphism $\Cl(\ds A/C,U) \to \CR(\stacks_A,C,\pole_U)$. On both sides, $n\circ n$ is equivalent to $(*\to U)\to U$. Therefore the functors $n\circ -$ are weak inverses of each other, and the triposes are equivalent.
\end{proof}

So relative realizability is a source of models for classical realizability. In the other direction, classical realizability may give more insight to Boolean subtoposes of relative realizability toposes.

\subsection{Lazy partial combinatory algebras}
In the implementation of Krivine's machine in an order partial applicative structure, we never use the fact that the right hand side of the application operator preserves the order. This is why we still get an abstract Krivine structure if we work with the following weaker type of applicative structure.

\begin{definition} A \xemph{lazy partial applicative structure} is an object $A$ with a partial operator $(x,y)\mapsto xy:A^2\partar A$ and a preorder $\leq$, such that if $x\leq x'$ and $x'y\converges$ then $xy\converges$ and $xy\leq x'y$. An arbitrary-ary partial $f:A^k \to A$ is realized if there is an $r\in A$ such that $((rx_1)\dotsm )x_{k-1}\converges$ for all $\vec x\in A^{k-1}$, and $r\vec y\leq f(\vec y)$ for all $\vec y\in\dom f$. A lazy partial applicative structure that has realizers for all partial combinatory functions is a \xemph{lazy partial combinatory algebra}.
\end{definition}

\begin{example} Note that $\CT$ with simple juxtaposition is an example of a lazy partial combinatory algebra. \end{example}

We extend our results on interpreting combinatory logic in partial combinatory algebras. This doesn't require any extra work. 

\begin{definition}
A \xemph{partial interpretation} of combinatory terms, is a partial function $f:\CT \partar A$ that satisfies $f(MN)\converges$ if and only if $f(M)f(N)\converges$ and in that case $f(MN) = f(M)f(N)$ and for all $x,y,z\in A$:
\begin{align}
&f(\comb b)xy\converges & x(yz)\converges \Longrightarrow & f(\comb b)xyz\converges, f(\comb b)xyz \leq x(yz) \\
&f(\comb c)xy\converges & xzy\converges \Longrightarrow & f(\comb c)xyz\converges, f(\comb c)xyz \leq xzy \\ 
&& & f(\comb k)xy\converges, f(\comb k)xy \leq x \\ 
&f(\comb w)x\converges & xyy\converges \Longrightarrow &  f(\comb w)xy\converges, f(\comb b)xyz \leq x(yy)
\end{align}
\end{definition}

\begin{lemma}[generalization of lemma \ref{monotony}] For all terms $M,N\in \CT$ and every partial interpretation $f$, if $M\whrto N$ and $f(N)\converges$, then $f(M)\converges$ and $f(M)\leq f(N)$. \end{lemma}

We can still build realizability triposes with combinatory complete versions of lazy partial applicative structures, and embed the classical realizability tripos into the relative realizability topos. Unfortunately, these models are even wilder than the realizability models we saw up to now, and it is hard to say the appropriate notion of combinatory complete filter actually is.



Example \ref{cpals} in subsection \ref{fcpal} shows how to construct a complete fibred partial applicative lattice out of a complete partial applicative lattice. We still get a complete partial applicative lattice out of a lazy partial applicative structure $A$ with the downset construction.

\begin{definition} The \xemph{algebra of downsets} of $A$ is the set of downsets $DA$ of $A$ together with the inclusion ordering and the following partial operator.
\[ UV\converges \iff \forall x\in U,y\in V.xy\converges\quad UV\converges \Longrightarrow UV = \set{z|\exists x\in U,y\in V. z\leq xy} \]
\end{definition}

\begin{lemma} The set $DA$ is a complete partial applicative lattice. \end{lemma}

\begin{proof} The ordered set $DA$ is a completions under arbitrary joins, and the preservation of joins by application follows from the definition of to application operator as a direct image map followed by a downward closure. Both operations are left adjoints and therefore preserve joins.

By the way application is defined, each partial combinatory function $f:DA^n\partar DA$ satisfies $f(\vec U) = \set{y\in A| \exists \vec x\in\prod\vec U. y\leq g(\vec x)}$ for some partial applicative function $f:A^n\partar A$. The function $g$ has a realizer in $t\in C$ and the downset of $t$ is a member of $C$. Therefore some member of $\chi$ represents $f$.
\end{proof}

\begin{corollary} Let $\ds A$ be the complete fibred partial applicative lattice generated by $A$ and let $\cat C\subseteq \ds A/\cat C$ be a combinatory complete fibred filter closed under indexed joins. The filter quotient $\ds A/\cat C$ is a tripos. \end{corollary}

\begin{proof} See lemma \ref{fcHa} for the first part. The evaluation map $DA^X\times X\to DA$ induces a membership predicate for $X$. \end{proof}

This says it all: we can still get a sound realizability interpretation out of a lazy partial combinatory algebra. We can even go further and drop the requirements that application preserves the ordering in the first variable and that the application operator is single valued.

\begin{remark} The reason we have not written this thesis based on lazy partial applicative structures, or even more general structures is the following. Let $A$ be a lazy partial combinatory algebra and consider any function $f:A\to A$. The set of realizers for $\forall x\oftype.\A(x)\to \A(f(x))$ is $\set{r\in A| \forall x\in A, x'\leq x. rx'\converges \land rx'\leq f(x)}$. With order partial combinatory algebras $rx\leq f(x)$ implies $rx'\converges$ and $rx'\leq f(x)$ for all $x'\leq x$, but with lazy partial combinatory algebras this is no longer valid. For this reason $\A$ may not be a combinatory complete filter for all combinatory complete external filters $\phi$. In particular, $\A$ may not even be closed under application! \end{remark}

\begin{remark}[combinatory bases]
A lazy partial applicative structure $A$ represents all partial applicative functions, if it has the combinators $\comb b$, $\comb c$, $\comb k$ and $\comb w$ of the \xemph{Curry basis}. This follows from fact that we have abstraction operators. It seems that the more popular \xemph{Feferman basis} $\comb k$, $\comb s$ for order partial combinatory algebras is not a basis for lazy partial applicative structures, i.e. we have not found a combination $M$ of $\comb k$ and $\comb s$ that satisfies $Mx\whrto xx$, where $\whrto$ is the weak head reduction ordering. \end{remark}

%% file: Appendix.tex
This is a chapter of concluding remarks, containing a summary of our results and some topics for future research.

\section{Summary}
In chapter \ref{axiom}, we applied the construction of the relative realizability tripos (\cite{MR2479466} p. 106) to arbitrary order partial applicative structures in arbitrary Heyting categories, and generalized it to external filters. Then we showed a handful of properties that single out the resulting realizability fibrations. We analyzed which properties can be expressed in the internal language, and which can't, leading to a `theory of realizability' that is satisfied by any realizability model.

In chapter \ref{realcat} we constructed regular and exact categories from the realizability fibrations, following the tripos-to-topos construction. We considered how the universal property of a realizability fibration translates into universal properties of the resulting categories, and found that realizability categories are reflective inclusions of Heyting categories with order partial applicative structures. We started the analysis of functors between realizability categories, by showing that these categories are pseudoinitial objects of certain 2-categories. This is yet another universal property.

In the last section of chapter \ref{realcat} we analyzed which realizability categories are reg/lex and ex/lex completions, and which are relative completions. This condition puts a restriction on realizability models: external filters must be generated by singletons. This justifies taking an alternative route to the construction of realizability categories.

In chapter \ref{apps}, we looked at the effective topos, and how it generalizes to other base categories than the topos of sets. We started doing a kind of \xemph{synthetic recursive realizability}, i.e. proving properties of effective categories by reasoning from the universal properties rather than from the construction. First we considered general properties, then we spent some time proving the existence of a nontrivial algebraically compact internal subcategory.

The last section of chapter \ref{apps} looks at the connection between relative and classical realizability. Relative realizability models can provide many classical realizability models.

\section{Comparisons}
We compare our work with that of others in the area of realizability and category theory, in particular in the definition and characterization of realizability categories. In general, other people have focused more on generalizing to typed realizability, and less on generalizing to different base categories, so that my work is `orthogonal'. Here we focus on the intersection, and show that the characterizations coincide there.

\subsection{Jonas Frey}
Jonas Frey is a Ph.D. student from Paris, who also works on realizability and category theory. We make a survey of his work here and compare it to our own. Frey does not have any journal publications at the time of writing, so our survey is based on material available on the internet, see:\\ 
\url{www.pps.univ-paris-diderot.fr/~frey/}\\
\url{arxiv.org/abs/1104.2776}\\
\url{lama.univ-savoie.fr/\~{}hyvernat/Realisabilite2012/Files/Frey-slides.pdf}

\begin{definition}[Frey's realizability toposes] 
Frey first introduces the following concepts for general regular $F:\cat C\to\cat D$.
\begin{itemize}
\item An arrow $i:X\to FY$ is \xemph{indecomposable} if for each $j:X\to FZ$ there is a unique $f:Y\to Z$ such that $j=Ff\circ i$. This is the same thing as an initial object in $X/F$ and therefore called an \xemph{initial morphism} elsewhere.
\item As arrow $p:X\to FY$ is \xemph{projective} if for each regular epimorphism $e:W\to Z$ in $\cat C$, each $f:I\to Y$ in $\cat D$, and $g:FI\times_{FY} X \to Z$ in $\cat C$ there is an $h:J\to I$ and a $k:FJ\times_{FY} X \to W$ such that $e\circ k = e\circ (Fh\times_Y \id_X)$.
\[\xymatrix{
FJ\ar@{.>}[d]_{Fh}& \ar@{.>}[l]\ar@{.>}[d]\ar@{}[dl]|<\llcorner \bullet \ar@{.>}[r]^k & W\ar[d]^e \\
FI\ar[d]_{Ff}& \ar[l]\ar[d]\ar@{}[dl]|<\llcorner \bullet \ar[r]_g & Z \\
FY& X\ar[l]^p
}\]

If we apply this definition to the functor $FX = \termo$ in a Cartesian closed regular category, then this results in \xemph{internally projective} objects. See exercise IV.16 in \cite{MR1300636}.

\item Borrowing terminology from the fibred category $\Sub(F-)$ we let a morphism $f:X\to Y$ in $\cat D$ be \xemph{prone}, if it is the pullback of a morphism $Fg:FX'\to FY'$ along a monomorphism $Y\to FY'$. An object $M$ is \xemph{modest} if it is right orthogonal to all prone epimorphisms.

For $\nabla:\cat H\to\Asm(\ds A/\phi)$ this coincides with the objects $M$ for which $!:M\to\termo$ is modest by definition \ref{modest}. For $\nabla:\cat H\to\RT(A,\phi)$ this coincides with \xemph{discrete} in \cite{MR1023803}.
\end{itemize}

By Frey's characterization a realizability topos is:
\begin{itemize}
\item an exact locally Cartesian closed category $\cat X$,
\item whose global sections functor $\Gamma$ has a fully faithful regular right adjoint $\Delta$,
\item and there is a monic $\phi:M \to \Delta A$ such that
\begin{itemize}
\item $\phi$ is \xemph{indecomposable},
\item $\phi$ is \xemph{projective},
\item $M$ is \xemph{modest},
\item the objects with a prone morphism to $M$ are closed under finite limits, and they cover all other objects. Here, prone in the sense defined right above is equivalent to prone relative to $\Gamma$.
\end{itemize}
\end{itemize}
\end{definition}

\begin{lemma} Let $A$ be a partial combinatory algebra over the category of sets. In that case the triple $(\RT(A,A),\nabla,\A)$ satisfies Frey's definition, with $\RT(A,A)$ for $\cat X$, $\nabla$ for $\Delta$ and $\A\to\nabla A$ for $\phi$. \end{lemma}

\begin{proof} The category $\RT(A,A)$ is a topos, and therefore exact and locally Cartesian closed. Because $\termo$ is projective in $\Set$, it is projective in $\RT(A,A)$, so that the global sections functor is regular. It also satisfies $\Gamma\A\simeq A$ because each element of $A$ realizes itself as global section of $\A$, and $\Gamma\nabla X \simeq X$ because there are uniform realizers for all global sections of all sheaves. Since $\Gamma$ and $\supp$ induce the same regular model -- $(\id_\Set,A)$  -- the functors are isomorphic: $\Gamma\simeq \supp$. Hence $\nabla$ is indeed right adjoint to $\Gamma$.

The unit of the adjunction gives an inclusion $\A\to\nabla A$ that is an initial object in $\A/\nabla$, $A$ is projective and $\A$ is a partitioned assembly. So it is an actual projective object, and therefore trivially satisfies Frey's weakening of this condition. And of course $\A$ is a modest object. The last condition of Frey's definition is precisely our weak genericity.
\end{proof}

\begin{lemma} Frey's realizability toposes are realizability toposes. \end{lemma}

\begin{proof} The difficult part is finding a weak partial combinatory algebra. Because $\cat X$ is locally Cartesian closed, we can construct a set $M^*$ of partial morphisms $M\partar M$ with prone domain, i.e. subobjects of $M$ of the form $M\cap\Delta U$, for $U$ in the powerset $\pow \Gamma M$.
\[ M^* = \coprod_{U\in \pow \Gamma M} M^{M\cap \Delta U} \]
This is some exponential in the fibre over $\nabla\pow \Gamma M$, and since right orthogonal objects are closed under these exponentials, $M^*$ is another discrete object. Clearly, $M$ is a weakly generic object for $\Gamma$, so there is an object $Y$ with a regular epimorphism $e:Y\to M^*$ and a prone $p:Y\to M$. This $p$ factors as a prone regular epimorphism follows by a prone monomorphism $m$, and because $M^*$ is \xemph{discrete}, this means $M^*$ is a quotient of a prone subobject of $M$. We can now construct a partial application operator $\alpha:M\times M \partar M$ from the regular-epi-prone-mono span. Now $M$ is a (weak) partial applicative structure \xemph{that represents all partial endomorphisms with prone domain}. Because $\cat X$ is locally Cartesian closed, $M$ is combinatory complete. We may therefore conclude that $(\Gamma M,\Gamma\alpha)$ is an order partial combinatory algebra.

Now we just strike off characteristic properties from the definition before theorem \ref{charexact}.
\begin{itemize}
\item The right adjoint $\Delta$ is regular and fully faithful, and $\Gamma$ preserves finite limits because it is the global section functor. The unit $M\to \Delta\Gamma M$ is monic, and $M$ is a combinatory complete filter.
\item We already mentioned that $M$ is weakly generic.
\item That $M$ is discrete, projective and that it represents all of its endomorphisms implies that Church's rule and the uniformity rule hold.
\end{itemize}
\end{proof}

\begin{theorem}[Frey] Frey's definition captures all realizability toposes for all partial combinatory algebras in $\Set$. \end{theorem}

Frey's ambition is to characterize all categories that come from the tripos-to-topos construction, in the way that Giraud characterized Grothendieck toposes. He starts with a diverse category of indexed posets and works his way down to indexed posets are constructed from partial combinatory algebras.

We have been looking at characterizing realizability toposes too, but our focus was to find out what realizability can do. Since the partial combinatory algebra in a realizability model is given in advance, we have no need to derive it from other structure in the category. On the other hand, I made my construction work in arbitrary Heyting categories, where power objects and the axiom of choice are unavailable, in order to remove properties from realizability categories that aren't universal.

\subsection{Pieter Hofstra}
Pieter Hofstra, one of my predecessors as Ph.D. candidate under Ieke Moerdijk and Jaap van Oosten, worked on \xemph{indexed preorders}, \xemph{ordered partial combinatory algebras} and \xemph{relative completions}. His latest work in the characterization of realizability toposes seems to be \cite{MR2265872}, where he generalizes order partial combinatory algebras by abstracting the set of representable functions.

\begin{definition} A \xemph{(saturated) basic combinatory object} is a partially ordered set $(\Sigma,\leq)$, together with a monoid $\mathcal F_\Sigma$ of partial monotone functions $f:\Sigma\partar \Sigma$ whose domains are downsets, which is upward closed for the following ordering on partial monotone functions $\Sigma\partar \Sigma$: $f\leq g$ if for all $x\in \dom g$, $x\in \dom f$ and $f(x)\leq g(x)$. \end{definition}

In \cite{MR2265872}, Hofstra first defines a non-saturated version, then defines when these structures are equivalent, i.e. when they induce equivalent complete fibred Heyting algebras, and then shows that each basic combinatory object is equivalent to a saturated one.

\comment{
\begin{remark} As long as we are working with sets, or at least with objects of a topos, we can extend a partial ordered set $\Sigma$ with a top-element $\infty$ resulting is a partial order $\Sigma_\infty$. The set of partial monotone functions with downwards closed domain $\Sigma\partar \Sigma$ is isomorphic to the set of $\infty$-preserving monotone functions $\Sigma_\infty\to\Sigma_\infty$. The ordering on partial functions corresponds to the ordering of total functions. 
\end{remark}
}

\begin{example} If $A$ is a partial combinatory algebra and $\phi$ and external filter, then the set $\mathcal F_A$ of partial monotone functions $f:A\partar A$ such that $\db f = \set{a\in A|\forall x\in\dom f. ax\converges\land ax\leq f(x)}\in\phi$ makes the structure a basic combinatory object. \label{opca is bco}
\end{example}

We construct a fibred preorder out of a basic combinatory object by the following construction.

\begin{definition} Let $\cat P(\Sigma,\leq,\mathcal F)$ be the category whose objects are pairs $(X,f:X\to\Sigma)$ and where a morphism $(X,f) \to (Y,g)$ is a function $h:X\to Y$, such that for some $k\in \cat F$, $k(f(x)) \leq g(h(x))$ for all $x\in X$. The functor $P(\Sigma,\leq,\mathcal F): \cat P(\Sigma,\leq,\mathcal F) \to\Set$ is simply the forgetful functor, that sends $(X,f)$ to $X$ and is the identity on morphisms. \end{definition}

This faithful fibration lacks the indexed coproducts and the finite products of a fibred locale, which are problems that Hofstra addresses as follows.

\begin{definition} A \xemph{top element} in a basic combinatory object $(\Sigma,\leq,\mathcal F)$ is an element $\top\in \Sigma$ such that there is an $f\in\mathcal F$ such that for all $x\in\Sigma$, $x\in\dom f$ and $f(x)\leq \top$. A basic combinatory object \xemph{has binary products} if there is a function $p:\Sigma\times \Sigma \to \Sigma$ such that there are $d,p_0,p_1\in \mathcal F$ such that for all $x,y\in \Sigma$, $x\in \dom d$ and, $d(x)\leq p(x,x)$, and $p(x,y)\in \dom p_0\cap \dom p_1$, $p_0(p(x,y))\leq x$ and $p_1(p(x,y))\leq y$. If a basic combinatory object has a top element and binary products, then it \xemph{has finite limits} \end{definition}

\begin{lemma} $P(\Sigma,\leq,\mathcal F)$ has finite limits if and only if $(\Sigma,\leq,\mathcal F)$ does. \end{lemma}

\begin{proof} Hofstra sketches a proof in \cite{MR2265872} under `Indexed finite limits'. \end{proof}

\begin{definition} The \xemph{completion} of $(\Sigma,\leq,\mathcal F)$ is $(D\Sigma, \subseteq, \mathcal F')$ where $D\Sigma$ is the set of downsets, and $f\in \mathcal F^+$ if for some $g\in \mathcal F$, $x\in\dom g$ and $g(x)\in f(\xi)$ for all $x\in \xi \in \dom f$. \end{definition}

\begin{lemma} If $(\Sigma,\leq,\mathcal F)$ has finite limits, then $P(D\Sigma, \subseteq, \mathcal F')$ is a fibred locale. \end{lemma}

\begin{proof} This follows from propositions 4.3 and 4.4 in \cite{MR2265872}. \end{proof}

\begin{theorem}[Hofstra] Let $(\Sigma,\leq,\mathcal F)$ be a basic combinatory object with finite limits. The functor $P(D\Sigma,\leq,\mathcal F^+)$ is a tripos if and only if $(\Sigma,\leq,\mathcal F)$ is constructed from an order partial combinatory algebra and a filter generated by singletons as in \ref{opca is bco}. \end{theorem}

\begin{proof} See theorem 6.9 in \cite{MR2265872}. In the topos of sets, subobject are determined by their set of global sections. Hence there is no distinction between external filters that are generated by singletons and external filters of subsets that intersect some internal filter. \end{proof}

\begin{definition}
The completion construction corresponds to Hofstra's \xemph{relative completion} construction, which we saw in theorem \ref{relcomp}. As we mentioned there, this construction can only produce realizability toposes for external filters that are generated by singletons. We can, however, get all realizability triposes from the inhabited joins completions of order partial combinatory algebras as we saw in subsection \ref{downsetmonad}. In fact, the category $\cat P(\Sigma,\leq,\mathcal F)$ is the category of partitioned assemblies, if $(\Sigma,\leq,\mathcal F)$ comes from an order partial combinatory algebra, and $P(\Sigma,\leq,\mathcal F)$ is the functor we called $\supp$.
\end{definition}

I have not attempted the generality of these \xemph{basic combinatory objects}. Instead, I worked out how to generalize the most useful case, namely that of order partial combinatory algebras with filters, to arbitrary Heyting categories. Breaking the need for filters that are generated by singletons was the first step in the direction of external filters.


\subsection{John Longley}
John Longley also presented a universal property of $\Asm(\ds A/A)$ in his thesis \cite{RTnLS}, where $A$ is a partial combinatory algebra in $\Set$. The category $\Asm(\ds A/ A)$ has a subcategory $\Mod(A)$ of \xemph{modest sets}, because definition \ref{modest} makes sense in any realizability category. Longley shows that $\Asm(\ds A/A)$ sort of is the result of freely adding a fully faithful regular right adjoint to $\supp:\Mod(A) \to \Set$.

A precise formulation follows. We work with a category where the objects are regular functors to $\Set$.

\begin{definition} A \xemph{$\Gamma$-category} is a regular category $\cat C$ with a regular functor $\Gamma_{\cat C}:\cat R\to\Set$. A $\Gamma$-functor $(\cat C,\Gamma_{\cat C})\to(\cat D,\Gamma_{\cat D})$ is a regular functor $F:\cat C\to\cat D$ such that $\Gamma_{\cat D}F\simeq \Gamma_{\cat C}$. The category of $\Gamma$-categories and -functors is $\Gamma\reg$.

A \xemph{$\nabla\Gamma$-category} is a $\Gamma$-category $(\cat C,\Gamma_{\cat C})$ where $\Gamma_{\cat C}$ has a right adjoint $\nabla$ that is fully faithful and regular. A $\nabla\Gamma$-functor is a morphism of $\Gamma$-categories that commutes with the $\nabla$'s. The category of $\nabla\Gamma$-categories and $\nabla\Gamma$-functors is $\nabla\Gamma\reg$.
\end{definition}

\begin{theorem}[Longley] Let $A$ be a partial combinatory algebra in $\Set$. For each $\nabla\Gamma$-category $(\cat C,\Gamma_{\cat C})$, the inclusion $J:\Mod(A)\to\Asm(\ds A/A)$ induces an equivalence of categories:
\[ \Gamma\reg(\Mod(A),(\cat C,\Gamma))\cong \nabla\Gamma\reg(\Asm(\ds A/A),(\cat C,\Gamma_{\cat C}))\]
\end{theorem}

\begin{proof} Note that $\Mod(A)$ is a regular subcategory that contains $\A$.
A $\Gamma$-functor $F: (\Mod(A),\supp) \to (\cat C,\Gamma_\cat C)$ gives us a regular model: $(\nabla_\cat C,F\A)$. This model induces an up to isomorphism unique regular functor $G:\Asm(\ds A/A) \to \cat C$ such that $G\A\simeq F\A$ and $G\nabla\simeq \nabla_{\cat C}$. We use regular models to show that $\Gamma_{\cat C} G\simeq \supp$ too.
\[ \Gamma_{\cat C} G\A\simeq \Gamma_{\cat C} F\A \simeq \supp A \quad \Gamma_{\cat C} G\nabla\simeq \Gamma_{\cat C} \nabla_{\cat C} \simeq \id_\cat C \simeq \supp \nabla \]
The functors $\Gamma_\cat C G$ and $\supp$ induce regular models that are isomorphic too $(\id_\Set, A)$, and hence are isomorphic functors. Last but not least, $GJ\simeq F$, because $G\A\simeq F\A$ and all objects in $\Mod(A)$ are subquotients of $\A$ (see lemma \ref{mod is per}). \end{proof}

We conclude that Longley was just a characterization of $\Mod(A)$ short of giving a universal property of $\Asm(\ds A/A)$ for some partial combinatory algebra $A\in\Set$. 

\comment{
\begin{remark} 
We have not found a definite solution yet, but here is how far we have come. Let $\supp':\Mod(A) \to\Set$ be the restriction of $\supp:\Asm(\ds A/A)\to \Set$ for some partial combinatory algebra $A$. There is a functor $F:\Asm(\ds A/A) \to \Set/\supp$: for each assembly $X$ there is a modest reflection $X_{\rm mod}$, a regular epimorphism $X\to X_{\rm mod}$ and therefore a surjection $FX:\Gamma X \to \Gamma X_{\rm mod}$. On the other hand, there is a functor $G:\Set/\supp \to\Asm(A)$: any morphism $f: X\to \supp Y$ induces a unique prone morphism $Gf\to Y$. We believe that $F$ and $G$ are inverses between subcategory of $\Set/\supp'$ on the surjections and the subcategory $\Asm(\ds A/A)$ on the assemblies for which the unit $X\to X_{\rm mod}$ is a prone morphism. There is an embedding $I:\Mod(A)\to \cat S$ to either of these subcategories, and we believe that $\Asm(\ds A/A)$ is the regular completion of $\cat S$ \xemph{relative} to this embedding. By the way, for this construction to work it is vital that epimorphisms split in $\Set$.

For every regular functor $\Gamma:\cat C\to\Set$ we can construct this subcategory $\cat S(\Gamma)\subseteq\Set/\Gamma$ on the surjections and then the relative regular completion with respect to the canonical embedding $\cat C \to \cat S(\Gamma)$. This seems to deliver the desired $\nabla\Gamma$-completion, but we have not proved this yet.
\end{remark}
}

Longley brings up the question whether each $\Gamma$-category has a `$\nabla\Gamma$-completion', or $\Mod(A)$ has a special property that makes this completion possible. We don't have a definite answer to this question, and this makes it hard to see how it should be generalized to our more general setting. Already when working with an order partial combinatory algebra $A$ whose ordering is not $=$, we run into the problem of how to define $\Mod(A)$: should we consider quotients of all subobjects of $\A$ or only quotients of downward closed subobjects?

We essentially replaced $\Mod(A)$ with the subcategory of prone subobjects of Cartesian powers $\A$ throughout this thesis, and didn't bother to hide this structure in a $\Gamma$-category. Longley's approach is more convenient for working with typed realizability, because the $\Gamma$-categories hide the many types of realizers.

\section{Directions for future research}
The most straightforward direction is subtoposes of realizability toposes. We have seen a lot of research in these areas, and feel that the tools developed in this thesis may shed some light on there too.

We have this in mind: for each topos $\cat E$, each complete partial applicative lattice $L\in\cat E$ and each combinatory complete filter of global sections of $L$, we can construct a tripos $L/\phi$, following example \ref{cpals}. 
We expect that triposes of this form are easier to characterize than realizability triposes. The resulting class of triposes is probably closed under (geometric) subtriposes and filter quotients, and all realizability triposes and all triposes derived from internal locales are of this form.

We have to note that these structures occur as special cases of Hofstra's basic combinatory objects, though, and hence are not completely new. This makes us wonder whether there are any interesting and challenging problems in this direction, and whether we shouldn't look at far more general structures instead, e.g. involving typed realizability.

Another direction for future research is completion constructions. We know why the usual ex/lex completion fails if the base category does not satisfy choice, we know conditions under which the relative completion works, and we know that there are exact completions that preserve a preselected class of regular epimorphisms (see \cite{ECnSS} for this). But maybe there are more subtle completion constructions that also work.

Consider that realizability toposes over $\Set$ are enriched in $\Set$, while realizability categories over another base category $\cat H$ are not always enriched in $\cat H$. If $\cat H$ is a topos there may be an enriched version of realizability toposes that are enriched ex/lex completions -- that is, if such things exist -- of enriched categories of partitioned assemblies. However, we may be overlooking some choice principles that are implied if the effective topos is an enriched ex/lex completion.

We have just started to generalize known results about realizability categories, and to show how to derive them directly from the universal properties; there is still a long way to go in this third direction. Here, we have to consider what role the properties of $\Set$ play in the proofs of properties of realizability toposes. In particular, it is not even clear what modest sets are for an order partial combinatory algebra whose order is not $=$.

For a fourth direction, consider the natural transformation $\eta: \id_{\Asm(\ds A/A)} \to \nabla\supp$: over partitioned assemblies all naturality squares are pullbacks. This is connected to the fact that $\A$ is an elementary substructure of $\nabla A$ for regular logic. Therefore, the realizability construction is a way of adding an elementary substructure to a model of a regular theory. Maybe this generalizes to other regular theories, and maybe we can construct free elementary substructures in first order classical and intuitionistic logic using variations of realizability.

%% file: samenvatting3.tex
\chapter{Samenvatting}

\xemph{Realiseerbaarheid} is een verzameling van technieken voor het bestuderen van \xemph{constructieve logica}. \xemph{Categorie\"entheorie} gaat over verbanden tussen uiteenlopende takken van de wiskunde. Ik zal hier aandacht besteden aan de achtergronden van mijn werk, om vervolgens een korte samenvatting te geven van de resultaten die u op de pagina's hierboven kunt bewonderen.

\section*{Achtergrond}
Wiskundige logica gaat over waarheid in de wiskunde, maar ook over kennis. Om te weten dat een propositie waar is hebben we een bewijs nodig en veel belangrijke resultaten binnen we wiskundige logica, waaronder G\"odels onvolledigheidsstellingen, gaan over \xemph{bewijsbaarheid}.

Binnen de \xemph{constructieve logica} worden er bepaalde beperkingen aan bewijsbaarheid opgelegd ten opzichte van de \xemph{klassieke logica}, die de meeste wiskundigen gebruiken. Klassieke en constructieve wiskundigen geven verschillende betekenissen aan het be\-staan van wiskundige objecten: een klassiek wiskundige accepteert dat er wiskundige objecten bestaan waarvan geen voorbeelden gegeven kunnen worden en de constructieve wiskundige accepteert dat (in veel gevallen) niet. De klassieke logica omvat principes die het bestaan van dergelijke \xemph{niet construeerbare} objecten kan aantonen. Die principes worden in de constructieve logica weggelaten. Het meest opvallende is het \xemph{principe van de uitgesloten derde}, dat zegt dat elke propositie ofwel waar ofwel onwaar is. Dit is een axioma van de klassieke logica maar niet van de constructieve.

Het uitgangspunt van \xemph{realiseerbaarheid} is dat een propositie van de vorm `voor alle $x$ is er een $y$ zodat $x\mathrel R y$' constructief alleen waar kan zijn als er een constructie is om voor iedere $x$ een $y$ te maken zodanig dat `$x\mathrel R y$'. In een \xemph{realiseerbaarheidsmodel} kiezen we een geschikte klasse van parti\"ele  functies $A\rightharpoonup A$ op een verzameling $A$ om de rol van constructie te spelen. Vervolgens defini\"eren we een \xemph{realiseerbaarheidsrelatie} tussen proposities en elementen van $A$, met behulp van deze constructies. Of een propositie \xemph{geldig} is binnen het model hangt weer af van de verzameling van realisatoren die aan iedere propositie wordt toegewezen. 

Het bekendste voorbeeld van realiseerbaarheid, \xemph{recursieve realiseerbaarheid}, staat uitgewerkt in de inleiding van dit proefschrift. Het is gebaseerd op de verzameling van niet negatieve gehele getallen en de parti\"eel recursieve functies. 

Ik heb realiseerbaarheid bestudeerd met behulp van \xemph{categorie\"entheorie}. In plaats van structuren in isolatie te bestuderen, legt categorie\"entheorie de nadruk op de afbeeldingen die structuren met elkaar verbinden. De theorie heeft daarom toepassingen in veel verschillende takken van de wiskunde.

Categorie\"entheorie speelt op drie manieren een rol in realiseerbaarheid. Ten eerste heeft ieder realiseerbaarheidsmodel een categorie van realiseerbare afbeeldingen -- dat zijn de `\xemph{realiseerbaarheidcategorie\"en}' uit de titel. Ten tweede is een realiseerbaarheidsmodel een verbinding tussen twee verschillende werkelijkheden, een klassieke en een constructieve bijvoorbeeld, en daarmee zelf een soort functie. Ten slotte verbinden we verschillende realiseerbaarheidsmodellen in een categorie met elkaar.

\section*{Samenvatting}
Er zijn proposities die in elk realiseerbaarheidsmodel worden gerealiseerd, hoewel ze niet bewijsbaar zijn met constructieve logica. Ook als we het begrip `realiseerbaarheidsmodel' drastisch oprekken en een veel grotere klasse van modellen toelaten, blijven die proposities geldig. Die proposities heb ik in kaart gebracht in het eerste hoofdstuk van mijn proefschrift. Daarna heb ik geanalyseerd in hoeverre ze gebruikt kunnen worden als axioma's van realiseerbaarheid.

De verzameling van functies in een realiseerbaarheidsmodel vormt een \xemph{realiseerbaarheidscategorie}. Deze realiseerbaarheidscategorie\"en zijn met elkaar verbonden door functoren. Het was bekend dat er een verband was tussen \xemph{reguliere functoren} en \xemph{applicatieve morfismes}, waarbij die laatste een soort functies tussen geordende parti\"eel combinatorische algebra's zijn. Het tweede hoofdstuk van mijn proefschrift verklaart waarom dit verband bestaat. We generaliseren het verband tussen applicatieve morfismes en reguliere functoren naar de nieuwe realiseerbaarheidsmodellen van hoofdstuk \'e\'en en tonen vergelijkbare verbanden voor reguliere functoren van realiseerbaarheidscategorie\"en naar willekeurige andere categorie\"en.

In het derde hoofdstuk kijken we naar toepassingen van de theorie die in de eerste twee is ontwikkeld. Het idee is dat bekende eigenschappen van realiseerbaarheidsmodellen veel makkelijker bewezen kunnen worden door ze af te leiden uit de axioma's van hoofdstuk \'e\'en. Dit stuk is noodzakelijkerwijs een samenraapsel van verschillende resultaten. In sectie 3.1 kijken we hoe recursieve realiseerbaarheid eruit ziet voor een constructivist. Sectie 3.2 is een herhaling van mijn masterscriptie. In sectie 3.3 kijken we naar realiseerbaarheidsmodellen van klassieke logica.

Het vierde hoofdstuk rond het geheel af en omvat een vergelijking van mijn werk met dat van anderen op het gebied van realiseerbaarheid en categorie\"entheorie.

\chapter{Dankwoord, Acknowledgements}
Bij mij onderzoek en bij het schrijven van dit proefschrift hebben velen mij geholpen of gesteund. Een aantal van hen wil ik hier bedanken.

Allereerst wil mijn copromotor Jaap van Oosten bedanken. Zonder hem was dit proefschrift niet mogelijk geweest. Ik ben dankbaar voor de vrijheid die ik gehad heb in het kiezen van onderwerpen. De wekelijkse vergaderingen gaven me altijd weer energie om door te gaan. Ten slotte ben ik dankbaar voor alle hulp bij het leesbaar maken van mijn cryptische en spelfoutrijke schrijfwerk.

Ik dank Ieke Moerdijk voor het vertrouwen dat hij in Jaap en mij toont door op te treden als mijn promotor, hoewel hij niet veel tijd heeft gehad om zich met mijn onderzoek te bemoeien.

I thank the members of my assessment committee, Martin Hyland, Bart Jacobs, Giuseppe Rosolini, Thomas Streicher and Benno van der Berg for careful reading my thesis and for providing useful commentary.

Ik dank mijn kamergenote Janne Kool, voor alle gezelligheid en alle praktische hulp die ik heb gekregen; ik zal je missen. Ik bedank ook mijn andere collega's voor de gezamelijke lunches en de koffies in de Gutenberg, het mountainbiken, de uitjes naar Berlijn en Brugge en natuurlijk ook voor alle hulp en gezelligheid. Ik heb veel gelachen met Albert Jan, Arjen, Arthur, Bart, Bas (F.), Bas (J.), Charlene, Dali, Dana, Esther, Ionut, Jaap, Jan Jitse, Jantien, Jan Willem, Jeroen, Job,  Kayin, Lee, Sander, Sasha, Sebastiaan, Sebastian, Tammo Jan, Timo, Vincent en Wilfred.

Ik dank mijn familie, de leden van Utrechts Studenten Koor en Orkest, de Utrecht\-se Studenten Schaats Vereniging ``Softijs'', en de tafelgenoten die ik had bij ``Spek en Bonen'', omdat jullie mij herinnerden aan het leven buiten de wiskunde.

Tot slot bedank ik Saskia en Bram omdat ze mijn paranimfen willen zijn.

\chapter{Curriculum Vitae}
Wouter Stekelenburg werd op 9 juli 1984 in Huizen geboren. In 2002 behaalde hij zijn VWO-diploma aan het Goois Lyceum te Bussum. Hij studeerde wiskunde aan de Universiteit van Utrecht. Daar haalde hij in 2006 zijn bachelor in wiskunde met muziekwetenschappen als bijvak; vervolgens haalde hij in 2008 cum laude zijn master in de wiskunde met een scriptie over \xemph{algebra\"isch compacte categorie\"en in de effectieve topos}; daarna is hij er assistent in opleiding geworden in de onderzoeksgroep van Ieke Moerdijk onder begeleiding van Jaap van Oosten, die ook de masterscriptie begeleid heeft. Een onderzoek naar realizeerbaarheidstopossen heeft geleid tot het proefschrift dat nu voor u ligt.

Buiten de minor muziekwetenschappen heeft Wouters interesse in muziek zich geuit in lidmaatschap van het \emph{Utrechts Studenten Koor en Orkest}\index{USKO} en in deelname aan diverse zangprojecten. Wouter heeft zich voor het USKO ingezet in de PR-commissie, de logistieke commissie en de archiefcommissie.

Wouter is lid geweest van de \emph{Utrechtse Studenten Schaatsvereniging `Softijs'}\index{Softijs}, waar hij 's winters schaatste en 's zomers aan inline-skating deed. Hij is daar lid geweest van de skeelercommissie, die verantwoordelijk is voor de organisatie van de inline-skatetrainingen.